\documentclass[numbers=enddot,12pt,final,onecolumn,notitlepage]{scrartcl}%
\usepackage[headsepline,footsepline,manualmark]{scrlayer-scrpage}
\usepackage[all,cmtip]{xy}
\usepackage{amsfonts}
\usepackage{amssymb}
\usepackage{framed}
\usepackage{amsmath}
\usepackage{comment}
\usepackage{color}
\usepackage[breaklinks=true]{hyperref}
\usepackage[sc]{mathpazo}
\usepackage[T1]{fontenc}
\usepackage{amsthm}
\usepackage{needspace}
\providecommand{\U}[1]{\protect\rule{.1in}{.1in}}
\theoremstyle{definition}
\newtheorem{theo}{Theorem}[section]
\newenvironment{theorem}[1][]
{\begin{theo}[#1]\begin{leftbar}}
{\end{leftbar}\end{theo}}
\newtheorem{lem}[theo]{Lemma}
\newenvironment{lemma}[1][]
{\begin{lem}[#1]\begin{leftbar}}
{\end{leftbar}\end{lem}}
\newtheorem{prop}[theo]{Proposition}
\newenvironment{proposition}[1][]
{\begin{prop}[#1]\begin{leftbar}}
{\end{leftbar}\end{prop}}
\newtheorem{defi}[theo]{Definition}
\newenvironment{definition}[1][]
{\begin{defi}[#1]\begin{leftbar}}
{\end{leftbar}\end{defi}}
\newtheorem{remk}[theo]{Remark}
\newenvironment{remark}[1][]
{\begin{remk}[#1]\begin{leftbar}}
{\end{leftbar}\end{remk}}
\newtheorem{coro}[theo]{Corollary}
\newenvironment{corollary}[1][]
{\begin{coro}[#1]\begin{leftbar}}
{\end{leftbar}\end{coro}}
\newtheorem{conv}[theo]{Convention}
\newenvironment{convention}[1][]
{\begin{conv}[#1]\begin{leftbar}}
{\end{leftbar}\end{conv}}
\newtheorem{quest}[theo]{Question}

\newtheorem{warn}[theo]{Warning}

\newtheorem{conj}[theo]{Conjecture}

\newtheorem{exmp}[theo]{Example}
\newenvironment{example}[1][]
{\begin{exmp}[#1]\begin{leftbar}}
{\end{leftbar}\end{exmp}}
\newenvironment{statement}{\begin{quote}}{\end{quote}}
\newenvironment{verlong}{}{}
\newenvironment{vershort}{}{}
\newenvironment{noncompile}{}{}
\excludecomment{verlong}
\includecomment{vershort}
\excludecomment{noncompile}

\let\sumnonlimits\sum
\let\prodnonlimits\prod
\renewcommand{\sum}{\sumnonlimits\limits}
\renewcommand{\prod}{\prodnonlimits\limits}
\ihead{Integrality over ideal semifiltrations}
\ohead{\today}
\begin{document}

\title{Integrality over ideal semifiltrations}
\author{Darij Grinberg}
\date{
\today
\footnote{updated and improved version of undergraduate work from 2010}}
\maketitle

\begin{abstract}
\textbf{Abstract.} We study integrality over rings (all commutative in this
paper) and over ideal semifiltrations (a generalization of integrality over
ideals). We begin by reproving classical results, such as a version of the
\textquotedblleft faithful module\textquotedblright\ criterion for integrality
over a ring, the transitivity of integrality, and the theorem that sums and
products of integral elements are again integral. Then, we define the notion
of integrality over an ideal semifiltration (a sequence $\left(  I_{0}%
,I_{1},I_{2},\ldots\right)  $ of ideals satisfying $I_{0}=A$ and $I_{a}%
I_{b}\subseteq I_{a+b}$ for all $a,b\in\mathbb{N}$), which generalizes both
integrality over a ring and integrality over an ideal (as considered, e.g., in
Swanson/Huneke \cite{2}). We prove a criterion that reduces this general
notion to integrality over a ring using a variant of the Rees algebra. Using
this criterion, we study this notion further and obtain transitivity and
closedness under sums and products for it as well. Finally, we prove the
curious fact that if $u$, $x$ and $y$ are three elements of a (commutative)
$A$-algebra (for $A$ a ring) such that $u$ is both integral over $A\left[
x\right]  $ and integral over $A\left[  y\right]  $, then $u$ is integral over
$A\left[  xy\right]  $. We generalize this to integrality over ideal
semifiltrations, too.

\end{abstract}
\tableofcontents

\section*{Introduction}

The purpose of this paper is to state (and prove) some theorems and proofs
related to integrality in commutative algebra in somewhat greater generality
than is common in the literature. I claim no novelty, at least not for the
underlying ideas, but I hope that this paper will be useful as a reference (at
least for myself).

Section \ref{sect.1} (Integrality over rings) mainly consists of known facts
(Theorem~\ref{Theorem1}, Theorem~\ref{Theorem4}, Theorem~\ref{Theorem5bc}) and
a generalized exercise from \cite{1} (Corollary~\ref{Corollary3}) with a few
minor variations (Theorem~\ref{Theorem2} and Corollary~\ref{Corollary6}).

Section \ref{sect.2} (Integrality over ideal semifiltrations) merges the
concept of integrality over rings (as considered in Section \ref{sect.1}) and
integrality over ideals (a less popular but still highly useful notion; the
book \cite{2} is devoted to it) into one general notion: that of integrality
over ideal semifiltrations (Definition~\ref{Definition9}). This notion is very
general, yet it can be reduced to the basic notion of integrality over rings
by a suitable change of base ring (Theorem~\ref{Theorem7}). This reduction
allows to extend some standard properties of integrality over rings to the
general case (Theorem~\ref{Theorem8b}, Theorem~\ref{Theorem8c} and
Theorem~\ref{Theorem9}).

Section \ref{sect.3} (Generalizing to two ideal semifiltrations) continues
Section \ref{sect.2}, adding one more layer of generality. Its main results
are a \textquotedblleft relative\textquotedblright\ version of
Theorem~\ref{Theorem7} (Theorem~\ref{Theorem11}) and a known fact generalized
once more (Theorem~\ref{Theorem13}).

Section \ref{sect.4} (Accelerating ideal semifiltrations) generalizes
Theorem~\ref{Theorem11} (and thus also Theorem~\ref{Theorem7}) a bit further
by considering accelerated ideal semifiltrations (a generalization of powers
of an ideal).

Section \ref{sect.5} (On a lemma by Lombardi) is about an auxiliary result
Henri Lombardi used in \cite{3} to prove Kronecker's
Theorem\footnote{\textbf{Kronecker's Theorem.} Let $B$ be a ring
(\textquotedblleft ring\textquotedblright\ always means \textquotedblleft
commutative ring with unity\textquotedblright\ in this paper). Let $g$ and $h$
be two elements of the polynomial ring $B\left[  X\right]  $. Let $g_{\alpha}$
be any coefficient of the polynomial $g$. Let $h_{\beta}$ be any coefficient
of the polynomial $h$. Let $A$ be a subring of $B$ which contains all
coefficients of the polynomial $gh$. Then, the element $g_{\alpha}h_{\beta}$
of $B$ is integral over the subring $A$.}. Here we show a variant of this
result (generalized in one direction, less general in another).

This paper is supposed to be self-contained (only linear algebra and basic
knowledge about rings, modules, ideals and polynomials is assumed).

\begin{vershort}
All proofs given in this paper are constructive (with the exception of the
proof of Lemma~\ref{Lemma18}, which proceeds by contradiction out of sheer
convenience; a constructive version of this proof can be found in
\cite{verlong}).
\end{vershort}

\begin{verlong}
All proofs given in this paper are constructive.
\end{verlong}

\subsection*{Note on the level of detail}

\begin{vershort}
This is the short version of this paper. It constitutes an attempt to balance
clarity and brevity in the proofs. See \cite{verlong} for a more detailed version.
\end{vershort}

\begin{verlong}
This is the long version of this paper, with all proofs maximally detailed.
For all practical purposes, the brief version \cite{vershort} should be
sufficient (and quite possibly easier to read).
\end{verlong}

\subsection*{Note on an old preprint}

This is an updated and somewhat generalized version of my preprint
\textquotedblleft A few facts on integrality\textquotedblright, which is still
available in its old form as well:

\begin{itemize}
\item brief version: \newline\url{https://www.cip.ifi.lmu.de/~grinberg/IntegralityBRIEF.pdf}

\item long version: \newline%
\url{https://www.cip.ifi.lmu.de/~grinberg/Integrality.pdf} .
\end{itemize}

\noindent Be warned that said preprint has been written in 2009--2010 when I
was an undergraduate, and suffers from bad writing and formatting.

\subsection*{Acknowledgments}

I thank Irena Swanson and Marco Fontana for enlightening conversations, and
Irena Swanson in particular for making her book \cite{2} freely available
(which helped me discover the subject as an undergraduate).

\setcounter{section}{-1}

\section{Definitions and notations}

We begin our study of integrality with some classical definitions and
conventions from commutative algebra:

\begin{definition}
\label{Definition1} In the following, \textquotedblleft ring\textquotedblright%
\ will always mean \textquotedblleft commutative ring with
unity\textquotedblright. Furthermore, if $A$ is a ring, then \textquotedblleft%
$A$-algebra\textquotedblright\ shall always mean \textquotedblleft commutative
$A$-algebra with unity\textquotedblright. The unity of a ring $A$ will be
denoted by $1_{A}$ or by $1$ if no confusion can arise.

We denote the set $\left\{  0,1,2,\ldots\right\}  $ by $\mathbb{N}$, and the
set $\left\{  1,2,3,\ldots\right\}  $ by $\mathbb{N}^{+}$.
\end{definition}

\begin{definition}
\label{Definition2} Let $A$ be a ring. Let $M$ be an $A$-module.

If $n\in\mathbb{N}$, and if $m_{1},m_{2},\ldots,m_{n}$ are $n$ elements of
$M$, then we define an $A$-submodule $\left\langle m_{1},m_{2},\ldots
,m_{n}\right\rangle _{A}$ of $M$ by%
\[
\left\langle m_{1},m_{2},\ldots,m_{n}\right\rangle _{A}=\left\{
\sum\limits_{i=1}^{n}a_{i}m_{i}\ \mid\ \left(  a_{1},a_{2},\ldots
,a_{n}\right)  \in A^{n}\right\}  .
\]
This $A$-submodule $\left\langle m_{1},m_{2},\ldots,m_{n}\right\rangle _{A}$
is known as the $A$-submodule of $M$ generated by $m_{1},m_{2},\ldots,m_{n}$
(or as the $A$\textit{-linear span} of $m_{1},m_{2},\ldots,m_{n}$). It
consists of all $A$-linear combinations of $m_{1},m_{2},\ldots,m_{n}$, and in
particular contains all $n$ elements $m_{1},m_{2},\ldots,m_{n}$. Thus, it
satisfies $\left\{  m_{1},m_{2},\ldots,m_{n}\right\}  \subseteq\left\langle
m_{1},m_{2},\ldots,m_{n}\right\rangle _{A}$.

Also, if $S$ is a finite set, and $m_{s}$ is an element of $M$ for every $s\in
S$, then we define an $A$-submodule $\left\langle m_{s}\ \mid\ s\in
S\right\rangle _{A}$ of $M$ by%
\[
\left\langle m_{s}\ \mid\ s\in S\right\rangle _{A}=\left\{  \sum\limits_{s\in
S}a_{s}m_{s}\ \mid\ \left(  a_{s}\right)  _{s\in S}\in A^{S}\right\}  .
\]
This $A$-submodule $\left\langle m_{s}\ \mid\ s\in S\right\rangle _{A}$ is
known as the $A$-submodule of $M$ generated by the family $\left(
m_{s}\right)  _{s\in S}$ (or as the $A$\textit{-linear span} of $\left(
m_{s}\right)  _{s\in S}$). It consists of all $A$-linear combinations of the
elements $m_{s}$ with $s\in S$, and in particular contains all these elements themselves.

Of course, if $m_{1},m_{2},\ldots,m_{n}$ are $n$ elements of $M$, then
\[
\left\langle m_{1},m_{2},\ldots,m_{n}\right\rangle _{A}=\left\langle
m_{s}\ \mid\ s\in\left\{  1,2,\ldots,n\right\}  \right\rangle _{A}.
\]

\end{definition}

Let us observe a trivial fact that we shall use (often tacitly):

\begin{lemma}
\label{lem.mil} Let $A$ be a ring. Let $M$ be an $A$-module. Let $N$ be an
$A$-submodule of $M$. Let $S$ be a finite set; let $m_{s}$ be an element of
$N$ for every $s\in S$. Then, $\left\langle m_{s}\ \mid\ s\in S\right\rangle
_{A}\subseteq N$.
\end{lemma}

\begin{verlong}
\begin{proof}
[Proof of Lemma \ref{lem.mil}.]We have $m_{s}\in N$ for every $s\in S$. Thus,
$\sum\limits_{s\in S}a_{s}m_{s}\in N$ for every $\left(  a_{s}\right)  _{s\in
S}\in A^{S}$ (since $N$ is an $A$-submodule of $M$, and thus is closed under
$A$-linear combination). But the definition of $\left\langle m_{s}\ \mid\ s\in
S\right\rangle _{A}$ yields
\[
\left\langle m_{s}\ \mid\ s\in S\right\rangle _{A}=\left\{  \sum\limits_{s\in
S}a_{s}m_{s}\ \mid\ \left(  a_{s}\right)  _{s\in S}\in A^{S}\right\}
\subseteq N
\]
(since $\sum\limits_{s\in S}a_{s}m_{s}\in N$ for every $\left(  a_{s}\right)
_{s\in S}\in A^{S}$). This proves Lemma \ref{lem.mil}.
\end{proof}
\end{verlong}

\begin{definition}
\label{Definition3} Let $A$ be a ring, and let $n\in\mathbb{N}$. Let $M$ be an
$A$-module. We say that the $A$-module $M$ is $n$\textit{-generated} if there
exist $n$ elements $m_{1},m_{2},\ldots,m_{n}$ of $M$ such that $M=\left\langle
m_{1},m_{2},\ldots,m_{n}\right\rangle _{A}$. In other words, the $A$-module
$M$ is $n$-generated if and only if there exists a set $S$ and an element
$m_{s}$ of $M$ for every $s\in S$ such that $\left\vert S\right\vert =n$ and
$M=\left\langle m_{s}\ \mid\ s\in S\right\rangle _{A}$.
\end{definition}

\begin{verlong}
We shall use the standard basic properties of submodules of algebras, such as
the following:

\begin{proposition}
Let $A$ be a ring. Let $B$ be an $A$-algebra.

\textbf{(a)} For any two $A$-submodules $U$ and $V$ of $B$, we let $U\cdot V$
denote the $A$-submodule of $B$ spanned by all products of the form $uv$ with
$\left(  u,v\right)  \in U\times V$. This $A$-submodule $U\cdot V$ is also
denoted by $UV$. Thus we have defined a binary operation $\cdot$ on the set of
all $A$-submodules of $B$. Equipped with this operation, the set of all
$A$-submodules of $B$ becomes an abelian monoid, with neutral element
$A\cdot1_{B}$.

This all applies, in particular, to the case when $B=A$; in this case, the
$A$-submodules of $B$ are the same as the ideals of $A$. Thus, the set of all
ideals of $A$ becomes an abelian monoid, with neutral element $A\cdot1_{A}=A$.

Likewise, we can define $U\cdot V$ when $U$ is an ideal of $A$ while $V$ is an
$A$-submodule of $B$. These \textquotedblleft product\textquotedblright%
\ operations satisfy the rules one would expect, such as%
\[
U\left(  V+W\right)  =UV+UW;\ \ \ \ \ \ \ \ \ \ \left(  U+V\right)
W=UW+VW;\ \ \ \ \ \ \ \ \ \ \left(  UV\right)  W=U\left(  VW\right)
\]
(whenever these expressions make sense).

\textbf{(b)} Let $S$ be a finite set. Let $m_{s}$ be an element of $B$ for
each $s\in S$. Then, for any $b\in B$, we have%
\[
b\cdot\left\langle m_{s}\ \mid\ s\in S\right\rangle _{A}=\left\langle
bm_{s}\ \mid\ s\in S\right\rangle _{A}.
\]

\textbf{(c)} Let $S$ be a finite set. Let $m_{s}$ be an element of $B$ for
each $s\in S$. Let $T$ be a finite set. Let $n_{t}$ be an element of $B$ for
each $t\in T$. Then,%
\[
\left\langle m_{s}\ \mid\ s\in S\right\rangle _{A}\cdot\left\langle
n_{t}\ \mid\ t\in T\right\rangle _{A}=\left\langle m_{s}n_{t}\ \mid\ \left(
s,t\right)  \in S\times T\right\rangle _{A}.
\]

\end{proposition}
\end{verlong}

\begin{definition}
\label{Definition4} Let $A$ be a ring. Let $B$ be an $A$-algebra. (Let us
recall that both rings and algebras are always understood to be commutative
and unital in this paper.)

If $u_{1},u_{2},\ldots,u_{n}$ are $n$ elements of $B$, then we define an
$A$-subalgebra $A\left[  u_{1},u_{2},\ldots,u_{n}\right]  $ of $B$ by%
\[
A\left[  u_{1},u_{2},\ldots,u_{n}\right]  =\left\{  P\left(  u_{1}%
,u_{2},\ldots,u_{n}\right)  \ \mid\ P\in A\left[  X_{1},X_{2},\ldots
,X_{n}\right]  \right\}
\]
(where $A\left[  X_{1},X_{2},\ldots,X_{n}\right]  $ denotes the polynomial
ring in $n$ indeterminates $X_{1},X_{2},\ldots,X_{n}$ over $A$).

In particular, if $u$ is an element of $B$, then the $A$-subalgebra $A\left[
u\right]  $ of $B$ is defined by%
\[
A\left[  u\right]  =\left\{  P\left(  u\right)  \ \mid\ P\in A\left[
X\right]  \right\}
\]
(where $A\left[  X\right]  $ denotes the polynomial ring in a single
indeterminate $X$ over $A$). Since
\[
A\left[  X\right]  =\left\{  \sum\limits_{i=0}^{m}a_{i}X^{i}\ \mid
\ m\in\mathbb{N}\text{ and }\left(  a_{0},a_{1},\ldots,a_{m}\right)  \in
A^{m+1}\right\}  ,
\]
this becomes
\begin{align*}
A\left[  u\right]   &  =\left\{  \left(  \sum\limits_{i=0}^{m}a_{i}%
X^{i}\right)  \left(  u\right)  \ \mid\ m\in\mathbb{N}\text{ and }\left(
a_{0},a_{1},\ldots,a_{m}\right)  \in A^{m+1}\right\} \\
&  \ \ \ \ \ \ \ \ \ \ \left(
\begin{array}
[c]{c}%
\text{where }\left(  \sum\limits_{i=0}^{m}a_{i}X^{i}\right)  \left(  u\right)
\text{ means the}\\
\text{polynomial }\sum\limits_{i=0}^{m}a_{i}X^{i}\text{ evaluated at }X=u
\end{array}
\right) \\
&  =\left\{  \sum\limits_{i=0}^{m}a_{i}u^{i}\ \mid\ m\in\mathbb{N}\text{ and
}\left(  a_{0},a_{1},\ldots,a_{m}\right)  \in A^{m+1}\right\} \\
&  \ \ \ \ \ \ \ \ \ \ \left(  \text{because }\left(  \sum\limits_{i=0}%
^{m}a_{i}X^{i}\right)  \left(  u\right)  =\sum\limits_{i=0}^{m}a_{i}%
u^{i}\right)  .
\end{align*}
Obviously, $uA\left[  u\right]  \subseteq A\left[  u\right]  $ (since
$A\left[  u\right]  $ is an $A$-algebra and $u\in A\left[  u\right]  $).
\end{definition}

\begin{definition}
\label{def.superring-as-alg}Let $B$ be a ring, and let $A$ be a subring of
$B$. Then, $B$ canonically becomes an $A$-algebra. The $A$-module structure of
this $A$-algebra $B$ is given by multiplication inside $B$.
\end{definition}

Definition \ref{def.superring-as-alg} shows that theorems about $A$-algebras
(for a ring $A$) are always more general than theorems about rings that
contain $A$ as a subring. Hence, we shall study $A$-algebras in the following,
even though most of the applications of the results we shall see are found at
the level of rings containing $A$.

\section{\label{sect.1}Integrality over rings}

\subsection{The fundamental equivalence}

Most of the theory of integrality is based upon the following result:

\begin{theorem}
\label{Theorem1} Let $A$ be a ring. Let $B$ be an $A$-algebra. Thus, $B$ is
canonically an $A$-module. Let $n\in\mathbb{N}$. Let $u\in B$. Then, the
following four assertions $\mathcal{A}$, $\mathcal{B}$, $\mathcal{C}$ and
$\mathcal{D}$ are equivalent:

\begin{itemize}
\item \textit{Assertion }$\mathcal{A}$\textit{:} There exists a monic
polynomial $P\in A\left[  X\right]  $ with $\deg P=n$ and $P\left(  u\right)
=0$.

\item \textit{Assertion }$\mathcal{B}$\textit{:} There exist a $B$-module $C$
and an $n$-generated $A$-submodule $U$ of $C$ such that $uU\subseteq U$ and
such that every $v\in B$ satisfying $vU=0$ satisfies $v=0$. (Here, $C$ is an
$A$-module, since $C$ is a $B$-module and $B$ is an $A$-algebra.)

\item \textit{Assertion }$\mathcal{C}$\textit{:} There exists an $n$-generated
$A$-submodule $U$ of $B$ such that $1\in U$ and $uU\subseteq U$. (Here and in
the following, \textquotedblleft$1$\textquotedblright\ means \textquotedblleft%
$1_{B}$\textquotedblright, that is, the unity of the ring $B$.)

\item \textit{Assertion }$\mathcal{D}$\textit{:} We have $A\left[  u\right]
=\left\langle u^{0},u^{1},\ldots,u^{n-1}\right\rangle _{A}$.
\end{itemize}
\end{theorem}

We shall soon prove this theorem; first, let us explain what it is for:

\begin{definition}
\label{Definition5} Let $A$ be a ring. Let $B$ be an $A$-algebra. Let
$n\in\mathbb{N}$. Let $u\in B$. We say that the element $u$ of $B$ is
$n$\textit{-integral over }$A$ if it satisfies the four equivalent assertions
$\mathcal{A}$, $\mathcal{B}$, $\mathcal{C}$ and $\mathcal{D}$ of
Theorem~\ref{Theorem1}.

Hence, in particular, the element $u$ of $B$ is $n$-integral over $A$ if and
only if it satisfies the assertion $\mathcal{A}$ of Theorem~\ref{Theorem1}. In
other words, $u$ is $n$-integral over $A$ if and only if there exists a monic
polynomial $P\in A\left[  X\right]  $ with $\deg P=n$ and $P\left(  u\right)
=0$.
\end{definition}

The notion of \textquotedblleft$n$-integral\textquotedblright\ elements that
we have just defined is a refinement of the classical notion of integrality of
elements over rings (see, e.g., \cite[Definition (10.21)]{AllKle14} or
\cite[Chapter V, \S 1.1, Definition 1]{Bourba72} or \cite[Definition
8.1.1]{ChaLoi14} for this classical notion, and \cite[Definition 2.1.1]{2} for
its particular case when $A$ is a subring of $B$). Indeed, the classical
notion defines an element $u$ of $B$ to be \textit{integral} over $A$ if and
only if (using the language of our Definition~\ref{Definition5}) there exists
some $n\in\mathbb{N}$ such that $u$ is $n$-integral over $A$. Since I believe
the concrete value of $n$ to be worth more than its mere existence, I prefer
the specificity of the \textquotedblleft$n$-integral\textquotedblright%
\ concept to the slickness of \textquotedblleft integral\textquotedblright.

Theorem~\ref{Theorem1} is one of several similar results providing equivalent
criteria for the integrality of an element of an $A$-algebra. See
\cite[Proposition (10.23)]{AllKle14}, \cite[Chapter V, Section 1.1, Theorem
1]{Bourba72} or \cite[Theorem 8.1.6]{ChaLoi14} for other such results (some
very close to Theorem~\ref{Theorem1}, and all proven in similar ways).

Before we prove Theorem~\ref{Theorem1}, let us recall a classical property of matrices:

\begin{lemma}
\label{lem.adj} Let $B$ be a ring. Let $n\in\mathbb{N}$. Let $M$ be an
$n\times n$-matrix over $B$. Then,%
\[
\det M\cdot I_{n}=\operatorname*{adj}M\cdot M.
\]
(Here, $I_{n}$ means the $n\times n$ identity matrix and $\operatorname*{adj}%
M$ denotes the adjugate of the matrix $M$. The expressions \textquotedblleft%
$\det M\cdot I_{n}$\textquotedblright\ and \textquotedblleft%
$\operatorname*{adj}M\cdot M$\textquotedblright\ have to be understood as
\textquotedblleft$\left(  \det M\right)  \cdot I_{n}$\textquotedblright\ and
\textquotedblleft$\left(  \operatorname*{adj}M\right)  \cdot M$%
\textquotedblright, respectively.)
\end{lemma}

Lemma~\ref{lem.adj} is well-known (for example, it follows from \cite[Theorem
6.100]{detnotes}, applied to $\mathbb{K}=B$ and $A=M$).

\begin{proof}
[Proof of Theorem~\ref{Theorem1}.]We will prove the implications
$\mathcal{A}\Longrightarrow\mathcal{C}$, $\mathcal{C}\Longrightarrow
\mathcal{B}$, $\mathcal{B}\Longrightarrow\mathcal{A}$, $\mathcal{A}%
\Longrightarrow\mathcal{D}$ and $\mathcal{D}\Longrightarrow\mathcal{C}$.

\textit{Proof of the implication }$\mathcal{A}\Longrightarrow\mathcal{C}%
$\textit{.} Assume that Assertion $\mathcal{A}$ holds. Then, there exists a
monic polynomial $P\in A\left[  X\right]  $ with $\deg P=n$ and $P\left(
u\right)  =0$. Consider this $P$. Since $P\in A\left[  X\right]  $ is a monic
polynomial with $\deg P=n$, there exist elements $a_{0},a_{1},\ldots,a_{n-1}$
of $A$ such that $P\left(  X\right)  =X^{n}+\sum\limits_{k=0}^{n-1}a_{k}X^{k}%
$. Consider these $a_{0},a_{1},\ldots,a_{n-1}$. Substituting $u$ for $X$ in
the equality $P\left(  X\right)  =X^{n}+\sum\limits_{k=0}^{n-1}a_{k}X^{k}$, we
find $P\left(  u\right)  =u^{n}+\sum\limits_{k=0}^{n-1}a_{k}u^{k}$. Hence, the
equality $P\left(  u\right)  =0$ (which holds by definition of $P$) rewrites
as $u^{n}+\sum\limits_{k=0}^{n-1}a_{k}u^{k}=0$. Hence, $u^{n}=-\sum
\limits_{k=0}^{n-1}a_{k}u^{k}$.

Let $U$ be the $A$-submodule $\left\langle u^{0},u^{1},\ldots,u^{n-1}%
\right\rangle _{A}$ of $B$. Then, $U=\left\langle u^{0},u^{1},\ldots
,u^{n-1}\right\rangle _{A}$ and
\[
u^{n}=-\sum\limits_{k=0}^{n-1}a_{k}u^{k}\in\left\langle u^{0},u^{1}%
,\ldots,u^{n-1}\right\rangle _{A}=U.
\]
Moreover, the $n$ elements $u^{0},u^{1},\ldots,u^{n-1}$ belong to $U$ (since
$U=\left\langle u^{0},u^{1},\ldots,u^{n-1}\right\rangle _{A}$). In other
words,%
\begin{equation}
u^{i}\in U\ \ \ \ \ \ \ \ \ \ \text{for each }i\in\left\{  0,1,\ldots
,n-1\right\}  . \label{pf.Theorem1.AC.1}%
\end{equation}
This relation also holds for $i=n$ (since $u^{n}\in U$); thus, it holds for
all $i\in\left\{  0,1,\ldots,n\right\}  $. In other words, we have%
\begin{equation}
u^{i}\in U\ \ \ \ \ \ \ \ \ \ \text{for each }i\in\left\{  0,1,\ldots
,n\right\}  . \label{pf.Theorem1.AC.2}%
\end{equation}

\begin{vershort}
\noindent Applying this to $i=0$, we find $u^{0}\in U$. Thus, $1=u^{0}\in U$.
\end{vershort}

\begin{verlong}
\noindent Applying this to $i=0$, we find $u^{0}\in U$ (since $0\in\left\{
0,1,\ldots,n\right\}  $). This rewrites as $1\in U$ (since $u^{0}=1$).
\end{verlong}

Recall that $U=\left\langle u^{0},u^{1},\ldots,u^{n-1}\right\rangle _{A}$.
Hence, $U$ is an $n$-generated $A$-module (since $u^{0},u^{1},\ldots,u^{n-1}$
are $n$ elements of $U$).

Now, for each $s\in\left\{  0,1,\ldots,n-1\right\}  $, we have $s+1\in\left\{
1,2,\ldots,n\right\}  \subseteq\left\{  0,1,\ldots,n\right\}  $ and thus
$u^{s+1}\in U$ (by (\ref{pf.Theorem1.AC.2}), applied to $i=s+1$). Hence,
Lemma~\ref{lem.mil} (applied to $M=B$, $N=U$, $S=\left\{  0,1,\ldots
,n-1\right\}  $ and $m_{s}=u^{s+1}$) yields%
\[
\left\langle u^{s+1}\ \mid\ s\in\left\{  0,1,\ldots,n-1\right\}  \right\rangle
_{A}\subseteq U.
\]
Now, from $U=\left\langle u^{0},u^{1},\ldots,u^{n-1}\right\rangle _{A}$, we
obtain%
\begin{align*}
uU  &  =u\left\langle u^{0},u^{1},\ldots,u^{n-1}\right\rangle _{A}%
=\left\langle u\cdot u^{0},u\cdot u^{1},\ldots,u\cdot u^{n-1}\right\rangle
_{A}\\
&  =\left\langle \underbrace{u\cdot u^{s}}_{=u^{s+1}}\ \mid\ s\in\left\{
0,1,\ldots,n-1\right\}  \right\rangle _{A}=\left\langle u^{s+1}\ \mid
\ s\in\left\{  0,1,\ldots,n-1\right\}  \right\rangle _{A}\subseteq U.
\end{align*}

Thus, we have found an $n$-generated $A$-submodule $U$ of $B$ such that $1\in
U$ and $uU\subseteq U$. Hence, Assertion $\mathcal{C}$ holds. Hence, we have
proved that $\mathcal{A}\Longrightarrow\mathcal{C}$.

\textit{Proof of the implication }$\mathcal{C}\Longrightarrow\mathcal{B}%
$\textit{.} Assume that Assertion $\mathcal{C}$ holds. Then, there exists an
$n$-generated $A$-submodule $U$ of $B$ such that $1\in U$ and $uU\subseteq U$.
Consider this $U$. Every $v\in B$ satisfying $vU=0$ satisfies $v=0$ (since
$1\in U$ and $vU=0$ yield $v\cdot\underbrace{1}_{\in U}\in vU=0$ and thus
$v\cdot1=0$, so that $v=0$). Set $C=B$. Then, $C$ is a $B$-module, and $U$ is
an $n$-generated $A$-submodule of $C$ (since $U$ is an $n$-generated
$A$-submodule of $B$, and $C=B$) such that $uU\subseteq U$ and such that every
$v\in B$ satisfying $vU=0$ satisfies $v=0$. Thus, Assertion $\mathcal{B}$
holds. Hence, we have proved that $\mathcal{C}\Longrightarrow\mathcal{B}$.

\textit{Proof of the implication }$\mathcal{B}\Longrightarrow\mathcal{A}%
$\textit{.} Assume that Assertion $\mathcal{B}$ holds. Then, there exist a
$B$-module $C$ and an $n$-generated $A$-submodule\footnote{where $C$ is an
$A$-module, since $C$ is a $B$-module and $B$ is an $A$-algebra} $U$ of $C$
such that $uU\subseteq U$, and such that every $v\in B$ satisfying $vU=0$
satisfies $v=0$. Consider these $C$ and $U$.

The $A$-module $U$ is $n$-generated. In other words, there exist $n$ elements
$m_{1},m_{2},\ldots,m_{n}$ of $U$ such that $U=\left\langle m_{1},m_{2}%
,\ldots,m_{n}\right\rangle _{A}$. Consider these $m_{1},m_{2},\ldots,m_{n}$.
For any $k\in\left\{  1,2,\ldots,n\right\}  $, we have $m_{k}\in U$ (since
$U=\left\langle m_{1},m_{2},\ldots,m_{n}\right\rangle _{A}$) and thus%
\[
um_{k}\in uU\subseteq U=\left\langle m_{1},m_{2},\ldots,m_{n}\right\rangle
_{A},
\]
so that there exist $n$ elements $a_{k,1},a_{k,2},\ldots,a_{k,n}$ of $A$ such
that
\begin{equation}
um_{k}=\sum\limits_{i=1}^{n}a_{k,i}m_{i}. \label{pf.Theorem1.BA.defaki}%
\end{equation}
Consider these $a_{k,1},a_{k,2},\ldots,a_{k,n}$.

The $A$-algebra $B$ gives rise to a canonical ring homomorphism $\iota
:A\rightarrow B$ (sending each $a\in A$ to $a\cdot1_{B}\in B$). This ring
homomorphism, in turn, induces a ring homomorphism $\iota^{n\times
n}:A^{n\times n}\rightarrow B^{n\times n}$ (which acts on an $n\times
n$-matrix by applying $\iota$ to each entry of the matrix).

\begin{vershort}
We are also going to work with matrices over $U$ (that is, matrices whose
entries lie in $U$). This might sound somewhat strange, because $U$ is not a
ring; however, we can still define matrices over $U$ just as one defines
matrices over any ring. While we cannot multiply two matrices over $U$
(because $U$ is not a ring), we can define the product of a matrix over $A$
with a matrix over $U$ as follows: If $P\in A^{\alpha\times\beta}$ is a matrix
over $A$, and $Q\in U^{\beta\times\gamma}$ is a matrix over $U$ (where
$\alpha,\beta,\gamma\in\mathbb{N}$), then we define the product $PQ\in
U^{\alpha\times\gamma}$ by setting%
\[
\left(  PQ\right)  _{x,y}=\sum_{z=1}^{\beta}P_{x,z}Q_{z,y}%
\ \ \ \ \ \ \ \ \ \ \text{for all }x\in\left\{  1,2,\ldots,\alpha\right\}
\text{ and }y\in\left\{  1,2,\ldots,\gamma\right\}  .
\]
(Here, for any matrix $T$ and any integers $x$ and $y$, we denote by $T_{x,y}$
the entry of the matrix $T$ in the $x$-th row and the $y$-th column.)

It is easy to see that whenever $P\in A^{\alpha\times\beta}$, $Q\in
A^{\beta\times\gamma}$ and $R\in U^{\gamma\times\delta}$ are three matrices,
then
\begin{equation}
\left(  PQ\right)  R=P\left(  QR\right)  . \label{(PQ)R=P(QR)}%
\end{equation}
This is proven in the same way as the fact that the multiplication of matrices
over a ring is associative.

Now define a matrix $V\in U^{n\times1}$ by setting
\[
V_{i,1}=m_{i}\ \ \ \ \ \ \ \ \ \ \text{for all }i\in\left\{  1,2,\ldots
,n\right\}  .
\]

Define another matrix $S\in A^{n\times n}$ by setting
\[
S_{k,i}=a_{k,i}\ \ \ \ \ \ \ \ \ \ \text{for all }k\in\left\{  1,2,\ldots
,n\right\}  \text{ and }i\in\left\{  1,2,\ldots,n\right\}  .
\]

Then, for any $k\in\left\{  1,2,\ldots,n\right\}  $, we have
\begin{align*}
\left(  uV\right)  _{k,1}  &  =u\underbrace{V_{k,1}}_{\substack{=m_{k}%
\\\text{(by the definition of }V\text{)}}}=um_{k}%
\ \ \ \ \ \ \ \ \ \ \text{and}\\
\left(  SV\right)  _{k,1}  &  =\sum\limits_{i=1}^{n}\underbrace{S_{k,i}%
}_{\substack{=a_{k,i}\\\text{(by the definition of }S\text{)}}%
}\underbrace{V_{i,1}}_{\substack{=m_{i}\\\text{(by the definition of
}V\text{)}}}=\sum\limits_{i=1}^{n}a_{k,i}m_{i}.
\end{align*}
Hence, the equality (\ref{pf.Theorem1.BA.defaki}) rewrites as $\left(
uV\right)  _{k,1}=\left(  SV\right)  _{k,1}$. Since this holds for every
$k\in\left\{  1,2,\ldots,n\right\}  $, we conclude that $uV=SV$. Thus,%
\begin{equation}
0=uV-SV=uI_{n}V-SV=\left(  uI_{n}-S\right)  V.
\label{pf.Theorem1.BA.short.matprod=0}%
\end{equation}
Here, the \textquotedblleft$S$\textquotedblright\ in \textquotedblleft%
$uI_{n}-S$\textquotedblright\ means not the matrix $S\in A^{n\times n}$
itself, but rather its image under the ring homomorphism $\iota^{n\times
n}:A^{n\times n}\rightarrow B^{n\times n}$; thus, the matrix $uI_{n}-S$ is a
well-defined matrix in $B^{n\times n}$.

Now, let $P\in A\left[  X\right]  $ be the characteristic polynomial of the
matrix $S\in A^{n\times n}$. Then, $P$ is monic, and $\deg P=n$. Besides,
$P\left(  X\right)  =\det\left(  XI_{n}-S\right)  $, so that $P\left(
u\right)  =\det\left(  uI_{n}-S\right)  $. Thus,%
\begin{align*}
P\left(  u\right)  \cdot V  &  =\det\left(  uI_{n}-S\right)  \cdot
V=\underbrace{\det\left(  uI_{n}-S\right)  I_{n}}%
_{\substack{=\operatorname{adj}\left(  uI_{n}-S\right)  \cdot\left(
uI_{n}-S\right)  \\\text{(by Lemma \ref{lem.adj},}\\\text{applied to }%
M=uI_{n}-S\text{)}}}\cdot V\\
&  =\left(  \operatorname{adj}\left(  uI_{n}-S\right)  \cdot\left(
uI_{n}-S\right)  \right)  \cdot V\\
&  =\operatorname{adj}\left(  uI_{n}-S\right)  \cdot\underbrace{\left(
\left(  uI_{n}-S\right)  V\right)  }_{\substack{=0\\\text{(by
\eqref{pf.Theorem1.BA.short.matprod=0})}}}\ \ \ \ \ \ \ \ \ \ \left(  \text{by
\eqref{(PQ)R=P(QR)}}\right) \\
&  =0.
\end{align*}
Since the entries of the matrix $V$ are $m_{1},m_{2},\ldots,m_{n}$, this
yields $P\left(  u\right)  \cdot m_{k}=0$ for every $k\in\left\{
1,2,\ldots,n\right\}  $. Now, from $U=\left\langle m_{1},m_{2},\ldots
,m_{n}\right\rangle _{A}$, we obtain%
\begin{align*}
P\left(  u\right)  \cdot U  &  =P\left(  u\right)  \cdot\left\langle
m_{1},m_{2},\ldots,m_{n}\right\rangle _{A}=\left\langle P\left(  u\right)
\cdot m_{1},P\left(  u\right)  \cdot m_{2},\ldots,P\left(  u\right)  \cdot
m_{n}\right\rangle _{A}\\
&  =\left\langle 0,0,\ldots,0\right\rangle _{A}\ \ \ \ \ \ \ \ \ \ \left(
\text{since }P\left(  u\right)  \cdot m_{k}=0\text{ for any }k\in\left\{
1,2,\ldots,n\right\}  \right) \\
&  =0.
\end{align*}
This implies $P\left(  u\right)  =0$ (since every $v\in B$ satisfying $vU=0$
satisfies $v=0$). Thus, Assertion $\mathcal{A}$ holds. Hence, we have proved
that $\mathcal{B}\Longrightarrow\mathcal{A}$.
\end{vershort}

\begin{verlong}
We introduce two notations:

\begin{itemize}
\item For any matrix $T$ and any integers $x$ and $y$, we denote by $T_{x,y}$
the entry of the matrix $T$ in the $x$-th row and the $y$-th column.

\item For any assertion $\mathcal{U}$, we denote by $\left[  \mathcal{U}%
\right]  $ the Boolean value of the assertion $\mathcal{U}$ (that is, $\left[
\mathcal{U}\right]  =%
\begin{cases}
1, & \text{if }\mathcal{U}\text{ is true;}\\
0, & \text{if }\mathcal{U}\text{ is false}%
\end{cases}
$). This value $\left[  \mathcal{U}\right]  $ is an element of $\left\{
0,1\right\}  $ and is also known as the \textit{truth value} of $\mathcal{U}$.
\end{itemize}

Clearly, the $n\times n$ identity matrix $I_{n}$ satisfies
\[
\left(  I_{n}\right)  _{k,i}=\left[  k=i\right]  \ \ \ \ \ \ \ \ \ \ \text{for
every }k\in\left\{  1,2,\ldots,n\right\}  \text{ and }i\in\left\{
1,2,\ldots,n\right\}  .
\]

Note that for every $k\in\left\{  1,2,\ldots,n\right\}  $, we have%
\begin{equation}
m_{k}=\sum\limits_{i=1}^{n}\left(  I_{n}\right)  _{k,i}m_{i}, \label{1}%
\end{equation}
since%
\begin{align*}
\underbrace{\sum\limits_{i=1}^{n}}_{=\sum\limits_{i\in\left\{  1,2,\ldots
,n\right\}  }}\underbrace{\left(  I_{n}\right)  _{k,i}}_{\substack{=\left[
k=i\right]  \\=\left[  i=k\right]  }}m_{i}  &  =\sum\limits_{i\in\left\{
1,2,\ldots,n\right\}  }\left[  i=k\right]  m_{i}\\
&  =\sum\limits_{\substack{i\in\left\{  1,2,\ldots,n\right\}  \\\text{such
that }i=k}}\underbrace{\left[  i=k\right]  }_{\substack{=1\\\text{(since
}i=k\text{ is true)}}}m_{i}+\sum\limits_{\substack{i\in\left\{  1,2,\ldots
,n\right\}  \\\text{such that }i\neq k}}\underbrace{\left[  i=k\right]
}_{\substack{=0\\\text{(since }i=k\text{ is false}\\\text{(since }i\neq
k\text{))}}}m_{i}\\
&  =\sum\limits_{\substack{i\in\left\{  1,2,\ldots,n\right\}  \\\text{such
that }i=k}}\underbrace{1m_{i}}_{=m_{i}}+\underbrace{\sum
\limits_{\substack{i\in\left\{  1,2,\ldots,n\right\}  \\\text{such that }i\neq
k}}0m_{i}}_{=0}=\sum\limits_{\substack{i\in\left\{  1,2,\ldots,n\right\}
\\\text{such that }i=k}}m_{i}+0\\
&  =\sum\limits_{\substack{i\in\left\{  1,2,\ldots,n\right\}  \\\text{such
that }i=k}}m_{i}=\sum\limits_{i\in\left\{  k\right\}  }m_{i}\\
&  \ \ \ \ \ \ \ \ \ \ \left(
\begin{array}
[c]{c}%
\text{since }\left\{  i\in\left\{  1,2,\ldots,n\right\}  \ \mid\ i=k\right\}
=\left\{  k\right\} \\
\text{(because }k\in\left\{  1,2,\ldots,n\right\}  \text{)}%
\end{array}
\right) \\
&  =m_{k}.
\end{align*}
Hence, for every $k\in\left\{  1,2,\ldots,n\right\}  $, we have%
\begin{align}
\sum\limits_{i=1}^{n}\left(  u\left(  I_{n}\right)  _{k,i}-a_{k,i}\right)
m_{i}  &  =\sum\limits_{i=1}^{n}\left(  u\left(  I_{n}\right)  _{k,i}%
m_{i}-a_{k,i}m_{i}\right)  =u\underbrace{\sum\limits_{i=1}^{n}\left(
I_{n}\right)  _{k,i}m_{i}}_{\substack{=m_{k}\\\text{(by \eqref{1})}}%
}-\sum\limits_{i=1}^{n}a_{k,i}m_{i}\nonumber\\
&  =um_{k}-\sum\limits_{i=1}^{n}a_{k,i}m_{i}=0\ \ \ \ \ \ \ \ \ \ \left(
\text{by \eqref{pf.Theorem1.BA.defaki}}\right)  .
\label{pf.Theorem1.BA.long.matprod=0}%
\end{align}

Define a matrix $S\in A^{n\times n}$ by%
\[
\left(  S_{k,i}=a_{k,i}\text{ for all }k\in\left\{  1,2,\ldots,n\right\}
\text{ and }i\in\left\{  1,2,\ldots,n\right\}  \right)  .
\]

Define a matrix $T\in B^{n\times n}$ by%
\[
T=\operatorname{adj}\left(  uI_{n}-S\right)  .
\]
Here, the \textquotedblleft$S$\textquotedblright\ in \textquotedblleft%
$uI_{n}-S$\textquotedblright\ means not the matrix $S\in A^{n\times n}$
itself, but rather its image under the ring homomorphism $\iota^{n\times
n}:A^{n\times n}\rightarrow B^{n\times n}$; thus, the matrix $uI_{n}-S$ is a
well-defined matrix in $B^{n\times n}$.

Let $P\in A\left[  X\right]  $ be the characteristic polynomial of the matrix
$S\in A^{n\times n}$. Then, $P$ is monic, and $\deg P=n$. Besides, the
definition of $P$ yields $P\left(  X\right)  =\det\left(  XI_{n}-S\right)  $,
so that $P\left(  u\right)  =\det\left(  uI_{n}-S\right)  $. Therefore,%
\begin{align*}
P\left(  u\right)  \cdot I_{n}  &  =\det\left(  uI_{n}-S\right)  \cdot
I_{n}=\underbrace{\operatorname{adj}\left(  uI_{n}-S\right)  }_{=T}%
\cdot\left(  uI_{n}-S\right) \\
&  \ \ \ \ \ \ \ \ \ \ \left(  \text{by Lemma \ref{lem.adj}, applied to
}M=uI_{n}-S\right) \\
&  =T\cdot\left(  uI_{n}-S\right)  .
\end{align*}

Now, for every $\tau\in\left\{  1,2,\ldots,n\right\}  $, we have%
\begin{align*}
P\left(  u\right)  \cdot m_{\tau}  &  =P\left(  u\right)  \cdot\sum
\limits_{i=1}^{n}\left(  I_{n}\right)  _{\tau,i}m_{i}\\
&  \ \ \ \ \ \ \ \ \ \ \left(  \text{since \eqref{1} (applied to }%
k=\tau\text{) yields }m_{\tau}=\sum\limits_{i=1}^{n}\left(  I_{n}\right)
_{\tau,i}m_{i}\right) \\
&  =\sum_{i=1}^{n}\underbrace{P\left(  u\right)  \cdot\left(  I_{n}\right)
_{\tau,i}}_{=\left(  P\left(  u\right)  \cdot I_{n}\right)  _{\tau,i}}%
m_{i}=\sum_{i=1}^{n}\left(  \underbrace{P\left(  u\right)  \cdot I_{n}%
}_{=T\cdot\left(  uI_{n}-S\right)  }\right)  _{\tau,i}m_{i}\\
&  =\sum_{i=1}^{n}\underbrace{\left(  T\cdot\left(  uI_{n}-S\right)  \right)
_{\tau,i}}_{\substack{=\sum\limits_{k=1}^{n}T_{\tau,k}\left(  uI_{n}-S\right)
_{k,i}\\\text{(by the definition of}\\\text{the product of two matrices)}%
}}m_{i}=\sum_{i=1}^{n}\sum\limits_{k=1}^{n}T_{\tau,k}\left(  uI_{n}-S\right)
_{k,i}m_{i}\\
&  =\sum\limits_{k=1}^{n}T_{\tau,k}\sum_{i=1}^{n}\underbrace{\left(
uI_{n}-S\right)  _{k,i}}_{=u\left(  I_{n}\right)  _{k,i}-S_{k,i}}m_{i}%
=\sum\limits_{k=1}^{n}T_{\tau,k}\sum_{i=1}^{n}\left(  u\left(  I_{n}\right)
_{k,i}-\underbrace{S_{k,i}}_{=a_{k,i}}\right)  m_{i}\\
&  =\sum\limits_{k=1}^{n}T_{\tau,k}\underbrace{\sum_{i=1}^{n}\left(  u\left(
I_{n}\right)  _{k,i}-a_{k,i}\right)  m_{i}}_{\substack{=0\\\text{(by
\eqref{pf.Theorem1.BA.long.matprod=0})}}}=0.
\end{align*}
But from $U=\left\langle m_{1},m_{2},\ldots,m_{n}\right\rangle _{A}$, we
obtain%
\begin{align*}
P\left(  u\right)  \cdot U  &  =P\left(  u\right)  \cdot\left\langle
m_{1},m_{2},\ldots,m_{n}\right\rangle _{A}=\left\langle P\left(  u\right)
\cdot m_{1},P\left(  u\right)  \cdot m_{2},\ldots,P\left(  u\right)  \cdot
m_{n}\right\rangle _{A}\\
&  =\left\langle 0,0,\ldots,0\right\rangle _{A}\ \ \ \ \ \ \ \ \ \ \left(
\text{since }P\left(  u\right)  \cdot m_{\tau}=0\text{ for any }\tau
\in\left\{  1,2,\ldots,n\right\}  \right) \\
&  =0.
\end{align*}

But recall that every $v\in B$ satisfying $vU=0$ satisfies $v=0$. Applying
this to $v=P\left(  u\right)  $, we find $P\left(  u\right)  =0$ (since
$P\left(  u\right)  \cdot U=0$). Thus, we have found a monic polynomial $P\in
A\left[  X\right]  $ with $\deg P=n$ and $P\left(  u\right)  =0$. Therefore,
Assertion $\mathcal{A}$ holds. Hence, we have proved that $\mathcal{B}%
\Longrightarrow\mathcal{A}$.
\end{verlong}

\textit{Proof of the implication }$\mathcal{A}\Longrightarrow\mathcal{D}%
$\textit{.} Assume that Assertion $\mathcal{A}$ holds. Then, there exists a
monic polynomial $P\in A\left[  X\right]  $ with $\deg P=n$ and $P\left(
u\right)  =0$. Consider this $P$.

Let $U$ be the $A$-submodule $\left\langle u^{0},u^{1},\ldots,u^{n-1}%
\right\rangle _{A}$ of $B$. As in the Proof of the implication $\mathcal{A}%
\Longrightarrow\mathcal{C}$, we can show that $U$ is an $n$-generated
$A$-module, and that $1\in U$ and $uU\subseteq U$.

Now, it is easy to show that
\begin{equation}
u^{i}\in U\ \ \ \ \ \ \ \ \ \ \text{for any }i\in\mathbb{N}. \label{2}%
\end{equation}

\begin{vershort}
[Indeed, this follows easily from $uU\subseteq U$ and $1 \in U$ by induction
on $i$.]
\end{vershort}

\begin{verlong}
[\textit{Proof of \eqref{2}.} We will prove \eqref{2} by induction over $i$:

\textit{Induction base:} The assertion \eqref{2} holds for $i=0$ (since
$u^{0}=1\in U$). This completes the induction base.

\textit{Induction step:} Let $\tau\in\mathbb{N}$. If the assertion \eqref{2}
holds for $i=\tau$, then the assertion \eqref{2} holds for $i=\tau+1$ (because
if the assertion \eqref{2} holds for $i=\tau$, then $u^{\tau}\in U$, so that
$u^{\tau+1}=u\cdot\underbrace{u^{\tau}}_{\in U}\in uU\subseteq U$, so that
$u^{\tau+1}\in U$, and thus the assertion \eqref{2} holds for $i=\tau+1$).
This completes the induction step.

Hence, the induction is complete, and \eqref{2} is proven.]
\end{verlong}

\begin{vershort}
Hence, for any $m\in\mathbb{N}$ and any $\left(  a_{0},a_{1},\ldots
,a_{m}\right)  \in A^{m+1}$, we have $\sum\limits_{i=0}^{m}a_{i}u^{i}\in U$
(since $U$ is an $A$-module, and thus is closed under $A$-linear combination).

Now, the definition of $A\left[  u\right]  $ yields
\begin{align*}
A\left[  u\right]   &  =\left\{  \sum\limits_{i=0}^{m}a_{i}u^{i}\ \mid
\ m\in\mathbb{N}\text{ and }\left(  a_{0},a_{1},\ldots,a_{m}\right)  \in
A^{m+1}\right\} \\
&  \subseteq U\ \ \ \ \ \ \ \ \ \ \left(  \text{since }\sum\limits_{i=0}%
^{m}a_{i}u^{i}\in U\text{ for any }m\in\mathbb{N}\text{ and }\left(
a_{0},a_{1},\ldots,a_{m}\right)  \in A^{m+1}\right) \\
&  =\left\langle u^{0},u^{1},\ldots,u^{n-1}\right\rangle _{A}.
\end{align*}
Combining this with the (obvious) relation $\left\langle u^{0},u^{1}%
,\ldots,u^{n-1}\right\rangle _{A}\subseteq A\left[  u\right]  $, we find
$A\left[  u\right]  =\left\langle u^{0},u^{1},\ldots,u^{n-1}\right\rangle
_{A}$.
\end{vershort}

\begin{verlong}
But recall that $U$ is an $A$-module, and therefore is closed under $A$-linear
combination. Thus, for any $m\in\mathbb{N}$ and any $\left(  a_{0}%
,a_{1},\ldots,a_{m}\right)  \in A^{m+1}$, we have $\sum\limits_{i=0}^{m}%
a_{i}u^{i}\in U$, because each $i\in\left\{  0,1,\ldots,m\right\}  $ satisfies
$a_{i}\in A$ and $u^{i}\in U$ (by \eqref{2}).

Now, the definition of $A\left[  u\right]  $ yields%
\[
A\left[  u\right]  =\left\{  \sum\limits_{i=0}^{m}a_{i}u^{i}\ \mid
\ m\in\mathbb{N}\text{ and }\left(  a_{0},a_{1},\ldots,a_{m}\right)  \in
A^{m+1}\right\}  \subseteq U
\]
(since $\sum\limits_{i=0}^{m}a_{i}u^{i}\in U$ for any $m\in\mathbb{N}$ and any
$\left(  a_{0},a_{1},\ldots,a_{m}\right)  \in A^{m+1}$). On the other hand,
$U\subseteq A\left[  u\right]  $, since%
\begin{align*}
U  &  =\left\langle u^{0},u^{1},\ldots,u^{n-1}\right\rangle _{A}=\left\{
\sum\limits_{i=0}^{n-1}a_{i}u^{i}\ \mid\ \left(  a_{0},a_{1},\ldots
,a_{n-1}\right)  \in A^{n}\right\} \\
&  \subseteq\left\{  \sum\limits_{i=0}^{m}a_{i}u^{i}\ \mid\ m\in
\mathbb{N}\text{ and }\left(  a_{0},a_{1},\ldots,a_{m}\right)  \in
A^{m+1}\right\}  =A\left[  u\right]  .
\end{align*}
Combining this with $A\left[  u\right]  \subseteq U$, we obtain $U=A\left[
u\right]  $. Comparing this with $U=\left\langle u^{0},u^{1},\ldots
,u^{n-1}\right\rangle _{A}$, we obtain $A\left[  u\right]  =\left\langle
u^{0},u^{1},\ldots,u^{n-1}\right\rangle _{A}$.
\end{verlong}

Thus, Assertion $\mathcal{D}$ holds. Hence, we have proved that $\mathcal{A}%
\Longrightarrow\mathcal{D}$.

\textit{Proof of the implication }$\mathcal{D}\Longrightarrow\mathcal{C}%
$\textit{.} Assume that Assertion $\mathcal{D}$ holds. Then, $A\left[
u\right]  =\left\langle u^{0},u^{1},\ldots,u^{n-1}\right\rangle _{A}$.

\begin{vershort}
Let $U$ be the $A$-submodule $A\left[  u\right]  $ of $B$. Then, $U$ is an
$n$-generated $A$-module (since $U=A\left[  u\right]  =\left\langle
u^{0},u^{1},\ldots,u^{n-1}\right\rangle _{A}$). Besides, $1=u^{0}\in A\left[
u\right]  =U$. Finally, $U=A\left[  u\right]  $ yields $uU\subseteq U$. Thus,
Assertion $\mathcal{C}$ holds. Hence, we have proved that $\mathcal{D}%
\Longrightarrow\mathcal{C}$.
\end{vershort}

\begin{verlong}
Let $U$ be the $A$-submodule $\left\langle u^{0},u^{1},\ldots,u^{n-1}%
\right\rangle _{A}$ of $B$. Then, $u^{0},u^{1},\ldots,u^{n-1}$ are $n$
elements of $U$. Hence, $U$ is an $n$-generated $A$-module (since
$U=\left\langle u^{0},u^{1},\ldots,u^{n-1}\right\rangle _{A}$). Comparing
$U=\left\langle u^{0},u^{1},\ldots,u^{n-1}\right\rangle _{A}$ with $A\left[
u\right]  =\left\langle u^{0},u^{1},\ldots,u^{n-1}\right\rangle _{A}$, we
obtain $U=A\left[  u\right]  $. Now, $1=u^{0}\in A\left[  u\right]  =U$.

Also, from $U=A\left[  u\right]  $, we obtain $uU=u\cdot A\left[  u\right]
\subseteq A\left[  u\right]  =U$.

Thus, we have found an $n$-generated $A$-submodule $U$ of $B$ such that $1\in
U$ and $uU\subseteq U$. Hence, Assertion $\mathcal{C}$ holds. Thus, we have
proved that $\mathcal{D}\Longrightarrow\mathcal{C}$.
\end{verlong}

Now, we have proved the implications $\mathcal{A}\Longrightarrow\mathcal{D}$,
$\mathcal{D}\Longrightarrow\mathcal{C}$, $\mathcal{C}\Longrightarrow
\mathcal{B}$ and $\mathcal{B}\Longrightarrow\mathcal{A}$ above. Thus, all four
assertions $\mathcal{A}$, $\mathcal{B}$, $\mathcal{C}$ and $\mathcal{D}$ are
equivalent, and Theorem~\ref{Theorem1} is proven.
\end{proof}

For the sake of completeness (and as a very easy exercise), let us state a
basic property of integrality that we will not ever use:

\begin{proposition}
\label{prop.integr.u-v}Let $A$ be a ring. Let $B$ be an $A$-algebra. Let $u\in
B$. Let $q\in\mathbb{N}$ and $p\in\mathbb{N}$ be such that $p\geq q$. Assume
that $u$ is $q$-integral over $A$. Then, $u$ is $p$-integral over $A$.
\end{proposition}

\begin{verlong}
\begin{proof}
[Proof of Proposition~\ref{prop.integr.u-v}.]The element $u$ is $q$-integral
over $A$. Thus, it satisfies the Assertion $\mathcal{D}$ of
Theorem~\ref{Theorem1}, stated for $q$ in lieu of $n$. In other words, it
satisfies $A\left[  u\right]  =\left\langle u^{0},u^{1},...,u^{q-1}%
\right\rangle _{A}$. Note that $\left\langle u^{0},u^{1},...,u^{p-1}%
\right\rangle _{A}$ is an $A$-submodule of $B$.

But $p\geq q$, thus $q\leq p$ and therefore $q-1\leq p-1$. Every $s\in\left\{
0,1,...,q-1\right\}  $ satisfies $s\in\left\{  0,1,...,q-1\right\}
\subseteq\left\{  0,1,...,p-1\right\}  $ (since $q-1\leq p-1$) and therefore
$u^{s}\in\left\{  u^{0},u^{1},...,u^{p-1}\right\}  \subseteq\left\langle
u^{0},u^{1},...,u^{p-1}\right\rangle _{A}$. Thus, $u^{s}$ is an element of
$\left\langle u^{0},u^{1},...,u^{p-1}\right\rangle _{A}$ for every
$s\in\left\{  0,1,\ldots,q-1\right\}  $. Hence, Lemma \ref{lem.mil} (applied
to $M=B$, $N=\left\langle u^{0},u^{1},...,u^{p-1}\right\rangle _{A}$,
$S=\left\{  0,1,\ldots,q-1\right\}  $ and $m_{s}=u^{s}$) shows that
\[
\left\langle u^{s}\ \mid\ s\in\left\{  0,1,\ldots,q-1\right\}  \right\rangle
_{A}\subseteq\left\langle u^{0},u^{1},...,u^{p-1}\right\rangle _{A}.
\]
Now,%
\[
A\left[  u\right]  =\left\langle u^{0},u^{1},...,u^{q-1}\right\rangle
_{A}=\left\langle u^{s}\ \mid\ s\in\left\{  0,1,\ldots,q-1\right\}
\right\rangle _{A}\subseteq\left\langle u^{0},u^{1},...,u^{p-1}\right\rangle
_{A}.
\]
Combining this with $\left\langle u^{0},u^{1},...,u^{p-1}\right\rangle
_{A}\subseteq A\left[  u\right]  $ (which is obvious, since every $A$-linear
combination of $u^{0},u^{1},...,u^{p-1}$ is a polynomial in $u$ with
coefficients in $A$), we obtain $A\left[  u\right]  =\left\langle u^{0}%
,u^{1},...,u^{p-1}\right\rangle _{A}$. In other words, $u$ satisfies the
Assertion $\mathcal{D}$ of Theorem~\ref{Theorem1}, stated for $p$ in lieu of
$n$. Hence, $u$ is $p$-integral over $A$. This proves
Proposition~\ref{prop.integr.u-v}.
\end{proof}
\end{verlong}

\subsection{Transitivity of integrality}

Let us now prove the first and probably most important consequence of
Theorem~\ref{Theorem1}:

\begin{theorem}
\label{Theorem4} Let $A$ be a ring. Let $B$ be an $A$-algebra. Let $v\in B$
and $u\in B$. Let $m\in\mathbb{N}$ and $n\in\mathbb{N}$. Assume that $v$ is
$m$-integral over $A$, and that $u$ is $n$-integral over $A\left[  v\right]
$. Then, $u$ is $nm$-integral over $A$.
\end{theorem}

(Here, we are using the fact that if $A$ is a ring, and if $v$ is an element
of an $A$-algebra $B$, then $A\left[  v\right]  $ is a subring of $B$, and
therefore $B$ is an $A\left[  v\right]  $-algebra.)

\begin{proof}
[Proof of Theorem~\ref{Theorem4}.]Since $v$ is $m$-integral over $A$, we have
$A\left[  v\right]  =\left\langle v^{0},v^{1},\ldots,v^{m-1}\right\rangle
_{A}$ (this is the Assertion $\mathcal{D}$ of Theorem~\ref{Theorem1}, stated
for $v$ and $m$ in lieu of $u$ and $n$).

Since $u$ is $n$-integral over $A\left[  v\right]  $, we have $\left(
A\left[  v\right]  \right)  \left[  u\right]  =\left\langle u^{0},u^{1}%
,\ldots,u^{n-1}\right\rangle _{A\left[  v\right]  }$ (this is the Assertion
$\mathcal{D}$ of Theorem~\ref{Theorem1}, stated for $A\left[  v\right]  $ in
lieu of $A$).

Let $S=\left\{  0,1,\ldots,n-1\right\}  \times\left\{  0,1,\ldots,m-1\right\}
$. Then, $S$ is a finite set with size $\left\vert S\right\vert =nm$.

\begin{verlong}
[\textit{Proof:} From $S=\left\{  0,1,\ldots,n-1\right\}  \times\left\{
0,1,\ldots,m-1\right\}  $, we obtain
\begin{align*}
\left\vert S\right\vert  &  =\left\vert \left\{  0,1,\ldots,n-1\right\}
\times\left\{  0,1,\ldots,m-1\right\}  \right\vert \\
&  =\underbrace{\left\vert \left\{  0,1,\ldots,n-1\right\}  \right\vert }%
_{=n}\cdot\underbrace{\left\vert \left\{  0,1,\ldots,m-1\right\}  \right\vert
}_{=m}=nm.
\end{align*}
Thus, $S$ is finite.]
\end{verlong}

Let $x\in\left(  A\left[  v\right]  \right)  \left[  u\right]  $. Then, there
exist $n$ elements $b_{0},b_{1},\ldots,b_{n-1}$ of $A\left[  v\right]  $ such
that $x=\sum\limits_{i=0}^{n-1}b_{i}u^{i}$ (since $x\in\left(  A\left[
v\right]  \right)  \left[  u\right]  =\left\langle u^{0},u^{1},\ldots
,u^{n-1}\right\rangle _{A\left[  v\right]  }$). Consider these $b_{0}%
,b_{1},\ldots,b_{n-1}$.

For each $i\in\left\{  0,1,\ldots,n-1\right\}  $, there exist $m$ elements
$a_{i,0},a_{i,1},\ldots,a_{i,m-1}$ of $A$ such that $b_{i}=\sum\limits_{j=0}%
^{m-1}a_{i,j}v^{j}$ (because $b_{i}\in A\left[  v\right]  =\left\langle
v^{0},v^{1},\ldots,v^{m-1}\right\rangle _{A}$). Consider these $a_{i,0}%
,a_{i,1},\ldots,a_{i,m-1}$. Thus,%
\begin{align*}
x  &  =\sum\limits_{i=0}^{n-1}\underbrace{b_{i}}_{=\sum\limits_{j=0}%
^{m-1}a_{i,j}v^{j}}u^{i}=\sum\limits_{i=0}^{n-1}\sum\limits_{j=0}^{m-1}%
a_{i,j}v^{j}u^{i}=\sum\limits_{\left(  i,j\right)  \in\left\{  0,1,\ldots
,n-1\right\}  \times\left\{  0,1,\ldots,m-1\right\}  }a_{i,j}v^{j}u^{i}\\
&  =\sum\limits_{\left(  i,j\right)  \in S}a_{i,j}v^{j}u^{i}%
\ \ \ \ \ \ \ \ \ \ \left(  \text{since }\left\{  0,1,\ldots,n-1\right\}
\times\left\{  0,1,\ldots,m-1\right\}  =S\right) \\
&  \in\left\langle v^{j}u^{i}\ \mid\ \left(  i,j\right)  \in S\right\rangle
_{A}\ \ \ \ \ \ \ \ \ \ \left(  \text{since }a_{i,j}\in A\text{ for every
}\left(  i,j\right)  \in S\right)  .
\end{align*}

\begin{vershort}
So we have proved that $x\in\left\langle v^{j}u^{i}\ \mid\ \left(  i,j\right)
\in S\right\rangle _{A}$ for every $x\in\left(  A\left[  v\right]  \right)
\left[  u\right]  $. In other words, $\left(  A\left[  v\right]  \right)
\left[  u\right]  \subseteq\left\langle v^{j}u^{i}\ \mid\ \left(  i,j\right)
\in S\right\rangle _{A}$. Conversely, $\left\langle v^{j}u^{i}\ \mid\ \left(
i,j\right)  \in S\right\rangle _{A}\subseteq\left(  A\left[  v\right]
\right)  \left[  u\right]  $ (this is trivial). Combining these two relations,
we find $\left(  A\left[  v\right]  \right)  \left[  u\right]  =\left\langle
v^{j}u^{i}\ \mid\ \left(  i,j\right)  \in S\right\rangle _{A}$. Thus, the
$A$-module $\left(  A\left[  v\right]  \right)  \left[  u\right]  $ is
$nm$-generated (since $\left\vert S\right\vert =nm$).
\end{vershort}

\begin{verlong}
Now, forget that we fixed $x$. So we have proved that $x\in\left\langle
v^{j}u^{i}\ \mid\ \left(  i,j\right)  \in S\right\rangle _{A}$ for every
$x\in\left(  A\left[  v\right]  \right)  \left[  u\right]  $. In other words,
$\left(  A\left[  v\right]  \right)  \left[  u\right]  \subseteq\left\langle
v^{j}u^{i}\ \mid\ \left(  i,j\right)  \in S\right\rangle _{A}$. Conversely,
$\left\langle v^{j}u^{i}\ \mid\ \left(  i,j\right)  \in S\right\rangle
_{A}\subseteq\left(  A\left[  v\right]  \right)  \left[  u\right]  $ (since
$v^{j}\in A\left[  v\right]  $ for every $\left(  i,j\right)  \in S$, and thus
$\underbrace{v^{j}}_{\in A\left[  v\right]  }u^{i}\in\left(  A\left[
v\right]  \right)  \left[  u\right]  $ for every $\left(  i,j\right)  \in S$,
and therefore%
\begin{align*}
\left\langle v^{j}u^{i}\ \mid\ \left(  i,j\right)  \in S\right\rangle _{A}  &
=\left\{  \underbrace{\sum\limits_{\left(  i,j\right)  \in S}a_{i,j}v^{j}%
u^{i}}_{\substack{\in\left(  A\left[  v\right]  \right)  \left[  u\right]
\\\text{(since }v^{j}u^{i}\in\left(  A\left[  v\right]  \right)  \left[
u\right]  \text{ for every }\left(  i,j\right)  \in S\\\text{and since
}\left(  A\left[  v\right]  \right)  \left[  u\right]  \text{ is an
}A\text{-module)}}}\ \mid\ \left(  a_{i,j}\right)  _{\left(  i,j\right)  \in
S}\in A^{S}\right\} \\
&  \subseteq\left(  A\left[  v\right]  \right)  \left[  u\right]
\end{align*}
). Combining these two relations, we find $\left(  A\left[  v\right]  \right)
\left[  u\right]  =\left\langle v^{j}u^{i}\ \mid\ \left(  i,j\right)  \in
S\right\rangle _{A}$. Thus, the $A$-module $\left(  A\left[  v\right]
\right)  \left[  u\right]  $ is $nm$-generated (since $\left\vert S\right\vert
=nm$).
\end{verlong}

\begin{vershort}
Let $U=\left(  A\left[  v\right]  \right)  \left[  u\right]  $. Then, the
$A$-module $U$ is $nm$-generated. Besides, $U$ is an $A$-submodule of $B$, and
we have $1\in U$ and $uU\subseteq U$.
\end{vershort}

\begin{verlong}
Let $U=\left(  A\left[  v\right]  \right)  \left[  u\right]  $. Thus, the
$A$-module $U$ is $nm$-generated (since the $A$-module $\left(  A\left[
v\right]  \right)  \left[  u\right]  $ is $nm$-generated). Besides, $U$ is an
$A$-submodule of $B$, and we have $1=u^{0}\in\left(  A\left[  v\right]
\right)  \left[  u\right]  =U$ and%
\begin{align*}
uU  &  =u\left(  A\left[  v\right]  \right)  \left[  u\right]  \subseteq
\left(  A\left[  v\right]  \right)  \left[  u\right] \\
&  \ \ \ \ \ \ \ \ \ \ \left(  \text{since }\left(  A\left[  v\right]
\right)  \left[  u\right]  \text{ is an }A\left[  v\right]  \text{-algebra and
}u\in\left(  A\left[  v\right]  \right)  \left[  u\right]  \right) \\
&  =U.
\end{align*}

Altogether, we now know that the $A$-submodule $U$ of $B$ is $nm$-generated
and satisfies $1\in U$ and $uU\subseteq U$.
\end{verlong}

Thus, the element $u$ of $B$ satisfies the Assertion $\mathcal{C}$ of
Theorem~\ref{Theorem1} with $n$ replaced by $nm$. Hence, $u\in B$ satisfies
the four equivalent assertions $\mathcal{A}$, $\mathcal{B}$, $\mathcal{C}$ and
$\mathcal{D}$ of Theorem~\ref{Theorem1}, all with $n$ replaced by $nm$. Thus,
$u$ is $nm$-integral over $A$. This proves Theorem~\ref{Theorem4}.
\end{proof}

\subsection{Integrality of sums and products}

Before the next significant consequence of Theorem~\ref{Theorem1}, let us show
an essentially trivial fact:

\begin{theorem}
\label{Theorem5a} Let $A$ be a ring. Let $B$ be an $A$-algebra. Let $a\in A$.
Then, $a\cdot1_{B}\in B$ is $1$-integral over $A$.
\end{theorem}

\begin{proof}
[Proof of Theorem~\ref{Theorem5a}.]The polynomial $X-a\in A\left[  X\right]  $
is monic and satisfies \newline$\deg\left(  X-a\right)  =1$; moreover,
evaluating this polynomial at $a\cdot1_{B}\in B$ yields $a\cdot1_{B}%
-a\cdot1_{B}=0$. Hence, there exists a monic polynomial $P\in A\left[
X\right]  $ with $\deg P=1$ and $P\left(  a\cdot1_{B}\right)  =0$ (namely, the
polynomial $P\in A\left[  X\right]  $ defined by $P\left(  X\right)  =X-a$).
Thus, $a\cdot1_{B}$ is $1$-integral over $A$. This proves
Theorem~\ref{Theorem5a}.
\end{proof}

The following theorem is a standard result, generalizing (for example) the
classical fact that sums and products of algebraic integers are again
algebraic integers:

\begin{theorem}
\label{Theorem5bc} Let $A$ be a ring. Let $B$ be an $A$-algebra. Let $x\in B$
and $y\in B$. Let $m\in\mathbb{N}$ and $n\in\mathbb{N}$. Assume that $x$ is
$m$-integral over $A$, and that $y$ is $n$-integral over $A$.

\textbf{(a)} Then, $x+y$ is $nm$-integral over $A$.

\textbf{(b)} Furthermore, $xy$ is $nm$-integral over $A$.
\end{theorem}

Our proof of this theorem will rely on a simple lemma:

\begin{lemma}
\label{lem.monic-shift} Let $A$ be a ring. Let $C$ be an $A$-algebra. Let
$x\in C$.

Let $n\in\mathbb{N}$. Let $P\in A\left[  X\right]  $ be a monic polynomial
with $\deg P=n$. Define a polynomial $Q\in C\left[  X\right]  $ by $Q\left(
X\right)  =P\left(  X-x\right)  $. Then, $Q$ is a monic polynomial with $\deg
Q=n$.
\end{lemma}

\begin{proof}
[Proof of Lemma \ref{lem.monic-shift}.]Recall that $P$ is a monic polynomial
with $\deg P=n$; hence, we can write $P$ in the form%
\begin{equation}
P=X^{n}+\sum_{i=0}^{n-1}a_{i}X^{i} \label{pf.Theorem5.b.1}%
\end{equation}
for some $a_{0},a_{1},\ldots,a_{n-1}\in A$. Consider these $a_{0},a_{1}%
,\ldots,a_{n-1}$.

\begin{vershort}
Now,%
\begin{align*}
Q\left(  X\right)   &  =P\left(  X-x\right)  =\underbrace{\left(  X-x\right)
^{n}}_{=X^{n}+\left(  \text{lower order terms}\right)  }+\underbrace{\sum
_{i=0}^{n-1}a_{i}\left(  X-x\right)  ^{i}}_{=\left(  \text{lower order
terms}\right)  }\\
&  \ \ \ \ \ \ \ \ \ \ \left(  \text{here, we have substituted }X-x\text{ for
}X\text{ in \eqref{pf.Theorem5.b.1}}\right) \\
&  =X^{n}+\left(  \text{lower order terms}\right)  ,
\end{align*}
where \textquotedblleft lower order terms\textquotedblright\ means a sum of
terms of the form $bX^{i}$ with $b\in C$ and $i<n$. Hence, $Q$ is a monic
polynomial with $\deg Q=n$. This proves Lemma \ref{lem.monic-shift}.
\qedhere

\end{vershort}

\begin{verlong}
Consider the $C$-submodule $\left\langle X^{0},X^{1},\ldots,X^{n-1}%
\right\rangle _{C}$ of $C\left[  X\right]  $. We have $X-x=X+\left(
-x\right)  $ and thus%
\begin{align}
\left(  X-x\right)  ^{n}  &  =\left(  X+\left(  -x\right)  \right)  ^{n}%
=\sum_{i=0}^{n}\dbinom{n}{i}X^{i}\left(  -x\right)  ^{n-i}%
\ \ \ \ \ \ \ \ \ \ \left(  \text{by the binomial formula}\right) \nonumber\\
&  =\sum_{i=0}^{n-1}\dbinom{n}{i}\underbrace{X^{i}\left(  -x\right)  ^{n-i}%
}_{=\left(  -x\right)  ^{n-i}X^{i}}+\underbrace{\dbinom{n}{n}}_{=1}%
X^{n}\underbrace{\left(  -x\right)  ^{n-n}}_{=\left(  -x\right)  ^{0}%
=1}\nonumber\\
&  \ \ \ \ \ \ \ \ \ \ \left(  \text{here, we have split off the addend for
}i=n\text{ from the sum}\right) \nonumber\\
&  =\sum_{i=0}^{n-1}\dbinom{n}{i}\left(  -x\right)  ^{n-i}X^{i}+X^{n}%
=X^{n}+\underbrace{\sum_{i=0}^{n-1}\dbinom{n}{i}\left(  -x\right)  ^{n-i}%
X^{i}}_{\in\left\langle X^{0},X^{1},\ldots,X^{n-1}\right\rangle _{C}%
}\nonumber\\
&  \in X^{n}+\left\langle X^{0},X^{1},\ldots,X^{n-1}\right\rangle _{C}.
\label{pf.Theorem5.b.2n}%
\end{align}
Furthermore, for each $i\in\left\{  0,1,\ldots,n-1\right\}  $, we have
\begin{align}
\left(  X-x\right)  ^{i}  &  =\left(  X+\left(  -x\right)  \right)
^{i}\ \ \ \ \ \ \ \ \ \ \left(  \text{since }X-x=X+\left(  -x\right)  \right)
\nonumber\\
&  =\sum_{j=0}^{i}\dbinom{i}{j}\underbrace{X^{j}\left(  -x\right)  ^{i-j}%
}_{=\left(  -x\right)  ^{i-j}X^{j}}\ \ \ \ \ \ \ \ \ \ \left(  \text{by the
binomial formula}\right) \nonumber\\
&  =\sum_{j=0}^{i}\dbinom{i}{j}\left(  -x\right)  ^{i-j}X^{j}\in\left\langle
X^{0},X^{1},\ldots,X^{i}\right\rangle _{C}\nonumber\\
&  \subseteq\left\langle X^{0},X^{1},\ldots,X^{n-1}\right\rangle _{C}
\label{pf.Theorem5.b.2i}%
\end{align}
(since $i\leq n-1$). Now,
\begin{align*}
Q\left(  X\right)   &  =P\left(  X-x\right)  =\underbrace{\left(  X-x\right)
^{n}}_{\substack{\in X^{n}+\left\langle X^{0},X^{1},\ldots,X^{n-1}%
\right\rangle _{C}\\\text{(by \eqref{pf.Theorem5.b.2n})}}}+\underbrace{\sum
_{i=0}^{n-1}a_{i}\left(  X-x\right)  ^{i}}_{\substack{\in\left\langle
X^{0},X^{1},\ldots,X^{n-1}\right\rangle _{C}\\\text{(by
\eqref{pf.Theorem5.b.2i})}}}\\
&  \ \ \ \ \ \ \ \ \ \ \left(  \text{here, we have substituted }X-x\text{ for
}X\text{ in \eqref{pf.Theorem5.b.1}}\right) \\
&  \in X^{n}+\underbrace{\left\langle X^{0},X^{1},\ldots,X^{n-1}\right\rangle
_{C}+\left\langle X^{0},X^{1},\ldots,X^{n-1}\right\rangle _{C}}%
_{\substack{\subseteq\left\langle X^{0},X^{1},\ldots,X^{n-1}\right\rangle
_{C}\\\text{(since }\left\langle X^{0},X^{1},\ldots,X^{n-1}\right\rangle
_{C}\text{ is a }C\text{-module)}}}\\
&  \subseteq X^{n}+\left\langle X^{0},X^{1},\ldots,X^{n-1}\right\rangle _{C}.
\end{align*}
In other words, $Q\left(  X\right)  =X^{n}+W$ for some $W\in\left\langle
X^{0},X^{1},\ldots,X^{n-1}\right\rangle _{C}$. Consider this $W$. We have
$W\in\left\langle X^{0},X^{1},\ldots,X^{n-1}\right\rangle _{C}$; thus, we can
write $W$ in the form $W=\sum_{i=0}^{n-1}w_{i}X^{i}$ for some $w_{0}%
,w_{1},\ldots,w_{n-1}\in C$. Consider these $w_{0},w_{1},\ldots,w_{n-1}$. Now,%
\[
Q\left(  X\right)  =X^{n}+\underbrace{W}_{=\sum_{i=0}^{n-1}w_{i}X^{i}}%
=X^{n}+\sum_{i=0}^{n-1}w_{i}X^{i}.
\]
Hence, $Q$ is a monic polynomial with $\deg Q=n$. This proves
Lemma~\ref{lem.monic-shift}.
\end{verlong}
\end{proof}

\begin{proof}
[Proof of Theorem~\ref{Theorem5bc}.]Since $y$ is $n$-integral over $A$, there
exists a monic polynomial $P\in A\left[  X\right]  $ with $\deg P=n$ and
$P\left(  y\right)  =0$. Consider this $P$.

\textbf{(a)} Let $C$ be the $A$-subalgebra $A\left[  x\right]  $ of $B$. Then,
$C=A\left[  x\right]  $, so that $x\in A\left[  x\right]  =C$.

Now, define a polynomial $Q\in C\left[  X\right]  $ by $Q\left(  X\right)
=P\left(  X-x\right)  $. Then, Lemma \ref{lem.monic-shift} shows that $Q$ is a
monic polynomial with $\deg Q=n$. Also, substituting $x+y$ for $X$ in the
equality $Q\left(  X\right)  =P\left(  X-x\right)  $, we obtain $Q\left(
x+y\right)  =P\left(  \underbrace{\left(  x+y\right)  -x}_{=y}\right)
=P\left(  y\right)  =0$.

Hence, there exists a monic polynomial $Q\in C\left[  X\right]  $ with $\deg
Q=n$ and $Q\left(  x+y\right)  =0$. Thus, $x+y$ is $n$-integral over $C$. In
other words, $x+y$ is $n$-integral over $A\left[  x\right]  $ (since
$C=A\left[  x\right]  $). Thus, Theorem~\ref{Theorem4} (applied to $v=x$ and
$u=x+y$) yields that $x+y$ is $nm$-integral over $A$. This proves
Theorem~\ref{Theorem5bc} \textbf{(a)}.

\textbf{(b)} Recall that $P\in A\left[  X\right]  $ is a monic polynomial with
$\deg P=n$. Thus, there exist elements $a_{0},a_{1},\ldots,a_{n-1}$ of $A$
such that $P\left(  X\right)  =X^{n}+\sum\limits_{k=0}^{n-1}a_{k}X^{k}$.
Consider these $a_{0},a_{1},\ldots,a_{n-1}$. Substituting $y$ for $X$ in
$P\left(  X\right)  =X^{n}+\sum\limits_{k=0}^{n-1}a_{k}X^{k}$, we find
$P\left(  y\right)  =y^{n}+\sum\limits_{k=0}^{n-1}a_{k}y^{k}$. Thus,%
\begin{equation}
y^{n}+\sum\limits_{k=0}^{n-1}a_{k}y^{k}=P\left(  y\right)  =0.
\label{pf.Theorem5.c.1}%
\end{equation}

Now, define a polynomial $Q\in\left(  A\left[  x\right]  \right)  \left[
X\right]  $ by $Q\left(  X\right)  =X^{n}+\sum\limits_{k=0}^{n-1}x^{n-k}%
a_{k}X^{k}$. Then,%
\begin{align*}
Q\left(  xy\right)   &  =\underbrace{\left(  xy\right)  ^{n}}_{=x^{n}y^{n}%
}+\sum\limits_{k=0}^{n-1}x^{n-k}\underbrace{a_{k}\left(  xy\right)  ^{k}%
}_{\substack{=a_{k}x^{k}y^{k}\\=x^{k}a_{k}y^{k}}}=x^{n}y^{n}+\sum
\limits_{k=0}^{n-1}\underbrace{x^{n-k}x^{k}}_{=x^{n}}a_{k}y^{k}\\
&  =x^{n}y^{n}+\sum\limits_{k=0}^{n-1}x^{n}a_{k}y^{k}=x^{n}\underbrace{\left(
y^{n}+\sum\limits_{k=0}^{n-1}a_{k}y^{k}\right)  }_{\substack{=0\\\text{(by
\eqref{pf.Theorem5.c.1})}}}=0.
\end{align*}
Also, recall that $Q\left(  X\right)  =X^{n}+\sum\limits_{k=0}^{n-1}%
x^{n-k}a_{k}X^{k}$; hence, the polynomial $Q\in\left(  A\left[  x\right]
\right)  \left[  X\right]  $ is monic and $\deg Q=n$. Thus, there exists a
monic polynomial $Q\in\left(  A\left[  x\right]  \right)  \left[  X\right]  $
with $\deg Q=n$ and $Q\left(  xy\right)  =0$. Thus, $xy$ is $n$-integral over
$A\left[  x\right]  $. Hence, Theorem~\ref{Theorem4} (applied to $v=x$ and
$u=xy$) yields that $xy$ is $nm$-integral over $A$. This proves
Theorem~\ref{Theorem5bc} \textbf{(b)}.
\end{proof}

\begin{corollary}
\label{cor.-x-int} Let $A$ be a ring. Let $B$ be an $A$-algebra. Let $x\in B$.
Let $m\in\mathbb{N}$. Assume that $x$ is $m$-integral over $A$. Then, $-x$ is
$m$-integral over $A$.
\end{corollary}

\begin{proof}
[Proof of Corollary~\ref{cor.-x-int}.]This is easy to prove directly (using
Assertion $\mathcal{A}$ of Theorem~\ref{Theorem1}), but the slickest proof is
using Theorem~\ref{Theorem5bc} \textbf{(b)}: The element $\left(  -1\right)
\cdot1_{B}\in B$ is $1$-integral over $A$ (by Theorem~\ref{Theorem5a}, applied
to $a=-1$). Thus, $x\cdot\left(  \left(  -1\right)  \cdot1_{B}\right)  $ is
$1m$-integral over $A$ (by Theorem~\ref{Theorem5bc} \textbf{(b)}, applied to
$y=\left(  -1\right)  \cdot1_{B}$ and $n=1$). In other words, $-x$ is
$m$-integral over $A$ (since $x\cdot\underbrace{\left(  \left(  -1\right)
\cdot1_{B}\right)  }_{=-1_{B}}=x\cdot\left(  -1_{B}\right)  =-x\cdot1_{B}=-x$
and $1m=m$). This proves Corollary~\ref{cor.-x-int}.
\end{proof}

\begin{corollary}
\label{cor.x-y-int} Let $A$ be a ring. Let $B$ be an $A$-algebra. Let $x\in B$
and $y\in B$. Let $m\in\mathbb{N}$ and $n\in\mathbb{N}$. Assume that $x$ is
$m$-integral over $A$, and that $y$ is $n$-integral over $A$. Then, $x-y$ is
$nm$-integral over $A$.
\end{corollary}

\begin{proof}
[Proof of Corollary~\ref{cor.x-y-int}.]We know that $y$ is $n$-integral over
$A$. Hence, Corollary~\ref{cor.-x-int} (applied to $y$ and $n$ instead of $x$
and $m$) shows that $-y$ is $n$-integral over $A$. Thus,
Theorem~\ref{Theorem5bc} \textbf{(a)} (applied to $-y$ instead of $y$) shows
that $x+\left(  -y\right)  $ is $nm$-integral over $A$. In other words, $x-y$
is $nm$-integral over $A$ (since $x+\left(  -y\right)  =x-y$). This proves
Corollary~\ref{cor.x-y-int}.
\end{proof}

\subsection{Some further consequences}

\begin{theorem}
\label{Theorem2} Let $A$ be a ring. Let $B$ be an $A$-algebra. Let
$n\in\mathbb{N}^{+}$. Let $v\in B$. Let $a_{0},a_{1},\ldots,a_{n}$ be $n+1$
elements of $A$ such that $\sum\limits_{i=0}^{n}a_{i}v^{i}=0$. Let
$k\in\left\{  0,1,\ldots,n\right\}  $. Then, $\sum\limits_{i=0}^{n-k}%
a_{i+k}v^{i}$ is $n$-integral over $A$.
\end{theorem}

\begin{proof}
[Proof of Theorem~\ref{Theorem2}.]Let $u=\sum\limits_{i=0}^{n-k}a_{i+k}v^{i}$.
Then,%
\begin{align*}
0  &  =\sum\limits_{i=0}^{n}a_{i}v^{i}=\sum\limits_{i=0}^{k-1}a_{i}v^{i}%
+\sum\limits_{i=k}^{n}a_{i}v^{i}=\sum\limits_{i=0}^{k-1}a_{i}v^{i}%
+\sum\limits_{i=0}^{n-k}a_{i+k}\underbrace{v^{i+k}}_{=v^{i}v^{k}}\\
&  \ \ \ \ \ \ \ \ \ \ \left(  \text{here, we substituted }i+k\text{ for
}i\text{\ in the second sum}\right) \\
&  =\sum\limits_{i=0}^{k-1}a_{i}v^{i}+\underbrace{\sum\limits_{i=0}%
^{n-k}a_{i+k}v^{i}v^{k}}_{=v^{k}\sum\limits_{i=0}^{n-k}a_{i+k}v^{i}}%
=\sum\limits_{i=0}^{k-1}a_{i}v^{i}+v^{k}\underbrace{\sum\limits_{i=0}%
^{n-k}a_{i+k}v^{i}}_{=u}=\sum\limits_{i=0}^{k-1}a_{i}v^{i}+v^{k}u,
\end{align*}
so that
\[
v^{k}u=-\sum\limits_{i=0}^{k-1}a_{i}v^{i}.
\]

\begin{vershort}
Let $U$ be the $A$-submodule $\left\langle v^{0},v^{1},\ldots,v^{n-1}%
\right\rangle _{A}$ of $B$. Thus, $U$ is an $n$-generated $A$-module, and
$1=v^{0}\in U$.
\end{vershort}

\begin{verlong}
Let $U$ be the $A$-submodule $\left\langle v^{0},v^{1},\ldots,v^{n-1}%
\right\rangle _{A}$ of $B$. Then, $v^{0},v^{1},\ldots,v^{n-1}$ are $n$
elements of $U$. Hence, $U$ is an $n$-generated $A$-module (since
$U=\left\langle v^{0},v^{1},\ldots,v^{n-1}\right\rangle _{A}$). Besides,
$n\in\mathbb{N}^{+}$ and thus $0\in\left\{  0,1,\ldots,n-1\right\}  $.
Therefore, $v^{0}$ is one of the $n$ elements $v^{0},v^{1},\ldots,v^{n-1}$;
hence, $v^{0}\in\left\langle v^{0},v^{1},\ldots,v^{n-1}\right\rangle _{A}=U$.
Thus, $1=v^{0}\in U$.

Note that $U$ is an $A$-submodule of $B$, and thus is closed under $A$-linear combination.
\end{verlong}

Now, we are going to show that%
\begin{equation}
uv^{s}\in U\ \ \ \ \ \ \ \ \ \ \text{for any }s\in\left\{  0,1,\ldots
,n-1\right\}  . \label{3}%
\end{equation}

\begin{vershort}
[\textit{Proof of \eqref{3}.} Let $s\in\left\{  0,1,\ldots,n-1\right\}  $.
Thus, we have either $s<k$ or $s\geq k$. In the case $s<k$, the relation
\eqref{3} follows from%
\begin{align*}
uv^{s}  &  =\sum\limits_{i=0}^{n-k}a_{i+k}\underbrace{v^{i}\cdot v^{s}%
}_{=v^{i+s}}\ \ \ \ \ \ \ \ \ \ \left(  \text{since }u=\sum\limits_{i=0}%
^{n-k}a_{i+k}v^{i}\right) \\
&  =\sum\limits_{i=0}^{n-k}a_{i+k}v^{i+s}\in U
\end{align*}
(since every $i\in\left\{  0,1,\ldots,n-k\right\}  $ satisfies $i+s\in\left\{
0,1,\ldots,n-1\right\}  $\ \ \ \ \footnote{Here we are using $s<k$.}, and thus
$\sum\limits_{i=0}^{n-k}a_{i+k}v^{i+s}\in\left\langle v^{0},v^{1}%
,\ldots,v^{n-1}\right\rangle _{A}=U$). In the case $s\geq k$, the relation
\eqref{3} follows from%
\begin{align*}
uv^{s}  &  =u\underbrace{v^{k+\left(  s-k\right)  }}_{=v^{k}v^{s-k}}%
=v^{k}u\cdot v^{s-k}=-\sum\limits_{i=0}^{k-1}a_{i}\underbrace{v^{i}\cdot
v^{s-k}}_{=v^{i+\left(  s-k\right)  }}\ \ \ \ \ \ \ \ \ \ \left(  \text{since
}v^{k}u=-\sum\limits_{i=0}^{k-1}a_{i}v^{i}\right) \\
&  =-\sum\limits_{i=0}^{k-1}a_{i}v^{i+\left(  s-k\right)  }\in U
\end{align*}
(since every $i\in\left\{  0,1,\ldots,k-1\right\}  $ satisfies $i+\left(
s-k\right)  \in\left\{  0,1,\ldots,n-1\right\}  $\ \ \ \ \footnote{Here we are
using $s\geq k$ and $s\in\left\{  0,1,\ldots,n-1\right\}  $.}, and thus
$-\sum\limits_{i=0}^{k-1}a_{i}v^{i+\left(  s-k\right)  }\in\left\langle
v^{0},v^{1},\ldots,v^{n-1}\right\rangle _{A}=U$). Hence, \eqref{3} is proven
in both possible cases, and thus the proof of \eqref{3} is complete.]
\end{vershort}

\begin{verlong}
[\textit{Proof of \eqref{3}.} Let $s\in\left\{  0,1,\ldots,n-1\right\}  $.
Thus, we are in one of the following two cases:

\textit{Case 1:} We have $s<k$.

\textit{Case 2:} We have $s\geq k$.

Let us first consider Case 1. In this case, we have $s<k$. Hence, $s\leq k-1$
(since $s$ and $k$ are integers) and thus $s\in\left\{  0,1,\ldots
,k-1\right\}  $ (since $s\in\left\{  0,1,\ldots,n-1\right\}  $).

For each $i\in\left\{  0,1,\ldots,n-k\right\}  $, we have $i\leq n-k$ and thus
$\underbrace{i}_{\leq n-k}+\underbrace{s}_{\leq k-1}\leq\left(  n-k\right)
+\left(  k-1\right)  =n-1$ and therefore $i+s\in\left\{  0,1,\ldots
,n-1\right\}  $ (since $\underbrace{i}_{\geq0}+\underbrace{s}_{\geq0}\geq0$)
and therefore%
\[
v^{i+s}\in\left\{  v^{0},v^{1},\ldots,v^{n-1}\right\}  \subseteq\left\langle
v^{0},v^{1},\ldots,v^{n-1}\right\rangle _{A}=U.
\]
Hence, $\sum\limits_{i=0}^{n-k}a_{i+k}v^{i+s}$ is an $A$-linear combination of
elements of the set $U$ (since the coefficients $a_{i+k}$ belong to $A$) and
therefore belongs to $U$ itself (since $U$ is closed under $A$-linear
combination). In other words, $\sum\limits_{i=0}^{n-k}a_{i+k}v^{i+s}\in U$.

Now, from $u=\sum\limits_{i=0}^{n-k}a_{i+k}v^{i}$, we obtain
\[
uv^{s}=\sum\limits_{i=0}^{n-k}a_{i+k}\underbrace{v^{i}\cdot v^{s}}_{=v^{i+s}%
}=\sum\limits_{i=0}^{n-k}a_{i+k}v^{i+s}\in U.
\]
Hence, \eqref{3} is proven in Case 1.

Let us next consider Case 2. In this case, we have $s\geq k$. Hence,
$s-k\geq0$. Also, $s\leq n-1$ (since $s\in\left\{  0,1,\ldots,n-1\right\}  $).

For each $i\in\left\{  0,1,\ldots,k-1\right\}  $, we have $i\geq0$ and $i\leq
k-1$ and thus $\underbrace{i}_{\leq k-1}+\left(  s-k\right)  \leq\left(
k-1\right)  +\left(  s-k\right)  =s-1\leq s\leq n-1$ and therefore $i+\left(
s-k\right)  \in\left\{  0,1,\ldots,n-1\right\}  $ (since $\underbrace{i}%
_{\geq0}+\left(  s-k\right)  \geq s-k\geq0$) and thus%
\[
v^{i+\left(  s-k\right)  }\in\left\{  v^{0},v^{1},\ldots,v^{n-1}\right\}
\subseteq\left\langle v^{0},v^{1},\ldots,v^{n-1}\right\rangle _{A}=U.
\]
Hence, $\sum\limits_{i=0}^{k-1}a_{i}v^{i+\left(  s-k\right)  }$ is an
$A$-linear combination of elements of the set $U$ (since the coefficients
$a_{i}$ belong to $A$) and therefore belongs to $U$ itself (since $U$ is
closed under $A$-linear combination). In other words, $\sum\limits_{i=0}%
^{k-1}a_{i}v^{i+\left(  s-k\right)  }\in U$.

From $s-k\geq0$, we obtain $s-k\in\mathbb{N}$ and thus $v^{s}=v^{k+\left(
s-k\right)  }=v^{k}v^{s-k}$. Hence,%
\begin{align*}
uv^{s}  &  =uv^{k}v^{s-k}=v^{k}u\cdot v^{s-k}=-\sum\limits_{i=0}^{k-1}%
a_{i}\underbrace{v^{i}\cdot v^{s-k}}_{=v^{i+\left(  s-k\right)  }%
}\ \ \ \ \ \ \ \ \ \ \left(  \text{since }v^{k}u=-\sum\limits_{i=0}^{k-1}%
a_{i}v^{i}\right) \\
&  =-\underbrace{\sum\limits_{i=0}^{k-1}a_{i}v^{i+\left(  s-k\right)  }}_{\in
U}\in-U\subseteq U\ \ \ \ \ \ \ \ \ \ \left(  \text{since }U\text{ is an
}A\text{-module}\right)  .
\end{align*}
Hence, \eqref{3} is proven in Case 2.

Hence, in both cases, we have proven \eqref{3}. This completes the proof of \eqref{3}.]
\end{verlong}

\begin{vershort}
Now, from $U=\left\langle v^{0},v^{1},\ldots,v^{n-1}\right\rangle _{A}$, we
obtain%
\begin{align*}
uU  &  =u\left\langle v^{0},v^{1},\ldots,v^{n-1}\right\rangle _{A}%
=\left\langle uv^{0},uv^{1},\ldots,uv^{n-1}\right\rangle _{A}\\
&  =\left\langle uv^{s}\ \mid\ s\in\left\{  0,1,\ldots,n-1\right\}
\right\rangle _{A}\subseteq U\ \ \ \ \ \ \ \ \ \ \left(  \text{by \eqref{3}
and Lemma \ref{lem.mil}}\right)  .
\end{align*}

\end{vershort}

\begin{verlong}
Thus we know that $uv^{s}\in U$ for every $s\in\left\{  0,1,\ldots
,n-1\right\}  $. Hence, Lemma \ref{lem.mil} (applied to $M=B$, $N=U$,
$S=\left\{  0,1,\ldots,n-1\right\}  $ and $m_{s}=uv^{s}$) yields%
\[
\left\langle uv^{s}\ \mid\ s\in\left\{  0,1,\ldots,n-1\right\}  \right\rangle
_{A}\subseteq U.
\]

Now, from $U=\left\langle v^{0},v^{1},\ldots,v^{n-1}\right\rangle _{A}$, we
obtain%
\begin{align*}
uU  &  =u\left\langle v^{0},v^{1},\ldots,v^{n-1}\right\rangle _{A}%
=\left\langle uv^{0},uv^{1},\ldots,uv^{n-1}\right\rangle _{A}\\
&  =\left\langle uv^{s}\ \mid\ s\in\left\{  0,1,\ldots,n-1\right\}
\right\rangle _{A}\subseteq U.
\end{align*}

\end{verlong}

Altogether, $U$ is an $n$-generated $A$-submodule of $B$ such that $1\in U$
and $uU\subseteq U$. Thus, $u\in B$ satisfies Assertion $\mathcal{C}$ of
Theorem~\ref{Theorem1}. Hence, $u\in B$ satisfies the four equivalent
assertions $\mathcal{A}$, $\mathcal{B}$, $\mathcal{C}$ and $\mathcal{D}$ of
Theorem~\ref{Theorem1}. Consequently, $u$ is $n$-integral over $A$. Since
$u=\sum\limits_{i=0}^{n-k}a_{i+k}v^{i}$, this means that $\sum\limits_{i=0}%
^{n-k}a_{i+k}v^{i}$ is $n$-integral over $A$. This proves
Theorem~\ref{Theorem2}.
\end{proof}

\begin{corollary}
\label{Corollary3} Let $A$ be a ring. Let $B$ be an $A$-algebra. Let
$\alpha\in\mathbb{N}$ and $\beta\in\mathbb{N}$ be such that $\alpha+\beta
\in\mathbb{N}^{+}$. Let $u\in B$ and $v\in B$. Let $s_{0},s_{1},\ldots
,s_{\alpha}$ be $\alpha+1$ elements of $A$ such that $\sum\limits_{i=0}%
^{\alpha}s_{i}v^{i}=u$. Let $t_{0},t_{1},\ldots,t_{\beta}$ be $\beta+1$
elements of $A$ such that $\sum\limits_{i=0}^{\beta}t_{i}v^{\beta-i}%
=uv^{\beta}$. Then, $u$ is $\left(  \alpha+\beta\right)  $-integral over $A$.
\end{corollary}

(This Corollary~\ref{Corollary3} generalizes \cite[Exercise 2-5]{1}, which
says that if $v$ is an invertible element of an $A$-algebra $B$, then every
element $u\in A\left[  v\right]  \cap A\left[  v^{-1}\right]  $ is integral
over $A$. To see how this follows from Corollary~\ref{Corollary3}, just pick
$\alpha\in\mathbb{N}^{+}$ and $\beta\in\mathbb{N}^{+}$ and $s_{0},s_{1}%
,\ldots,s_{\alpha}\in A$ and $t_{0},t_{1},\ldots,t_{\beta}\in A$ such that
$\sum\limits_{i=0}^{\alpha}s_{i}v^{i}=u$ and $\sum\limits_{i=0}^{\beta}%
t_{i}\left(  v^{-1}\right)  ^{i}=u$.)

\begin{proof}
[First proof of Corollary~\ref{Corollary3}.]Let $k=\beta$ and $n=\alpha+\beta
$. Then, $k\in\left\{  0,1,\ldots,n\right\}  $ (since $\alpha\in\mathbb{N}$
and $\beta\in\mathbb{N}$) and $n=\alpha+\beta\in\mathbb{N}^{+}$ and
$n-\beta=\alpha$ (since $n=\alpha+\beta$). Define $n+1$ elements $a_{0}%
,a_{1},\ldots,a_{n}$ of $A$ by%
\[
a_{i}=%
\begin{cases}
t_{\beta-i}, & \text{ if }i<\beta;\\
t_{0}-s_{0}, & \text{ if }i=\beta;\\
-s_{i-\beta}, & \text{ if }i>\beta
\end{cases}
\ \ \ \ \ \ \ \ \ \ \text{for every }i\in\left\{  0,1,\ldots,n\right\}  .
\]

\begin{vershort}
Then, from $n=\alpha+\beta$, we obtain%
\begin{align*}
\sum\limits_{i=0}^{n}a_{i}v^{i}  &  =\sum\limits_{i=0}^{\alpha+\beta}%
a_{i}v^{i}=\sum\limits_{i=0}^{\beta-1}\underbrace{a_{i}}_{=t_{\beta-i}}%
v^{i}+\underbrace{a_{\beta}}_{=t_{0}-s_{0}}v^{\beta}+\sum\limits_{i=\beta
+1}^{\alpha+\beta}\underbrace{a_{i}}_{\substack{=-s_{i-\beta}}}v^{i}\\
&  =\underbrace{\sum\limits_{i=0}^{\beta-1}t_{\beta-i}v^{i}}_{\substack{=\sum
\limits_{i=1}^{\beta}t_{i}v^{\beta-i}\\\text{(here, we have substituted
}i\\\text{for }\beta-i\text{ in the sum)}}}+\underbrace{\left(  t_{0}%
-s_{0}\right)  v^{\beta}}_{=t_{0}v^{\beta}-s_{0}v^{\beta}}+\underbrace{\sum
\limits_{i=\beta+1}^{\alpha+\beta}\left(  -s_{i-\beta}\right)  v^{i}%
}_{\substack{=-\sum\limits_{i=\beta+1}^{\alpha+\beta}s_{i-\beta}v^{i}%
=-\sum\limits_{i=1}^{\alpha}s_{i}v^{i+\beta}\\\text{(here, we have substituted
}i\\\text{for }i-\beta\text{ in the sum)}}}\\
&  =\sum\limits_{i=1}^{\beta}t_{i}v^{\beta-i}+t_{0}v^{\beta}-s_{0}v^{\beta
}-\sum\limits_{i=1}^{\alpha}s_{i}v^{i+\beta}\\
&  =\underbrace{\sum\limits_{i=1}^{\beta}t_{i}v^{\beta-i}+t_{0}v^{\beta}%
}_{=\sum\limits_{i=0}^{\beta}t_{i}v^{\beta-i}=uv^{\beta}}-\underbrace{\left(
s_{0}v^{\beta}+\sum\limits_{i=1}^{\alpha}s_{i}v^{i+\beta}\right)
}_{\substack{=\sum\limits_{i=0}^{\alpha}s_{i}v^{i+\beta}=\left(
\sum\limits_{i=0}^{\alpha}s_{i}v^{i}\right)  v^{\beta}=uv^{\beta
}\\\text{(since }\sum\limits_{i=0}^{\alpha}s_{i}v^{i}=u\text{)}}}=uv^{\beta
}-uv^{\beta}=0.
\end{align*}

\end{vershort}

\begin{verlong}
Then, from $n=\alpha+\beta$, we obtain%
\begin{align*}
\sum\limits_{i=0}^{n}a_{i}v^{i}  &  =\sum\limits_{i=0}^{\alpha+\beta}%
a_{i}v^{i}=\sum\limits_{i=0}^{\beta-1}\underbrace{a_{i}}_{\substack{=t_{\beta
-i}\\\text{(by the}\\\text{definition of }a_{i}\text{,}\\\text{since }%
i<\beta\text{)}}}v^{i}+\sum\limits_{i=\beta}^{\beta}\underbrace{a_{i}%
}_{\substack{=t_{0}-s_{0}\\\text{(by the}\\\text{definition of }a_{i}%
\text{,}\\\text{since }i=\beta\text{)}}}v^{i}+\sum\limits_{i=\beta+1}%
^{\alpha+\beta}\underbrace{a_{i}}_{\substack{=-s_{i-\beta}\\\text{(by
the}\\\text{definition of }a_{i}\text{,}\\\text{since }i>\beta\text{)}}%
}v^{i}\\
&  =\sum\limits_{i=0}^{\beta-1}t_{\beta-i}v^{i}+\underbrace{\sum
\limits_{i=\beta}^{\beta}\left(  t_{0}-s_{0}\right)  v^{i}}%
_{\substack{=\left(  t_{0}-s_{0}\right)  v^{\beta}\\=t_{0}v^{\beta}%
-s_{0}v^{\beta}}}+\underbrace{\sum\limits_{i=\beta+1}^{\alpha+\beta}\left(
-s_{i-\beta}\right)  v^{i}}_{=-\sum\limits_{i=\beta+1}^{\alpha+\beta
}s_{i-\beta}v^{i}}\\
&  =\sum\limits_{i=0}^{\beta-1}t_{\beta-i}v^{i}+t_{0}v^{\beta}-s_{0}v^{\beta
}-\sum\limits_{i=\beta+1}^{\alpha+\beta}s_{i-\beta}v^{i}\\
&  =\sum\limits_{i=0}^{\beta-1}t_{\beta-i}v^{i}+t_{0}v^{\beta}-\left(
s_{0}v^{\beta}+\sum\limits_{i=\beta+1}^{\alpha+\beta}s_{i-\beta}v^{i}\right)
\\
&  =\sum\limits_{i=0}^{\beta-1}t_{\beta-i}v^{i}+t_{0}v^{\beta}-\left(
s_{0}v^{\beta}+\sum\limits_{i=1}^{\alpha}\underbrace{s_{\left(  i+\beta
\right)  -\beta}}_{=s_{i}}\underbrace{v^{i+\beta}}_{=v^{i}v^{\beta}}\right) \\
&  \ \ \ \ \ \ \ \ \ \ \left(  \text{here, we substituted }i+\beta\text{ for
}i\text{ in the second sum}\right) \\
&  =\sum\limits_{i=0}^{\beta-1}t_{\beta-i}v^{i}+t_{0}v^{\beta}-\left(
s_{0}v^{\beta}+\sum\limits_{i=1}^{\alpha}s_{i}v^{i}v^{\beta}\right) \\
&  =\sum\limits_{i=1}^{\beta}\underbrace{t_{\beta-\left(  \beta-i\right)  }%
}_{=t_{i}}v^{\beta-i}+t_{0}\underbrace{v^{\beta}}_{=v^{\beta-0}}-\left(
s_{0}\underbrace{v^{\beta}}_{=v^{0}v^{\beta}}+\sum\limits_{i=1}^{\alpha}%
s_{i}v^{i}v^{\beta}\right) \\
&  \ \ \ \ \ \ \ \ \ \ \left(  \text{here, we substituted }\beta-i\text{ for
}i\text{ in the first sum}\right) \\
&  =\sum\limits_{i=1}^{\beta}t_{i}v^{\beta-i}+t_{0}v^{\beta-0}-\left(
s_{0}v^{0}v^{\beta}+\sum\limits_{i=1}^{\alpha}s_{i}v^{i}v^{\beta}\right) \\
&  =\underbrace{\sum\limits_{i=1}^{\beta}t_{i}v^{\beta-i}+t_{0}v^{\beta-0}%
}_{=\sum\limits_{i=0}^{\beta}t_{i}v^{\beta-i}=uv^{\beta}}-\left(
\underbrace{s_{0}v^{0}+\sum\limits_{i=1}^{\alpha}s_{i}v^{i}}_{=\sum
\limits_{i=0}^{\alpha}s_{i}v^{i}=u}\right)  v^{\beta}=uv^{\beta}-uv^{\beta}=0.
\end{align*}

\end{verlong}

Thus, Theorem~\ref{Theorem2} yields that $\sum\limits_{i=0}^{n-k}a_{i+k}v^{i}$
is $n$-integral over $A$.

\begin{vershort}
But $k=\beta$ and thus%
\begin{align*}
\sum\limits_{i=0}^{n-k}a_{i+k}v^{i}  &  =\sum\limits_{i=0}^{n-\beta}%
a_{i+\beta}v^{i}=\sum\limits_{i=\beta}^{n}a_{i}v^{i-\beta}%
\ \ \ \ \ \ \ \ \ \ \left(
\begin{array}
[c]{c}%
\text{here, we have substituted }i-\beta\\
\text{for }i\text{ in the sum}%
\end{array}
\right) \\
&  =\sum\limits_{i=\beta}^{\beta}\underbrace{a_{i}}_{\substack{=t_{0}%
-s_{0}\\\text{(by the}\\\text{definition of }a_{i}\text{,}\\\text{since
}i=\beta\text{)}}}v^{i-\beta}+\sum\limits_{i=\beta+1}^{n}\underbrace{a_{i}%
}_{\substack{=-s_{i-\beta}\\\text{(by the}\\\text{definition of }a_{i}%
\text{,}\\\text{since }i>\beta\text{)}}}v^{i-\beta}\\
&  =\underbrace{\sum\limits_{i=\beta}^{\beta}\left(  t_{0}-s_{0}\right)
v^{i-\beta}}_{\substack{_{\substack{=\left(  t_{0}-s_{0}\right)
v^{\beta-\beta}}}\\=\left(  t_{0}-s_{0}\right)  v^{0}\\=t_{0}v^{0}-s_{0}v^{0}%
}}+\underbrace{\sum\limits_{i=\beta+1}^{n}\left(  -s_{i-\beta}\right)
v^{i-\beta}}_{\substack{=-\sum\limits_{i=\beta+1}^{n}s_{i-\beta}v^{i-\beta
}=-\sum\limits_{i=1}^{n-\beta}s_{i}v^{i}\\\text{(here, we have substituted
}i\\\text{for }i-\beta\text{ in the sum)}}}\\
&  =t_{0}v^{0}-s_{0}v^{0}-\sum\limits_{i=1}^{n-\beta}s_{i}v^{i}=t_{0}%
v^{0}-s_{0}v^{0}-\sum\limits_{i=1}^{\alpha}s_{i}v^{i}%
\ \ \ \ \ \ \ \ \ \ \left(  \text{since }n-\beta=\alpha\right) \\
&  =t_{0}\underbrace{v^{0}}_{=1_{B}}-\underbrace{\left(  s_{0}v^{0}%
+\sum\limits_{i=1}^{\alpha}s_{i}v^{i}\right)  }_{=\sum\limits_{i=0}^{\alpha
}s_{i}v^{i}=u}=t_{0}\cdot1_{B}-u.
\end{align*}

\end{vershort}

\begin{verlong}
But $k=\beta$ and thus%
\begin{align*}
\sum\limits_{i=0}^{n-k}a_{i+k}v^{i}  &  =\sum\limits_{i=0}^{n-\beta}%
a_{i+\beta}v^{i}=\sum\limits_{i=\beta}^{n}\underbrace{a_{\left(
i-\beta\right)  +\beta}}_{=a_{i}}v^{i-\beta}\\
&  \ \ \ \ \ \ \ \ \ \ \left(  \text{here, we have substituted }i-\beta\text{
for }i\text{ in the sum}\right) \\
&  =\sum\limits_{i=\beta}^{n}a_{i}v^{i-\beta}=\sum\limits_{i=\beta}^{\beta
}\underbrace{a_{i}}_{\substack{=t_{0}-s_{0}\\\text{(by the}\\\text{definition
of }a_{i}\text{,}\\\text{since }i=\beta\text{)}}}v^{i-\beta}+\sum
\limits_{i=\beta+1}^{n}\underbrace{a_{i}}_{\substack{=-s_{i-\beta}\\\text{(by
the}\\\text{definition of }a_{i}\text{,}\\\text{since }i>\beta\text{)}%
}}v^{i-\beta}\\
&  =\underbrace{\sum\limits_{i=\beta}^{\beta}\left(  t_{0}-s_{0}\right)
v^{i-\beta}}_{\substack{=\left(  t_{0}-s_{0}\right)  v^{\beta-\beta}%
}}+\underbrace{\sum\limits_{i=\beta+1}^{n}\left(  -s_{i-\beta}\right)
v^{i-\beta}}_{\substack{=\sum\limits_{i=1}^{n-\beta}\left(  -s_{i}\right)
v^{i}\\\text{(here, we have substituted }i\\\text{for }i-\beta\text{ in the
sum)}}}\\
&  =\left(  t_{0}-s_{0}\right)  \underbrace{v^{\beta-\beta}}_{=v^{0}%
}+\underbrace{\sum\limits_{i=1}^{n-\beta}\left(  -s_{i}\right)  v^{i}%
}_{\substack{=-\sum\limits_{i=1}^{n-\beta}s_{i}v^{i}=-\sum\limits_{i=1}%
^{\alpha}s_{i}v^{i}\\\text{(since }n-\beta=\alpha\text{)}}%
}=\underbrace{\left(  t_{0}-s_{0}\right)  v^{0}}_{=t_{0}v^{0}-s_{0}v^{0}%
}+\left(  -\sum\limits_{i=1}^{\alpha}s_{i}v^{i}\right) \\
&  =t_{0}v^{0}-s_{0}v^{0}+\left(  -\sum\limits_{i=1}^{\alpha}s_{i}%
v^{i}\right)  =t_{0}\underbrace{v^{0}}_{=1_{B}}-\underbrace{\left(  s_{0}%
v^{0}+\sum\limits_{i=1}^{\alpha}s_{i}v^{i}\right)  }_{=\sum\limits_{i=0}%
^{\alpha}s_{i}v^{i}=u}\\
&  =t_{0}\cdot1_{B}-u.
\end{align*}

\end{verlong}

Thus, $t_{0}\cdot1_{B}-u$ is $n$-integral over $A$ (since $\sum\limits_{i=0}%
^{n-k}a_{i+k}v^{i}$ is $n$-integral over $A$). Thus,
Corollary~\ref{cor.-x-int} (applied to $x=t_{0}\cdot1_{B}-u$ and $m=n$) shows
that $-\left(  t_{0}\cdot1_{B}-u\right)  $ is $n$-integral over $A$. In other
words, $u-t_{0}\cdot1_{B}$ is $n$-integral over $A$ (since $-\left(
t_{0}\cdot1_{B}-u\right)  =u-t_{0}\cdot1_{B}$).

On the other hand, $t_{0}\cdot1_{B}$ is $1$-integral over $A$ (by
Theorem~\ref{Theorem5a}, applied to $a=t_{0}$). Thus, $t_{0}\cdot1_{B}+\left(
u-t_{0}\cdot1_{B}\right)  $ is $n\cdot1$-integral over $A$ (by
Theorem~\ref{Theorem5bc} \textbf{(a)}, applied to $x=t_{0}\cdot1_{B}$,
$y=u-t_{0}\cdot1_{B}$ and $m=1$). In other words, $u$ is $\left(  \alpha
+\beta\right)  $-integral over $A$ (since $t_{0}\cdot1_{B}+\left(
u-t_{0}\cdot1_{B}\right)  =u$ and $n\cdot1=n=\alpha+\beta$). This proves
Corollary~\ref{Corollary3}.
\end{proof}

We will provide a second proof of Corollary~\ref{Corollary3} in Section
\ref{sect.5}.

\begin{corollary}
\label{Corollary6} Let $A$ be a ring. Let $B$ be an $A$-algebra. Let
$n\in\mathbb{N}^{+}$ and $m\in\mathbb{N}$. Let $v\in B$. Let $b_{0}%
,b_{1},\ldots,b_{n-1}$ be $n$ elements of $A$, and let $u=\sum\limits_{i=0}%
^{n-1}b_{i}v^{i}$. Assume that $vu$ is $m$-integral over $A$. Then, $u$ is
$nm$-integral over $A$.
\end{corollary}

Corollary~\ref{Corollary6} generalizes a folklore fact about integrality,
which states that if $B$ is an $A$-algebra, and if an invertible $v\in B$
satisfies $v^{-1}\in A\left[  v\right]  $, then $v$ is integral over $A$.
(Indeed, this latter fact follows from Corollary~\ref{Corollary6} by setting
$u=v^{-1}$.)

\begin{proof}
[Proof of Corollary~\ref{Corollary6}.]Define $n+1$ elements $a_{0}%
,a_{1},\ldots,a_{n}$ of $A\left[  vu\right]  $ by%
\[
a_{i}=%
\begin{cases}
-vu, & \text{ if }i=0;\\
b_{i-1}\cdot1_{B}, & \text{ if }i>0
\end{cases}
\ \ \ \ \ \ \ \ \ \ \text{for every }i\in\left\{  0,1,\ldots,n\right\}  .
\]
(These are well-defined, since every positive $i\in\left\{  0,1,\ldots
,n\right\}  $ satisfies $i\in\left\{  1,2,\ldots,n\right\}  $ and thus
$i-1\in\left\{  0,1,\ldots,n-1\right\}  $ and thus $b_{i-1}\in A$ and
therefore $b_{i-1}\cdot1_{B}\in A\cdot1_{B}\subseteq A\left[  vu\right]  $.)

The definition of $a_{0}$ yields $a_{0}=-vu$. Also,%
\begin{align*}
\sum\limits_{i=0}^{n}a_{i}v^{i}  &  =\underbrace{a_{0}}_{=-vu}%
\underbrace{v^{0}}_{=1}+\sum\limits_{i=1}^{n}\underbrace{a_{i}}%
_{\substack{=b_{i-1}\cdot1_{B}\\\text{(by the definition}\\\text{of }%
a_{i}\text{, since }i>0\text{)}}}v^{i}=-vu+\sum\limits_{i=1}^{n}b_{i-1}%
\cdot\underbrace{1_{B}v^{i}}_{=v^{i}=v^{i-1}v}\\
&  =-vu+\sum\limits_{i=1}^{n}b_{i-1}v^{i-1}v=-vu+\underbrace{\sum
\limits_{i=0}^{n-1}b_{i}v^{i}}_{=u}v\\
&  \ \ \ \ \ \ \ \ \ \ \left(  \text{here, we substituted }i\text{ for
}i-1\text{\ in the sum}\right) \\
&  =-vu+uv=0.
\end{align*}

Let $k=1$. Then, $k=1\in\left\{  0,1,\ldots,n\right\}  $ (since $n\in
\mathbb{N}^{+}$).

Now, $A\left[  vu\right]  $ is a subring of $B$; hence, $B$ is an $A\left[
vu\right]  $-algebra. The $n+1$ elements $a_{0},a_{1},\ldots,a_{n}$ of
$A\left[  vu\right]  $ satisfy $\sum\limits_{i=0}^{n}a_{i}v^{i}=0$.

Hence, Theorem~\ref{Theorem2} (applied to the ring $A\left[  vu\right]  $ in
lieu of $A$) yields that $\sum\limits_{i=0}^{n-k}a_{i+k}v^{i}$ is $n$-integral
over $A\left[  vu\right]  $. But from $k=1$, we obtain%
\[
\sum\limits_{i=0}^{n-k}a_{i+k}v^{i}=\sum\limits_{i=0}^{n-1}\underbrace{a_{i+1}%
}_{\substack{=b_{\left(  i+1\right)  -1}\cdot1_{B}\\\text{(by the
definition}\\\text{of }a_{i+1}\text{, since }i+1>0\text{)}}}v^{i}%
=\sum\limits_{i=0}^{n-1}\underbrace{b_{\left(  i+1\right)  -1}}_{=b_{i}}%
\cdot\underbrace{1_{B}v^{i}}_{=v^{i}}=\sum\limits_{i=0}^{n-1}b_{i}v^{i}=u.
\]
Hence, $u$ is $n$-integral over $A\left[  vu\right]  $ (since $\sum
\limits_{i=0}^{n-k}a_{i+k}v^{i}$ is $n$-integral over $A\left[  vu\right]  $).
But $vu$ is $m$-integral over $A$. Thus, Theorem~\ref{Theorem4} (applied to
$vu$ in lieu of $v$) yields that $u$ is $nm$-integral over $A$. This proves
Corollary~\ref{Corollary6}.
\end{proof}

\section{\label{sect.2}Integrality over ideal semifiltrations}

\subsection{Definitions of ideal semifiltrations and integrality over them}

We now set our sights at a more general notion of integrality.

\begin{definition}
\label{Definition6} Let $A$ be a ring, and let $\left(  I_{\rho}\right)
_{\rho\in\mathbb{N}}$ be a sequence of ideals of $A$. Then, $\left(  I_{\rho
}\right)  _{\rho\in\mathbb{N}}$ is called an \textit{ideal semifiltration} of
$A$ if and only if it satisfies the two conditions%
\begin{align*}
I_{0}  &  =A;\\
I_{a}I_{b}  &  \subseteq I_{a+b}\ \ \ \ \ \ \ \ \ \ \text{for every }%
a\in\mathbb{N}\text{ and }b\in\mathbb{N}.
\end{align*}

\end{definition}

Two simple examples of ideal semifiltrations can easily be constructed from
any ideal:

\begin{example}
\label{exa.semifil.powers}Let $A$ be a ring. Let $I$ be an ideal of $A$. Then:

\textbf{(a)} The sequence $\left(  I^{\rho}\right)  _{\rho\in\mathbb{N}}$ is
an ideal semifiltration of $A$. (Here, $I^{\rho}$ denotes the $\rho$-th power
of $I$ in the multiplicative monoid of ideals of $A$; in particular, $I^{0}=A$.)

\textbf{(b)} The sequence $\left(  A,I,I,I,\ldots\right)  =\left(
\begin{cases}
A, & \text{if }\rho=0;\\
I, & \text{if }\rho>0
\end{cases}
\right)  _{\rho\in\mathbb{N}}$ is an ideal semifiltration of $A$.
\end{example}

\begin{proof}
[Proof of Example \ref{exa.semifil.powers}.]This is a straightforward exercise
in checking axioms.
\end{proof}

\begin{definition}
\label{Definition9} Let $A$ be a ring. Let $B$ be an $A$-algebra. Let $\left(
I_{\rho}\right)  _{\rho\in\mathbb{N}}$ be an ideal semifiltration of $A$. Let
$n\in\mathbb{N}$. Let $u\in B$.

We say that the element $u$ of $B$ is $n$\textit{-integral over }$\left(
A,\left(  I_{\rho}\right)  _{\rho\in\mathbb{N}}\right)  $ if there exists some
$\left(  a_{0},a_{1},\ldots,a_{n}\right)  \in A^{n+1}$ such that%
\[
\sum\limits_{k=0}^{n}a_{k}u^{k}=0,\ \ \ \ \ \ \ \ \ \ a_{n}%
=1,\ \ \ \ \ \ \ \ \ \ \text{and}\ \ \ \ \ \ \ \ \ \ a_{i}\in I_{n-i}\text{
for every }i\in\left\{  0,1,\ldots,n\right\}  .
\]

\end{definition}

This definition generalizes \cite[Definition 1.1.1]{2} in multiple ways.
Indeed, if $I$ is an ideal of a ring $A$, and if $u\in A$ and $n\in\mathbb{N}%
$, then $u$ is $n$-integral over $\left(  A,\left(  I_{\rho}\right)  _{\rho
\in\mathbb{N}}\right)  $ (here, $\left(  I^{\rho}\right)  _{\rho\in\mathbb{N}%
}$ is the ideal semifiltration from Example \ref{exa.semifil.powers}
\textbf{(a)}) if and only if there is an equation of integral dependence of
$u$ over $I$ (in the sense of \cite[Definition 1.1.1]{2}).

We further notice that integrality over an ideal semifiltration of a ring $A$
is a stronger claim than integrality over $A$:

\begin{proposition}
\label{prop.semifil.stronger-than-ring}Let $A$ be a ring. Let $B$ be an
$A$-algebra. Let $\left(  I_{\rho}\right)  _{\rho\in\mathbb{N}}$ be an ideal
semifiltration of $A$. Let $n\in\mathbb{N}$. Let $u\in B$ be such that $u$ is
$n$-integral over $\left(  A,\left(  I_{\rho}\right)  _{\rho\in\mathbb{N}%
}\right)  $. Then, $u$ is $n$-integral over $A$.
\end{proposition}

\begin{proof}
[Proof of Proposition \ref{prop.semifil.stronger-than-ring}.]We know that $u$
is $n$-integral over $\left(  A,\left(  I_{\rho}\right)  _{\rho\in\mathbb{N}%
}\right)  $. Thus, by Definition~\ref{Definition9}, there exists some $\left(
a_{0},a_{1},\ldots,a_{n}\right)  \in A^{n+1}$ such that%
\[
\sum\limits_{k=0}^{n}a_{k}u^{k}=0,\ \ \ \ \ \ \ \ \ \ a_{n}%
=1,\ \ \ \ \ \ \ \ \ \ \text{and}\ \ \ \ \ \ \ \ \ \ a_{i}\in I_{n-i}\text{
for every }i\in\left\{  0,1,\ldots,n\right\}  .
\]
Consider this $\left(  a_{0},a_{1},\ldots,a_{n}\right)  $.

\begin{vershort}
Thus, there exists a monic polynomial $P\in A\left[  X\right]  $ with $\deg
P=n$ and $P\left(  u\right)  =0$ (namely, $P\left(  X\right)  =\sum
\limits_{k=0}^{n}a_{k}X^{k}$).
\end{vershort}

\begin{verlong}
For each $k\in\left\{  0,1,\ldots,n\right\}  $, we have $a_{k}\in I_{n-k}$
(since $a_{i}\in I_{n-i}$ for every $i\in\left\{  0,1,\ldots,n\right\}  $) and
therefore $a_{k}\in I_{n-k}\subseteq A$. Thus, we can define a polynomial
$P\in A\left[  X\right]  $ by $P\left(  X\right)  =\sum\limits_{k=0}^{n}%
a_{k}X^{k}$. Consider this $P$. This polynomial $P$ satisfies $\deg P\leq n$,
and its coefficient before $X^{n}$ is $a_{n}=1$. Hence, this polynomial $P$ is
monic and satisfies $\deg P=n$. Also, by substituting $u$ for $X$ in $P\left(
X\right)  =\sum\limits_{k=0}^{n}a_{k}X^{k}$, we obtain $P\left(  u\right)
=\sum\limits_{k=0}^{n}a_{k}u^{k}=0$. Hence, we have found a monic polynomial
$P\in A\left[  X\right]  $ with $\deg P=n$ and $P\left(  u\right)  =0$.
\end{verlong}

In other words, $u$ is $n$-integral over $A$. This proves Proposition
\ref{prop.semifil.stronger-than-ring}.
\end{proof}

We leave it to the reader to prove the following simple fact, which shows that
nilpotency is an instance of integrality over ideal semifiltrations:

\begin{proposition}
\label{prop.semifil.int-0}Let $A$ be a ring. Let $0A$ be the zero ideal of
$A$. Let $n\in\mathbb{N}$. Let $u\in A$. Then, the element $u$ of $A$ is
$n$-integral over $\left(  A,\left(  \left(  0A\right)  ^{\rho}\right)
_{\rho\in\mathbb{N}}\right)  $ if and only if $u^{n}=0$.
\end{proposition}

\subsection{Polynomial rings and Rees algebras}

In order to study integrality over ideal semifiltrations, we shall now
introduce the concept of a Rees algebra -- a subalgebra of a polynomial ring
that conveniently encodes an ideal semifiltration of the base ring. This,
again, generalizes a classical notion for ideals (namely, the Rees algebra of
an ideal -- see \cite[Definition 5.1.1]{2}). First, we recall a basic fact:

\begin{definition}
\label{Definition7} Let $A$ be a ring. Let $B$ be an $A$-algebra. Then, there
is a canonical ring homomorphism $\iota:A\rightarrow B$ that sends each $a\in
A$ to $a\cdot1_{B}\in B$. This ring homomorphism $\iota$ induces a canonical
ring homomorphism $\iota\left[  Y\right]  :A\left[  Y\right]  \rightarrow
B\left[  Y\right]  $ between the polynomial rings $A\left[  Y\right]  $ and
$B\left[  Y\right]  $ that sends each polynomial $\sum\limits_{i=0}^{m}%
a_{i}Y^{i}\in A\left[  Y\right]  $ (with $m\in\mathbb{N}$ and $\left(
a_{0},a_{1},\ldots,a_{m}\right)  \in A^{m+1}$) to the polynomial
$\sum\limits_{i=0}^{m}\iota\left(  a_{i}\right)  Y^{i}\in B\left[  Y\right]
$. Thus, the polynomial ring $B\left[  Y\right]  $ becomes an $A\left[
Y\right]  $-algebra.
\end{definition}

\begin{definition}
\label{Definition8} Let $A$ be a ring, and let $\left(  I_{\rho}\right)
_{\rho\in\mathbb{N}}$ be an ideal semifiltration of $A$. Thus, $I_{0}%
,I_{1},I_{2},\ldots$ are ideals of $A$, and we have $I_{0}=A$.

Consider the polynomial ring $A\left[  Y\right]  $. For each $i\in\mathbb{N}$,
the subset $I_{i}Y^{i}$ of $A\left[  Y\right]  $ is an $A$-submodule of the
$A$-algebra $A\left[  Y\right]  $ (since $I_{i}$ is an ideal of $A$). Hence,
the sum $\sum\limits_{i\in\mathbb{N}}I_{i}Y^{i}$ of these $A$-submodules must
also be an $A$-submodule of the $A$-algebra $A\left[  Y\right]  $.

Let $A\left[  \left(  I_{\rho}\right)  _{\rho\in\mathbb{N}}\ast Y\right]  $
denote this $A$-submodule $\sum\limits_{i\in\mathbb{N}}I_{i}Y^{i}$ of the
$A$-algebra $A\left[  Y\right]  $. Then,%
\begin{align*}
A\left[  \left(  I_{\rho}\right)  _{\rho\in\mathbb{N}}\ast Y\right]   &
=\sum\limits_{i\in\mathbb{N}}I_{i}Y^{i}\\
&  =\left\{  \sum_{i\in\mathbb{N}}a_{i}Y^{i}\ \mid\ \left(  a_{i}\in
I_{i}\text{ for all }i\in\mathbb{N}\right)  \text{, }\right. \\
&  \ \ \ \ \ \ \ \ \ \ \left.  \text{and }\left(  \text{only finitely many
}i\in\mathbb{N}\text{ satisfy }a_{i}\neq0\right)
\vphantom{\sum_{i\in\mathbb{N}}}\right\} \\
&  =\left\{  P\in A\left[  Y\right]  \ \mid\ \text{the }i\text{-th coefficient
of the polynomial }P\right. \\
&  \ \ \ \ \ \ \ \ \ \ \left.  \text{lies in }I_{i}\text{ for every }%
i\in\mathbb{N}\right\}  .
\end{align*}

Clearly, $A\subseteq A\left[  \left(  I_{\rho}\right)  _{\rho\in\mathbb{N}%
}\ast Y\right]  $, since
\[
A\left[  \left(  I_{\rho}\right)  _{\rho\in\mathbb{N}}\ast Y\right]
=\sum\limits_{i\in\mathbb{N}}I_{i}Y^{i}\supseteq\underbrace{I_{0}}%
_{=A}\underbrace{Y^{0}}_{=1}=A\cdot1=A.
\]

Hence, $1\in A\subseteq A\left[  \left(  I_{\rho}\right)  _{\rho\in\mathbb{N}%
}\ast Y\right]  $. Also, the $A$-submodule $A\left[  \left(  I_{\rho}\right)
_{\rho\in\mathbb{N}}\ast Y\right]  $ of $A\left[  Y\right]  $ is an
$A$-subalgebra of the $A$-algebra $A\left[  Y\right]  $ (by
Lemma~\ref{lem.rees.mult} below), and thus is a subring of $A\left[  Y\right]
$.

This $A$-subalgebra $A\left[  \left(  I_{\rho}\right)  _{\rho\in\mathbb{N}%
}\ast Y\right]  $ is called the \textit{Rees algebra} of the ideal
semifiltration $\left(  I_{\rho}\right)  _{\rho\in\mathbb{N}}$.
\end{definition}

\begin{lemma}
\label{lem.rees.mult}Let $A$ be a ring, and let $\left(  I_{\rho}\right)
_{\rho\in\mathbb{N}}$ be an ideal semifiltration of $A$. Then, the
$A$-submodule $A\left[  \left(  I_{\rho}\right)  _{\rho\in\mathbb{N}}\ast
Y\right]  $ of $A\left[  Y\right]  $ is an $A$-subalgebra of the $A$-algebra
$A\left[  Y\right]  $.
\end{lemma}

\begin{vershort}
\begin{proof}
[Proof of Lemma \ref{lem.rees.mult}.]This is an easy exercise. (Use
$I_{a}I_{b}\subseteq I_{a+b}$ to prove that $A\left[  \left(  I_{\rho}\right)
_{\rho\in\mathbb{N}}\ast Y\right]  $ is closed under multiplication.)
\end{proof}
\end{vershort}

\begin{verlong}
\begin{proof}
[Proof of Lemma \ref{lem.rees.mult}.]Multiplying the equality $A\left[
\left(  I_{\rho}\right)  _{\rho\in\mathbb{N}}\ast Y\right]  =\sum
\limits_{i\in\mathbb{N}}I_{i}Y^{i}$ with itself, we find
\begin{align*}
&  A\left[  \left(  I_{\rho}\right)  _{\rho\in\mathbb{N}}\ast Y\right]  \cdot
A\left[  \left(  I_{\rho}\right)  _{\rho\in\mathbb{N}}\ast Y\right] \\
&  =\left(  \sum\limits_{i\in\mathbb{N}}I_{i}Y^{i}\right)  \cdot\left(
\sum\limits_{i\in\mathbb{N}}I_{i}Y^{i}\right)  =\left(  \sum\limits_{i\in
\mathbb{N}}I_{i}Y^{i}\right)  \cdot\left(  \sum\limits_{j\in\mathbb{N}}%
I_{j}Y^{j}\right) \\
&  \ \ \ \ \ \ \ \ \ \ \left(  \text{here we renamed the index }i\text{ as
}j\text{ in the second sum}\right) \\
&  =\sum\limits_{i\in\mathbb{N}}\sum\limits_{j\in\mathbb{N}}I_{i}%
\underbrace{Y^{i}I_{j}}_{=I_{j}Y^{i}}Y^{j}=\sum\limits_{i\in\mathbb{N}}%
\sum\limits_{j\in\mathbb{N}}\underbrace{I_{i}I_{j}}_{\substack{\subseteq
I_{i+j}\\\text{(since }\left(  I_{\rho}\right)  _{\rho\in\mathbb{N}}\\\text{is
an ideal}\\\text{semifiltration)}}}\underbrace{Y^{i}Y^{j}}_{=Y^{i+j}}\\
&  \subseteq\sum\limits_{i\in\mathbb{N}}\sum\limits_{j\in\mathbb{N}%
}\underbrace{I_{i+j}Y^{i+j}}_{\subseteq\sum_{k\in\mathbb{N}}I_{k}Y^{k}%
}\subseteq\sum\limits_{i\in\mathbb{N}}\sum\limits_{j\in\mathbb{N}}\sum
_{k\in\mathbb{N}}I_{k}Y^{k}\\
&  \subseteq\sum_{k\in\mathbb{N}}I_{k}Y^{k}\ \ \ \ \ \ \ \ \ \ \left(
\text{since }\sum_{k\in\mathbb{N}}I_{k}Y^{k}\text{ is an }A\text{-module}%
\right) \\
&  =\sum\limits_{i\in\mathbb{N}}I_{i}Y^{i}\ \ \ \ \ \ \ \ \ \ \left(
\text{here we renamed the index }k\text{ as }i\text{ in the sum}\right) \\
&  =A\left[  \left(  I_{\rho}\right)  _{\rho\in\mathbb{N}}\ast Y\right]  .
\end{align*}
Hence, the $A$-submodule $A\left[  \left(  I_{\rho}\right)  _{\rho
\in\mathbb{N}}\ast Y\right]  $ of $A\left[  Y\right]  $ is closed under
multiplication. Thus, it is an $A$-subalgebra of the $A$-algebra $A\left[
Y\right]  $ (since $1\in A\left[  \left(  I_{\rho}\right)  _{\rho\in
\mathbb{N}}\ast Y\right]  $). This proves Lemma \ref{lem.rees.mult}.
\end{proof}
\end{verlong}

\begin{remark}
\label{rmk.BY-over-AIY}Let $A$ be a ring. Let $B$ be an $A$-algebra. Let
$\left(  I_{\rho}\right)  _{\rho\in\mathbb{N}}$ be an ideal semifiltration of
$A$.

Consider the polynomial ring $A\left[  Y\right]  $ and its $A$-subalgebra
$A\left[  \left(  I_{\rho}\right)  _{\rho\in\mathbb{N}}\ast Y\right]  $
defined in Definition~\ref{Definition8}.

The polynomial ring $B\left[  Y\right]  $ is an $A\left[  Y\right]  $-algebra
(as explained in Definition~\ref{Definition7}), and thus is an $A\left[
\left(  I_{\rho}\right)  _{\rho\in\mathbb{N}}\ast Y\right]  $-algebra (since
$A\left[  \left(  I_{\rho}\right)  _{\rho\in\mathbb{N}}\ast Y\right]  $ is a
subring of $A\left[  Y\right]  $). Hence, if $p\in B\left[  Y\right]  $ is a
polynomial and $n\in\mathbb{N}$, then it makes sense to ask whether $p$ is
$n$-integral over $A\left[  \left(  I_{\rho}\right)  _{\rho\in\mathbb{N}}\ast
Y\right]  $. Questions of this form will often appear in what follows.
\end{remark}

We note in passing that the notion of a Rees algebra helps set up a bijection
between ideal semifiltrations of a ring $A$ and graded $A$-subalgebras of the
polynomial ring $A\left[  Y\right]  $:

\begin{proposition}
\label{prop.rees=grad-sub}Let $A$ be a ring. Consider the polynomial ring
$A\left[  Y\right]  $ as a graded $A$-algebra (with the usual degree of polynomials).

\textbf{(a)} If $\left(  I_{\rho}\right)  _{\rho\in\mathbb{N}}$ is an ideal
semifiltration of $A$, then its Rees algebra $A\left[  \left(  I_{\rho
}\right)  _{\rho\in\mathbb{N}}\ast Y\right]  $ is a graded $A$-subalgebra of
$A\left[  Y\right]  $.

\textbf{(b)} If $B$ is any graded $A$-subalgebra of $A\left[  Y\right]  $, and
if $\rho\in\mathbb{N}$, then we let $I_{B,\rho}$ denote the subset $\left\{
a\in A\ \mid\ aY^{\rho}\in B\right\}  $ of $A$. Then, $I_{B,\rho}$ is an ideal
of $A$.

\textbf{(c)} The maps%
\begin{align*}
\left\{  \text{ideal semifiltrations of }A\right\}   &  \rightarrow\left\{
\text{graded }A\text{-subalgebras of }A\left[  Y\right]  \right\}  ,\\
\left(  I_{\rho}\right)  _{\rho\in\mathbb{N}}  &  \mapsto A\left[  \left(
I_{\rho}\right)  _{\rho\in\mathbb{N}}\ast Y\right]
\end{align*}
and%
\begin{align*}
\left\{  \text{graded }A\text{-subalgebras of }A\left[  Y\right]  \right\}
&  \rightarrow\left\{  \text{ideal semifiltrations of }A\right\}  ,\\
B  &  \mapsto\left(  I_{B,\rho}\right)  _{\rho\in\mathbb{N}}%
\end{align*}
are mutually inverse bijections.
\end{proposition}

We shall not have any need for this proposition, so we omit its
(straightforward and easy) proof.

\subsection{Reduction to integrality over rings}

We start with a theorem which reduces the question of $n$-integrality over
$\left(  A,\left(  I_{\rho}\right)  _{\rho\in\mathbb{N}}\right)  $ to that of
$n$-integrality over a ring\footnote{Theorem~\ref{Theorem7} is inspired by
\cite[Proposition 5.2.1]{2}.}:

\begin{theorem}
\label{Theorem7} Let $A$ be a ring. Let $B$ be an $A$-algebra. Let $\left(
I_{\rho}\right)  _{\rho\in\mathbb{N}}$ be an ideal semifiltration of $A$. Let
$n\in\mathbb{N}$. Let $u\in B$.

Consider the polynomial ring $A\left[  Y\right]  $ and its $A$-subalgebra
$A\left[  \left(  I_{\rho}\right)  _{\rho\in\mathbb{N}}\ast Y\right]  $
defined in Definition~\ref{Definition8}.

Then, the element $u$ of $B$ is $n$-integral over $\left(  A,\left(  I_{\rho
}\right)  _{\rho\in\mathbb{N}}\right)  $ if and only if the element $uY$ of
the polynomial ring $B\left[  Y\right]  $ is $n$-integral over the ring
$A\left[  \left(  I_{\rho}\right)  _{\rho\in\mathbb{N}}\ast Y\right]  $.
(Here, $B\left[  Y\right]  $ is an $A\left[  \left(  I_{\rho}\right)
_{\rho\in\mathbb{N}}\ast Y\right]  $-algebra, as explained in Remark
\ref{rmk.BY-over-AIY}.)
\end{theorem}

\begin{vershort}
\begin{proof}
[Proof of Theorem~\ref{Theorem7}.]$\Longrightarrow:$ Assume that $u$ is
$n$-integral over $\left(  A,\left(  I_{\rho}\right)  _{\rho\in\mathbb{N}%
}\right)  $. Thus, by Definition~\ref{Definition9}, there exists some $\left(
a_{0},a_{1},\ldots,a_{n}\right)  \in A^{n+1}$ such that%
\[
\sum\limits_{k=0}^{n}a_{k}u^{k}=0,\ \ \ \ \ \ \ \ \ \ a_{n}%
=1,\ \ \ \ \ \ \ \ \ \ \text{and}\ \ \ \ \ \ \ \ \ \ a_{i}\in I_{n-i}\text{
for every }i\in\left\{  0,1,\ldots,n\right\}  .
\]
Consider this $\left(  a_{0},a_{1},\ldots,a_{n}\right)  $.

Hence, there exists a monic polynomial $P\in\left(  A\left[  \left(  I_{\rho
}\right)  _{\rho\in\mathbb{N}}\ast Y\right]  \right)  \left[  X\right]  $ with
$\deg P=n$ and $P\left(  uY\right)  =0$ (viz., the polynomial $P\left(
X\right)  =\sum\limits_{k=0}^{n}a_{k}Y^{n-k}X^{k}$). Hence, $uY$ is
$n$-integral over $A\left[  \left(  I_{\rho}\right)  _{\rho\in\mathbb{N}}\ast
Y\right]  $. This proves the $\Longrightarrow$ direction of
Theorem~\ref{Theorem7}.

$\Longleftarrow:$ Assume that $uY$ is $n$-integral over $A\left[  \left(
I_{\rho}\right)  _{\rho\in\mathbb{N}}\ast Y\right]  $. Thus, there exists a
monic polynomial $P\in\left(  A\left[  \left(  I_{\rho}\right)  _{\rho
\in\mathbb{N}}\ast Y\right]  \right)  \left[  X\right]  $ with $\deg P=n$ and
$P\left(  uY\right)  =0$. Consider this $P$. Since $P\in\left(  A\left[
\left(  I_{\rho}\right)  _{\rho\in\mathbb{N}}\ast Y\right]  \right)  \left[
X\right]  $ satisfies $\deg P=n$, there exists $\left(  p_{0},p_{1}%
,\ldots,p_{n}\right)  \in\left(  A\left[  \left(  I_{\rho}\right)  _{\rho
\in\mathbb{N}}\ast Y\right]  \right)  ^{n+1}$ such that $P\left(  X\right)
=\sum\limits_{k=0}^{n}p_{k}X^{k}$. Consider this $\left(  p_{0},p_{1}%
,\ldots,p_{n}\right)  $. Note that $p_{n}=1$ (since $P$ is monic and $\deg
P=n$).

For every $k\in\left\{  0,1,\ldots,n\right\}  $, we have $p_{k}\in A\left[
\left(  I_{\rho}\right)  _{\rho\in\mathbb{N}}\ast Y\right]  =\sum
\limits_{i\in\mathbb{N}}I_{i}Y^{i}$, and thus there exists a sequence $\left(
p_{k,i}\right)  _{i\in\mathbb{N}}\in A^{\mathbb{N}}$ such that $p_{k}%
=\sum\limits_{i\in\mathbb{N}}p_{k,i}Y^{i}$, such that $\left(  p_{k,i}\in
I_{i}\text{ for every }i\in\mathbb{N}\right)  $, and such that only finitely
many $i\in\mathbb{N}$ satisfy $p_{k,i}\neq0$. Consider this sequence. Thus,
$P\left(  X\right)  =\sum\limits_{k=0}^{n}p_{k}X^{k}$ rewrites as $P\left(
X\right)  =\sum\limits_{k=0}^{n}\sum\limits_{i\in\mathbb{N}}p_{k,i}Y^{i}X^{k}$
(since $p_{k}=\sum\limits_{i\in\mathbb{N}}p_{k,i}Y^{i}$). Hence,
\[
P\left(  uY\right)  =\sum\limits_{k=0}^{n}\sum\limits_{i\in\mathbb{N}}%
p_{k,i}Y^{i}\left(  uY\right)  ^{k}=\sum\limits_{k=0}^{n}\sum\limits_{i\in
\mathbb{N}}p_{k,i}Y^{i+k}u^{k}.
\]
Therefore, $P\left(  uY\right)  =0$ rewrites as $\sum\limits_{k=0}^{n}%
\sum\limits_{i\in\mathbb{N}}p_{k,i}Y^{i+k}u^{k}=0$. In other words, the
polynomial $\sum\limits_{k=0}^{n}\sum\limits_{i\in\mathbb{N}}p_{k,i}%
Y^{i+k}u^{k}\in B\left[  Y\right]  $ equals $0$. Hence, its coefficient before
$Y^{n}$ equals $0$ as well. But its coefficient before $Y^{n}$ is
$\sum\limits_{k=0}^{n}p_{k,n-k}u^{k}$, so we get $\sum\limits_{k=0}%
^{n}p_{k,n-k}u^{k}=0$.

Recall that $\sum\limits_{i\in\mathbb{N}}p_{k,i}Y^{i}=p_{k}$ for every
$k\in\left\{  0,1,\ldots,n\right\}  $ (by the definition of the $p_{k,i}$).
Thus, $\sum\limits_{i\in\mathbb{N}}p_{n,i}Y^{i}=p_{n}=1$ in $A\left[
Y\right]  $, and thus $p_{n,0}=1$ (by comparing coefficients before $Y^{0}$).

Define an $\left(  n+1\right)  $-tuple $\left(  a_{0},a_{1},\ldots
,a_{n}\right)  \in A^{n+1}$ by $a_{k}=p_{k,n-k}$ for every $k\in\left\{
0,1,\ldots,n\right\}  $. Then, $a_{n}=p_{n,0}=1$. Besides, $\sum
\limits_{k=0}^{n}\underbrace{a_{k}}_{=p_{k,n-k}}u^{k}=\sum\limits_{k=0}%
^{n}p_{k,n-k}u^{k}=0$. Finally, $a_{k}=p_{k,n-k}\in I_{n-k}$ (since
$p_{k,i}\in I_{i}$ for every $i\in\mathbb{N}$) for every $k\in\left\{
0,1,\ldots,n\right\}  $. In other words, $a_{i}\in I_{n-i}$ for every
$i\in\left\{  0,1,\ldots,n\right\}  $.

Altogether, we now know that%
\[
\sum\limits_{k=0}^{n}a_{k}u^{k}=0,\ \ \ \ \ \ \ \ \ \ a_{n}%
=1,\ \ \ \ \ \ \ \ \ \ \text{and}\ \ \ \ \ \ \ \ \ \ a_{i}\in I_{n-i}\text{
for every }i\in\left\{  0,1,\ldots,n\right\}  .
\]
Thus, by Definition~\ref{Definition9}, the element $u$ is $n$-integral over
$\left(  A,\left(  I_{\rho}\right)  _{\rho\in\mathbb{N}}\right)  $. This
proves the $\Longleftarrow$ direction of Theorem~\ref{Theorem7}.
\end{proof}
\end{vershort}

\begin{verlong}
\begin{proof}
[Proof of Theorem~\ref{Theorem7}.]In order to verify Theorem~\ref{Theorem7},
we have to prove the following two lemmata:

\begin{statement}
\textit{Lemma }$\mathcal{E}$\textit{:} If $u$ is $n$-integral over $\left(
A,\left(  I_{\rho}\right)  _{\rho\in\mathbb{N}}\right)  $, then $uY$ is
$n$-integral over $A\left[  \left(  I_{\rho}\right)  _{\rho\in\mathbb{N}}\ast
Y\right]  $.
\end{statement}

\begin{statement}
\textit{Lemma} $\mathcal{F}$\textit{:} If $uY$ is $n$-integral over $A\left[
\left(  I_{\rho}\right)  _{\rho\in\mathbb{N}}\ast Y\right]  $, then $u$ is
$n$-integral over $\left(  A,\left(  I_{\rho}\right)  _{\rho\in\mathbb{N}%
}\right)  $.
\end{statement}

[\textit{Proof of Lemma }$\mathcal{E}$\textit{:} Assume that $u$ is
$n$-integral over $\left(  A,\left(  I_{\rho}\right)  _{\rho\in\mathbb{N}%
}\right)  $. Thus, by Definition~\ref{Definition9}, there exists some $\left(
a_{0},a_{1},\ldots,a_{n}\right)  \in A^{n+1}$ such that%
\[
\sum\limits_{k=0}^{n}a_{k}u^{k}=0,\ \ \ \ \ \ \ \ \ \ a_{n}%
=1,\ \ \ \ \ \ \ \ \ \ \text{and}\ \ \ \ \ \ \ \ \ \ a_{i}\in I_{n-i}\text{
for every }i\in\left\{  0,1,\ldots,n\right\}  .
\]
Consider this $\left(  a_{0},a_{1},\ldots,a_{n}\right)  $.

For each $k\in\left\{  0,1,\ldots,n\right\}  $, we have $a_{k}\in I_{n-k}$
(since $a_{i}\in I_{n-i}$ for every $i\in\left\{  0,1,\ldots,n\right\}  $) and
therefore%
\[
\underbrace{a_{k}}_{\in I_{n-k}}Y^{n-k}\in I_{n-k}Y^{n-k}\subseteq
\sum\limits_{i\in\mathbb{N}}I_{i}Y^{i}=A\left[  \left(  I_{\rho}\right)
_{\rho\in\mathbb{N}}\ast Y\right]
\]
(since Definition~\ref{Definition8} yields $A\left[  \left(  I_{\rho}\right)
_{\rho\in\mathbb{N}}\ast Y\right]  =\sum\limits_{i\in\mathbb{N}}I_{i}Y^{i}$).

Thus, we can define a polynomial $P\in\left(  A\left[  \left(  I_{\rho
}\right)  _{\rho\in\mathbb{N}}\ast Y\right]  \right)  \left[  X\right]  $ by
$P\left(  X\right)  =\sum\limits_{k=0}^{n}a_{k}Y^{n-k}X^{k}$. Consider this
$P$. This polynomial $P$ satisfies $\deg P\leq n$, and its coefficient before
$X^{n}$ is $\underbrace{a_{n}}_{=1}\underbrace{Y^{n-n}}_{=Y^{0}=1}=1$. Hence,
this polynomial $P$ is monic and satisfies $\deg P=n$. Also, by substituting
$uY$ for $X$ in $P\left(  X\right)  =\sum\limits_{k=0}^{n}a_{k}Y^{n-k}X^{k}$,
we obtain%
\[
P\left(  uY\right)  =\sum\limits_{k=0}^{n}a_{k}Y^{n-k}\underbrace{\left(
uY\right)  ^{k}}_{=u^{k}Y^{k}}=\sum\limits_{k=0}^{n}a_{k}Y^{n-k}u^{k}%
Y^{k}=\sum\limits_{k=0}^{n}a_{k}u^{k}\underbrace{Y^{n-k}Y^{k}}_{=Y^{n}}%
=Y^{n}\cdot\underbrace{\sum\limits_{k=0}^{n}a_{k}u^{k}}_{=0}=0.
\]
Thus, there exists a monic polynomial $P\in\left(  A\left[  \left(  I_{\rho
}\right)  _{\rho\in\mathbb{N}}\ast Y\right]  \right)  \left[  X\right]  $ with
$\deg P=n$ and $P\left(  uY\right)  =0$. Hence, $uY$ is $n$-integral over
$A\left[  \left(  I_{\rho}\right)  _{\rho\in\mathbb{N}}\ast Y\right]  $. This
proves Lemma $\mathcal{E}$.]

[\textit{Proof of Lemma }$\mathcal{F}$\textit{:} Assume that $uY$ is
$n$-integral over $A\left[  \left(  I_{\rho}\right)  _{\rho\in\mathbb{N}}\ast
Y\right]  $. Thus, there exists a monic polynomial $P\in\left(  A\left[
\left(  I_{\rho}\right)  _{\rho\in\mathbb{N}}\ast Y\right]  \right)  \left[
X\right]  $ with $\deg P=n$ and $P\left(  uY\right)  =0$. Consider this $P$.
Since $P\in\left(  A\left[  \left(  I_{\rho}\right)  _{\rho\in\mathbb{N}}\ast
Y\right]  \right)  \left[  X\right]  $ satisfies $\deg P=n$, there exists
$\left(  p_{0},p_{1},\ldots,p_{n}\right)  \in\left(  A\left[  \left(  I_{\rho
}\right)  _{\rho\in\mathbb{N}}\ast Y\right]  \right)  ^{n+1}$ such that
$P\left(  X\right)  =\sum\limits_{k=0}^{n}p_{k}X^{k}$. Consider this $\left(
p_{0},p_{1},\ldots,p_{n}\right)  $. Note that $p_{n}=1$ (since $P$ is monic
and $\deg P=n$).

Recall that $\left(  p_{0},p_{1},\ldots,p_{n}\right)  \in\left(  A\left[
\left(  I_{\rho}\right)  _{\rho\in\mathbb{N}}\ast Y\right]  \right)  ^{n+1}$.
Hence, for every $k\in\left\{  0,1,\ldots,n\right\}  $, we have $p_{k}\in
A\left[  \left(  I_{\rho}\right)  _{\rho\in\mathbb{N}}\ast Y\right]
=\sum\limits_{i\in\mathbb{N}}I_{i}Y^{i}$, and thus there exists a sequence
$\left(  p_{k,i}\right)  _{i\in\mathbb{N}}\in A^{\mathbb{N}}$ such that
$p_{k}=\sum\limits_{i\in\mathbb{N}}p_{k,i}Y^{i}$, such that $\left(
p_{k,i}\in I_{i}\text{ for every }i\in\mathbb{N}\right)  $, and such that only
finitely many $i\in\mathbb{N}$ satisfy $p_{k,i}\neq0$. Consider this sequence.
Thus,
\[
P\left(  X\right)  =\sum\limits_{k=0}^{n}\underbrace{p_{k}}_{=\sum
\limits_{i\in\mathbb{N}}p_{k,i}Y^{i}}X^{k}=\sum\limits_{k=0}^{n}%
\sum\limits_{i\in\mathbb{N}}p_{k,i}Y^{i}X^{k}.
\]
Substituting $uY$ for $X$ in this equality, we find
\begin{align*}
P\left(  uY\right)   &  =\sum\limits_{k=0}^{n}\sum\limits_{i\in\mathbb{N}%
}p_{k,i}Y^{i}\underbrace{\left(  uY\right)  ^{k}}_{\substack{=u^{k}%
Y^{k}\\=Y^{k}u^{k}}}=\sum\limits_{k=0}^{n}\sum\limits_{i\in\mathbb{N}}%
p_{k,i}\underbrace{Y^{i}Y^{k}}_{=Y^{i+k}}u^{k}\\
&  =\sum\limits_{k=0}^{n}\sum\limits_{i\in\mathbb{N}}p_{k,i}Y^{i+k}u^{k}%
=\sum\limits_{k\in\left\{  0,1,\ldots,n\right\}  }\sum\limits_{i\in\mathbb{N}%
}p_{k,i}Y^{i+k}u^{k}\\
&  =\sum\limits_{\left(  k,i\right)  \in\left\{  0,1,\ldots,n\right\}
\times\mathbb{N}}p_{k,i}Y^{i+k}u^{k}=\sum_{\ell\in\mathbb{N}}\sum
\limits_{\substack{\left(  k,i\right)  \in\left\{  0,1,\ldots,n\right\}
\times\mathbb{N};\\i+k=\ell}}p_{k,i}\underbrace{Y^{i+k}}_{\substack{=Y^{\ell
}\\\text{(since }i+k=\ell\text{)}}}u^{k}\\
&  =\sum_{\ell\in\mathbb{N}}\sum\limits_{\substack{\left(  k,i\right)
\in\left\{  0,1,\ldots,n\right\}  \times\mathbb{N};\\i+k=\ell}}p_{k,i}Y^{\ell
}u^{k}=\sum_{\ell\in\mathbb{N}}\sum\limits_{\substack{\left(  k,i\right)
\in\left\{  0,1,\ldots,n\right\}  \times\mathbb{N};\\i+k=\ell}}p_{k,i}%
u^{k}Y^{\ell}.
\end{align*}
Comparing this with $P\left(  uY\right)  =0$, we find $\sum\limits_{\ell
\in\mathbb{N}}\sum\limits_{\substack{\left(  k,i\right)  \in\left\{
0,1,\ldots,n\right\}  \times\mathbb{N};\\i+k=\ell}}p_{k,i}u^{k}Y^{\ell}=0$. In
other words, the polynomial $\sum\limits_{\ell\in\mathbb{N}}\underbrace{\sum
\limits_{\substack{\left(  k,i\right)  \in\left\{  0,1,\ldots,n\right\}
\times\mathbb{N};\\i+k=\ell}}p_{k,i}u^{k}}_{\in B}Y^{\ell}\in B\left[
Y\right]  $ equals $0$. Hence, its coefficient before $Y^{n}$ equals $0$ as
well. But its coefficient before $Y^{n}$ is $\sum\limits_{\substack{\left(
k,i\right)  \in\left\{  0,1,\ldots,n\right\}  \times\mathbb{N};\\i+k=n}%
}p_{k,i}u^{k}$. Comparing the preceding two sentences, we see that
$\sum\limits_{\substack{\left(  k,i\right)  \in\left\{  0,1,\ldots,n\right\}
\times\mathbb{N};\\i+k=n}}p_{k,i}u^{k}$ equals $0$. Thus,%
\begin{equation}
0=\sum\limits_{\substack{\left(  k,i\right)  \in\left\{  0,1,\ldots,n\right\}
\times\mathbb{N};\\i+k=n}}p_{k,i}u^{k}=\sum\limits_{k\in\left\{
0,1,\ldots,n\right\}  }\sum_{\substack{i\in\mathbb{N};\\i+k=n}}p_{k,i}u^{k}.
\label{pf.Theorem7.5}%
\end{equation}

For each $k\in\left\{  0,1,\ldots,n\right\}  $, we have%
\[
\left\{  i\in\mathbb{N}\text{\ }\mid\ \underbrace{i+k=n}_{\Longleftrightarrow
\ \left(  i=n-k\right)  }\right\}  =\left\{  i\in\mathbb{N}\ \mid
\ i=n-k\right\}  =\left\{  n-k\right\}
\]
(because $n-k\in\mathbb{N}$ (since $k\in\left\{  0,1,\ldots,n\right\}  $)) and
thus%
\[
\sum\limits_{\substack{i\in\mathbb{N};\\i+k=n}}p_{k,i}u^{k}=\sum
\limits_{i\in\left\{  n-k\right\}  }p_{k,i}u^{k}=p_{k,n-k}u^{k}.
\]

Thus, (\ref{pf.Theorem7.5}) becomes%
\begin{equation}
0=\sum\limits_{k\in\left\{  0,1,\ldots,n\right\}  }\underbrace{\sum
_{\substack{i\in\mathbb{N};\\i+k=n}}p_{k,i}u^{k}}_{=p_{k,n-k}u^{k}}%
=\sum\limits_{k\in\left\{  0,1,\ldots,n\right\}  }p_{k,n-k}u^{k}.
\label{pf.Theorem7.7}%
\end{equation}

Recall that $p_{k}=\sum\limits_{i\in\mathbb{N}}p_{k,i}Y^{i}$ for every
$k\in\left\{  0,1,\ldots,n\right\}  $. Applying this to $k=n$, we find
$p_{n}=\sum\limits_{i\in\mathbb{N}}p_{n,i}Y^{i}$. Comparing this with
$p_{n}=1=1\cdot Y^{0}$, we find
\[
\sum\limits_{i\in\mathbb{N}}p_{n,i}Y^{i}=1\cdot Y^{0}%
\ \ \ \ \ \ \ \ \ \ \text{in }A\left[  Y\right]  .
\]
Hence, the coefficient of the polynomial $\sum\limits_{i\in\mathbb{N}}%
p_{n,i}Y^{i}\in A\left[  Y\right]  $ before $Y^{0}$ is $1$. But the
coefficient of the polynomial $\sum\limits_{i\in\mathbb{N}}p_{n,i}Y^{i}\in
A\left[  Y\right]  $ before $Y^{0}$ is $p_{n,0}$ (since $p_{n,i}\in A$ for all
$i\in\mathbb{N}$). Comparing the preceding two sentences, we see that
$p_{n,0}=1$.

Define an $\left(  n+1\right)  $-tuple $\left(  a_{0},a_{1},\ldots
,a_{n}\right)  \in A^{n+1}$ by setting
\[
\left(  a_{k}=p_{k,n-k}\text{ for every }k\in\left\{  0,1,\ldots,n\right\}
\right)  .
\]
Then, $a_{n}=p_{n,n-n}=p_{n,0}=1$. Besides,%
\[
\sum\limits_{k=0}^{n}\underbrace{a_{k}}_{\substack{=p_{k,n-k}\\\text{(by the
definition}\\\text{of }a_{k}\text{)}}}u^{k}=\sum\limits_{k=0}^{n}%
p_{k,n-k}u^{k}=\sum\limits_{k\in\left\{  0,1,\ldots,n\right\}  }p_{k,n-k}%
u^{k}=0
\]
(by (\ref{pf.Theorem7.7})). Finally, for every $k\in\left\{  0,1,\ldots
,n\right\}  $, we have $n-k\in\mathbb{N}$ and thus $a_{k}=p_{k,n-k}\in
I_{n-k}$ (since $p_{k,i}\in I_{i}$ for every $i\in\mathbb{N}$). Renaming the
variable $k$ as $i$ in this statement, we obtain the following: For every
$i\in\left\{  0,1,\ldots,n\right\}  $, we have $a_{i}\in I_{n-i}$.

Altogether, we now know that the $\left(  n+1\right)  $-tuple $\left(
a_{0},a_{1},\ldots,a_{n}\right)  \in A^{n+1}$ satisfies%
\[
\sum\limits_{k=0}^{n}a_{k}u^{k}=0,\ \ \ \ \ \ \ \ \ \ a_{n}%
=1,\ \ \ \ \ \ \ \ \ \ \text{and}\ \ \ \ \ \ \ \ \ \ a_{i}\in I_{n-i}\text{
for every }i\in\left\{  0,1,\ldots,n\right\}  .
\]
Thus, by Definition~\ref{Definition9}, the element $u$ is $n$-integral over
$\left(  A,\left(  I_{\rho}\right)  _{\rho\in\mathbb{N}}\right)  $. This
proves Lemma $\mathcal{F}$.]

Combining Lemma $\mathcal{E}$ and Lemma $\mathcal{F}$, we obtain that $u$ is
$n$-integral over $\left(  A,\left(  I_{\rho}\right)  _{\rho\in\mathbb{N}%
}\right)  $ if and only if $uY$ is $n$-integral over $A\left[  \left(
I_{\rho}\right)  _{\rho\in\mathbb{N}}\ast Y\right]  $. This proves
Theorem~\ref{Theorem7}.
\end{proof}
\end{verlong}

\subsection{Sums and products again}

Let us next state an analogue of Theorem~\ref{Theorem5a} for integrality over
ideal semifiltrations:

\begin{theorem}
\label{Theorem8a} Let $A$ be a ring. Let $B$ be an $A$-algebra. Let $\left(
I_{\rho}\right)  _{\rho\in\mathbb{N}}$ be an ideal semifiltration of $A$. Let
$u\in A$. Then, $u\cdot1_{B}$ is $1$-integral over $\left(  A,\left(  I_{\rho
}\right)  _{\rho\in\mathbb{N}}\right)  $ if and only if $u\cdot1_{B}\in
I_{1}\cdot1_{B}$.
\end{theorem}

\begin{vershort}
\begin{proof}
[Proof of Theorem~\ref{Theorem8a}.]Straightforward and left to the reader.
\end{proof}
\end{vershort}

\begin{verlong}
\begin{proof}
[Proof of Theorem~\ref{Theorem8a}.]In order to verify Theorem~\ref{Theorem8a},
we have to prove the following two lemmata:

\begin{statement}
\textit{Lemma }$\mathcal{G}$\textit{:} If $u\cdot1_{B}$ is $1$-integral over
$\left(  A,\left(  I_{\rho}\right)  _{\rho\in\mathbb{N}}\right)  $, then
$u\cdot1_{B}\in I_{1}\cdot1_{B}$.
\end{statement}

\begin{statement}
\textit{Lemma} $\mathcal{H}$\textit{:} If $u\cdot1_{B}\in I_{1}\cdot1_{B}$,
then $u\cdot1_{B}$ is $1$-integral over $\left(  A,\left(  I_{\rho}\right)
_{\rho\in\mathbb{N}}\right)  $.
\end{statement}

[\textit{Proof of Lemma }$\mathcal{G}$\textit{:} Assume that $u\cdot1_{B}$ is
$1$-integral over $\left(  A,\left(  I_{\rho}\right)  _{\rho\in\mathbb{N}%
}\right)  $. Thus, by Definition~\ref{Definition9} (applied to $u\cdot1_{B}$
and $1$ instead of $u$ and $n$), there exists some $\left(  a_{0}%
,a_{1}\right)  \in A^{2}$ such that%
\[
\sum\limits_{k=0}^{1}a_{k}\left(  u\cdot1_{B}\right)  ^{k}%
=0,\ \ \ \ \ \ \ \ \ \ a_{1}=1,\ \ \ \ \ \ \ \ \ \ \text{and}%
\ \ \ \ \ \ \ \ \ \ a_{i}\in I_{1-i}\text{ for every }i\in\left\{
0,1\right\}  .
\]
Consider this $\left(  a_{0},a_{1}\right)  $. Thus, $a_{0}\in I_{1-0}$ (since
$a_{i}\in I_{1-i}$ for every $i\in\left\{  0,1\right\}  $), so that $a_{0}\in
I_{1-0}=I_{1}$ and thus $-a_{0}\in-I_{1}\subseteq I_{1}$ (since $I_{1}$ is an
ideal of $A$). Also,%
\[
0=\sum\limits_{k=0}^{1}a_{k}\left(  u\cdot1_{B}\right)  ^{k}=a_{0}%
\underbrace{\left(  u\cdot1_{B}\right)  ^{0}}_{=1_{B}}+\underbrace{a_{1}}%
_{=1}\underbrace{\left(  u\cdot1_{B}\right)  ^{1}}_{=u\cdot1_{B}}=a_{0}%
\cdot1_{B}+u\cdot1_{B},
\]
so that $u\cdot1_{B}=\underbrace{-a_{0}}_{\in I_{1}}\cdot1_{B}\in I_{1}%
\cdot1_{B}$. This proves Lemma $\mathcal{G}$.]

[\textit{Proof of Lemma }$\mathcal{H}$\textit{:} Assume that $u\cdot1_{B}\in
I_{1}\cdot1_{B}$. Thus, $u\cdot1_{B}=w\cdot1_{B}$ for some $w\in I_{1}$.
Consider this $w$. Then, $w\in I_{1}$, so that $-w\in-I_{1}\subseteq I_{1}$
(since $I_{1}$ is an ideal of $A$). Define a $2$-tuple $\left(  a_{0}%
,a_{1}\right)  \in A^{2}$ by setting $a_{0}=-w$ and $a_{1}=1$. Then,
\begin{align*}
\sum\limits_{k=0}^{1}a_{k}\left(  u\cdot1_{B}\right)  ^{k}  &
=\underbrace{a_{0}}_{=-w}\underbrace{\left(  u\cdot1_{B}\right)  ^{0}}%
_{=1_{B}}+\underbrace{a_{1}}_{=1}\underbrace{\left(  u\cdot1_{B}\right)  ^{1}%
}_{=u\cdot1_{B}}=-w\cdot1_{B}+\underbrace{u\cdot1_{B}}_{=w\cdot1_{B}}\\
&  =-w\cdot1_{B}+w\cdot1_{B}=0.
\end{align*}
Also, $a_{i}\in I_{1-i}$ for every $i\in\left\{  0,1\right\}  $ (since
$a_{0}=-w\in I_{1}=I_{1-0}$ and $a_{1}=1\in A=I_{0}=I_{1-1}$). Altogether, we
now know that $\left(  a_{0},a_{1}\right)  \in A^{2}$ and%
\[
\sum\limits_{k=0}^{1}a_{k}\left(  u\cdot1_{B}\right)  ^{k}%
=0,\ \ \ \ \ \ \ \ \ \ a_{1}=1,\ \ \ \ \ \ \ \ \ \ \text{and}%
\ \ \ \ \ \ \ \ \ \ a_{i}\in I_{1-i}\text{ for every }i\in\left\{
0,1\right\}  .
\]
Thus, by Definition~\ref{Definition9} (applied to $u\cdot1_{B}$ and $1$
instead of $u$ and $n$), the element $u\cdot1_{B}$ is $1$-integral over
$\left(  A,\left(  I_{\rho}\right)  _{\rho\in\mathbb{N}}\right)  $. This
proves Lemma $\mathcal{H}$.]

Combining Lemma $\mathcal{G}$ and Lemma $\mathcal{H}$, we obtain that
$u\cdot1_{B}$ is $1$-integral over $\left(  A,\left(  I_{\rho}\right)
_{\rho\in\mathbb{N}}\right)  $ if and only if $u\cdot1_{B}\in I_{1}\cdot1_{B}%
$. This proves Theorem~\ref{Theorem8a}.
\end{proof}
\end{verlong}

The next theorem is an analogue of Theorem~\ref{Theorem5bc} \textbf{(a)} for
integrality over ideal semifiltrations:

\begin{theorem}
\label{Theorem8b} Let $A$ be a ring. Let $B$ be an $A$-algebra. Let $\left(
I_{\rho}\right)  _{\rho\in\mathbb{N}}$ be an ideal semifiltration of $A$. Let
$x\in B$ and $y\in B$. Let $m\in\mathbb{N}$ and $n\in\mathbb{N}$. Assume that
$x$ is $m$-integral over $\left(  A,\left(  I_{\rho}\right)  _{\rho
\in\mathbb{N}}\right)  $, and that $y$ is $n$-integral over $\left(  A,\left(
I_{\rho}\right)  _{\rho\in\mathbb{N}}\right)  $.

Then, $x+y$ is $nm$-integral over $\left(  A,\left(  I_{\rho}\right)
_{\rho\in\mathbb{N}}\right)  $.
\end{theorem}

\begin{proof}
[Proof of Theorem~\ref{Theorem8b}.]Consider the polynomial ring $A\left[
Y\right]  $ and its $A$-subalgebra $A\left[  \left(  I_{\rho}\right)
_{\rho\in\mathbb{N}}\ast Y\right]  $. The polynomial ring $B\left[  Y\right]
$ is an $A\left[  \left(  I_{\rho}\right)  _{\rho\in\mathbb{N}}\ast Y\right]
$-algebra (as explained in Remark \ref{rmk.BY-over-AIY}).

Theorem~\ref{Theorem7} (applied to $x$ and $m$ instead of $u$ and $n$) yields
that $xY$ is $m$-integral over $A\left[  \left(  I_{\rho}\right)  _{\rho
\in\mathbb{N}}\ast Y\right]  $ (since $x$ is $m$-integral over $\left(
A,\left(  I_{\rho}\right)  _{\rho\in\mathbb{N}}\right)  $). Also,
Theorem~\ref{Theorem7} (applied to $y$ instead of $u$) yields that $yY$ is
$n$-integral over \newline$A\left[  \left(  I_{\rho}\right)  _{\rho
\in\mathbb{N}}\ast Y\right]  $ (since $y$ is $n$-integral over $\left(
A,\left(  I_{\rho}\right)  _{\rho\in\mathbb{N}}\right)  $).

Hence, Theorem~\ref{Theorem5bc} \textbf{(a)} (applied to $A\left[  \left(
I_{\rho}\right)  _{\rho\in\mathbb{N}}\ast Y\right]  $, $B\left[  Y\right]  $,
$xY$ and $yY$ instead of $A$, $B$, $x$ and $y$, respectively) yields that
$xY+yY$ is $nm$-integral over $A\left[  \left(  I_{\rho}\right)  _{\rho
\in\mathbb{N}}\ast Y\right]  $. Since $xY+yY=\left(  x+y\right)  Y$, this
means that $\left(  x+y\right)  Y$ is $nm$-integral over $A\left[  \left(
I_{\rho}\right)  _{\rho\in\mathbb{N}}\ast Y\right]  $. Hence,
Theorem~\ref{Theorem7} (applied to $x+y$ and $nm$ instead of $u$ and $n$)
yields that $x+y$ is $nm$-integral over $\left(  A,\left(  I_{\rho}\right)
_{\rho\in\mathbb{N}}\right)  $. This proves Theorem~\ref{Theorem8b}.
\end{proof}

Our next theorem is a somewhat asymmetric analogue of Theorem~\ref{Theorem5bc}
\textbf{(b)} for integrality over ideal semifiltrations:

\begin{theorem}
\label{Theorem8c} Let $A$ be a ring. Let $B$ be an $A$-algebra. Let $\left(
I_{\rho}\right)  _{\rho\in\mathbb{N}}$ be an ideal semifiltration of $A$. Let
$x\in B$ and $y\in B$. Let $m\in\mathbb{N}$ and $n\in\mathbb{N}$. Assume that
$x$ is $m$-integral over $\left(  A,\left(  I_{\rho}\right)  _{\rho
\in\mathbb{N}}\right)  $, and that $y$ is $n$-integral over $A$.

Then, $xy$ is $nm$-integral over $\left(  A,\left(  I_{\rho}\right)  _{\rho
\in\mathbb{N}}\right)  $.
\end{theorem}

Before we prove this theorem, we require a trivial observation:

\begin{lemma}
\label{lem.I}Let $A$ be a ring. Let $A^{\prime}$ be an $A$-algebra. Let
$B^{\prime}$ be an $A^{\prime}$-algebra. Let $v\in B^{\prime}$. Let
$n\in\mathbb{N}$. Assume that $v$ is $n$-integral over $A$. (Here, of course,
we are using the fact that $B^{\prime}$ is an $A$-algebra, since $B^{\prime}$
is an $A^{\prime}$-algebra while $A^{\prime}$ is an $A$-algebra.)

Then, $v$ is $n$-integral over $A^{\prime}$.
\end{lemma}

\begin{verlong}
\begin{proof}
[Proof of Lemma \ref{lem.I}.]We know that $v$ is $n$-integral over $A$. In
other words, there exists a monic polynomial $P\in A\left[  X\right]  $ with
$\deg P=n$ and $P\left(  v\right)  =0$. Consider this $P$, and denote it by
$Q$. Thus, $Q$ is a monic polynomial in $A\left[  X\right]  $ with $\deg Q=n$
and $Q\left(  v\right)  =0$.

Consider the canonical ring homomorphism $A\rightarrow A^{\prime}$ sending
each $a\in A$ to $a\cdot1_{A^{\prime}}\in A^{\prime}$. This homomorphism is
defined because $A^{\prime}$ is an $A$-algebra, and it in turn induces a
canonical ring homomorphism $A\left[  X\right]  \rightarrow A^{\prime}\left[
X\right]  $. Let $\widetilde{Q}\in A^{\prime}\left[  X\right]  $ be the image
of the polynomial $Q\in A\left[  X\right]  $ under this latter homomorphism.
Then, $\widetilde{Q}$ is a monic polynomial with $\deg\widetilde{Q}=n$ (since
$Q$ is a monic polynomial with $\deg Q=n$). Furthermore, the definition of
$\widetilde{Q}$ yields $\widetilde{Q}\left(  v\right)  =Q\left(  v\right)
=0$. Thus, there exists a monic polynomial $P\in A^{\prime}\left[  X\right]  $
with $\deg P=n$ and $P\left(  v\right)  =0$ (namely, $P=\widetilde{Q}$). In
other words, $v$ is $n$-integral over $A^{\prime}$. This proves
Lemma~\ref{lem.I}.
\end{proof}
\end{verlong}

\begin{proof}
[Proof of Theorem \ref{Theorem8c}.]Consider the polynomial ring $A\left[
Y\right]  $ and its $A$-subalgebra $A\left[  \left(  I_{\rho}\right)
_{\rho\in\mathbb{N}}\ast Y\right]  $. The polynomial ring $B\left[  Y\right]
$ is an $A\left[  \left(  I_{\rho}\right)  _{\rho\in\mathbb{N}}\ast Y\right]
$-algebra (as explained in Remark \ref{rmk.BY-over-AIY}).

Theorem~\ref{Theorem7} (applied to $x$ and $m$ instead of $u$ and $n$) yields
that $xY$ is $m$-integral over $A\left[  \left(  I_{\rho}\right)  _{\rho
\in\mathbb{N}}\ast Y\right]  $ (since $x$ is $m$-integral over $\left(
A,\left(  I_{\rho}\right)  _{\rho\in\mathbb{N}}\right)  $). Also, we know that
$y$ is $n$-integral over $A$. Thus, Lemma~\ref{lem.I} (applied to $A^{\prime
}=A\left[  \left(  I_{\rho}\right)  _{\rho\in\mathbb{N}}\ast Y\right]  $,
$B^{\prime}=B\left[  Y\right]  $ and $v=y$) yields that $y$ is $n$-integral
over $A\left[  \left(  I_{\rho}\right)  _{\rho\in\mathbb{N}}\ast Y\right]  $
(since $A\left[  \left(  I_{\rho}\right)  _{\rho\in\mathbb{N}}\ast Y\right]  $
is an $A$-algebra, and $B\left[  Y\right]  $ is an $A\left[  \left(  I_{\rho
}\right)  _{\rho\in\mathbb{N}}\ast Y\right]  $-algebra). On the other hand, we
know that $xY$ is $m$-integral over $A\left[  \left(  I_{\rho}\right)
_{\rho\in\mathbb{N}}\ast Y\right]  $. Hence, Theorem~\ref{Theorem5bc}
\textbf{(b)} (applied to $A\left[  \left(  I_{\rho}\right)  _{\rho
\in\mathbb{N}}\ast Y\right]  $, $B\left[  Y\right]  $ and $xY$ instead of $A$,
$B$ and $x$, respectively) yields that $xY\cdot y$ is $nm$-integral over
$A\left[  \left(  I_{\rho}\right)  _{\rho\in\mathbb{N}}\ast Y\right]  $. Since
$xY\cdot y=xyY$, this means that $xyY$ is $nm$-integral over $A\left[  \left(
I_{\rho}\right)  _{\rho\in\mathbb{N}}\ast Y\right]  $. Hence,
Theorem~\ref{Theorem7} (applied to $xy$ and $nm$ instead of $u$ and $n$)
yields that $xy$ is $nm$-integral over $\left(  A,\left(  I_{\rho}\right)
_{\rho\in\mathbb{N}}\right)  $. This proves Theorem~\ref{Theorem8c}.
\end{proof}

It is easy to state analogues of Corollary~\ref{cor.-x-int} and
Corollary~\ref{cor.x-y-int} for ideal semifiltrations. These analogues can be
derived from Corollary~\ref{cor.-x-int} and Corollary~\ref{cor.x-y-int} in the
same way as how we derived Theorem~\ref{Theorem8b} from
Theorem~\ref{Theorem5bc} \textbf{(a)}.

\subsection{Transitivity again}

The next theorem imitates Theorem~\ref{Theorem4} for integrality over ideal semifiltrations:

\begin{theorem}
\label{Theorem9} Let $A$ be a ring. Let $B$ be an $A$-algebra. Let $\left(
I_{\rho}\right)  _{\rho\in\mathbb{N}}$ be an ideal semifiltration of $A$.

Let $v\in B$ and $u\in B$. Let $m\in\mathbb{N}$ and $n\in\mathbb{N}$.

\textbf{(a)} Then, $\left(  I_{\rho}A\left[  v\right]  \right)  _{\rho
\in\mathbb{N}}$ is an ideal semifiltration of $A\left[  v\right]  $. (See
Convention \ref{conv.IAv} below for the meaning of \textquotedblleft$I_{\rho
}A\left[  v\right]  $\textquotedblright.)

\textbf{(b)} Assume that $v$ is $m$-integral over $A$, and that $u$ is
$n$-integral over $\left(  A\left[  v\right]  ,\left(  I_{\rho}A\left[
v\right]  \right)  _{\rho\in\mathbb{N}}\right)  $. Then, $u$ is $nm$-integral
over $\left(  A,\left(  I_{\rho}\right)  _{\rho\in\mathbb{N}}\right)  $.
\end{theorem}

Here and in the following, we are using the following convention:

\begin{convention}
\label{conv.IAv}Let $A$ be a ring. Let $B$ be an $A$-algebra. Let $v\in B$,
and let $I$ be an ideal of $A$. Then, you should read the expression
\textquotedblleft$IA\left[  v\right]  $\textquotedblright\ as $I\cdot\left(
A\left[  v\right]  \right)  $, not as $\left(  IA\right)  \left[  v\right]  $.
For instance, you should read the term \textquotedblleft$I_{\rho}A\left[
v\right]  $\textquotedblright\ (in Theorem~\ref{Theorem9} \textbf{(a)}) as
$I_{\rho}\cdot\left(  A\left[  v\right]  \right)  $, not as $\left(  I_{\rho
}A\right)  \left[  v\right]  $.
\end{convention}

Before we prove Theorem~\ref{Theorem9}, let us state two lemmas. The first is
a more general (but still obvious) version of Theorem~\ref{Theorem9}
\textbf{(a)}:

\begin{lemma}
\label{lem.J}Let $A$ be a ring. Let $A^{\prime}$ be an $A$-algebra. Let
$\left(  I_{\rho}\right)  _{\rho\in\mathbb{N}}$ be an ideal semifiltration of
$A$. Then, $\left(  I_{\rho}A^{\prime}\right)  _{\rho\in\mathbb{N}}$ is an
ideal semifiltration of $A^{\prime}$.
\end{lemma}

\begin{vershort}
\begin{proof}
[Proof of Lemma \ref{lem.J}.]This is a straightforward verification of axioms.
\end{proof}
\end{vershort}

\begin{verlong}
\begin{proof}
[Proof of Lemma \ref{lem.J}.]We know that $\left(  I_{\rho}\right)  _{\rho
\in\mathbb{N}}$ is an ideal semifiltration of $A$. In other words, $\left(
I_{\rho}\right)  _{\rho\in\mathbb{N}}$ is a sequence of ideals of $A$ and
satisfies%
\begin{align*}
I_{0}  &  =A;\\
I_{a}I_{b}  &  \subseteq I_{a+b}\ \ \ \ \ \ \ \ \ \ \text{for every }%
a\in\mathbb{N}\text{ and }b\in\mathbb{N}%
\end{align*}
(by Definition~\ref{Definition6}). The set $I_{\rho}$ is an ideal of $A$ for
every $\rho\in\mathbb{N}$ (since $\left(  I_{\rho}\right)  _{\rho\in
\mathbb{N}}$ is a sequence of ideals of $A$).

Now, the set $I_{\rho}A^{\prime}$ is an ideal of $A^{\prime}$ for every
$\rho\in\mathbb{N}$ (since $I_{\rho}$ is an ideal of $A$). Hence, $\left(
I_{\rho}A^{\prime}\right)  _{\rho\in\mathbb{N}}$ is a sequence of ideals of
$A^{\prime}$. It satisfies%
\begin{align*}
\underbrace{I_{0}}_{=A}A^{\prime}  &  =AA^{\prime}=A^{\prime}%
\ \ \ \ \ \ \ \ \ \ \left(  \text{since }A^{\prime}\text{ is an }%
A\text{-algebra}\right)  ;\\
I_{a}A^{\prime}\cdot I_{b}A^{\prime}  &  =\underbrace{I_{a}I_{b}}_{\subseteq
I_{a+b}}A^{\prime}\subseteq I_{a+b}A^{\prime}\ \ \ \ \ \ \ \ \ \ \text{for
every }a\in\mathbb{N}\text{ and }b\in\mathbb{N}.
\end{align*}
Thus, by Definition~\ref{Definition6} (applied to $A^{\prime}$ and $\left(
I_{\rho}A^{\prime}\right)  _{\rho\in\mathbb{N}}$ instead of $A$ and $\left(
I_{\rho}\right)  _{\rho\in\mathbb{N}}$), it follows that $\left(  I_{\rho
}A^{\prime}\right)  _{\rho\in\mathbb{N}}$ is an ideal semifiltration of
$A^{\prime}$. This proves Lemma \ref{lem.J}.
\end{proof}
\end{verlong}

\begin{lemma}
\label{lem.K}Let $A$ be a ring. Let $A^{\prime}$ be an $A$-algebra. Let
$B^{\prime}$ be an $A^{\prime}$-algebra. Let $v\in B^{\prime}$. Then,
$A^{\prime}\cdot A\left[  v\right]  =A^{\prime}\left[  v\right]  $ (an
equality between $A$-submodules of $B^{\prime}$). (Here, we are using the fact
that $B^{\prime}$ is an $A$-algebra, because $B^{\prime}$ is an $A^{\prime}%
$-algebra while $A^{\prime}$ is an $A$-algebra.)
\end{lemma}

Here, of course, the expression \textquotedblleft$A^{\prime}\cdot A\left[
v\right]  $\textquotedblright\ means \textquotedblleft$A^{\prime}\cdot\left(
A\left[  v\right]  \right)  $\textquotedblright, not \textquotedblleft$\left(
A^{\prime}\cdot A\right)  \left[  v\right]  $\textquotedblright.

\begin{vershort}
\begin{proof}
[Proof of Lemma \ref{lem.K}.]Left to the reader (see \cite{verlong}).
\end{proof}
\end{vershort}

\begin{verlong}
\begin{proof}
[Proof of Lemma \ref{lem.K}.]We have $A\left[  v\right]  \subseteq A^{\prime
}\left[  v\right]  $ (since the ring $A$ acts on $B^{\prime}$ through the
canonical ring homomorphism $A\rightarrow A^{\prime}$). Hence, $A^{\prime
}\cdot\underbrace{A\left[  v\right]  }_{\subseteq A^{\prime}\left[  v\right]
}\subseteq A^{\prime}\cdot A^{\prime}\left[  v\right]  \subseteq A^{\prime
}\left[  v\right]  $ (since $A^{\prime}\left[  v\right]  $ is an $A^{\prime}%
$-algebra). On the other hand, let $x$ be an element of $A^{\prime}\left[
v\right]  $. Then, there exist some $n\in\mathbb{N}$ and some $\left(
a_{0},a_{1},\ldots,a_{n}\right)  \in\left(  A^{\prime}\right)  ^{n+1}$ such
that $x=\sum\limits_{k=0}^{n}a_{k}v^{k}$. Consider this $n$ and this $\left(
a_{0},a_{1},\ldots,a_{n}\right)  $. Thus,%
\[
x=\sum\limits_{k=0}^{n}\underbrace{a_{k}}_{\in A^{\prime}}\underbrace{v^{k}%
}_{\in A\left[  v\right]  }\in\sum\limits_{k=0}^{n}A^{\prime}\cdot A\left[
v\right]  \subseteq A^{\prime}\cdot A\left[  v\right]
\ \ \ \ \ \ \ \ \ \ \left(  \text{since }A^{\prime}\cdot A\left[  v\right]
\text{ is an additive group}\right)  .
\]

Now, forget that we fixed $x$. Thus, we have proved that $x\in A^{\prime}\cdot
A\left[  v\right]  $ for every $x\in A^{\prime}\left[  v\right]  $. Therefore,
$A^{\prime}\left[  v\right]  \subseteq A^{\prime}\cdot A\left[  v\right]  $.
Combined with $A^{\prime}\cdot A\left[  v\right]  \subseteq A^{\prime}\left[
v\right]  $, this yields $A^{\prime}\cdot A\left[  v\right]  =A^{\prime
}\left[  v\right]  $. Hence, we have established Lemma~\ref{lem.K}.
\end{proof}
\end{verlong}

We are now ready to prove Theorem~\ref{Theorem9}:

\begin{proof}
[Proof of Theorem~\ref{Theorem9}.]\textbf{(a)} Lemma \ref{lem.J} (applied to
$A^{\prime}=A\left[  v\right]  $) yields that $\left(  I_{\rho}A\left[
v\right]  \right)  _{\rho\in\mathbb{N}}$ is an ideal semifiltration of
$A\left[  v\right]  $. This proves Theorem~\ref{Theorem9} \textbf{(a)}.

\textbf{(b)} Consider the polynomial ring $A\left[  Y\right]  $ and its
$A$-subalgebra $A\left[  \left(  I_{\rho}\right)  _{\rho\in\mathbb{N}}\ast
Y\right]  $. Then, $\left(  A\left[  v\right]  \right)  \left[  Y\right]  $ is
an $A\left[  Y\right]  $-algebra (since $A\left[  v\right]  $ is an
$A$-algebra) and therefore also an $A\left[  \left(  I_{\rho}\right)
_{\rho\in\mathbb{N}}\ast Y\right]  $-algebra (since $A\left[  \left(  I_{\rho
}\right)  _{\rho\in\mathbb{N}}\ast Y\right]  $ is a subring of $A\left[
Y\right]  $). Hence, $\left(  A\left[  \left(  I_{\rho}\right)  _{\rho
\in\mathbb{N}}\ast Y\right]  \right)  \left[  v\right]  $ is an $A$-subalgebra
of $\left(  A\left[  v\right]  \right)  \left[  Y\right]  $ (since $v\in
A\left[  v\right]  \subseteq\left(  A\left[  v\right]  \right)  \left[
Y\right]  $). On the other hand, $\left(  A\left[  v\right]  \right)  \left[
\left(  I_{\rho}A\left[  v\right]  \right)  _{\rho\in\mathbb{N}}\ast Y\right]
$ is an $A$-subalgebra of $\left(  A\left[  v\right]  \right)  \left[
Y\right]  $ (by its definition).

Note that $B$ is an $A\left[  v\right]  $-algebra (since $A\left[  v\right]  $
is a subring of $B$). Hence, (as explained in Definition~\ref{Definition7})
the polynomial ring $B\left[  Y\right]  $ is an $\left(  A\left[  v\right]
\right)  \left[  Y\right]  $-algebra. Moreover, $B\left[  Y\right]  $ is an
$A\left[  Y\right]  $-algebra (as explained in Definition~\ref{Definition7})
and also an $A\left[  \left(  I_{\rho}\right)  _{\rho\in\mathbb{N}}\ast
Y\right]  $-algebra (as explained in Remark \ref{rmk.BY-over-AIY}).

Now, we will show that $\left(  A\left[  v\right]  \right)  \left[  \left(
I_{\rho}A\left[  v\right]  \right)  _{\rho\in\mathbb{N}}\ast Y\right]
=\left(  A\left[  \left(  I_{\rho}\right)  _{\rho\in\mathbb{N}}\ast Y\right]
\right)  \left[  v\right]  $. (This is an equality between two subrings of
$\left(  A\left[  v\right]  \right)  \left[  Y\right]  $.)

In fact, Definition~\ref{Definition8} yields $A\left[  \left(  I_{\rho
}\right)  _{\rho\in\mathbb{N}}\ast Y\right]  =\sum\limits_{i\in\mathbb{N}%
}I_{i}Y^{i}$. The same definition (but applied to $A\left[  v\right]  $ and
$\left(  I_{\rho}A\left[  v\right]  \right)  _{\rho\in\mathbb{N}}$ instead of
$A$ and $\left(  I_{\rho}\right)  _{\rho\in\mathbb{N}}$) yields%
\begin{align}
\left(  A\left[  v\right]  \right)  \left[  \left(  I_{\rho}A\left[  v\right]
\right)  _{\rho\in\mathbb{N}}\ast Y\right]   &  =\sum\limits_{i\in\mathbb{N}%
}I_{i}\underbrace{A\left[  v\right]  \cdot Y^{i}}_{=Y^{i}\cdot A\left[
v\right]  }=\sum\limits_{i\in\mathbb{N}}I_{i}Y^{i}\cdot A\left[  v\right]
\nonumber\\
&  =\underbrace{\left(  \sum\limits_{i\in\mathbb{N}}I_{i}Y^{i}\right)
}_{=A\left[  \left(  I_{\rho}\right)  _{\rho\in\mathbb{N}}\ast Y\right]
}\cdot A\left[  v\right]  =A\left[  \left(  I_{\rho}\right)  _{\rho
\in\mathbb{N}}\ast Y\right]  \cdot A\left[  v\right] \nonumber\\
&  =\left(  A\left[  \left(  I_{\rho}\right)  _{\rho\in\mathbb{N}}\ast
Y\right]  \right)  \left[  v\right]  \label{pf.Theorem9.3}%
\end{align}
(by Lemma \ref{lem.K}, applied to $A^{\prime}=A\left[  \left(  I_{\rho
}\right)  _{\rho\in\mathbb{N}}\ast Y\right]  $ and $B^{\prime}=\left(
A\left[  v\right]  \right)  \left[  Y\right]  $).

Recall that $B\left[  Y\right]  $ is an $A\left[  \left(  I_{\rho}\right)
_{\rho\in\mathbb{N}}\ast Y\right]  $-algebra. Hence, Lemma \ref{lem.I}
(applied to $A\left[  \left(  I_{\rho}\right)  _{\rho\in\mathbb{N}}\ast
Y\right]  $, $B\left[  Y\right]  $ and $m$ instead of $A^{\prime}$,
$B^{\prime}$ and $n$) yields that $v$ is $m$-integral over $A\left[  \left(
I_{\rho}\right)  _{\rho\in\mathbb{N}}\ast Y\right]  $ (since $v$ is
$m$-integral over $A$).

Now, Theorem~\ref{Theorem7} (applied to $A\left[  v\right]  $ and $\left(
I_{\rho}A\left[  v\right]  \right)  _{\rho\in\mathbb{N}}$ instead of $A$ and
$\left(  I_{\rho}\right)  _{\rho\in\mathbb{N}}$) yields that the element $uY$
is $n$-integral over $\left(  A\left[  v\right]  \right)  \left[  \left(
I_{\rho}A\left[  v\right]  \right)  _{\rho\in\mathbb{N}}\ast Y\right]  $
(since $u$ is $n$-integral over $\left(  A\left[  v\right]  ,\left(  I_{\rho
}A\left[  v\right]  \right)  _{\rho\in\mathbb{N}}\right)  $). In view of
(\ref{pf.Theorem9.3}), this rewrites as follows: The element $uY$ is
$n$-integral over $\left(  A\left[  \left(  I_{\rho}\right)  _{\rho
\in\mathbb{N}}\ast Y\right]  \right)  \left[  v\right]  $. Hence,
Theorem~\ref{Theorem4} (applied to $A\left[  \left(  I_{\rho}\right)
_{\rho\in\mathbb{N}}\ast Y\right]  $, $B\left[  Y\right]  $ and $uY$ instead
of $A$, $B$ and $u$) yields that $uY$ is $nm$-integral over $A\left[  \left(
I_{\rho}\right)  _{\rho\in\mathbb{N}}\ast Y\right]  $ (since $v$ is
$m$-integral over $A\left[  \left(  I_{\rho}\right)  _{\rho\in\mathbb{N}}\ast
Y\right]  $). Thus, Theorem~\ref{Theorem7} (applied to $nm$ instead of $n$)
yields that $u$ is $nm$-integral over $\left(  A,\left(  I_{\rho}\right)
_{\rho\in\mathbb{N}}\right)  $. This proves Theorem~\ref{Theorem9}
\textbf{(b)}.
\end{proof}

\section{\label{sect.3}Generalizing to two ideal semifiltrations}

Theorem~\ref{Theorem8c} can be generalized: Instead of requiring $y$ to be
integral over the ring $A$, we can require $y$ to be integral over a further
ideal semifiltration $\left(  J_{\rho}\right)  _{\rho\in\mathbb{N}}$ of $A$.
In that case, $xy$ will be integral over the ideal semifiltration $\left(
I_{\rho}J_{\rho}\right)  _{\rho\in\mathbb{N}}$ (see Theorem~\ref{Theorem13}
for the precise statement). To get a grip on this, let us study two ideal semifiltrations.

\subsection{The product of two ideal semifiltrations}

\begin{theorem}
\label{Theorem10} Let $A$ be a ring.

\textbf{(a)} Then, $\left(  A\right)  _{\rho\in\mathbb{N}}$ is an ideal
semifiltration of $A$.

\textbf{(b)} Let $\left(  I_{\rho}\right)  _{\rho\in\mathbb{N}}$ and $\left(
J_{\rho}\right)  _{\rho\in\mathbb{N}}$ be two ideal semifiltrations of $A$.
Then, $\left(  I_{\rho}J_{\rho}\right)  _{\rho\in\mathbb{N}}$ is an ideal
semifiltration of $A$.
\end{theorem}

\begin{vershort}
\begin{proof}
[Proof of Theorem~\ref{Theorem10}.]The proof of this is just basic axiom
checking (see \cite{verlong} for details).
\end{proof}
\end{vershort}

\begin{verlong}
\begin{proof}
[Proof of Theorem~\ref{Theorem10}.]\textbf{(a)} Clearly, $\left(  A\right)
_{\rho\in\mathbb{N}}$ is a sequence of ideals of $A$. Hence, in order to prove
that $\left(  A\right)  _{\rho\in\mathbb{N}}$ is an ideal semifiltration of
$A$, it is enough to verify that it satisfies the two conditions%
\begin{align*}
A  &  =A;\\
AA  &  \subseteq A\ \ \ \ \ \ \ \ \ \ \text{for every }a\in\mathbb{N}\text{
and }b\in\mathbb{N}.
\end{align*}
But these two conditions are obviously satisfied. Hence, $\left(  A\right)
_{\rho\in\mathbb{N}}$ is an ideal semifiltration of $A$ (by
Definition~\ref{Definition6}, applied to $\left(  A\right)  _{\rho
\in\mathbb{N}}$ instead of $\left(  I_{\rho}\right)  _{\rho\in\mathbb{N}}$).
This proves Theorem~\ref{Theorem10} \textbf{(a)}.

\textbf{(b)} Since $\left(  I_{\rho}\right)  _{\rho\in\mathbb{N}}$ is an ideal
semifiltration of $A$, it is a sequence of ideals of $A$, and it satisfies the
two conditions%
\begin{align}
I_{0}  &  =A;\nonumber\\
I_{a}I_{b}  &  \subseteq I_{a+b}\ \ \ \ \ \ \ \ \ \ \text{for every }%
a\in\mathbb{N}\text{ and }b\in\mathbb{N} \label{pf.Theorem10.b.2}%
\end{align}
(by Definition~\ref{Definition6}). Since $\left(  J_{\rho}\right)  _{\rho
\in\mathbb{N}}$ is an ideal semifiltration of $A$, it is a sequence of ideals
of $A$, and it satisfies the two conditions%
\begin{align}
J_{0}  &  =A;\nonumber\\
J_{a}J_{b}  &  \subseteq J_{a+b}\ \ \ \ \ \ \ \ \ \ \text{for every }%
a\in\mathbb{N}\text{ and }b\in\mathbb{N} \label{pf.Theorem10.b.4}%
\end{align}
(by Definition~\ref{Definition6}, applied to $\left(  J_{\rho}\right)
_{\rho\in\mathbb{N}}$ instead of $\left(  I_{\rho}\right)  _{\rho\in
\mathbb{N}}$).

Now, we know that both $\left(  I_{\rho}\right)  _{\rho\in\mathbb{N}}$ and
$\left(  J_{\rho}\right)  _{\rho\in\mathbb{N}}$ are sequences of ideals of
$A$. Hence, if $\rho\in\mathbb{N}$, then both $I_{\rho}$ and $J_{\rho}$ are
ideals of $A$, and therefore $I_{\rho}J_{\rho}$ is an ideal of $A$ as well
(since the product of any two ideals of $A$ is an ideal of $A$). Thus,
$I_{\rho}J_{\rho}$ is an ideal of $A$ for each $\rho\in\mathbb{N}$. In other
words, $\left(  I_{\rho}J_{\rho}\right)  _{\rho\in\mathbb{N}}$ is a sequence
of ideals of $A$. Thus, in order to prove that $\left(  I_{\rho}J_{\rho
}\right)  _{\rho\in\mathbb{N}}$ is an ideal semifiltration of $A$, it is
enough to verify that it satisfies the two conditions%
\begin{align*}
I_{0}J_{0}  &  =A;\\
I_{a}J_{a}\cdot I_{b}J_{b}  &  \subseteq I_{a+b}J_{a+b}%
\ \ \ \ \ \ \ \ \ \ \text{for every }a\in\mathbb{N}\text{ and }b\in\mathbb{N}.
\end{align*}
But these two conditions are satisfied, since%
\begin{align*}
\underbrace{I_{0}}_{=A}\underbrace{J_{0}}_{=A}  &  =AA=A;\\
I_{a}J_{a}\cdot I_{b}J_{b}  &  =\underbrace{I_{a}I_{b}}_{\substack{\subseteq
I_{a+b}\\\text{(by (\ref{pf.Theorem10.b.2}))}}}\underbrace{J_{a}J_{b}%
}_{\substack{\subseteq J_{a+b}\\\text{(by (\ref{pf.Theorem10.b.4}))}%
}}\subseteq I_{a+b}J_{a+b}\ \ \ \ \ \ \ \ \ \ \text{for every }a\in
\mathbb{N}\text{ and }b\in\mathbb{N}.
\end{align*}
Hence, $\left(  I_{\rho}J_{\rho}\right)  _{\rho\in\mathbb{N}}$ is an ideal
semifiltration of $A$ (by Definition~\ref{Definition6}, applied to $\left(
I_{\rho}J_{\rho}\right)  _{\rho\in\mathbb{N}}$ instead of $\left(  I_{\rho
}\right)  _{\rho\in\mathbb{N}}$). This proves Theorem~\ref{Theorem10}
\textbf{(b)}.
\end{proof}
\end{verlong}

\subsection{Half-reduction}

Now let us generalize Theorem~\ref{Theorem7}:

\begin{theorem}
\label{Theorem11} Let $A$ be a ring. Let $B$ be an $A$-algebra. Let $\left(
I_{\rho}\right)  _{\rho\in\mathbb{N}}$ and $\left(  J_{\rho}\right)  _{\rho
\in\mathbb{N}}$ be two ideal semifiltrations of $A$. Let $n\in\mathbb{N}$. Let
$u\in B$.

We know that $\left(  I_{\rho}J_{\rho}\right)  _{\rho\in\mathbb{N}}$ is an
ideal semifiltration of $A$ (according to Theorem~\ref{Theorem10} \textbf{(b)}).

Consider the polynomial ring $A\left[  Y\right]  $ and its $A$-subalgebra
$A\left[  \left(  I_{\rho}\right)  _{\rho\in\mathbb{N}}\ast Y\right]  $.

We will abbreviate this $A$-subalgebra $A\left[  \left(  I_{\rho}\right)
_{\rho\in\mathbb{N}}\ast Y\right]  $ by $A_{\left[  I\right]  }$.

\textbf{(a)} The sequence $\left(  J_{\tau}A_{\left[  I\right]  }\right)
_{\tau\in\mathbb{N}}$ is an ideal semifiltration of $A_{\left[  I\right]  }$.

\textbf{(b)} The element $u$ of $B$ is $n$-integral over $\left(  A,\left(
I_{\rho}J_{\rho}\right)  _{\rho\in\mathbb{N}}\right)  $ if and only if the
element $uY$ of the polynomial ring $B\left[  Y\right]  $ is $n$-integral over
$\left(  A_{\left[  I\right]  },\left(  J_{\tau}A_{\left[  I\right]  }\right)
_{\tau\in\mathbb{N}}\right)  $. (Here, we are using the fact that $B\left[
Y\right]  $ is an $A_{\left[  I\right]  }$-algebra, because $A_{\left[
I\right]  }=A\left[  \left(  I_{\rho}\right)  _{\rho\in\mathbb{N}}\ast
Y\right]  $ is a subring of $A\left[  Y\right]  $ and because $B\left[
Y\right]  $ is an $A\left[  Y\right]  $-algebra as explained in
Definition~\ref{Definition7}.)
\end{theorem}

\begin{vershort}
\begin{proof}
[Proof of Theorem~\ref{Theorem11}.]\textbf{(a)} We know that $\left(  J_{\tau
}\right)  _{\tau\in\mathbb{N}}=\left(  J_{\rho}\right)  _{\rho\in\mathbb{N}}$
is an ideal semifiltration of $A$. Thus, by Lemma \ref{lem.J} (applied to
$A_{\left[  I\right]  }$ and $\left(  J_{\tau}\right)  _{\tau\in\mathbb{N}}$
instead of $A^{\prime}$ and $\left(  I_{\rho}\right)  _{\rho\in\mathbb{N}}$),
the sequence $\left(  J_{\tau}A_{\left[  I\right]  }\right)  _{\tau
\in\mathbb{N}}$ is an ideal semifiltration of $A_{\left[  I\right]  }$. This
proves Theorem~\ref{Theorem11} \textbf{(a)}.

\textbf{(b)} In order to verify Theorem~\ref{Theorem11} \textbf{(b)}, we have
to prove the $\Longrightarrow$ and $\Longleftarrow$ statements.

$\Longrightarrow:$ Assume that $u$ is $n$-integral over $\left(  A,\left(
I_{\rho}J_{\rho}\right)  _{\rho\in\mathbb{N}}\right)  $. Thus, by
Definition~\ref{Definition9} (applied to $\left(  I_{\rho}J_{\rho}\right)
_{\rho\in\mathbb{N}}$ instead of $\left(  I_{\rho}\right)  _{\rho\in
\mathbb{N}}$), there exists some $\left(  a_{0},a_{1},\ldots,a_{n}\right)  \in
A^{n+1}$ such that%
\[
\sum\limits_{k=0}^{n}a_{k}u^{k}=0,\ \ \ \ \ \ \ \ \ \ a_{n}%
=1,\ \ \ \ \ \ \ \ \ \ \text{and}\ \ \ \ \ \ \ \ \ \ a_{i}\in I_{n-i}%
J_{n-i}\text{ for every }i\in\left\{  0,1,\ldots,n\right\}  .
\]
Consider this $\left(  a_{0},a_{1},\ldots,a_{n}\right)  $.

For each $k\in\left\{  0,1,\ldots,n\right\}  $, we have $a_{k}\in
I_{n-k}J_{n-k}\subseteq I_{n-k}$ (since $I_{n-k}$ is an ideal of $A$) and thus
$a_{k}Y^{n-k}\in I_{n-k}Y^{n-k}\subseteq\sum_{i\in\mathbb{N}}I_{i}%
Y^{i}=A_{\left[  I\right]  }$. Thus, we can define an $\left(  n+1\right)
$-tuple $\left(  b_{0},b_{1},\ldots,b_{n}\right)  \in\left(  A_{\left[
I\right]  }\right)  ^{n+1}$ by $b_{k}=a_{k}Y^{n-k}$ for every $k\in\left\{
0,1,\ldots,n\right\}  $. This $\left(  n+1\right)  $-tuple satisfies%
\[
\sum\limits_{k=0}^{n}b_{k}\cdot\left(  uY\right)  ^{k}%
=0,\ \ \ \ \ \ \ \ \ \ b_{n}=1,\ \ \ \ \ \ \ \ \ \ \text{and}%
\ \ \ \ \ \ \ \ \ \ b_{i}\in J_{n-i}A_{\left[  I\right]  }\text{ for every
}i\in\left\{  0,1,\ldots,n\right\}
\]
(as can be easily checked). Hence, by Definition~\ref{Definition9} (applied to
$A_{\left[  I\right]  }$, $B\left[  Y\right]  $, $\left(  J_{\tau}A_{\left[
I\right]  }\right)  _{\tau\in\mathbb{N}}$, $uY$ and $\left(  b_{0}%
,b_{1},\ldots,b_{n}\right)  $ instead of $A$, $B$, $\left(  I_{\rho}\right)
_{\rho\in\mathbb{N}}$, $u$ and $\left(  a_{0},a_{1},\ldots,a_{n}\right)  $),
the element $uY$ is $n$-integral over $\left(  A_{\left[  I\right]  },\left(
J_{\tau}A_{\left[  I\right]  }\right)  _{\tau\in\mathbb{N}}\right)  $. This
proves the $\Longrightarrow$ direction of Theorem~\ref{Theorem11} \textbf{(b)}.

$\Longleftarrow:$ Assume that $uY$ is $n$-integral over $\left(  A_{\left[
I\right]  },\left(  J_{\tau}A_{\left[  I\right]  }\right)  _{\tau\in
\mathbb{N}}\right)  $. Thus, by Definition~\ref{Definition9} (applied to
$A_{\left[  I\right]  }$, $B\left[  Y\right]  $, $\left(  J_{\tau}A_{\left[
I\right]  }\right)  _{\tau\in\mathbb{N}}$, $uY$ and $\left(  p_{0}%
,p_{1},\ldots,p_{n}\right)  $ instead of $A$, $B$, $\left(  I_{\rho}\right)
_{\rho\in\mathbb{N}}$, $u$ and $\left(  a_{0},a_{1},\ldots,a_{n}\right)  $),
there exists some $\left(  p_{0},p_{1},\ldots,p_{n}\right)  \in\left(
A_{\left[  I\right]  }\right)  ^{n+1}$ such that%
\[
\sum\limits_{k=0}^{n}p_{k}\cdot\left(  uY\right)  ^{k}%
=0,\ \ \ \ \ \ \ \ \ \ p_{n}=1,\ \ \ \ \ \ \ \ \ \ \text{and}%
\ \ \ \ \ \ \ \ \ \ p_{i}\in J_{n-i}A_{\left[  I\right]  }\text{ for every
}i\in\left\{  0,1,\ldots,n\right\}  .
\]
Consider this $\left(  p_{0},p_{1},\ldots,p_{n}\right)  $. For every
$k\in\left\{  0,1,\ldots,n\right\}  $, we have%
\begin{align*}
p_{k}  &  \in J_{n-k}A_{\left[  I\right]  }\ \ \ \ \ \ \ \ \ \ \left(
\text{since }p_{i}\in J_{n-i}A_{\left[  I\right]  }\text{ for every }%
i\in\left\{  0,1,\ldots,n\right\}  \right) \\
&  =J_{n-k}\sum\limits_{i\in\mathbb{N}}I_{i}Y^{i}\ \ \ \ \ \ \ \ \ \ \left(
\text{since }A_{\left[  I\right]  }=A\left[  \left(  I_{\rho}\right)
_{\rho\in\mathbb{N}}\ast Y\right]  =\sum\limits_{i\in\mathbb{N}}I_{i}%
Y^{i}\right) \\
&  =\sum\limits_{i\in\mathbb{N}}J_{n-k}I_{i}Y^{i}=\sum\limits_{i\in\mathbb{N}%
}I_{i}J_{n-k}Y^{i},
\end{align*}
and thus there exists a sequence $\left(  p_{k,i}\right)  _{i\in\mathbb{N}}\in
A^{\mathbb{N}}$ such that $p_{k}=\sum\limits_{i\in\mathbb{N}}p_{k,i}Y^{i}$,
such that $\left(  p_{k,i}\in I_{i}J_{n-k}\text{ for every }i\in
\mathbb{N}\right)  $, and such that only finitely many $i\in\mathbb{N}$
satisfy $p_{k,i}\neq0$. Consider this sequence. Thus,%
\[
\sum\limits_{k=0}^{n}\underbrace{p_{k}}_{=\sum\limits_{i\in\mathbb{N}}%
p_{k,i}Y^{i}}\cdot\underbrace{\left(  uY\right)  ^{k}}_{\substack{=u^{k}%
Y^{k}\\=Y^{k}u^{k}}}=\sum\limits_{k=0}^{n}\sum\limits_{i\in\mathbb{N}}%
p_{k,i}\underbrace{Y^{i}\cdot Y^{k}}_{=Y^{i+k}}u^{k}=\sum\limits_{k=0}^{n}%
\sum\limits_{i\in\mathbb{N}}p_{k,i}Y^{i+k}u^{k}.
\]
Hence, $\sum\limits_{k=0}^{n}p_{k}\cdot\left(  uY\right)  ^{k}=0$ rewrites as
$\sum\limits_{k=0}^{n}\sum\limits_{i\in\mathbb{N}}p_{k,i}Y^{i+k}u^{k}=0$. In
other words, the polynomial $\sum\limits_{k=0}^{n}\sum\limits_{i\in\mathbb{N}%
}p_{k,i}Y^{i+k}u^{k}\in B\left[  Y\right]  $ equals $0$. Hence, its
coefficient before $Y^{n}$ equals $0$ as well. But its coefficient before
$Y^{n}$ is $\sum\limits_{k=0}^{n}p_{k,n-k}u^{k}$. Hence, we obtain
$\sum\limits_{k=0}^{n}p_{k,n-k}u^{k}=0$.

Recall that $\sum\limits_{i\in\mathbb{N}}p_{k,i}Y^{i}=p_{k}$ for every
$k\in\left\{  0,1,\ldots,n\right\}  $ (by the definition of the $p_{k,i}$).
Thus, $\sum\limits_{i\in\mathbb{N}}p_{n,i}Y^{i}=p_{n}=1$ in $A\left[
Y\right]  $, and thus $p_{n,0}=1$ (by comparing coefficients before $Y^{0}$).

Define an $\left(  n+1\right)  $-tuple $\left(  a_{0},a_{1},\ldots
,a_{n}\right)  \in A^{n+1}$ by $a_{k}=p_{k,n-k}$ for every $k\in\left\{
0,1,\ldots,n\right\}  $. Then, $a_{n}=p_{n,0}=1$. Besides,%
\[
\sum\limits_{k=0}^{n}\underbrace{a_{k}}_{=p_{k,n-k}}u^{k}=\sum\limits_{k=0}%
^{n}p_{k,n-k}u^{k}=0.
\]
Finally, for every $k\in\left\{  0,1,\ldots,n\right\}  $, we have
$n-k\in\mathbb{N}$ and thus $a_{k}=p_{k,n-k}\in I_{n-k}J_{n-k}$ (since
$p_{k,i}\in I_{i}J_{n-k}$ for every $i\in\mathbb{N}$). Renaming the variable
$k$ as $i$ in this statement, we obtain the following: For every $i\in\left\{
0,1,\ldots,n\right\}  $, we have $a_{i}\in I_{n-i}J_{n-i}$.

Altogether, we now know that the $\left(  n+1\right)  $-tuple $\left(
a_{0},a_{1},\ldots,a_{n}\right)  \in A^{n+1}$ satisfies%
\[
\sum\limits_{k=0}^{n}a_{k}u^{k}=0,\ \ \ \ \ \ \ \ \ \ a_{n}%
=1,\ \ \ \ \ \ \ \ \ \ \text{and}\ \ \ \ \ \ \ \ \ \ a_{i}\in I_{n-i}%
J_{n-i}\text{ for every }i\in\left\{  0,1,\ldots,n\right\}  .
\]
Thus, by Definition~\ref{Definition9} (applied to $\left(  I_{\rho}J_{\rho
}\right)  _{\rho\in\mathbb{N}}$ instead of $\left(  I_{\rho}\right)  _{\rho
\in\mathbb{N}}$), the element $u$ is $n$-integral over $\left(  A,\left(
I_{\rho}J_{\rho}\right)  _{\rho\in\mathbb{N}}\right)  $. This proves the
$\Longleftarrow$ direction of Theorem~\ref{Theorem11} \textbf{(b)}, and thus
Theorem~\ref{Theorem11} \textbf{(b)} is shown.
\end{proof}
\end{vershort}

\begin{verlong}
\begin{proof}
[Proof of Theorem~\ref{Theorem11}.]The definition of $A_{\left[  I\right]  }$
yields%
\begin{align*}
A_{\left[  I\right]  }  &  =A\left[  \left(  I_{\rho}\right)  _{\rho
\in\mathbb{N}}\ast Y\right]  =\sum\limits_{i\in\mathbb{N}}I_{i}Y^{i}%
\ \ \ \ \ \ \ \ \ \ \left(  \text{by Definition~\ref{Definition8}}\right) \\
&  =\sum\limits_{\ell\in\mathbb{N}}I_{\ell}Y^{\ell}\ \ \ \ \ \ \ \ \ \ \left(
\text{here we renamed }i\text{ as }\ell\text{ in the sum}\right)  .
\end{align*}
As a consequence of this chain of equalities, we have $\sum\limits_{i\in
\mathbb{N}}I_{i}Y^{i}=A_{\left[  I\right]  }$ and $\sum\limits_{\ell
\in\mathbb{N}}I_{\ell}Y^{\ell}=A_{\left[  I\right]  }$.

\textbf{(a)} We know that $\left(  J_{\rho}\right)  _{\rho\in\mathbb{N}}$ is
an ideal semifiltration of $A$. In other words, $\left(  J_{\tau}\right)
_{\tau\in\mathbb{N}}$ is an ideal semifiltration of $A$ (since $\left(
J_{\tau}\right)  _{\tau\in\mathbb{N}}=\left(  J_{\rho}\right)  _{\rho
\in\mathbb{N}}$). Thus, by Lemma \ref{lem.J} (applied to $A_{\left[  I\right]
}$ and $\left(  J_{\tau}\right)  _{\tau\in\mathbb{N}}$ instead of $A^{\prime}$
and $\left(  I_{\rho}\right)  _{\rho\in\mathbb{N}}$), the sequence $\left(
J_{\tau}A_{\left[  I\right]  }\right)  _{\tau\in\mathbb{N}}$ is an ideal
semifiltration of $A_{\left[  I\right]  }$. This proves
Theorem~\ref{Theorem11} \textbf{(a)}.

\textbf{(b)} In order to verify Theorem~\ref{Theorem11} \textbf{(b)}, we have
to prove the following two lemmata:

\begin{statement}
\textit{Lemma }$\mathcal{E}^{\prime}$\textit{:} If $u$ is $n$-integral over
$\left(  A,\left(  I_{\rho}J_{\rho}\right)  _{\rho\in\mathbb{N}}\right)  $,
then $uY$ is $n$-integral over $\left(  A_{\left[  I\right]  },\left(
J_{\tau}A_{\left[  I\right]  }\right)  _{\tau\in\mathbb{N}}\right)  $.
\end{statement}

\begin{statement}
\textit{Lemma} $\mathcal{F}^{\prime}$\textit{:} If $uY$ is $n$-integral over
$\left(  A_{\left[  I\right]  },\left(  J_{\tau}A_{\left[  I\right]  }\right)
_{\tau\in\mathbb{N}}\right)  $, then $u$ is $n$-integral over $\left(
A,\left(  I_{\rho}J_{\rho}\right)  _{\rho\in\mathbb{N}}\right)  $.
\end{statement}

[\textit{Proof of Lemma }$\mathcal{E}^{\prime}$\textit{:} Assume that $u$ is
$n$-integral over $\left(  A,\left(  I_{\rho}J_{\rho}\right)  _{\rho
\in\mathbb{N}}\right)  $. Thus, by Definition~\ref{Definition9} (applied to
$\left(  I_{\rho}J_{\rho}\right)  _{\rho\in\mathbb{N}}$ instead of $\left(
I_{\rho}\right)  _{\rho\in\mathbb{N}}$), there exists some $\left(
a_{0},a_{1},\ldots,a_{n}\right)  \in A^{n+1}$ such that%
\[
\sum\limits_{k=0}^{n}a_{k}u^{k}=0,\ \ \ \ \ \ \ \ \ \ a_{n}%
=1,\ \ \ \ \ \ \ \ \ \ \text{and}\ \ \ \ \ \ \ \ \ \ a_{i}\in I_{n-i}%
J_{n-i}\text{ for every }i\in\left\{  0,1,\ldots,n\right\}  .
\]
Consider this $\left(  a_{0},a_{1},\ldots,a_{n}\right)  $.

For every $k\in\left\{  0,1,\ldots,n\right\}  $, we have
\begin{align*}
a_{k}  &  \in I_{n-k}\underbrace{J_{n-k}}_{\subseteq A}%
\ \ \ \ \ \ \ \ \ \ \left(  \text{since }a_{i}\in I_{n-i}J_{n-i}\text{ for
every }i\in\left\{  0,1,\ldots,n\right\}  \right) \\
&  \subseteq I_{n-k}A\subseteq I_{n-k}\ \ \ \ \ \ \ \ \ \ \left(  \text{since
}I_{n-k}\text{ is an ideal of }A\right)
\end{align*}
and thus
\[
\underbrace{a_{k}}_{\in I_{n-k}}Y^{n-k}\in I_{n-k}Y^{n-k}\subseteq
\sum\limits_{i\in\mathbb{N}}I_{i}Y^{i}=A_{\left[  I\right]  }.
\]

Thus, we can define an $\left(  n+1\right)  $-tuple $\left(  b_{0}%
,b_{1},\ldots,b_{n}\right)  \in\left(  A_{\left[  I\right]  }\right)  ^{n+1}$
by setting
\[
\left(  b_{k}=a_{k}Y^{n-k}\text{ for every }k\in\left\{  0,1,\ldots,n\right\}
\right)  .
\]
Consider this $\left(  n+1\right)  $-tuple. The definition of this $\left(
n+1\right)  $-tuple yields%
\begin{align*}
\sum\limits_{k=0}^{n}\underbrace{b_{k}}_{=a_{k}Y^{n-k}}\cdot
\underbrace{\left(  uY\right)  ^{k}}_{=u^{k}Y^{k}}  &  =\sum\limits_{k=0}%
^{n}a_{k}Y^{n-k}u^{k}Y^{k}=\sum\limits_{k=0}^{n}a_{k}u^{k}\underbrace{Y^{n-k}%
Y^{k}}_{=Y^{n}}=Y^{n}\cdot\underbrace{\sum\limits_{k=0}^{n}a_{k}u^{k}}%
_{=0}=0;\\
b_{n}  &  =\underbrace{a_{n}}_{=1}\underbrace{Y^{n-n}}_{=Y^{0}=1}=1,
\end{align*}
and%
\[
b_{i}=\underbrace{a_{i}}_{\substack{\in I_{n-i}J_{n-i}\\=J_{n-i}I_{n-i}%
}}Y^{n-i}\in J_{n-i}\underbrace{I_{n-i}Y^{n-i}}_{\substack{\subseteq
\sum\limits_{\ell\in\mathbb{N}}I_{\ell}Y^{\ell}\\=A_{\left[  I\right]  }%
}}\subseteq J_{n-i}A_{\left[  I\right]  }\ \ \ \ \ \ \ \ \ \ \text{for every
}i\in\left\{  0,1,\ldots,n\right\}  .
\]

Altogether, we now know that $\left(  b_{0},b_{1},\ldots,b_{n}\right)
\in\left(  A_{\left[  I\right]  }\right)  ^{n+1}$ and%
\[
\sum\limits_{k=0}^{n}b_{k}\cdot\left(  uY\right)  ^{k}%
=0,\ \ \ \ \ \ \ \ \ \ b_{n}=1,\ \ \ \ \ \ \ \ \ \ \text{and}%
\ \ \ \ \ \ \ \ \ \ b_{i}\in J_{n-i}A_{\left[  I\right]  }\text{ for every
}i\in\left\{  0,1,\ldots,n\right\}  .
\]
Hence, by Definition~\ref{Definition9} (applied to $A_{\left[  I\right]  }$,
$B\left[  Y\right]  $, $\left(  J_{\tau}A_{\left[  I\right]  }\right)
_{\tau\in\mathbb{N}}$, $uY$ and $\left(  b_{0},b_{1},\ldots,b_{n}\right)  $
instead of $A$, $B$, $\left(  I_{\rho}\right)  _{\rho\in\mathbb{N}}$, $u$ and
$\left(  a_{0},a_{1},\ldots,a_{n}\right)  $), the element $uY$ is $n$-integral
over $\left(  A_{\left[  I\right]  },\left(  J_{\tau}A_{\left[  I\right]
}\right)  _{\tau\in\mathbb{N}}\right)  $. This proves Lemma $\mathcal{E}%
^{\prime}$.]

[\textit{Proof of Lemma }$\mathcal{F}^{\prime}$\textit{:} Assume that $uY$ is
$n$-integral over $\left(  A_{\left[  I\right]  },\left(  J_{\tau}A_{\left[
I\right]  }\right)  _{\tau\in\mathbb{N}}\right)  $. Thus, by
Definition~\ref{Definition9} (applied to $A_{\left[  I\right]  }$, $B\left[
Y\right]  $, $\left(  J_{\tau}A_{\left[  I\right]  }\right)  _{\tau
\in\mathbb{N}}$, $uY$ and $\left(  p_{0},p_{1},\ldots,p_{n}\right)  $ instead
of $A$, $B$, $\left(  I_{\rho}\right)  _{\rho\in\mathbb{N}}$, $u$ and $\left(
a_{0},a_{1},\ldots,a_{n}\right)  $), there exists some $\left(  p_{0}%
,p_{1},\ldots,p_{n}\right)  \in\left(  A_{\left[  I\right]  }\right)  ^{n+1}$
such that%
\[
\sum\limits_{k=0}^{n}p_{k}\cdot\left(  uY\right)  ^{k}%
=0,\ \ \ \ \ \ \ \ \ \ p_{n}=1,\ \ \ \ \ \ \ \ \ \ \text{and}%
\ \ \ \ \ \ \ \ \ \ p_{i}\in J_{n-i}A_{\left[  I\right]  }\text{ for every
}i\in\left\{  0,1,\ldots,n\right\}  .
\]
Consider this $\left(  p_{0},p_{1},\ldots,p_{n}\right)  $. For every
$k\in\left\{  0,1,\ldots,n\right\}  $, we have%
\begin{align*}
p_{k}  &  \in J_{n-k}A_{\left[  I\right]  }\ \ \ \ \ \ \ \ \ \ \left(
\text{since }p_{i}\in J_{n-i}A_{\left[  I\right]  }\text{ for every }%
i\in\left\{  0,1,\ldots,n\right\}  \right) \\
&  =J_{n-k}\sum\limits_{i\in\mathbb{N}}I_{i}Y^{i}\ \ \ \ \ \ \ \ \ \ \left(
\text{since }A_{\left[  I\right]  }=\sum\limits_{i\in\mathbb{N}}I_{i}%
Y^{i}\right) \\
&  =\sum\limits_{i\in\mathbb{N}}\underbrace{J_{n-k}I_{i}}_{=I_{i}J_{n-k}}%
Y^{i}=\sum\limits_{i\in\mathbb{N}}I_{i}J_{n-k}Y^{i},
\end{align*}
and thus there exists a sequence $\left(  p_{k,i}\right)  _{i\in\mathbb{N}}\in
A^{\mathbb{N}}$ such that $p_{k}=\sum\limits_{i\in\mathbb{N}}p_{k,i}Y^{i}$,
such that $\left(  p_{k,i}\in I_{i}J_{n-k}\text{ for every }i\in
\mathbb{N}\right)  $, and such that only finitely many $i\in\mathbb{N}$
satisfy $p_{k,i}\neq0$. Consider this sequence. Thus,%
\begin{align*}
&  \sum\limits_{k=0}^{n}\underbrace{p_{k}}_{=\sum\limits_{i\in\mathbb{N}%
}p_{k,i}Y^{i}}\cdot\underbrace{\left(  uY\right)  ^{k}}_{\substack{=u^{k}%
Y^{k}\\=Y^{k}u^{k}}}\\
&  =\sum\limits_{k=0}^{n}\left(  \sum\limits_{i\in\mathbb{N}}p_{k,i}%
Y^{i}\right)  \cdot Y^{k}u^{k}=\underbrace{\sum\limits_{k=0}^{n}}%
_{=\sum\limits_{k\in\left\{  0,1,\ldots,n\right\}  }}\sum\limits_{i\in
\mathbb{N}}p_{k,i}\underbrace{Y^{i}\cdot Y^{k}}_{=Y^{i+k}}u^{k}=\sum
\limits_{k\in\left\{  0,1,\ldots,n\right\}  }\sum\limits_{i\in\mathbb{N}%
}p_{k,i}Y^{i+k}u^{k}\\
&  =\sum\limits_{\left(  k,i\right)  \in\left\{  0,1,\ldots,n\right\}
\times\mathbb{N}}p_{k,i}Y^{i+k}u^{k}=\sum_{\ell\in\mathbb{N}}\sum
\limits_{\substack{\left(  k,i\right)  \in\left\{  0,1,\ldots,n\right\}
\times\mathbb{N};\\i+k=\ell}}p_{k,i}\underbrace{Y^{i+k}}_{\substack{=Y^{\ell
}\\\text{(since }i+k=\ell\text{)}}}u^{k}\\
&  =\sum_{\ell\in\mathbb{N}}\sum\limits_{\substack{\left(  k,i\right)
\in\left\{  0,1,\ldots,n\right\}  \times\mathbb{N};\\i+k=\ell}}p_{k,i}Y^{\ell
}u^{k}=\sum_{\ell\in\mathbb{N}}\sum\limits_{\substack{\left(  k,i\right)
\in\left\{  0,1,\ldots,n\right\}  \times\mathbb{N};\\i+k=\ell}}p_{k,i}%
u^{k}Y^{\ell}.
\end{align*}
Comparing this with $\sum\limits_{k=0}^{n}p_{k}\cdot\left(  uY\right)  ^{k}%
=0$, we obtain $\sum\limits_{\ell\in\mathbb{N}}\sum\limits_{\substack{\left(
k,i\right)  \in\left\{  0,1,\ldots,n\right\}  \times\mathbb{N};\\i+k=\ell
}}p_{k,i}u^{k}Y^{\ell}=0$. In other words, the polynomial $\sum\limits_{\ell
\in\mathbb{N}}\underbrace{\sum\limits_{\substack{\left(  k,i\right)
\in\left\{  0,1,\ldots,n\right\}  \times\mathbb{N};\\i+k=\ell}}p_{k,i}u^{k}%
}_{\in B}Y^{\ell}\in B\left[  Y\right]  $ equals $0$. Hence, its coefficient
before $Y^{n}$ equals $0$ as well. But its coefficient before $Y^{n}$ is
$\sum\limits_{\substack{\left(  k,i\right)  \in\left\{  0,1,\ldots,n\right\}
\times\mathbb{N};\\i+k=n}}p_{k,i}u^{k}$. Comparing the preceding two
sentences, we see that $\sum\limits_{\substack{\left(  k,i\right)  \in\left\{
0,1,\ldots,n\right\}  \times\mathbb{N};\\i+k=n}}p_{k,i}u^{k}$ equals $0$.
Thus,%
\begin{equation}
0=\sum\limits_{\substack{\left(  k,i\right)  \in\left\{  0,1,\ldots,n\right\}
\times\mathbb{N};\\i+k=n}}p_{k,i}u^{k}=\sum\limits_{k\in\left\{
0,1,\ldots,n\right\}  }\sum_{\substack{i\in\mathbb{N};\\i+k=n}}p_{k,i}u^{k}.
\label{T11.pf.6}%
\end{equation}

But for any given $k\in\left\{  0,1,\ldots,n\right\}  $, we have%
\[
\left\{  i\in\mathbb{N}\text{\ }\mid\ \underbrace{i+k=n}_{\Longleftrightarrow
\ \left(  i=n-k\right)  }\right\}  =\left\{  i\in\mathbb{N}\ \mid
\ i=n-k\right\}  =\left\{  n-k\right\}
\]
(since $n-k\in\mathbb{N}$ (because $k\in\left\{  0,1,\ldots,n\right\}  $)) and
therefore%
\[
\sum\limits_{\substack{i\in\mathbb{N};\\i+k=n}}p_{k,i}u^{k}=\sum
\limits_{i\in\left\{  n-k\right\}  }p_{k,i}u^{k}=p_{k,n-k}u^{k}.
\]
Hence, (\ref{T11.pf.6}) becomes%
\begin{equation}
0=\sum\limits_{k\in\left\{  0,1,\ldots,n\right\}  }\underbrace{\sum
_{\substack{i\in\mathbb{N};\\i+k=n}}p_{k,i}u^{k}}_{=p_{k,n-k}u^{k}}%
=\sum\limits_{k\in\left\{  0,1,\ldots,n\right\}  }p_{k,n-k}u^{k}.
\label{pf.Theorem11.7}%
\end{equation}

Recall that $p_{k}=\sum\limits_{i\in\mathbb{N}}p_{k,i}Y^{i}$ for every
$k\in\left\{  0,1,\ldots,n\right\}  $. Applying this to $k=n$, we find
$p_{n}=\sum\limits_{i\in\mathbb{N}}p_{n,i}Y^{i}$. Comparing this with
$p_{n}=1=1\cdot Y^{0}$, we find
\[
\sum\limits_{i\in\mathbb{N}}p_{n,i}Y^{i}=1\cdot Y^{0}%
\ \ \ \ \ \ \ \ \ \ \text{in }A\left[  Y\right]  .
\]
Hence, the coefficient of the polynomial $\sum\limits_{i\in\mathbb{N}}%
p_{n,i}Y^{i}\in A\left[  Y\right]  $ before $Y^{0}$ is $1$. But the
coefficient of the polynomial $\sum\limits_{i\in\mathbb{N}}p_{n,i}Y^{i}\in
A\left[  Y\right]  $ before $Y^{0}$ is $p_{n,0}$ (since $p_{n,i}\in A$ for all
$i\in\mathbb{N}$). Comparing the preceding two sentences, we see that
$p_{n,0}=1$.

Define an $\left(  n+1\right)  $-tuple $\left(  a_{0},a_{1},\ldots
,a_{n}\right)  \in A^{n+1}$ by setting%
\[
\left(  a_{k}=p_{k,n-k}\text{ for every }k\in\left\{  0,1,\ldots,n\right\}
\right)  .
\]
Then, $a_{n}=p_{n,n-n}=p_{n,0}=1$. Besides,%
\[
\sum\limits_{k=0}^{n}\underbrace{a_{k}}_{\substack{=p_{k,n-k}\\\text{(by the
definition}\\\text{of }a_{k}\text{)}}}u^{k}=\sum\limits_{k=0}^{n}%
p_{k,n-k}u^{k}=\sum\limits_{k\in\left\{  0,1,\ldots,n\right\}  }p_{k,n-k}%
u^{k}=0\ \ \ \ \ \ \ \ \ \ \left(  \text{by (\ref{pf.Theorem11.7})}\right)  .
\]
Finally, for every $k\in\left\{  0,1,\ldots,n\right\}  $, we have
$n-k\in\mathbb{N}$ and thus $a_{k}=p_{k,n-k}\in I_{n-k}J_{n-k}$ (since
$p_{k,i}\in I_{i}J_{n-k}$ for every $i\in\mathbb{N}$). Renaming the variable
$k$ as $i$ in this statement, we obtain the following: For every $i\in\left\{
0,1,\ldots,n\right\}  $, we have $a_{i}\in I_{n-i}J_{n-i}$.

Altogether, we now know that the $\left(  n+1\right)  $-tuple $\left(
a_{0},a_{1},\ldots,a_{n}\right)  \in A^{n+1}$ satisfies%
\[
\sum\limits_{k=0}^{n}a_{k}u^{k}=0,\ \ \ \ \ \ \ \ \ \ a_{n}%
=1,\ \ \ \ \ \ \ \ \ \ \text{and}\ \ \ \ \ \ \ \ \ \ a_{i}\in I_{n-i}%
J_{n-i}\text{ for every }i\in\left\{  0,1,\ldots,n\right\}  .
\]
Thus, by Definition~\ref{Definition9} (applied to $\left(  I_{\rho}J_{\rho
}\right)  _{\rho\in\mathbb{N}}$ instead of $\left(  I_{\rho}\right)  _{\rho
\in\mathbb{N}}$), the element $u$ is $n$-integral over $\left(  A,\left(
I_{\rho}J_{\rho}\right)  _{\rho\in\mathbb{N}}\right)  $. This proves Lemma
$\mathcal{F}^{\prime}$.]

Combining Lemma $\mathcal{E}^{\prime}$ and Lemma $\mathcal{F}^{\prime}$, we
obtain that $u$ is $n$-integral over $\left(  A,\left(  I_{\rho}J_{\rho
}\right)  _{\rho\in\mathbb{N}}\right)  $ if and only if $uY$ is $n$-integral
over $\left(  A_{\left[  I\right]  },\left(  J_{\tau}A_{\left[  I\right]
}\right)  _{\tau\in\mathbb{N}}\right)  $. This proves Theorem~\ref{Theorem11}
\textbf{(b)}.
\end{proof}
\end{verlong}

The reason why Theorem~\ref{Theorem11} \textbf{(b)} generalizes
Theorem~\ref{Theorem7} (more precisely, Theorem~\ref{Theorem7} is the
particular case of Theorem~\ref{Theorem11} \textbf{(b)} for $J_{\rho}=A$) is
the following fact, which we mention here for the pure sake of completeness:

\begin{theorem}
\label{Theorem12} Let $A$ be a ring. Let $B$ be an $A$-algebra. Let
$n\in\mathbb{N}$. Let $u\in B$.

We know that $\left(  A\right)  _{\rho\in\mathbb{N}}$ is an ideal
semifiltration of $A$ (according to Theorem~\ref{Theorem10} \textbf{(a)}).

Then, the element $u$ of $B$ is $n$-integral over $\left(  A,\left(  A\right)
_{\rho\in\mathbb{N}}\right)  $ if and only if $u$ is $n$-integral over $A$.
\end{theorem}

\begin{verlong}
\begin{proof}
[Proof of Theorem~\ref{Theorem12}.]In order to verify Theorem~\ref{Theorem12},
we have to prove the following two lemmata:

\begin{statement}
\textit{Lemma }$\mathcal{L}$\textit{:} If $u$ is $n$-integral over $\left(
A,\left(  A\right)  _{\rho\in\mathbb{N}}\right)  $, then $u$ is $n$-integral
over $A$.
\end{statement}

\begin{statement}
\textit{Lemma} $\mathcal{M}$\textit{:} If $u$ is $n$-integral over $A$, then
$u$ is $n$-integral over $\left(  A,\left(  A\right)  _{\rho\in\mathbb{N}%
}\right)  $.
\end{statement}

[\textit{Proof of Lemma }$\mathcal{L}$\textit{:} Assume that $u$ is
$n$-integral over $\left(  A,\left(  A\right)  _{\rho\in\mathbb{N}}\right)  $.
Thus, by Definition~\ref{Definition9} (applied to $\left(  A\right)  _{\rho
\in\mathbb{N}}$ instead of $\left(  I_{\rho}\right)  _{\rho\in\mathbb{N}}$),
there exists some $\left(  a_{0},a_{1},\ldots,a_{n}\right)  \in A^{n+1}$ such
that%
\[
\sum\limits_{k=0}^{n}a_{k}u^{k}=0,\ \ \ \ \ \ \ \ \ \ a_{n}%
=1,\ \ \ \ \ \ \ \ \ \ \text{and}\ \ \ \ \ \ \ \ \ \ a_{i}\in A\text{ for
every }i\in\left\{  0,1,\ldots,n\right\}  .
\]
Consider this $\left(  a_{0},a_{1},\ldots,a_{n}\right)  $.

Define a polynomial $P\in A\left[  X\right]  $ by $P\left(  X\right)
=\sum\limits_{k=0}^{n}a_{k}X^{k}$. Then, $P\left(  X\right)  =\sum
\limits_{k=0}^{n}a_{k}X^{k}=\underbrace{a_{n}}_{=1}X^{n}+\sum\limits_{k=0}%
^{n-1}a_{k}X^{k}=X^{n}+\sum\limits_{k=0}^{n-1}a_{k}X^{k}$. Hence, the
polynomial $P$ is monic, and $\deg P=n$. Besides, $P\left(  u\right)  =0$
(since $P\left(  X\right)  =\sum\limits_{k=0}^{n}a_{k}X^{k}$ yields $P\left(
u\right)  =\sum\limits_{k=0}^{n}a_{k}u^{k}=0$). Thus, there exists a monic
polynomial $P\in A\left[  X\right]  $ with $\deg P=n$ and $P\left(  u\right)
=0$. Hence, $u$ is $n$-integral over $A$. This proves Lemma $\mathcal{L}$.]

[\textit{Proof of Lemma }$\mathcal{M}$\textit{:} Assume that $u$ is
$n$-integral over $A$. Thus, there exists a monic polynomial $P\in A\left[
X\right]  $ with $\deg P=n$ and $P\left(  u\right)  =0$. Consider this $P$.
Since $\deg P=n$, there exists some $\left(  n+1\right)  $-tuple $\left(
a_{0},a_{1},\ldots,a_{n}\right)  \in A^{n+1}$ such that $P\left(  X\right)
=\sum\limits_{k=0}^{n}a_{k}X^{k}$. Consider this $\left(  a_{0},a_{1}%
,\ldots,a_{n}\right)  $. Thus, $a_{n}=1$ (since $P$ is monic, and $\deg P=n$).
Also, substituting $u$ for $X$ in the equality $\sum\limits_{k=0}^{n}%
a_{k}X^{k}=P\left(  X\right)  $ yields $\sum\limits_{k=0}^{n}a_{k}%
u^{k}=P\left(  u\right)  =0$. Altogether, we now know that $\left(
a_{0},a_{1},\ldots,a_{n}\right)  \in A^{n+1}$ and%
\[
\sum\limits_{k=0}^{n}a_{k}u^{k}=0,\ \ \ \ \ \ \ \ \ \ a_{n}%
=1,\ \ \ \ \ \ \ \ \ \ \text{and}\ \ \ \ \ \ \ \ \ \ a_{i}\in A\text{ for
every }i\in\left\{  0,1,\ldots,n\right\}  .
\]
Hence, by Definition~\ref{Definition9} (applied to $\left(  A\right)
_{\rho\in\mathbb{N}}$ instead of $\left(  I_{\rho}\right)  _{\rho\in
\mathbb{N}}$), the element $u$ is $n$-integral over $\left(  A,\left(
A\right)  _{\rho\in\mathbb{N}}\right)  $. This proves Lemma $\mathcal{M}$.]

Combining Lemma $\mathcal{L}$ and Lemma $\mathcal{M}$, we obtain that $u$ is
$n$-integral over $\left(  A,\left(  A\right)  _{\rho\in\mathbb{N}}\right)  $
if and only if $u$ is $n$-integral over $A$. This proves
Theorem~\ref{Theorem12}.
\end{proof}
\end{verlong}

\subsection{Integrality of products over the product semifiltration}

Finally, let us generalize Theorem~\ref{Theorem8c}:

\begin{theorem}
\label{Theorem13} Let $A$ be a ring. Let $B$ be an $A$-algebra. Let $\left(
I_{\rho}\right)  _{\rho\in\mathbb{N}}$ and $\left(  J_{\rho}\right)  _{\rho
\in\mathbb{N}}$ be two ideal semifiltrations of $A$.

Let $x\in B$ and $y\in B$. Let $m\in\mathbb{N}$ and $n\in\mathbb{N}$. Assume
that $x$ is $m$-integral over $\left(  A,\left(  I_{\rho}\right)  _{\rho
\in\mathbb{N}}\right)  $, and that $y$ is $n$-integral over $\left(  A,\left(
J_{\rho}\right)  _{\rho\in\mathbb{N}}\right)  $. Then, $xy$ is $nm$-integral
over $\left(  A,\left(  I_{\rho}J_{\rho}\right)  _{\rho\in\mathbb{N}}\right)
$.
\end{theorem}

The proof of this theorem will require a generalization of Lemma \ref{lem.I}:

\begin{lemma}
\label{lem.I'}Let $A$ be a ring. Let $A^{\prime}$ be an $A$-algebra. Let
$B^{\prime}$ be an $A^{\prime}$-algebra. Let $\left(  I_{\rho}\right)
_{\rho\in\mathbb{N}}$ be an ideal semifiltration of $A$. Let $v\in B^{\prime}%
$. Let $n\in\mathbb{N}$. Assume that $v$ is $n$-integral over $\left(
A,\left(  I_{\rho}\right)  _{\rho\in\mathbb{N}}\right)  $. (Here, of course,
we are using the fact that $B^{\prime}$ is an $A$-algebra, since $B^{\prime}$
is an $A^{\prime}$-algebra while $A^{\prime}$ is an $A$-algebra.)

Then, $v$ is $n$-integral over $\left(  A^{\prime},\left(  I_{\rho}A^{\prime
}\right)  _{\rho\in\mathbb{N}}\right)  $. (Note that $\left(  I_{\rho
}A^{\prime}\right)  _{\rho\in\mathbb{N}}$ is an ideal semifiltration of
$A^{\prime}$, according to Lemma \ref{lem.J}.)
\end{lemma}

\begin{vershort}
\begin{proof}
[Proof of Lemma \ref{lem.I'}.]This becomes obvious upon unraveling the
definitions of \textquotedblleft$n$-integral over $\left(  A,\left(  I_{\rho
}\right)  _{\rho\in\mathbb{N}}\right)  $\textquotedblright\ and of
\textquotedblleft$n$-integral over $\left(  A^{\prime},\left(  I_{\rho
}A^{\prime}\right)  _{\rho\in\mathbb{N}}\right)  $\textquotedblright, and by
realizing that every $\rho\in\mathbb{N}$ and every $a\in I_{\rho}$ satisfy
$a\cdot1_{A^{\prime}}\in I_{\rho}A^{\prime}$. (See \cite{verlong} for details.)
\end{proof}
\end{vershort}

\begin{verlong}
\begin{proof}
[Proof of Lemma \ref{lem.I'}.]We know that $v$ is $n$-integral over $\left(
A,\left(  I_{\rho}\right)  _{\rho\in\mathbb{N}}\right)  $. Thus, by
Definition~\ref{Definition9} (applied to $B=B^{\prime}$ and $u=v$), there
exists some $\left(  a_{0},a_{1},\ldots,a_{n}\right)  \in A^{n+1}$ such that%
\[
\sum\limits_{k=0}^{n}a_{k}v^{k}=0,\ \ \ \ \ \ \ \ \ \ a_{n}%
=1,\ \ \ \ \ \ \ \ \ \ \text{and}\ \ \ \ \ \ \ \ \ \ a_{i}\in I_{n-i}\text{
for every }i\in\left\{  0,1,\ldots,n\right\}  .
\]
Consider this $\left(  a_{0},a_{1},\ldots,a_{n}\right)  $.

Now, recall that $A^{\prime}$ is an $A$-algebra. Define an $\left(
n+1\right)  $-tuple $\left(  b_{0},b_{1},\ldots,b_{n}\right)  \in\left(
A^{\prime}\right)  ^{n+1}$ by setting%
\[
\left(  b_{i}=a_{i}\cdot1_{A^{\prime}}\ \ \ \ \ \ \ \ \ \ \text{for each }%
i\in\left\{  0,1,\ldots,n\right\}  \right)  .
\]
Then, we have $b_{i}=\underbrace{a_{i}}_{\in I_{n-i}}\cdot
\underbrace{1_{A^{\prime}}}_{\in A^{\prime}}\in I_{n-i}A^{\prime}$ for every
$i\in\left\{  0,1,\ldots,n\right\}  $. Also,
\[
\sum\limits_{k=0}^{n}\underbrace{b_{k}}_{\substack{=a_{k}\cdot1_{A^{\prime}%
}\\\text{(by the definition of }b_{k}\text{)}}}v^{k}=\sum\limits_{k=0}%
^{n}\underbrace{\left(  a_{k}\cdot1_{A^{\prime}}\right)  v^{k}}_{=a_{k}v^{k}%
}=\sum\limits_{k=0}^{n}a_{k}v^{k}=0.
\]
Furthermore, the definition of $b_{n}$ yields $b_{n}=\underbrace{a_{n}}%
_{=1}\cdot1_{A^{\prime}}=1_{A^{\prime}}=1$.

Thus, $\left(  b_{0},b_{1},\ldots,b_{n}\right)  \in\left(  A^{\prime}\right)
^{n+1}$ and%
\[
\sum\limits_{k=0}^{n}b_{k}v^{k}=0,\ \ \ \ \ \ \ \ \ \ b_{n}%
=1,\ \ \ \ \ \ \ \ \ \ \text{and}\ \ \ \ \ \ \ \ \ \ b_{i}\in I_{n-i}%
A^{\prime}\text{ for every }i\in\left\{  0,1,\ldots,n\right\}  .
\]
Hence, by Definition~\ref{Definition9} (applied to $B^{\prime}$, $A^{\prime}$,
$\left(  I_{\rho}A^{\prime}\right)  _{\rho\in\mathbb{N}}$, $v$ and $\left(
b_{0},b_{1},\ldots,b_{n}\right)  $ instead of $B$, $A$, $\left(  I_{\rho
}\right)  _{\rho\in\mathbb{N}}$, $u$ and $\left(  a_{0},a_{1},\ldots
,a_{n}\right)  $), the element $v$ is $n$-integral over $\left(  A^{\prime
},\left(  I_{\rho}A^{\prime}\right)  _{\rho\in\mathbb{N}}\right)  $. This
proves Lemma~\ref{lem.I'}.
\end{proof}
\end{verlong}

\begin{proof}
[Proof of Theorem~\ref{Theorem13}.]We have $\left(  J_{\rho}\right)  _{\rho
\in\mathbb{N}}=\left(  J_{\tau}\right)  _{\tau\in\mathbb{N}}$. Hence, $y$ is
$n$-integral over $\left(  A,\left(  J_{\tau}\right)  _{\tau\in\mathbb{N}%
}\right)  $ (since $y$ is $n$-integral over $\left(  A,\left(  J_{\rho
}\right)  _{\rho\in\mathbb{N}}\right)  $). Also, $\left(  J_{\tau}\right)
_{\tau\in\mathbb{N}}$ is an ideal semifiltration of $A$ (since $\left(
J_{\rho}\right)  _{\rho\in\mathbb{N}}$ is an ideal semifiltration of $A$, but
we have $\left(  J_{\rho}\right)  _{\rho\in\mathbb{N}}=\left(  J_{\tau
}\right)  _{\tau\in\mathbb{N}}$). Thus, $\left(  J_{\tau}A_{\left[  I\right]
}\right)  _{\tau\in\mathbb{N}}$ is an ideal semifiltration of $A_{\left[
I\right]  }$ (by Lemma \ref{lem.J}, applied to $A_{\left[  I\right]  }$ and
$\left(  J_{\tau}\right)  _{\tau\in\mathbb{N}}$ instead of $A^{\prime}$ and
$\left(  I_{\rho}\right)  _{\rho\in\mathbb{N}}$).

Consider the polynomial ring $A\left[  Y\right]  $ and its $A$-subalgebra
$A\left[  \left(  I_{\rho}\right)  _{\rho\in\mathbb{N}}\ast Y\right]  $. We
will abbreviate this $A$-subalgebra $A\left[  \left(  I_{\rho}\right)
_{\rho\in\mathbb{N}}\ast Y\right]  $ by $A_{\left[  I\right]  }$. Thus,
$A_{\left[  I\right]  }=A\left[  \left(  I_{\rho}\right)  _{\rho\in\mathbb{N}%
}\ast Y\right]  $ is a subring of $A\left[  Y\right]  $. Hence, $B\left[
Y\right]  $ is an $A_{\left[  I\right]  }$-algebra (since $B\left[  Y\right]
$ is an $A\left[  Y\right]  $-algebra as explained in
Definition~\ref{Definition7}).

Theorem~\ref{Theorem7} (applied to $x$ and $m$ instead of $u$ and $n$) yields
that $xY$ is $m$-integral over $A\left[  \left(  I_{\rho}\right)  _{\rho
\in\mathbb{N}}\ast Y\right]  $ (since $x$ is $m$-integral over $\left(
A,\left(  I_{\rho}\right)  _{\rho\in\mathbb{N}}\right)  $). In other words,
$xY$ is $m$-integral over $A_{\left[  I\right]  }$ (since $A\left[  \left(
I_{\rho}\right)  _{\rho\in\mathbb{N}}\ast Y\right]  =A_{\left[  I\right]  }$).

On the other hand, $A_{\left[  I\right]  }$ is an $A$-algebra, and $B\left[
Y\right]  $ is an $A_{\left[  I\right]  }$-algebra. Hence, Lemma \ref{lem.I'}
(applied to $A_{\left[  I\right]  }$, $B\left[  Y\right]  $, $\left(  J_{\tau
}\right)  _{\tau\in\mathbb{N}}$ and $y$ instead of $A^{\prime}$, $B^{\prime}$,
$\left(  I_{\rho}\right)  _{\rho\in\mathbb{N}}$ and $v$) yields that $y$ is
$n$-integral over $\left(  A_{\left[  I\right]  },\left(  J_{\tau}A_{\left[
I\right]  }\right)  _{\tau\in\mathbb{N}}\right)  $ (since $y$ is $n$-integral
over $\left(  A,\left(  J_{\tau}\right)  _{\tau\in\mathbb{N}}\right)  $).

Hence, Theorem~\ref{Theorem8c} (applied to $A_{\left[  I\right]  }$, $B\left[
Y\right]  $, $\left(  J_{\tau}A_{\left[  I\right]  }\right)  _{\tau
\in\mathbb{N}}$, $y$, $xY$, $n$ and $m$ instead of $A$, $B$, $\left(  I_{\rho
}\right)  _{\rho\in\mathbb{N}}$, $x$, $y$, $m$ and $n$, respectively) yields
that $y\cdot xY$ is $mn$-integral over $\left(  A_{\left[  I\right]  },\left(
J_{\tau}A_{\left[  I\right]  }\right)  _{\tau\in\mathbb{N}}\right)  $ (since
$xY$ is $m$-integral over $A_{\left[  I\right]  }$).

Since $y\cdot xY=xyY$ and $mn=nm$, this means that $xyY$ is $nm$-integral over
$\left(  A_{\left[  I\right]  },\left(  J_{\tau}A_{\left[  I\right]  }\right)
_{\tau\in\mathbb{N}}\right)  $. Hence, Theorem~\ref{Theorem11} \textbf{(b)}
(applied to $xy$ and $nm$ instead of $u$ and $n$) yields that $xy$ is
$nm$-integral over $\left(  A,\left(  I_{\rho}J_{\rho}\right)  _{\rho
\in\mathbb{N}}\right)  $. This proves Theorem~\ref{Theorem13}.
\end{proof}

\section{\label{sect.4}Accelerating ideal semifiltrations}

\subsection{Definition of $\lambda$-acceleration}

We start this section with an obvious observation:

\begin{theorem}
\label{Theorem14} Let $A$ be a ring. Let $\left(  I_{\rho}\right)  _{\rho
\in\mathbb{N}}$ be an ideal semifiltration of $A$. Let $\lambda\in\mathbb{N}$.
Then, $\left(  I_{\lambda\rho}\right)  _{\rho\in\mathbb{N}}$ is an ideal
semifiltration of $A$.
\end{theorem}

\begin{verlong}
\begin{proof}
[Proof of Theorem~\ref{Theorem14}.]Since $\left(  I_{\rho}\right)  _{\rho
\in\mathbb{N}}$ is an ideal semifiltration of $A$, it is a sequence of ideals
of $A$, and it satisfies the two conditions%
\begin{align}
I_{0}  &  =A;\nonumber\\
I_{a}I_{b}  &  \subseteq I_{a+b}\ \ \ \ \ \ \ \ \ \ \text{for every }%
a\in\mathbb{N}\text{ and }b\in\mathbb{N} \label{pf.Theorem14.2}%
\end{align}
(by Definition~\ref{Definition6}).

Now, $I_{\lambda\rho}$ is an ideal of $A$ for every $\rho\in\mathbb{N}$ (since
$\left(  I_{\rho}\right)  _{\rho\in\mathbb{N}}$ is a sequence of ideals of
$A$). Hence, $\left(  I_{\lambda\rho}\right)  _{\rho\in\mathbb{N}}$ is a
sequence of ideals of $A$. Thus, in order to prove that $\left(
I_{\lambda\rho}\right)  _{\rho\in\mathbb{N}}$ is an ideal semifiltration of
$A$, it is enough to verify that it satisfies the two conditions%
\begin{align*}
I_{\lambda\cdot0}  &  =A;\\
I_{\lambda a}I_{\lambda b}  &  \subseteq I_{\lambda\left(  a+b\right)
}\ \ \ \ \ \ \ \ \ \ \text{for every }a\in\mathbb{N}\text{ and }b\in
\mathbb{N}.
\end{align*}
But these two conditions are satisfied, since%
\begin{align*}
I_{\lambda\cdot0}  &  =I_{0}=A;\\
I_{\lambda a}I_{\lambda b}  &  \subseteq I_{\lambda a+\lambda b}%
\ \ \ \ \ \ \ \ \ \ \left(  \text{by (\ref{pf.Theorem14.2}), applied to
}\lambda a\text{ and }\lambda b\text{ instead of }a\text{ and }b\right) \\
&  =I_{\lambda\left(  a+b\right)  }\ \ \ \ \ \ \ \ \ \ \left(  \text{since
}\lambda a+\lambda b=\lambda\left(  a+b\right)  \right)
\ \ \ \ \ \ \ \ \ \ \text{for every }a\in\mathbb{N}\text{ and }b\in\mathbb{N}.
\end{align*}
Hence, $\left(  I_{\lambda\rho}\right)  _{\rho\in\mathbb{N}}$ is an ideal
semifiltration of $A$ (by Definition~\ref{Definition6}, applied to $\left(
I_{\lambda\rho}\right)  _{\rho\in\mathbb{N}}$ instead of $\left(  I_{\rho
}\right)  _{\rho\in\mathbb{N}}$). This proves Theorem~\ref{Theorem14}.
\end{proof}
\end{verlong}

I refer to the ideal semifiltration $\left(  I_{\lambda\rho}\right)  _{\rho
\in\mathbb{N}}$ in Theorem~\ref{Theorem14} as the $\lambda$%
\textit{-acceleration} of the ideal semifiltration $\left(  I_{\rho}\right)
_{\rho\in\mathbb{N}}$.

\subsection{Half-reduction and reduction}

Now, Theorem~\ref{Theorem11}, itself a generalization of
Theorem~\ref{Theorem7}, can be generalized once more:

\begin{theorem}
\label{Theorem15} Let $A$ be a ring. Let $B$ be an $A$-algebra. Let $\left(
I_{\rho}\right)  _{\rho\in\mathbb{N}}$ and $\left(  J_{\rho}\right)  _{\rho
\in\mathbb{N}}$ be two ideal semifiltrations of $A$. Let $n\in\mathbb{N}$. Let
$u\in B$. Let $\lambda\in\mathbb{N}$.

We know that $\left(  I_{\lambda\rho}\right)  _{\rho\in\mathbb{N}}$ is an
ideal semifiltration of $A$ (according to Theorem~\ref{Theorem14}).

Hence, $\left(  I_{\lambda\rho}J_{\rho}\right)  _{\rho\in\mathbb{N}}$ is an
ideal semifiltration of $A$ (according to Theorem~\ref{Theorem10}
\textbf{(b)}, applied to $\left(  I_{\lambda\rho}\right)  _{\rho\in\mathbb{N}%
}$ instead of $\left(  I_{\rho}\right)  _{\rho\in\mathbb{N}}$).

Consider the polynomial ring $A\left[  Y\right]  $ and its $A$-subalgebra
$A\left[  \left(  I_{\rho}\right)  _{\rho\in\mathbb{N}}\ast Y\right]  $.

We will abbreviate this $A$-subalgebra $A\left[  \left(  I_{\rho}\right)
_{\rho\in\mathbb{N}}\ast Y\right]  $ by $A_{\left[  I\right]  }$.

\textbf{(a)} The sequence $\left(  J_{\tau}A_{\left[  I\right]  }\right)
_{\tau\in\mathbb{N}}$ is an ideal semifiltration of $A_{\left[  I\right]  }$.

\textbf{(b)} The element $u$ of $B$ is $n$-integral over $\left(  A,\left(
I_{\lambda\rho}J_{\rho}\right)  _{\rho\in\mathbb{N}}\right)  $ if and only if
the element $uY^{\lambda}$ of the polynomial ring $B\left[  Y\right]  $ is
$n$-integral over $\left(  A_{\left[  I\right]  },\left(  J_{\tau}A_{\left[
I\right]  }\right)  _{\tau\in\mathbb{N}}\right)  $. (Here, we are using the
fact that $B\left[  Y\right]  $ is an $A_{\left[  I\right]  }$-algebra,
because $A_{\left[  I\right]  }=A\left[  \left(  I_{\rho}\right)  _{\rho
\in\mathbb{N}}\ast Y\right]  $ is a subring of $A\left[  Y\right]  $ and
because $B\left[  Y\right]  $ is an $A\left[  Y\right]  $-algebra as explained
in Definition~\ref{Definition7}.)
\end{theorem}

\begin{vershort}
\begin{proof}
[Proof of Theorem~\ref{Theorem15}.]\textbf{(a)} This is precisely
Theorem~\ref{Theorem11} \textbf{(a)}.

\textbf{(b)} The definition of $A_{\left[  I\right]  }$ yields%
\begin{align*}
A_{\left[  I\right]  }  &  =A\left[  \left(  I_{\rho}\right)  _{\rho
\in\mathbb{N}}\ast Y\right]  =\sum\limits_{i\in\mathbb{N}}I_{i}Y^{i}%
\ \ \ \ \ \ \ \ \ \ \left(  \text{by Definition~\ref{Definition8}}\right) \\
&  =\sum\limits_{\ell\in\mathbb{N}}I_{\ell}Y^{\ell}\ \ \ \ \ \ \ \ \ \ \left(
\text{here we renamed }i\text{ as }\ell\text{ in the sum}\right)  .
\end{align*}

In order to verify Theorem~\ref{Theorem15} \textbf{(b)}, we have to prove the
$\Longrightarrow$ and $\Longleftarrow$ statements.

$\Longrightarrow:$ Assume that $u$ is $n$-integral over $\left(  A,\left(
I_{\lambda\rho}J_{\rho}\right)  _{\rho\in\mathbb{N}}\right)  $. Thus, by
Definition~\ref{Definition9} (applied to $\left(  I_{\lambda\rho}J_{\rho
}\right)  _{\rho\in\mathbb{N}}$ instead of $\left(  I_{\rho}\right)  _{\rho
\in\mathbb{N}}$), there exists some $\left(  a_{0},a_{1},\ldots,a_{n}\right)
\in A^{n+1}$ such that%
\[
\sum\limits_{k=0}^{n}a_{k}u^{k}=0,\ \ \ \ \ \ \ \ \ \ a_{n}%
=1,\ \ \ \ \ \ \ \ \ \ \text{and}\ \ \ \ \ \ \ \ \ \ a_{i}\in I_{\lambda
\left(  n-i\right)  }J_{n-i}\text{ for every }i\in\left\{  0,1,\ldots
,n\right\}  .
\]
Consider this $\left(  a_{0},a_{1},\ldots,a_{n}\right)  $.

For each $k\in\left\{  0,1,\ldots,n\right\}  $, we have $a_{k}\in
I_{\lambda\left(  n-k\right)  }J_{n-k}\subseteq I_{\lambda\left(  n-k\right)
}$ (since $I_{\lambda\left(  n-k\right)  }$ is an ideal of $A$) and thus
$a_{k}Y^{\lambda\left(  n-k\right)  }\in I_{\lambda\left(  n-k\right)
}Y^{\lambda\left(  n-k\right)  }\subseteq\sum_{i\in\mathbb{N}}I_{i}%
Y^{i}=A_{\left[  I\right]  }$. Thus, we can find an $\left(  n+1\right)
$-tuple $\left(  b_{0},b_{1},\ldots,b_{n}\right)  \in\left(  A_{\left[
I\right]  }\right)  ^{n+1}$ satisfying%
\[
\sum\limits_{k=0}^{n}b_{k}\cdot\left(  uY^{\lambda}\right)  ^{k}%
=0,\ \ \ \ \ \ \ \ \ \ b_{n}=1,\ \ \ \ \ \ \ \ \ \ \text{and}%
\ \ \ \ \ \ \ \ \ \ b_{i}\in J_{n-i}A_{\left[  I\right]  }\text{ for every
}i\in\left\{  0,1,\ldots,n\right\}  .
\]
\footnote{Namely, the $\left(  n+1\right)  $-tuple $\left(  b_{0},b_{1}%
,\ldots,b_{n}\right)  \in\left(  A_{\left[  I\right]  }\right)  ^{n+1}$
defined by $\left(  b_{k}=a_{k}Y^{\lambda\left(  n-k\right)  }\text{ for every
}k\in\left\{  0,1,\ldots,n\right\}  \right)  $ satisfies this. The proof is
very easy (see \cite{verlong} for details).} Hence, by
Definition~\ref{Definition9} (applied to $A_{\left[  I\right]  }$, $B\left[
Y\right]  $, $\left(  J_{\tau}A_{\left[  I\right]  }\right)  _{\tau
\in\mathbb{N}}$, $uY^{\lambda}$ and $\left(  b_{0},b_{1},\ldots,b_{n}\right)
$ instead of $A$, $B$, $\left(  I_{\rho}\right)  _{\rho\in\mathbb{N}}$, $u$
and $\left(  a_{0},a_{1},\ldots,a_{n}\right)  $), the element $uY^{\lambda}$
is $n$-integral over $\left(  A_{\left[  I\right]  },\left(  J_{\tau
}A_{\left[  I\right]  }\right)  _{\tau\in\mathbb{N}}\right)  $. This proves
the $\Longrightarrow$ direction of Theorem~\ref{Theorem15} \textbf{(b)}.

$\Longleftarrow:$ Assume that $uY^{\lambda}$ is $n$-integral over $\left(
A_{\left[  I\right]  },\left(  J_{\tau}A_{\left[  I\right]  }\right)
_{\tau\in\mathbb{N}}\right)  $. Thus, by Definition~\ref{Definition9} (applied
to $A_{\left[  I\right]  }$, $B\left[  Y\right]  $, $\left(  J_{\tau
}A_{\left[  I\right]  }\right)  _{\tau\in\mathbb{N}}$, $uY^{\lambda}$ and
$\left(  p_{0},p_{1},\ldots,p_{n}\right)  $ instead of $A$, $B$, $\left(
I_{\rho}\right)  _{\rho\in\mathbb{N}}$, $u$ and $\left(  a_{0},a_{1}%
,\ldots,a_{n}\right)  $), there exists some $\left(  p_{0},p_{1},\ldots
,p_{n}\right)  \in\left(  A_{\left[  I\right]  }\right)  ^{n+1}$ such that%
\[
\sum\limits_{k=0}^{n}p_{k}\cdot\left(  uY^{\lambda}\right)  ^{k}%
=0,\ \ \ \ \ \ \ \ \ \ p_{n}=1,\ \ \ \ \ \ \ \ \ \ \text{and}%
\ \ \ \ \ \ \ \ \ \ p_{i}\in J_{n-i}A_{\left[  I\right]  }\text{ for every
}i\in\left\{  0,1,\ldots,n\right\}  .
\]
Consider this $\left(  p_{0},p_{1},\ldots,p_{n}\right)  $. For every
$k\in\left\{  0,1,\ldots,n\right\}  $, we have%
\begin{align*}
p_{k}  &  \in J_{n-k}A_{\left[  I\right]  }=J_{n-k}\sum\limits_{i\in
\mathbb{N}}I_{i}Y^{i}\ \ \ \ \ \ \ \ \ \ \left(  \text{since }A_{\left[
I\right]  }=\sum\limits_{i\in\mathbb{N}}I_{i}Y^{i}\right) \\
&  =\sum\limits_{i\in\mathbb{N}}J_{n-k}I_{i}Y^{i}=\sum\limits_{i\in\mathbb{N}%
}I_{i}J_{n-k}Y^{i},
\end{align*}
and thus there exists a sequence $\left(  p_{k,i}\right)  _{i\in\mathbb{N}}\in
A^{\mathbb{N}}$ such that $p_{k}=\sum\limits_{i\in\mathbb{N}}p_{k,i}Y^{i}$,
such that $\left(  p_{k,i}\in I_{i}J_{n-k}\text{ for every }i\in
\mathbb{N}\right)  $, and such that only finitely many $i\in\mathbb{N}$
satisfy $p_{k,i}\neq0$. Consider this sequence. Thus,%
\[
\sum\limits_{k=0}^{n}\underbrace{p_{k}}_{=\sum\limits_{i\in\mathbb{N}}%
p_{k,i}Y^{i}}\cdot\left(  uY^{\lambda}\right)  ^{k}=\sum\limits_{k=0}^{n}%
\sum\limits_{i\in\mathbb{N}}p_{k,i}\underbrace{Y^{i}\cdot\left(  uY^{\lambda
}\right)  ^{k}}_{=u^{k}Y^{i+\lambda k}}=\sum\limits_{k=0}^{n}\sum
\limits_{i\in\mathbb{N}}p_{k,i}u^{k}Y^{i+\lambda k}.
\]
Compared with $\sum\limits_{k=0}^{n}p_{k}\cdot\left(  uY^{\lambda}\right)
^{k}=0$, this yields $\sum\limits_{k=0}^{n}\sum\limits_{i\in\mathbb{N}}%
p_{k,i}u^{k}Y^{i+\lambda k}=0$. In other words, the polynomial $\sum
\limits_{k=0}^{n}\sum\limits_{i\in\mathbb{N}}\underbrace{p_{k,i}u^{k}}_{\in
B}Y^{i+\lambda k}\in B\left[  Y\right]  $ equals $0$. Hence, its coefficient
before $Y^{\lambda n}$ equals $0$ as well. But its coefficient before
$Y^{\lambda n}$ is $\sum\limits_{k=0}^{n}p_{k,\lambda\left(  n-k\right)
}u^{k}$ (since $i+\lambda k=\lambda n$ holds if and only if $i=\lambda\left(
n-k\right)  $). Hence, $\sum\limits_{k=0}^{n}p_{k,\lambda\left(  n-k\right)
}u^{k}=0$.

Recall that $\sum\limits_{i\in\mathbb{N}}p_{k,i}Y^{i}=p_{k}$ for every
$k\in\left\{  0,1,\ldots,n\right\}  $ (by the definition of the $p_{k,i}$).
Thus, $\sum\limits_{i\in\mathbb{N}}p_{n,i}Y^{i}=p_{n}=1$ in $A\left[
Y\right]  $, and thus $p_{n,0}=1$ (by comparing coefficients before $Y^{0}$).

Define an $\left(  n+1\right)  $-tuple $\left(  a_{0},a_{1},\ldots
,a_{n}\right)  \in A^{n+1}$ by $a_{k}=p_{k,\lambda\left(  n-k\right)  }$ for
every $k\in\left\{  0,1,\ldots,n\right\}  $. Then, $a_{n}=p_{n,0}=1$. Besides,%
\[
\sum\limits_{k=0}^{n}\underbrace{a_{k}}_{=p_{k,\lambda\left(  n-k\right)  }%
}u^{k}=\sum\limits_{k=0}^{n}p_{k,\lambda\left(  n-k\right)  }u^{k}=0.
\]
Finally, for every $k\in\left\{  0,1,\ldots,n\right\}  $, we have
$a_{k}=p_{k,\lambda\left(  n-k\right)  }\in I_{\lambda\left(  n-k\right)
}J_{n-k}$ (since $p_{k,i}\in I_{i}J_{n-k}$ for every $i\in\mathbb{N}$).
Renaming the variable $k$ as $i$ in this statement, we obtain the following:
For every $i\in\left\{  0,1,\ldots,n\right\}  $, we have $a_{i}\in
I_{\lambda\left(  n-i\right)  }J_{n-i}$.

Altogether, we now know that%
\[
\sum\limits_{k=0}^{n}a_{k}u^{k}=0,\ \ \ \ \ \ \ \ \ \ a_{n}%
=1,\ \ \ \ \ \ \ \ \ \ \text{and}\ \ \ \ \ \ \ \ \ \ a_{i}\in I_{\lambda
\left(  n-i\right)  }J_{n-i}\text{ for every }i\in\left\{  0,1,\ldots
,n\right\}  .
\]
Thus, by Definition~\ref{Definition9} (applied to $\left(  I_{\lambda\rho
}J_{\rho}\right)  _{\rho\in\mathbb{N}}$ instead of $\left(  I_{\rho}\right)
_{\rho\in\mathbb{N}}$), the element $u$ is $n$-integral over $\left(
A,\left(  I_{\lambda\rho}J_{\rho}\right)  _{\rho\in\mathbb{N}}\right)  $. This
proves the $\Longleftarrow$ direction of Theorem~\ref{Theorem15} \textbf{(b)},
and thus completes the proof.
\end{proof}
\end{vershort}

\begin{verlong}
\begin{proof}
[Proof of Theorem~\ref{Theorem15}.]\textbf{(a)} This is precisely the claim of
Theorem~\ref{Theorem11} \textbf{(a)}; thus, we don't need to prove it again.

\textbf{(b)} The definition of $A_{\left[  I\right]  }$ yields%
\begin{align*}
A_{\left[  I\right]  }  &  =A\left[  \left(  I_{\rho}\right)  _{\rho
\in\mathbb{N}}\ast Y\right]  =\sum\limits_{i\in\mathbb{N}}I_{i}Y^{i}%
\ \ \ \ \ \ \ \ \ \ \left(  \text{by Definition~\ref{Definition8}}\right) \\
&  =\sum\limits_{\ell\in\mathbb{N}}I_{\ell}Y^{\ell}\ \ \ \ \ \ \ \ \ \ \left(
\text{here we renamed }i\text{ as }\ell\text{ in the sum}\right)  .
\end{align*}
As a consequence of this chain of equalities, we have $\sum\limits_{i\in
\mathbb{N}}I_{i}Y^{i}=A_{\left[  I\right]  }$ and $\sum\limits_{\ell
\in\mathbb{N}}I_{\ell}Y^{\ell}=A_{\left[  I\right]  }$.

In order to verify Theorem~\ref{Theorem15} \textbf{(b)}, we have to prove the
following two lemmata:

\begin{statement}
\textit{Lemma }$\mathcal{E}^{\prime\prime}$\textit{:} If $u$ is $n$-integral
over $\left(  A,\left(  I_{\lambda\rho}J_{\rho}\right)  _{\rho\in\mathbb{N}%
}\right)  $, then $uY^{\lambda}$ is $n$-integral over $\left(  A_{\left[
I\right]  },\left(  J_{\tau}A_{\left[  I\right]  }\right)  _{\tau\in
\mathbb{N}}\right)  $.
\end{statement}

\begin{statement}
\textit{Lemma} $\mathcal{F}^{\prime\prime}$\textit{:} If $uY^{\lambda}$ is
$n$-integral over $\left(  A_{\left[  I\right]  },\left(  J_{\tau}A_{\left[
I\right]  }\right)  _{\tau\in\mathbb{N}}\right)  $, then $u$ is $n$-integral
over $\left(  A,\left(  I_{\lambda\rho}J_{\rho}\right)  _{\rho\in\mathbb{N}%
}\right)  $.
\end{statement}

[\textit{Proof of Lemma }$\mathcal{E}^{\prime\prime}$\textit{:} Assume that
$u$ is $n$-integral over $\left(  A,\left(  I_{\lambda\rho}J_{\rho}\right)
_{\rho\in\mathbb{N}}\right)  $. Thus, by Definition~\ref{Definition9} (applied
to $\left(  I_{\lambda\rho}J_{\rho}\right)  _{\rho\in\mathbb{N}}$ instead of
$\left(  I_{\rho}\right)  _{\rho\in\mathbb{N}}$), there exists some $\left(
a_{0},a_{1},\ldots,a_{n}\right)  \in A^{n+1}$ such that%
\[
\sum\limits_{k=0}^{n}a_{k}u^{k}=0,\ \ \ \ \ \ \ \ \ \ a_{n}%
=1,\ \ \ \ \ \ \ \ \ \ \text{and}\ \ \ \ \ \ \ \ \ \ a_{i}\in I_{\lambda
\left(  n-i\right)  }J_{n-i}\text{ for every }i\in\left\{  0,1,\ldots
,n\right\}  .
\]
Consider this $\left(  a_{0},a_{1},\ldots,a_{n}\right)  $.

For each $k\in\left\{  0,1,\ldots,n\right\}  $, we have%
\begin{align*}
a_{k}  &  \in I_{\lambda\left(  n-k\right)  }\underbrace{J_{n-k}}_{\subseteq
A}\ \ \ \ \ \ \ \ \ \ \left(  \text{since }a_{i}\in I_{\lambda\left(
n-i\right)  }J_{n-i}\text{ for every }i\in\left\{  0,1,\ldots,n\right\}
\right) \\
&  \subseteq I_{\lambda\left(  n-k\right)  }A\subseteq I_{\lambda\left(
n-k\right)  }\ \ \ \ \ \ \ \ \ \ \left(  \text{since }I_{\lambda\left(
n-k\right)  }\text{ is an ideal of }A\right)
\end{align*}
and thus%
\[
\underbrace{a_{k}}_{\in I_{\lambda\left(  n-k\right)  }}Y^{\lambda\left(
n-k\right)  }\in I_{\lambda\left(  n-k\right)  }Y^{\lambda\left(  n-k\right)
}\subseteq\sum\limits_{i\in\mathbb{N}}I_{i}Y^{i}=A_{\left[  I\right]  }.
\]
Thus, we can define an $\left(  n+1\right)  $-tuple $\left(  b_{0}%
,b_{1},\ldots,b_{n}\right)  \in\left(  A_{\left[  I\right]  }\right)  ^{n+1}$
by%
\[
\left(  b_{k}=a_{k}Y^{\lambda\left(  n-k\right)  }\text{ for every }%
k\in\left\{  0,1,\ldots,n\right\}  \right)  .
\]
Consider this $\left(  n+1\right)  $-tuple. Then,%
\begin{align*}
\sum\limits_{k=0}^{n}\underbrace{b_{k}}_{\substack{=a_{k}Y^{\lambda\left(
n-k\right)  }\\\text{(by the}\\\text{definition of }b_{k}\text{)}}%
}\cdot\underbrace{\left(  uY^{\lambda}\right)  ^{k}}_{\substack{=u^{k}\left(
Y^{\lambda}\right)  ^{k}\\=u^{k}Y^{\lambda k}}}  &  =\sum\limits_{k=0}%
^{n}a_{k}\underbrace{Y^{\lambda\left(  n-k\right)  }u^{k}}_{=u^{k}%
Y^{\lambda\left(  n-k\right)  }}Y^{\lambda k}=\sum\limits_{k=0}^{n}a_{k}%
u^{k}\underbrace{Y^{\lambda\left(  n-k\right)  }Y^{\lambda k}}%
_{\substack{=Y^{\lambda\left(  n-k\right)  +\lambda k}\\=Y^{\lambda n}}}\\
&  =\sum\limits_{k=0}^{n}a_{k}u^{k}Y^{\lambda n}=Y^{\lambda n}\cdot
\underbrace{\sum\limits_{k=0}^{n}a_{k}u^{k}}_{=0}=0,
\end{align*}
Furthermore, the definition of $b_{n}$ yields%
\[
b_{n}=\underbrace{a_{n}}_{=1}\underbrace{Y^{\lambda\left(  n-n\right)  }%
}_{=Y^{\lambda\cdot0}=Y^{0}=1}=1.
\]
Finally, the definition of $b_{i}$ yields%
\[
b_{i}=\underbrace{a_{i}}_{\substack{\in I_{\lambda\left(  n-i\right)  }%
J_{n-i}\\=J_{n-i}I_{\lambda\left(  n-i\right)  }}}Y^{\lambda\left(
n-i\right)  }\in J_{n-i}\underbrace{I_{\lambda\left(  n-i\right)  }%
Y^{\lambda\left(  n-i\right)  }}_{\substack{\subseteq\sum\limits_{\ell
\in\mathbb{N}}I_{\ell}Y^{\ell}\\=A_{\left[  I\right]  }}}\subseteq
J_{n-i}A_{\left[  I\right]  }\ \ \ \ \ \ \ \ \ \ \text{for every }i\in\left\{
0,1,\ldots,n\right\}  .
\]

Altogether, we now know that $\left(  b_{0},b_{1},\ldots,b_{n}\right)
\in\left(  A_{\left[  I\right]  }\right)  ^{n+1}$ and%
\[
\sum\limits_{k=0}^{n}b_{k}\cdot\left(  uY^{\lambda}\right)  ^{k}%
=0,\ \ \ \ \ \ \ \ \ \ b_{n}=1,\ \ \ \ \ \ \ \ \ \ \text{and}%
\ \ \ \ \ \ \ \ \ \ b_{i}\in J_{n-i}A_{\left[  I\right]  }\text{ for every
}i\in\left\{  0,1,\ldots,n\right\}  .
\]
Hence, by Definition~\ref{Definition9} (applied to $A_{\left[  I\right]  }$,
$B\left[  Y\right]  $, $\left(  J_{\tau}A_{\left[  I\right]  }\right)
_{\tau\in\mathbb{N}}$, $uY^{\lambda}$ and $\left(  b_{0},b_{1},\ldots
,b_{n}\right)  $ instead of $A$, $B$, $\left(  I_{\rho}\right)  _{\rho
\in\mathbb{N}}$, $u$ and $\left(  a_{0},a_{1},\ldots,a_{n}\right)  $), the
element $uY^{\lambda}$ is $n$-integral over $\left(  A_{\left[  I\right]
},\left(  J_{\tau}A_{\left[  I\right]  }\right)  _{\tau\in\mathbb{N}}\right)
$. This proves Lemma $\mathcal{E}^{\prime\prime}$.]

[\textit{Proof of Lemma }$\mathcal{F}^{\prime\prime}$\textit{:} Assume that
$uY^{\lambda}$ is $n$-integral over $\left(  A_{\left[  I\right]  },\left(
J_{\tau}A_{\left[  I\right]  }\right)  _{\tau\in\mathbb{N}}\right)  $. Thus,
by Definition~\ref{Definition9} (applied to $A_{\left[  I\right]  }$,
$B\left[  Y\right]  $, $\left(  J_{\tau}A_{\left[  I\right]  }\right)
_{\tau\in\mathbb{N}}$, $uY^{\lambda}$ and $\left(  p_{0},p_{1},\ldots
,p_{n}\right)  $ instead of $A$, $B$, $\left(  I_{\rho}\right)  _{\rho
\in\mathbb{N}}$, $u$ and $\left(  a_{0},a_{1},\ldots,a_{n}\right)  $), there
exists some $\left(  p_{0},p_{1},\ldots,p_{n}\right)  \in\left(  A_{\left[
I\right]  }\right)  ^{n+1}$ such that%
\[
\sum\limits_{k=0}^{n}p_{k}\cdot\left(  uY^{\lambda}\right)  ^{k}%
=0,\ \ \ \ \ \ \ \ \ \ p_{n}=1,\ \ \ \ \ \ \ \ \ \ \text{and}%
\ \ \ \ \ \ \ \ \ \ p_{i}\in J_{n-i}A_{\left[  I\right]  }\text{ for every
}i\in\left\{  0,1,\ldots,n\right\}  .
\]
Consider this $\left(  p_{0},p_{1},\ldots,p_{n}\right)  $. For every
$k\in\left\{  0,1,\ldots,n\right\}  $, we have%
\begin{align*}
p_{k}  &  \in J_{n-k}A_{\left[  I\right]  }\ \ \ \ \ \ \ \ \ \ \left(
\text{since }p_{i}\in J_{n-i}A_{\left[  I\right]  }\text{ for every }%
i\in\left\{  0,1,\ldots,n\right\}  \right) \\
&  =J_{n-k}\sum\limits_{i\in\mathbb{N}}I_{i}Y^{i}\ \ \ \ \ \ \ \ \ \ \left(
\text{since }A_{\left[  I\right]  }=\sum\limits_{i\in\mathbb{N}}I_{i}%
Y^{i}\right) \\
&  =\sum\limits_{i\in\mathbb{N}}\underbrace{J_{n-k}I_{i}}_{=I_{i}J_{n-k}}%
Y^{i}=\sum\limits_{i\in\mathbb{N}}I_{i}J_{n-k}Y^{i},
\end{align*}
and thus there exists a sequence $\left(  p_{k,i}\right)  _{i\in\mathbb{N}}\in
A^{\mathbb{N}}$ such that $p_{k}=\sum\limits_{i\in\mathbb{N}}p_{k,i}Y^{i}$,
such that $\left(  p_{k,i}\in I_{i}J_{n-k}\text{ for every }i\in
\mathbb{N}\right)  $, and such that only finitely many $i\in\mathbb{N}$
satisfy $p_{k,i}\neq0$. Consider this sequence. Thus,%
\begin{align*}
\sum\limits_{k=0}^{n}\underbrace{p_{k}}_{=\sum\limits_{i\in\mathbb{N}}%
p_{k,i}Y^{i}}\cdot\underbrace{\left(  uY^{\lambda}\right)  ^{k}}%
_{\substack{=u^{k}\left(  Y^{\lambda}\right)  ^{k}\\=u^{k}Y^{\lambda
k}\\=Y^{\lambda k}u^{k}}}  &  =\sum\limits_{k=0}^{n}\left(  \sum
\limits_{i\in\mathbb{N}}p_{k,i}Y^{i}\right)  \cdot Y^{\lambda k}%
u^{k}\ \ \ \ \ \ \ \ \ \ \left(  \text{since }p_{k}=\sum\limits_{i\in
\mathbb{N}}p_{k,i}Y^{i}\right) \\
&  =\sum\limits_{k=0}^{n}\sum\limits_{i\in\mathbb{N}}p_{k,i}\underbrace{Y^{i}%
\cdot Y^{\lambda k}}_{=Y^{i+\lambda k}}u^{k}=\underbrace{\sum\limits_{k=0}%
^{n}}_{=\sum\limits_{k\in\left\{  0,1,\ldots,n\right\}  }}\sum\limits_{i\in
\mathbb{N}}p_{k,i}Y^{i+\lambda k}u^{k}\\
&  =\sum\limits_{k\in\left\{  0,1,\ldots,n\right\}  }\sum\limits_{i\in
\mathbb{N}}p_{k,i}Y^{i+\lambda k}u^{k}=\sum\limits_{\left(  k,i\right)
\in\left\{  0,1,\ldots,n\right\}  \times\mathbb{N}}p_{k,i}Y^{i+\lambda k}%
u^{k}\\
&  =\sum_{\ell\in\mathbb{N}}\sum\limits_{\substack{\left(  k,i\right)
\in\left\{  0,1,\ldots,n\right\}  \times\mathbb{N};\\i+\lambda k=\ell}%
}p_{k,i}\underbrace{Y^{i+\lambda k}}_{\substack{=Y^{\ell}\\\text{(since
}i+\lambda k=\ell\text{)}}}u^{k}\\
&  =\sum_{\ell\in\mathbb{N}}\sum\limits_{\substack{\left(  k,i\right)
\in\left\{  0,1,\ldots,n\right\}  \times\mathbb{N};\\i+\lambda k=\ell}%
}p_{k,i}\underbrace{Y^{\ell}u^{k}}_{=u^{k}Y^{\ell}}=\sum_{\ell\in\mathbb{N}%
}\sum\limits_{\substack{\left(  k,i\right)  \in\left\{  0,1,\ldots,n\right\}
\times\mathbb{N};\\i+\lambda k=\ell}}p_{k,i}u^{k}Y^{\ell}.
\end{align*}
Comparing this with $\sum\limits_{k=0}^{n}p_{k}\cdot\left(  uY^{\lambda
}\right)  ^{k}=0$, we obtain $\sum_{\ell\in\mathbb{N}}\sum
\limits_{\substack{\left(  k,i\right)  \in\left\{  0,1,\ldots,n\right\}
\times\mathbb{N};\\i+\lambda k=\ell}}p_{k,i}u^{k}Y^{\ell}=0$. In other words,
the polynomial $\sum\limits_{\ell\in\mathbb{N}}\underbrace{\sum
\limits_{\substack{\left(  k,i\right)  \in\left\{  0,1,\ldots,n\right\}
\times\mathbb{N};\\i+\lambda k=\ell}}p_{k,i}u^{k}}_{\in B}Y^{\ell}\in B\left[
Y\right]  $ equals $0$. Hence, its coefficient before $Y^{\lambda n}$ equals
$0$ as well. But its coefficient before $Y^{\lambda n}$ is $\sum
\limits_{\substack{\left(  k,i\right)  \in\left\{  0,1,\ldots,n\right\}
\times\mathbb{N};\\i+\lambda k=\lambda n}}p_{k,i}u^{k}$. Comparing the
preceding two sentences, we see that $\sum\limits_{\substack{\left(
k,i\right)  \in\left\{  0,1,\ldots,n\right\}  \times\mathbb{N};\\i+\lambda
k=\lambda n}}p_{k,i}u^{k}$ equals $0$. Thus,%
\begin{equation}
0=\sum\limits_{\substack{\left(  k,i\right)  \in\left\{  0,1,\ldots,n\right\}
\times\mathbb{N};\\i+\lambda k=\lambda n}}p_{k,i}u^{k}=\sum\limits_{k\in
\left\{  0,1,\ldots,n\right\}  }\sum_{\substack{i\in\mathbb{N};\\i+\lambda
k=\lambda n}}p_{k,i}u^{k}. \label{T15.pf.6}%
\end{equation}

But for each given $k\in\left\{  0,1,\ldots,n\right\}  $, we have
$n-k\in\mathbb{N}$ and thus $\lambda\left(  n-k\right)  \in\mathbb{N}$ (since
$\lambda\in\mathbb{N}$) and thus
\begin{align*}
\left\{  i\in\mathbb{N}\text{\ }\mid\ \underbrace{i+\lambda k=\lambda
n}_{\Longleftrightarrow\ \left(  i=\lambda n-\lambda k\right)  }\right\}   &
=\left\{  i\in\mathbb{N}\ \mid\ i=\underbrace{\lambda n-\lambda k}%
_{=\lambda\left(  n-k\right)  }\right\} \\
&  =\left\{  i\in\mathbb{N}\ \mid\ i=\lambda\left(  n-k\right)  \right\}
=\left\{  \lambda\left(  n-k\right)  \right\}
\end{align*}
(since $\lambda\left(  n-k\right)  \in\mathbb{N}$) and therefore%
\[
\sum_{\substack{i\in\mathbb{N};\\i+\lambda k=\lambda n}}p_{k,i}u^{k}%
=\sum_{i\in\left\{  \lambda\left(  n-k\right)  \right\}  }p_{k,i}%
u^{k}=p_{k,\lambda\left(  n-k\right)  }u^{k}.
\]
Hence, (\ref{T15.pf.6}) becomes%
\begin{equation}
0=\sum\limits_{k\in\left\{  0,1,\ldots,n\right\}  }\underbrace{\sum
_{\substack{i\in\mathbb{N};\\i+\lambda k=\lambda n}}p_{k,i}u^{k}%
}_{=p_{k,\lambda\left(  n-k\right)  }u^{k}}=\sum\limits_{k\in\left\{
0,1,\ldots,n\right\}  }p_{k,\lambda\left(  n-k\right)  }u^{k}.
\label{pf.Theorem15.7}%
\end{equation}

Recall that $p_{k}=\sum\limits_{i\in\mathbb{N}}p_{k,i}Y^{i}$ for every
$k\in\left\{  0,1,\ldots,n\right\}  $. Applying this to $k=n$, we find
$p_{n}=\sum\limits_{i\in\mathbb{N}}p_{n,i}Y^{i}$. Comparing this with
$p_{n}=1=1\cdot Y^{0}$, we find
\[
\sum\limits_{i\in\mathbb{N}}p_{n,i}Y^{i}=1\cdot Y^{0}%
\ \ \ \ \ \ \ \ \ \ \text{in }A\left[  Y\right]  .
\]
Hence, the coefficient of the polynomial $\sum\limits_{i\in\mathbb{N}}%
p_{n,i}Y^{i}\in A\left[  Y\right]  $ before $Y^{0}$ is $1$. But the
coefficient of the polynomial $\sum\limits_{i\in\mathbb{N}}p_{n,i}Y^{i}\in
A\left[  Y\right]  $ before $Y^{0}$ is $p_{n,0}$ (since $p_{n,i}\in A$ for all
$i\in\mathbb{N}$). Comparing the preceding two sentences, we see that
$p_{n,0}=1$.

Define an $\left(  n+1\right)  $-tuple $\left(  a_{0},a_{1},\ldots
,a_{n}\right)  \in A^{n+1}$ by setting
\[
\left(  a_{k}=p_{k,\lambda\left(  n-k\right)  }\text{ for every }k\in\left\{
0,1,\ldots,n\right\}  \right)  .
\]
Then, $a_{n}=p_{n,\lambda\left(  n-n\right)  }=p_{n,\lambda\cdot0}=p_{n,0}=1$.
Besides,%
\[
\sum\limits_{k=0}^{n}\underbrace{a_{k}}_{\substack{=p_{k,\lambda\left(
n-k\right)  }\\\text{(by the definition}\\\text{of }a_{k}\text{)}}}u^{k}%
=\sum\limits_{k=0}^{n}p_{k,\lambda\left(  n-k\right)  }u^{k}=\sum
\limits_{k\in\left\{  0,1,\ldots,n\right\}  }p_{k,\lambda\left(  n-k\right)
}u^{k}=0\ \ \ \ \ \ \ \ \ \ \left(  \text{by (\ref{pf.Theorem15.7})}\right)
.
\]
Finally, for every $k\in\left\{  0,1,\ldots,n\right\}  $, we have
$n-k\in\mathbb{N}$ and therefore $\lambda\left(  n-k\right)  \in\mathbb{N}$
(since $\lambda\in\mathbb{N}$) and thus $a_{k}=p_{k,\lambda\left(  n-k\right)
}\in I_{\lambda\left(  n-k\right)  }J_{n-k}$ (since $p_{k,i}\in I_{i}J_{n-k}$
for every $i\in\mathbb{N}$). Renaming the variable $k$ as $i$ in this
statement, we obtain the following: For every $i\in\left\{  0,1,\ldots
,n\right\}  $, we have $a_{i}\in I_{\lambda\left(  n-i\right)  }J_{n-i}$.

Altogether, we now know that the $\left(  n+1\right)  $-tuple $\left(
a_{0},a_{1},\ldots,a_{n}\right)  \in A^{n+1}$ satisfies%
\[
\sum\limits_{k=0}^{n}a_{k}u^{k}=0,\ \ \ \ \ \ \ \ \ \ a_{n}%
=1,\ \ \ \ \ \ \ \ \ \ \text{and}\ \ \ \ \ \ \ \ \ \ a_{i}\in I_{\lambda
\left(  n-i\right)  }J_{n-i}\text{ for every }i\in\left\{  0,1,\ldots
,n\right\}  .
\]
Thus, by Definition~\ref{Definition9} (applied to $\left(  I_{\lambda\rho
}J_{\rho}\right)  _{\rho\in\mathbb{N}}$ instead of $\left(  I_{\rho}\right)
_{\rho\in\mathbb{N}}$), the element $u$ is $n$-integral over $\left(
A,\left(  I_{\lambda\rho}J_{\rho}\right)  _{\rho\in\mathbb{N}}\right)  $. This
proves Lemma $\mathcal{F}^{\prime\prime}$.]

Combining Lemma $\mathcal{E}^{\prime\prime}$ and Lemma $\mathcal{F}%
^{\prime\prime}$, we obtain that $u$ is $n$-integral over $\left(  A,\left(
I_{\lambda\rho}J_{\rho}\right)  _{\rho\in\mathbb{N}}\right)  $ if and only if
$uY^{\lambda}$ is $n$-integral over $\left(  A_{\left[  I\right]  },\left(
J_{\tau}A_{\left[  I\right]  }\right)  _{\tau\in\mathbb{N}}\right)  $. This
proves Theorem~\ref{Theorem15} \textbf{(b)}.
\end{proof}
\end{verlong}

A particular case of Theorem~\ref{Theorem15} \textbf{(b)} is the following fact:

\begin{theorem}
\label{Theorem16} Let $A$ be a ring. Let $B$ be an $A$-algebra. Let $\left(
I_{\rho}\right)  _{\rho\in\mathbb{N}}$ be an ideal semifiltration of $A$. Let
$n\in\mathbb{N}$. Let $u\in B$. Let $\lambda\in\mathbb{N}$.

We know that $\left(  I_{\lambda\rho}\right)  _{\rho\in\mathbb{N}}$ is an
ideal semifiltration of $A$ (according to Theorem~\ref{Theorem14}).

Consider the polynomial ring $A\left[  Y\right]  $ and its $A$-subalgebra
$A\left[  \left(  I_{\rho}\right)  _{\rho\in\mathbb{N}}\ast Y\right]  $
defined in Definition~\ref{Definition8}.

Then, the element $u$ of $B$ is $n$-integral over $\left(  A,\left(
I_{\lambda\rho}\right)  _{\rho\in\mathbb{N}}\right)  $ if and only if the
element $uY^{\lambda}$ of the polynomial ring $B\left[  Y\right]  $ is
$n$-integral over the ring $A\left[  \left(  I_{\rho}\right)  _{\rho
\in\mathbb{N}}\ast Y\right]  $. (Here, we are using the fact that $B\left[
Y\right]  $ is an $A\left[  \left(  I_{\rho}\right)  _{\rho\in\mathbb{N}}\ast
Y\right]  $-algebra, because $A\left[  \left(  I_{\rho}\right)  _{\rho
\in\mathbb{N}}\ast Y\right]  $ is a subring of $A\left[  Y\right]  $ and
because $B\left[  Y\right]  $ is an $A\left[  Y\right]  $-algebra as explained
in Definition~\ref{Definition7}.)
\end{theorem}

\begin{vershort}
\begin{proof}
[Proof of Theorem~\ref{Theorem16}.]Theorem~\ref{Theorem10} \textbf{(a)} states
that $\left(  A\right)  _{\rho\in\mathbb{N}}$ is an ideal semifiltration of
$A$.

Every $\rho\in\mathbb{N}$ satisfies $I_{\lambda\rho}=I_{\lambda\rho}A$ (since
$I_{\lambda\rho}$ is an ideal of $A$). Thus, $\left(  I_{\lambda\rho}\right)
_{\rho\in\mathbb{N}}=\left(  I_{\lambda\rho}A\right)  _{\rho\in\mathbb{N}}$.

We will abbreviate the $A$-subalgebra $A\left[  \left(  I_{\rho}\right)
_{\rho\in\mathbb{N}}\ast Y\right]  $ of $A\left[  Y\right]  $ by $A_{\left[
I\right]  }$. Thus, $B\left[  Y\right]  $ is an $A_{\left[  I\right]  }%
$-algebra (since $B\left[  Y\right]  $ is an $A\left[  \left(  I_{\rho
}\right)  _{\rho\in\mathbb{N}}\ast Y\right]  $-algebra).

It is easy to see that $AA_{\left[  I\right]  }=A_{\left[  I\right]  }$ (since
$A_{\left[  I\right]  }$ is an $A$-algebra). Hence, $\left(  AA_{\left[
I\right]  }\right)  _{\tau\in\mathbb{N}}=\left(  A_{\left[  I\right]
}\right)  _{\tau\in\mathbb{N}}=\left(  A_{\left[  I\right]  }\right)
_{\rho\in\mathbb{N}}$.

Now, we have the following chain of equivalences:%
\begin{align*}
&  \ \left(  u\text{ is }n\text{-integral over }\left(  A,\left(
I_{\lambda\rho}\right)  _{\rho\in\mathbb{N}}\right)  \right) \\
&  \Longleftrightarrow\ \left(  u\text{ is }n\text{-integral over }\left(
A,\left(  I_{\lambda\rho}A\right)  _{\rho\in\mathbb{N}}\right)  \right) \\
&  \ \ \ \ \ \ \ \ \ \ \left(  \text{since }\left(  I_{\lambda\rho}\right)
_{\rho\in\mathbb{N}}=\left(  I_{\lambda\rho}A\right)  _{\rho\in\mathbb{N}%
}\right) \\
&  \Longleftrightarrow\ \left(  uY^{\lambda}\text{ is }n\text{-integral over
}\left(  A_{\left[  I\right]  },\left(  AA_{\left[  I\right]  }\right)
_{\tau\in\mathbb{N}}\right)  \right) \\
&  \ \ \ \ \ \ \ \ \ \ \left(  \text{by Theorem~\ref{Theorem15} \textbf{(b)},
applied to }\left(  J_{\rho}\right)  _{\rho\in\mathbb{N}}=\left(  A\right)
_{\rho\in\mathbb{N}}\right) \\
&  \Longleftrightarrow\ \left(  uY^{\lambda}\text{ is }n\text{-integral over
}\left(  A_{\left[  I\right]  },\left(  A_{\left[  I\right]  }\right)
_{\rho\in\mathbb{N}}\right)  \right) \\
&  \ \ \ \ \ \ \ \ \ \ \left(  \text{since }\left(  AA_{\left[  I\right]
}\right)  _{\tau\in\mathbb{N}}=\left(  A_{\left[  I\right]  }\right)
_{\rho\in\mathbb{N}}\right) \\
&  \Longleftrightarrow\ \left(  uY^{\lambda}\text{ is }n\text{-integral over
}A_{\left[  I\right]  }\right) \\
&  \ \ \ \ \ \ \ \ \ \ \left(  \text{by Theorem~\ref{Theorem12}, applied to
}A_{\left[  I\right]  }\text{, }B\left[  Y\right]  \text{ and }uY^{\lambda
}\text{ instead of }A\text{, }B\text{ and }u\right) \\
&  \Longleftrightarrow\ \left(  uY^{\lambda}\text{ is }n\text{-integral over
}A\left[  \left(  I_{\rho}\right)  _{\rho\in\mathbb{N}}\ast Y\right]  \right)
\\
&  \ \ \ \ \ \ \ \ \ \ \left(  \text{since }A_{\left[  I\right]  }=A\left[
\left(  I_{\rho}\right)  _{\rho\in\mathbb{N}}\ast Y\right]  \right)  .
\end{align*}
This proves Theorem~\ref{Theorem16}.
\end{proof}
\end{vershort}

\begin{verlong}
\begin{proof}
[Proof of Theorem~\ref{Theorem16}.]Theorem~\ref{Theorem10} \textbf{(a)} states
that $\left(  A\right)  _{\rho\in\mathbb{N}}$ is an ideal semifiltration of
$A$.

We have $\left(  I_{\lambda\rho}\right)  _{\rho\in\mathbb{N}}=\left(
I_{\lambda\rho}A\right)  _{\rho\in\mathbb{N}}$%
\ \ \ \ \footnote{\textit{Proof.} We know that $\left(  I_{\lambda\rho
}\right)  _{\rho\in\mathbb{N}}$ is an ideal semifiltration of $A$, thus a
sequence of ideals of $A$. In other words, for each $\rho\in\mathbb{N}$, the
set $I_{\lambda\rho}$ is an ideal of $A$.
\par
Now, let $\rho\in\mathbb{N}$. Then, the set $I_{\lambda\rho}$ is an ideal of
$A$ (as we have just seen). Hence, $I_{\lambda\rho}A\subseteq I_{\lambda\rho}%
$. Combining this with $I_{\lambda\rho}=I_{\lambda\rho}\underbrace{1_{A}}_{\in
A}\subseteq I_{\lambda\rho}A$, we obtain $I_{\lambda\rho}=I_{\lambda\rho}A$.
\par
Forget that we fixed $\rho$. We thus have shown that $I_{\lambda\rho
}=I_{\lambda\rho}A$ for each $\rho\in\mathbb{N}$. In other words, $\left(
I_{\lambda\rho}\right)  _{\rho\in\mathbb{N}}=\left(  I_{\lambda\rho}A\right)
_{\rho\in\mathbb{N}}$.}.

We will abbreviate the $A$-subalgebra $A\left[  \left(  I_{\rho}\right)
_{\rho\in\mathbb{N}}\ast Y\right]  $ of $A\left[  Y\right]  $ by $A_{\left[
I\right]  }$. Thus, $B\left[  Y\right]  $ is an $A_{\left[  I\right]  }%
$-algebra (since $B\left[  Y\right]  $ is an $A\left[  \left(  I_{\rho
}\right)  _{\rho\in\mathbb{N}}\ast Y\right]  $-algebra).

It is easy to see that $AA_{\left[  I\right]  }=A_{\left[  I\right]  }%
$\ \ \ \ \footnote{\textit{Proof.} We have $AA_{\left[  I\right]  }\subseteq
A_{\left[  I\right]  }$ (since $A_{\left[  I\right]  }$ is an $A$-algebra).
Combining this with $A_{\left[  I\right]  }=\underbrace{1_{A}}_{\in A}\cdot
A_{\left[  I\right]  }\subseteq AA_{\left[  I\right]  }$, we obtain
$AA_{\left[  I\right]  }=A_{\left[  I\right]  }$, qed.}. Hence, $\left(
\underbrace{AA_{\left[  I\right]  }}_{=A_{\left[  I\right]  }}\right)
_{\tau\in\mathbb{N}}=\left(  A_{\left[  I\right]  }\right)  _{\tau
\in\mathbb{N}}=\left(  A_{\left[  I\right]  }\right)  _{\rho\in\mathbb{N}}$.

We have the following five equivalences:

\begin{itemize}
\item The element $u$ of $B$ is $n$-integral over $\left(  A,\left(
I_{\lambda\rho}\right)  _{\rho\in\mathbb{N}}\right)  $ if and only if the
element $u$ of $B$ is $n$-integral over $\left(  A,\left(  I_{\lambda\rho
}A\right)  _{\rho\in\mathbb{N}}\right)  $ (since $\left(  I_{\lambda\rho
}\right)  _{\rho\in\mathbb{N}}=\left(  I_{\lambda\rho}A\right)  _{\rho
\in\mathbb{N}}$).

\item The element $u$ of $B$ is $n$-integral over $\left(  A,\left(
I_{\lambda\rho}A\right)  _{\rho\in\mathbb{N}}\right)  $ if and only if the
element $uY^{\lambda}$ of the polynomial ring $B\left[  Y\right]  $ is
$n$-integral over $\left(  A_{\left[  I\right]  },\left(  AA_{\left[
I\right]  }\right)  _{\tau\in\mathbb{N}}\right)  $ (according to
Theorem~\ref{Theorem15} \textbf{(b)}, applied to $\left(  A\right)  _{\rho
\in\mathbb{N}}$ instead of $\left(  J_{\rho}\right)  _{\rho\in\mathbb{N}}$).

\item The element $uY^{\lambda}$ of the polynomial ring $B\left[  Y\right]  $
is $n$-integral over $\left(  A_{\left[  I\right]  },\left(  AA_{\left[
I\right]  }\right)  _{\tau\in\mathbb{N}}\right)  $ if and only if the element
$uY^{\lambda}$ of the polynomial ring $B\left[  Y\right]  $ is $n$-integral
over $\left(  A_{\left[  I\right]  },\left(  A_{\left[  I\right]  }\right)
_{\rho\in\mathbb{N}}\right)  $ (since $\left(  AA_{\left[  I\right]  }\right)
_{\tau\in\mathbb{N}}=\left(  A_{\left[  I\right]  }\right)  _{\rho
\in\mathbb{N}}$).

\item The element $uY^{\lambda}$ of the polynomial ring $B\left[  Y\right]  $
is $n$-integral over $\left(  A_{\left[  I\right]  },\left(  A_{\left[
I\right]  }\right)  _{\rho\in\mathbb{N}}\right)  $ if and only if the element
$uY^{\lambda}$ of the polynomial ring $B\left[  Y\right]  $ is $n$-integral
over $A_{\left[  I\right]  }$ (by Theorem~\ref{Theorem12}, applied to
$A_{\left[  I\right]  }$, $B\left[  Y\right]  $ and $uY^{\lambda}$ instead of
$A$, $B$ and $u$).

\item The element $uY^{\lambda}$ of the polynomial ring $B\left[  Y\right]  $
is $n$-integral over $A_{\left[  I\right]  }$ if and only if the element
$uY^{\lambda}$ of the polynomial ring $B\left[  Y\right]  $ is $n$-integral
over $A\left[  \left(  I_{\rho}\right)  _{\rho\in\mathbb{N}}\ast Y\right]  $
(since $A_{\left[  I\right]  }=A\left[  \left(  I_{\rho}\right)  _{\rho
\in\mathbb{N}}\ast Y\right]  $).
\end{itemize}

Combining these five equivalences, we obtain that the element $u$ of $B$ is
$n$-integral over $\left(  A,\left(  I_{\lambda\rho}\right)  _{\rho
\in\mathbb{N}}\right)  $ if and only if the element $uY^{\lambda}$ of the
polynomial ring $B\left[  Y\right]  $ is $n$-integral over $A\left[  \left(
I_{\rho}\right)  _{\rho\in\mathbb{N}}\ast Y\right]  $. This proves
Theorem~\ref{Theorem16}.
\end{proof}
\end{verlong}

Note that Theorem~\ref{Theorem7} is the particular case of
Theorem~\ref{Theorem16} for $\lambda=1$.

Finally we can generalize even Theorem~\ref{Theorem2}:

\begin{theorem}
\label{Theorem17} Let $A$ be a ring. Let $B$ be an $A$-algebra. Let $\left(
I_{\rho}\right)  _{\rho\in\mathbb{N}}$ be an ideal semifiltration of $A$. Let
$n\in\mathbb{N}^{+}$. Let $v\in B$. Let $a_{0},a_{1},\ldots,a_{n}$ be $n+1$
elements of $A$ such that $\sum\limits_{i=0}^{n}a_{i}v^{i}=0$. Assume further
that $a_{i}\in I_{n-i}$ for every $i\in\left\{  0,1,\ldots,n\right\}  $.

Let $k\in\left\{  0,1,\ldots,n\right\}  $. We know that $\left(  I_{\left(
n-k\right)  \rho}\right)  _{\rho\in\mathbb{N}}$ is an ideal semifiltration of
$A$ (according to Theorem~\ref{Theorem14}, applied to $\lambda=n-k$).

Then, $\sum\limits_{i=0}^{n-k}a_{i+k}v^{i}$ is $n$-integral over $\left(
A,\left(  I_{\left(  n-k\right)  \rho}\right)  _{\rho\in\mathbb{N}}\right)  $.
\end{theorem}

\begin{proof}
[Proof of Theorem~\ref{Theorem17}.]Consider the polynomial ring $A\left[
Y\right]  $ and its $A$-subalgebra $A\left[  \left(  I_{\rho}\right)
_{\rho\in\mathbb{N}}\ast Y\right]  $ defined in Definition~\ref{Definition8}.
Note that $A\left[  \left(  I_{\rho}\right)  _{\rho\in\mathbb{N}}\ast
Y\right]  $ is a subring of $A\left[  Y\right]  $; hence, $B\left[  Y\right]
$ is an $A\left[  \left(  I_{\rho}\right)  _{\rho\in\mathbb{N}}\ast Y\right]
$-algebra (because $B\left[  Y\right]  $ is an $A\left[  Y\right]  $-algebra
as explained in Definition~\ref{Definition7}).

Definition~\ref{Definition8} yields%
\[
A\left[  \left(  I_{\rho}\right)  _{\rho\in\mathbb{N}}\ast Y\right]
=\sum\limits_{i\in\mathbb{N}}I_{i}Y^{i}=\sum\limits_{\ell\in\mathbb{N}}%
I_{\ell}Y^{\ell}\ \ \ \ \ \ \ \ \ \ \left(  \text{here we renamed }i\text{ as
}\ell\text{ in the sum}\right)  .
\]
Hence, $\sum\limits_{\ell\in\mathbb{N}}I_{\ell}Y^{\ell}=A\left[  \left(
I_{\rho}\right)  _{\rho\in\mathbb{N}}\ast Y\right]  $.

Define $u\in B$ by%
\begin{equation}
u=\sum\limits_{i=0}^{n-k}a_{i+k}v^{i}. \label{pf.Theorem17.u=}%
\end{equation}

In the ring $B\left[  Y\right]  $, we have%
\[
\sum_{i=0}^{n}a_{i}Y^{n-i}\underbrace{\left(  vY\right)  ^{i}}_{=v^{i}%
Y^{i}=Y^{i}v^{i}}=\sum_{i=0}^{n}a_{i}\underbrace{Y^{n-i}Y^{i}}_{=Y^{n}}%
v^{i}=Y^{n}\underbrace{\sum_{i=0}^{n}a_{i}v^{i}}_{=0}=0.
\]
Besides, every $i\in\left\{  0,1,\ldots,n\right\}  $ satisfies
\[
\underbrace{a_{i}}_{\substack{\in I_{n-i}\\\text{(by assumption)}}}Y^{n-i}\in
I_{n-i}Y^{n-i}\subseteq\sum\limits_{\ell\in\mathbb{N}}I_{\ell}Y^{\ell
}=A\left[  \left(  I_{\rho}\right)  _{\rho\in\mathbb{N}}\ast Y\right]  .
\]
In other words, $a_{0}Y^{n-0},a_{1}Y^{n-1},\ldots,a_{n}Y^{n-n}$ are $n+1$
elements of $A\left[  \left(  I_{\rho}\right)  _{\rho\in\mathbb{N}}\ast
Y\right]  $. Hence, Theorem~\ref{Theorem2} (applied to $A\left[  \left(
I_{\rho}\right)  _{\rho\in\mathbb{N}}\ast Y\right]  $, $B\left[  Y\right]  $,
$vY$ and $a_{i}Y^{n-i}$ instead of $A$, $B$, $v$ and $a_{i}$) yields that
$\sum\limits_{i=0}^{n-k}a_{i+k}Y^{n-\left(  i+k\right)  }\left(  vY\right)
^{i}$ is $n$-integral over $A\left[  \left(  I_{\rho}\right)  _{\rho
\in\mathbb{N}}\ast Y\right]  $. Since%
\[
\sum\limits_{i=0}^{n-k}a_{i+k}Y^{n-\left(  i+k\right)  }\underbrace{\left(
vY\right)  ^{i}}_{=v^{i}Y^{i}=Y^{i}v^{i}}=\sum\limits_{i=0}^{n-k}%
a_{i+k}\underbrace{Y^{n-\left(  i+k\right)  }Y^{i}}_{=Y^{\left(  n-\left(
i+k\right)  \right)  +i}=Y^{n-k}}v^{i}=\underbrace{\sum\limits_{i=0}%
^{n-k}a_{i+k}v^{i}}_{\substack{=u\\\text{(by (\ref{pf.Theorem17.u=}))}}}\cdot
Y^{n-k}=uY^{n-k},
\]
this means that $uY^{n-k}$ is $n$-integral over $A\left[  \left(  I_{\rho
}\right)  _{\rho\in\mathbb{N}}\ast Y\right]  $.

But Theorem~\ref{Theorem16} (applied to $\lambda=n-k$) yields that $u$ is
$n$-integral over \newline$\left(  A,\left(  I_{\left(  n-k\right)  \rho
}\right)  _{\rho\in\mathbb{N}}\right)  $ if and only if $uY^{n-k}$ is
$n$-integral over the ring $A\left[  \left(  I_{\rho}\right)  _{\rho
\in\mathbb{N}}\ast Y\right]  $. Since we know that $uY^{n-k}$ is $n$-integral
over the ring $A\left[  \left(  I_{\rho}\right)  _{\rho\in\mathbb{N}}\ast
Y\right]  $, this yields that $u$ is $n$-integral over $\left(  A,\left(
I_{\left(  n-k\right)  \rho}\right)  _{\rho\in\mathbb{N}}\right)  $. In other
words, $\sum\limits_{i=0}^{n-k}a_{i+k}v^{i}$ is $n$-integral over $\left(
A,\left(  I_{\left(  n-k\right)  \rho}\right)  _{\rho\in\mathbb{N}}\right)  $
(since $u=\sum\limits_{i=0}^{n-k}a_{i+k}v^{i}$). This proves
Theorem~\ref{Theorem17}.
\end{proof}

\section{\label{sect.5}On a lemma by Lombardi}

\subsection{A lemma on products of powers}

Now, we shall show a rather technical lemma:

\begin{lemma}
\label{Lemma18} Let $A$ be a ring. Let $B$ be an $A$-algebra. Let $x\in B$.
Let $m\in\mathbb{N}$ and $n\in\mathbb{N}$. Let $u\in B$. Let $\mu\in
\mathbb{N}$ and $\nu\in\mathbb{N}$ be such that $\mu+\nu\in\mathbb{N}^{+}$.
Assume that%
\begin{equation}
u^{n}\in\left\langle u^{0},u^{1},\ldots,u^{n-1}\right\rangle _{A}%
\cdot\left\langle x^{0},x^{1},\ldots,x^{\nu}\right\rangle _{A} \label{L18-1}%
\end{equation}
and that%
\begin{align}
u^{m}x^{\mu}  &  \in\left\langle u^{0},u^{1},\ldots,u^{m-1}\right\rangle
_{A}\cdot\left\langle x^{0},x^{1},\ldots,x^{\mu}\right\rangle _{A}\nonumber\\
&  \ \ \ \ \ \ \ \ \ \ +\left\langle u^{0},u^{1},\ldots,u^{m}\right\rangle
_{A}\cdot\left\langle x^{0},x^{1},\ldots,x^{\mu-1}\right\rangle _{A}.
\label{L18-2}%
\end{align}
Then, $u$ is $\left(  n\mu+m\nu\right)  $-integral over $A$.
\end{lemma}

This lemma can be seen as a variant of \cite[Theorem 2]{3}%
\footnote{\textit{Caveat:} The notion \textquotedblleft integral over $\left(
A,J\right)  $\textquotedblright\ defined in \cite{3} has nothing to do with
\textbf{our} notion \textquotedblleft$n$-integral over $\left(  A,\left(
I_{n}\right)  _{n\in\mathbb{N}}\right)  $\textquotedblright.}. Indeed, the
particular case of \cite[Theorem 2]{3} when $J=0$ can easily be obtained from
Lemma~\ref{Lemma18} (applied to $x$ and $\alpha$ instead of $u$ and $x$).

\begin{noncompile}
I used to believe that \textquotedblleft The general case of \cite[Theorem
2]{3} can, in turn, be obtained from this particular case by passing from the
ring $A$ to its localization $A_{1+J}$.\textquotedblright\ But now I am no
longer sure how this works.
\end{noncompile}

\begin{vershort}
The proof of Lemma \ref{Lemma18} is not difficult but rather elaborate. For a
completely detailed writeup of this proof, see \cite{verlong}. Here let me
give the skeleton of the proof:

\begin{proof}
[Proof of Lemma~\ref{Lemma18} (sketched).]Define the set%
\begin{align*}
S  &  =\left(  \left\{  0,1,\ldots,n-1\right\}  \times\left\{  0,1,\ldots
,\mu-1\right\}  \right) \\
&  \ \ \ \ \ \ \ \ \ \ \cup\left(  \left\{  0,1,\ldots,m-1\right\}
\times\left\{  \mu,\mu+1,\ldots,\mu+\nu-1\right\}  \right)  .
\end{align*}
Clearly, $\left\vert S\right\vert =n\mu+m\nu$ and%
\begin{equation}
j<\mu+\nu\text{ for every }\left(  i,j\right)  \in S. \label{L18-banal}%
\end{equation}

Let $U$ be the $A$-submodule $\left\langle u^{i}x^{j}\ \mid\ \left(
i,j\right)  \in S\right\rangle _{A}$ of $B$. Then, $U$ is an $\left(
n\mu+m\nu\right)  $-generated $A$-module (since $\left\vert S\right\vert
=n\mu+m\nu$). Besides, clearly,
\begin{equation}
u^{i}x^{j}\in U\text{ for every }\left(  i,j\right)  \in S. \label{L18-U}%
\end{equation}

Now, we will show that%
\begin{equation}
\text{every }i\in\mathbb{N}\text{ and }j\in\mathbb{N}\text{ satisfying }%
j<\mu+\nu\text{ satisfy }u^{i}x^{j}\in U. \label{L18-indU}%
\end{equation}

[The \textit{proof of (\ref{L18-indU})} can be done either by double induction
(over $i$ and over $j$) or by the well-ordering principle. The induction proof
has the advantage that it is completely constructive, but it is clumsy (I give
this induction proof in \cite{verlong}). So, for the sake of brevity, the
proof I am going to give here is by the well-ordering principle:

For the sake of contradiction, we assume that (\ref{L18-indU}) is not true.
That is, there exists some pair $\left(  i,j\right)  \in\mathbb{N}^{2}$
satisfying $j<\mu+\nu$ but \textbf{not} satisfying $u^{i}x^{j}\in U$. Let
$\left(  I,J\right)  $ be the lexicographically
smallest\footnote{\textquotedblleft Lexicographically
smallest\textquotedblright\ means \textquotedblleft smallest with respect to
the lexicographic order\textquotedblright. Here, the \textit{lexicographic
order} on $\mathbb{N}^{2}$ is defined to be the total order on $\mathbb{N}%
^{2}$ in which two pairs $\left(  a_{1},b_{1}\right)  \in\mathbb{N}^{2}$ and
$\left(  a_{2},b_{2}\right)  \in\mathbb{N}^{2}$ satisfy $\left(  a_{1}%
,b_{1}\right)  <\left(  a_{2},b_{2}\right)  $ if and only if either
$a_{1}<a_{2}$ or $\left(  a_{1}=a_{2}\text{ and }b_{1}<b_{2}\right)  $. It is
well-known that this total order is well-defined and turns $\mathbb{N}^{2}$
into a well-ordered set.} such pair\footnote{This is well-defined, since the
lexicographic order is a well-ordering on $\mathbb{N}^{2}$.}. Then, $J<\mu
+\nu$ but $u^{I}x^{J}\notin U$, and since $\left(  I,J\right)  $ is the
lexicographically smallest such pair, we have%
\begin{equation}
u^{I}x^{j}\in U\text{ for every }j\in\mathbb{N}\text{ such that }j<J
\label{L18-indL1}%
\end{equation}
and%
\begin{equation}
u^{i}x^{j}\in U\text{ for every }i\in\mathbb{N}\text{ and }j\in\mathbb{N}%
\text{ such that }i<I\text{ and }j<\mu+\nu. \label{L18-indL2}%
\end{equation}

Recall that $U$ is an $A$-module. Hence, (\ref{L18-indL1}) entails%
\begin{equation}
\left\langle u^{I}\right\rangle _{A}\cdot\left\langle x^{0},x^{1}%
,\ldots,x^{J-1}\right\rangle _{A}\subseteq U, \label{L18-indL1final}%
\end{equation}
and (\ref{L18-indL2}) entails
\begin{equation}
\left\langle u^{0},u^{1},\ldots,u^{I-1}\right\rangle _{A}\cdot\left\langle
x^{0},x^{1},\ldots,x^{\mu+\nu-1}\right\rangle _{A}\subseteq U.
\label{L18-indL2final}%
\end{equation}
Also, from $J<\mu+\nu$, we obtain $J\leq\mu+\nu-1$ (since $J$ and $\mu+\nu$
are integers).

We now want to prove that $u^{I}x^{J}\in U$.

We are in one of the following four cases:

\textit{Case 1:} We have $I\geq m$ and $J\geq\mu$.

\textit{Case 2:} We have $I<m$ and $J\geq\mu$.

\textit{Case 3:} We have $I\geq n$ and $J<\mu$.

\textit{Case 4:} We have $I<n$ and $J<\mu$.

In Case 1, we have $I-m\geq0$ (since $I\geq m$) and $J-\mu\geq0$ (since
$J\geq\mu$), thus%
\begin{align*}
&  \underbrace{u^{I}}_{=u^{I-m}u^{m}}\underbrace{x^{J}}_{=x^{\mu}x^{J-\mu}}\\
&  =u^{I-m}\underbrace{u^{m}x^{\mu}}_{\substack{\in\left\langle u^{0}%
,u^{1},\ldots,u^{m-1}\right\rangle _{A}\cdot\left\langle x^{0},x^{1}%
,\ldots,x^{\mu}\right\rangle _{A}+\left\langle u^{0},u^{1},\ldots
,u^{m}\right\rangle _{A}\cdot\left\langle x^{0},x^{1},\ldots,x^{\mu
-1}\right\rangle _{A}\\\left(  \text{by (\ref{L18-2})}\right)  }}x^{J-\mu}\\
&  \in u^{I-m}\left(  \left\langle u^{0},u^{1},\ldots,u^{m-1}\right\rangle
_{A}\cdot\left\langle x^{0},x^{1},\ldots,x^{\mu}\right\rangle _{A}\right. \\
&  \ \ \ \ \ \ \ \ \ \ \ \ \ \ \ \ \ \ \ \ \left.  +\left\langle u^{0}%
,u^{1},\ldots,u^{m}\right\rangle _{A}\cdot\left\langle x^{0},x^{1}%
,\ldots,x^{\mu-1}\right\rangle _{A}\right)  x^{J-\mu}\\
&  =\underbrace{u^{I-m}\left\langle u^{0},u^{1},\ldots,u^{m-1}\right\rangle
_{A}}_{\substack{=\left\langle u^{I-m},u^{I-m+1},\ldots,u^{I-1}\right\rangle
_{A}\\\subseteq\left\langle u^{0},u^{1},\ldots,u^{I-1}\right\rangle _{A}%
}}\cdot\underbrace{\left\langle x^{0},x^{1},\ldots,x^{\mu}\right\rangle
_{A}x^{J-\mu}}_{\substack{=\left\langle x^{J-\mu},x^{J-\mu+1},\ldots
,x^{J}\right\rangle _{A}\\\subseteq\left\langle x^{0},x^{1},\ldots,x^{\mu
+\nu-1}\right\rangle _{A}\\\text{(since }J-\mu\geq0\text{ and }J\leq\mu
+\nu-1\text{)}}}\\
&  \ \ \ \ \ \ \ \ \ \ +\underbrace{u^{I-m}\left\langle u^{0},u^{1}%
,\ldots,u^{m}\right\rangle _{A}}_{\substack{=\left\langle u^{I-m}%
,u^{I-m+1},\ldots,u^{I}\right\rangle _{A}\\\subseteq\left\langle u^{0}%
,u^{1},\ldots,u^{I}\right\rangle _{A}}}\cdot\underbrace{\left\langle
x^{0},x^{1},\ldots,x^{\mu-1}\right\rangle _{A}x^{J-\mu}}%
_{\substack{=\left\langle x^{J-\mu},x^{J-\mu+1},\ldots,x^{J-1}\right\rangle
_{A}\\\subseteq\left\langle x^{0},x^{1},\ldots,x^{J-1}\right\rangle _{A}}}\\
&  \subseteq\underbrace{\left\langle u^{0},u^{1},\ldots,u^{I-1}\right\rangle
_{A}\cdot\left\langle x^{0},x^{1},\ldots,x^{\mu+\nu-1}\right\rangle _{A}%
}_{\substack{\subseteq U\\\text{(by (\ref{L18-indL2final}))}}}\\
&  \ \ \ \ \ \ \ \ \ \ +\underbrace{\left\langle u^{0},u^{1},\ldots
,u^{I}\right\rangle _{A}}_{=\left\langle u^{0},u^{1},\ldots,u^{I-1}%
\right\rangle _{A}+\left\langle u^{I}\right\rangle _{A}}\cdot\left\langle
x^{0},x^{1},\ldots,x^{J-1}\right\rangle _{A}\\
&  \subseteq U+\underbrace{\left(  \left\langle u^{0},u^{1},\ldots
,u^{I-1}\right\rangle _{A}+\left\langle u^{I}\right\rangle _{A}\right)
\cdot\left\langle x^{0},x^{1},\ldots,x^{J-1}\right\rangle _{A}}_{=\left\langle
u^{0},u^{1},\ldots,u^{I-1}\right\rangle _{A}\cdot\left\langle x^{0}%
,x^{1},\ldots,x^{J-1}\right\rangle _{A}+\left\langle u^{I}\right\rangle
_{A}\cdot\left\langle x^{0},x^{1},\ldots,x^{J-1}\right\rangle _{A}}\\
&  =U+\left\langle u^{0},u^{1},\ldots,u^{I-1}\right\rangle _{A}\cdot
\underbrace{\left\langle x^{0},x^{1},\ldots,x^{J-1}\right\rangle _{A}%
}_{\substack{\subseteq\left\langle x^{0},x^{1},\ldots,x^{\mu+\nu
-1}\right\rangle _{A}\\\text{(since }J-1\leq J\leq\mu+\nu-1\text{)}%
}}+\left\langle u^{I}\right\rangle _{A}\cdot\left\langle x^{0},x^{1}%
,\ldots,x^{J-1}\right\rangle _{A}\\
&  \subseteq U+\underbrace{\left\langle u^{0},u^{1},\ldots,u^{I-1}%
\right\rangle _{A}\cdot\left\langle x^{0},x^{1},\ldots,x^{\mu+\nu
-1}\right\rangle _{A}}_{\substack{\subseteq U\\\text{(by (\ref{L18-indL2final}%
))}}}+\underbrace{\left\langle u^{I}\right\rangle _{A}\cdot\left\langle
x^{0},x^{1},\ldots,x^{J-1}\right\rangle _{A}}_{\substack{\subseteq
U\\\text{(by (\ref{L18-indL1final}))}}}\\
&  \subseteq U+U+U\subseteq U\ \ \ \ \ \ \ \ \ \ \left(  \text{since }U\text{
is an }A\text{-module}\right)  .
\end{align*}
Thus, we have proved that $u^{I}x^{J}\in U$ holds in Case 1.

In Case 2, we have $\left(  I,J\right)  \in\left\{  0,1,\ldots,m-1\right\}
\times\left\{  \mu,\mu+1,\ldots,\mu+\nu-1\right\}  \subseteq S$ and thus
$u^{I}x^{J}\in U$ (by (\ref{L18-U}), applied to $I$ and $J$ instead of $i$ and
$j$). Thus, we have proved that $u^{I}x^{J}\in U$ holds in Case 2.

In Case 3, we have $I-n\geq0$ (since $I\geq n$) and $J+\nu\leq\mu+\nu-1$
(since $J<\mu$ yields $J+\nu<\mu+\nu$, and since $J+\nu$ and $\mu+\nu$ are
integers), thus%
\begin{align*}
&  \underbrace{u^{I}}_{=u^{I-n}u^{n}}x^{J}\\
&  =u^{I-n}\underbrace{u^{n}}_{\substack{\in\left\langle u^{0},u^{1}%
,\ldots,u^{n-1}\right\rangle _{A}\cdot\left\langle x^{0},x^{1},\ldots,x^{\nu
}\right\rangle _{A}\\\left(  \text{by (\ref{L18-1})}\right)  }}x^{J}\\
&  \in\underbrace{u^{I-n}\left\langle u^{0},u^{1},\ldots,u^{n-1}\right\rangle
_{A}}_{\substack{=\left\langle u^{I-n},u^{I-n+1},\ldots,u^{I-1}\right\rangle
_{A}\\\subseteq\left\langle u^{0},u^{1},\ldots,u^{I-1}\right\rangle _{A}%
}}\cdot\underbrace{\left\langle x^{0},x^{1},\ldots,x^{\nu}\right\rangle
_{A}x^{J}}_{\substack{=\left\langle x^{J},x^{J+1},\ldots,x^{J+\nu
}\right\rangle _{A}\\\subseteq\left\langle x^{0},x^{1},\ldots,x^{\mu+\nu
-1}\right\rangle _{A}\\\text{(since }J+\nu\leq\mu+\nu-1\text{)}}}\\
&  \subseteq\left\langle u^{0},u^{1},\ldots,u^{I-1}\right\rangle _{A}%
\cdot\left\langle x^{0},x^{1},\ldots,x^{\mu+\nu-1}\right\rangle _{A}\subseteq
U\ \ \ \ \ \ \ \ \ \ \left(  \text{by (\ref{L18-indL2final})}\right)  .
\end{align*}
Thus, we have proved that $u^{I}x^{J}\in U$ holds in Case 3.

In Case 4, we have $\left(  I,J\right)  \in\left\{  0,1,\ldots,n-1\right\}
\times\left\{  0,1,\ldots,\mu-1\right\}  \subseteq S$ and thus $u^{I}x^{J}\in
U$ (by (\ref{L18-U}), applied to $I$ and $J$ instead of $i$ and $j$). Thus, we
have proved that $u^{I}x^{J}\in U$ holds in Case 4.

By now, we have proved that $u^{I}x^{J}\in U$ holds in each of the four cases
1, 2, 3 and 4. Hence, $u^{I}x^{J}\in U$ always holds, contradicting
$u^{I}x^{J}\notin U$. This contradiction completes the proof of
(\ref{L18-indU}).]

Now that (\ref{L18-indU}) is proven, we can easily conclude that $uU\subseteq
U$ (since every $\left(  i,j\right)  \in S$ satisfies $j<\mu+\nu$, and thus
(\ref{L18-indU}) shows that $u^{i+1}x^{j}\in U$) and $1\in U$ (this follows by
applying (\ref{L18-indU}) to $i=0$ and $j=0$). Altogether, $U$ is an $\left(
n\mu+m\nu\right)  $-generated $A$-submodule of $B$ such that $1\in U$ and
$uU\subseteq U$. Thus, $u\in B$ satisfies Assertion $\mathcal{C}$ of
Theorem~\ref{Theorem1} with $n$ replaced by $n\mu+m\nu$. Hence, $u\in B$
satisfies the four equivalent assertions $\mathcal{A}$, $\mathcal{B}$,
$\mathcal{C}$ and $\mathcal{D}$ of Theorem~\ref{Theorem1} with $n$ replaced by
$n\mu+m\nu$. Consequently, $u$ is $\left(  n\mu+m\nu\right)  $-integral over
$A$. This proves Lemma~\ref{Lemma18}.
\end{proof}
\end{vershort}

\begin{verlong}
Before we prove Lemma \ref{Lemma18}, we recall a basic mathematical principle:

\begin{proposition}
\label{prop.strind1}Let $\mathfrak{A}\left(  i\right)  $ be an assertion for
every $i\in\mathbb{N}$. If%
\[
\text{every }I\in\mathbb{N}\text{ satisfying }\left(  \mathfrak{A}\left(
i\right)  \text{ for every }i\in\mathbb{N}\text{ such that }i<I\right)  \text{
satisfies }\mathfrak{A}\left(  I\right)  ,
\]
then%
\[
\text{every }i\in\mathbb{N}\text{ satisfies }\mathfrak{A}\left(  i\right)
\text{.}%
\]

\end{proposition}

Proposition \ref{prop.strind1} is known as the \textit{principle of strong
induction}. By renaming $i$, $I$ and $\mathfrak{A}$ as $j$, $J$ and
$\mathfrak{B}$, respectively, we can rewrite this principle as follows:

\begin{proposition}
\label{prop.strind2}Let $\mathfrak{B}\left(  j\right)  $ be an assertion for
every $j\in\mathbb{N}$. If%
\[
\text{every }J\in\mathbb{N}\text{ satisfying }\left(  \mathfrak{B}\left(
j\right)  \text{ for every }j\in\mathbb{N}\text{ such that }j<J\right)  \text{
satisfies }\mathfrak{B}\left(  J\right)  ,
\]
then%
\[
\text{every }j\in\mathbb{N}\text{ satisfies }\mathfrak{B}\left(  j\right)
\text{.}%
\]

\end{proposition}

\begin{proof}
[Proof of Lemma~\ref{Lemma18}.]Define the set%
\begin{align}
S  &  =\left(  \left\{  0,1,\ldots,n-1\right\}  \times\left\{  0,1,\ldots
,\mu-1\right\}  \right) \nonumber\\
&  \ \ \ \ \ \ \ \ \ \ \cup\left(  \left\{  0,1,\ldots,m-1\right\}
\times\left\{  \mu,\mu+1,\ldots,\mu+\nu-1\right\}  \right)  .
\label{pf.Lemma18.S=}%
\end{align}
Then, $\left\vert S\right\vert =n\mu+m\nu\ \ \ \ $\footnote{\textit{Proof.} We
have $\left(  U\times V\right)  \cap\left(  X\times Y\right)  =\left(  U\cap
X\right)  \times\left(  V\cap Y\right)  $ for any four sets $U$, $V$, $X$ and
$Y$. Applying this to $U=\left\{  0,1,\ldots,n-1\right\}  $, $V=\left\{
0,1,\ldots,\mu-1\right\}  $, $X=\left\{  0,1,\ldots,m-1\right\}  $ and
$Y=\left\{  \mu,\mu+1,\ldots,\mu+\nu-1\right\}  $, we find%
\begin{align*}
&  \left(  \left\{  0,1,\ldots,n-1\right\}  \times\left\{  0,1,\ldots
,\mu-1\right\}  \right)  \cap\left(  \left\{  0,1,\ldots,m-1\right\}
\times\left\{  \mu,\mu+1,\ldots,\mu+\nu-1\right\}  \right) \\
&  =\left(  \left\{  0,1,\ldots,n-1\right\}  \cap\left\{  0,1,\ldots
,m-1\right\}  \right)  \times\underbrace{\left(  \left\{  0,1,\ldots
,\mu-1\right\}  \cap\left\{  \mu,\mu+1,\ldots,\mu+\nu-1\right\}  \right)
}_{=\varnothing}\\
&  =\left(  \left\{  0,1,\ldots,n-1\right\}  \cap\left\{  0,1,\ldots
,m-1\right\}  \right)  \times\varnothing=\varnothing.
\end{align*}
Hence,%
\begin{align*}
&  \left\vert \left(  \left\{  0,1,\ldots,n-1\right\}  \times\left\{
0,1,\ldots,\mu-1\right\}  \right)  \cup\left(  \left\{  0,1,\ldots
,m-1\right\}  \times\left\{  \mu,\mu+1,\ldots,\mu+\nu-1\right\}  \right)
\right\vert \\
&  =\underbrace{\left\vert \left\{  0,1,\ldots,n-1\right\}  \times\left\{
0,1,\ldots,\mu-1\right\}  \right\vert }_{=\left\vert \left\{  0,1,\ldots
,n-1\right\}  \right\vert \cdot\left\vert \left\{  0,1,\ldots,\mu-1\right\}
\right\vert }+\underbrace{\left\vert \left\{  0,1,\ldots,m-1\right\}
\times\left\{  \mu,\mu+1,\ldots,\mu+\nu-1\right\}  \right\vert }_{=\left\vert
\left\{  0,1,\ldots,m-1\right\}  \right\vert \cdot\left\vert \left\{  \mu
,\mu+1,\ldots,\mu+\nu-1\right\}  \right\vert }\\
&  \ \ \ \ \ \ \ \ \ \ \left(
\begin{array}
[c]{c}%
\text{because any two finite sets }U\text{ and }V\text{ satisfying }U\cap
V=\varnothing\\
\text{satisfy }\left\vert U\cup V\right\vert =\left\vert U\right\vert
+\left\vert V\right\vert
\end{array}
\right) \\
&  =\underbrace{\left\vert \left\{  0,1,\ldots,n-1\right\}  \right\vert }%
_{=n}\cdot\underbrace{\left\vert \left\{  0,1,\ldots,\mu-1\right\}
\right\vert }_{=\mu}+\underbrace{\left\vert \left\{  0,1,\ldots,m-1\right\}
\right\vert }_{=m}\cdot\underbrace{\left\vert \left\{  \mu,\mu+1,\ldots
,\mu+\nu-1\right\}  \right\vert }_{=\nu}=n\mu+m\nu.
\end{align*}
In view of
\[
S=\left(  \left\{  0,1,\ldots,n-1\right\}  \times\left\{  0,1,\ldots
,\mu-1\right\}  \right)  \cup\left(  \left\{  0,1,\ldots,m-1\right\}
\times\left\{  \mu,\mu+1,\ldots,\mu+\nu-1\right\}  \right)  ,
\]
this rewrites as $\left\vert S\right\vert =n\mu+m\nu$.}. Also,%
\begin{equation}
j<\mu+\nu\text{ for every }\left(  i,j\right)  \in S \label{L18-banal-long}%
\end{equation}
\footnote{In fact, $\nu\geq0$ (since $\nu\in\mathbb{N}$), so that
$\mu+\underbrace{\nu}_{\geq0}-1\geq\mu-1$. Hence, $\mu-1\leq\mu+\nu-1$, so
that%
\begin{align*}
S  &  =\left(  \left\{  0,1,\ldots,n-1\right\}  \times\underbrace{\left\{
0,1,\ldots,\mu-1\right\}  }_{\substack{\subseteq\left\{  0,1,\ldots,\mu
+\nu-1\right\}  \\\text{(since }\mu-1\leq\mu+\nu-1\text{)}}}\right) \\
&  \ \ \ \ \ \ \ \ \ \ \cup\left(  \left\{  0,1,\ldots,m-1\right\}
\times\underbrace{\left\{  \mu,\mu+1,\ldots,\mu+\nu-1\right\}  }%
_{\substack{\subseteq\left\{  0,1,\ldots,\mu+\nu-1\right\}  \\\text{(since
}\mu\geq0\text{)}}}\right) \\
&  \subseteq\left(  \left\{  0,1,\ldots,n-1\right\}  \times\left\{
0,1,\ldots,\mu+\nu-1\right\}  \right) \\
&  \ \ \ \ \ \ \ \ \ \ \cup\left(  \left\{  0,1,\ldots,m-1\right\}
\times\left\{  0,1,\ldots,\mu+\nu-1\right\}  \right) \\
&  =\left(  \left\{  0,1,\ldots,n-1\right\}  \cup\left\{  0,1,\ldots
,m-1\right\}  \right)  \times\left\{  0,1,\ldots,\mu+\nu-1\right\}
\end{align*}
(since $\left(  U\times X\right)  \cup\left(  V\times X\right)  =\left(  U\cup
V\right)  \times X$ for any three sets $U$, $V$ and $X$). Hence, for every
$\left(  i,j\right)  \in S$, we have $\left(  i,j\right)  \in S\subseteq
\left(  \left\{  0,1,\ldots,n-1\right\}  \cup\left\{  0,1,\ldots,m-1\right\}
\right)  \times\left\{  0,1,\ldots,\mu+\nu-1\right\}  $ and thus $j\in\left\{
0,1,\ldots,\mu+\nu-1\right\}  $ and thus $j<\mu+\nu$. This proves
(\ref{L18-banal-long}).}.

Let $U$ be the $A$-submodule $\left\langle u^{i}x^{j}\ \mid\ \left(
i,j\right)  \in S\right\rangle _{A}$ of $B$. Then, $U$ is an $\left(
n\mu+m\nu\right)  $-generated $A$-module (since $\left\vert S\right\vert
=n\mu+m\nu$). Besides, clearly,
\begin{equation}
u^{i}x^{j}\in U\text{ for every }\left(  i,j\right)  \in S \label{L18-U-long}%
\end{equation}
(since $U=\left\langle u^{i}x^{j}\ \mid\ \left(  i,j\right)  \in
S\right\rangle _{A}$).

Now, we will show that%
\begin{equation}
\text{every }i\in\mathbb{N}\text{ and }j\in\mathbb{N}\text{ satisfying }%
j<\mu+\nu\text{ satisfy }u^{i}x^{j}\in U. \label{L18-indU-long}%
\end{equation}

[\textit{Proof of (\ref{L18-indU-long}).} For every $i\in\mathbb{N}$, define
an assertion $\mathfrak{A}\left(  i\right)  $ by%
\[
\mathfrak{A}\left(  i\right)  =\left(  \text{every }j\in\mathbb{N}\text{
satisfies }\left(  \text{if }j<\mu+\nu\text{, then }u^{i}x^{j}\in U\right)
\right)  .
\]

Let us now show that%
\begin{equation}
\text{every }I\in\mathbb{N}\text{ satisfying }\left(  \mathfrak{A}\left(
i\right)  \text{ for every }i\in\mathbb{N}\text{ such that }i<I\right)  \text{
satisfies }\mathfrak{A}\left(  I\right)  . \label{L18-indI-long}%
\end{equation}

[\textit{Proof of (\ref{L18-indI-long}).} Let $I\in\mathbb{N}$ be such that%
\begin{equation}
\left(  \mathfrak{A}\left(  i\right)  \text{ for every }i\in\mathbb{N}\text{
such that }i<I\right)  . \label{L18-indIass-long}%
\end{equation}
We must prove that $\mathfrak{A}\left(  I\right)  $ holds.

The definition of the assertion $\mathfrak{A}\left(  I\right)  $ yields%
\[
\mathfrak{A}\left(  I\right)  =\left(  \text{every }j\in\mathbb{N}\text{
satisfies }\left(  \text{if }j<\mu+\nu\text{, then }u^{I}x^{j}\in U\right)
\right)  .
\]

For every $j\in\mathbb{N}$, define an assertion $\mathfrak{B}\left(  j\right)
$ by%
\begin{equation}
\mathfrak{B}\left(  j\right)  =\left(  \text{if }j<\mu+\nu\text{, then }%
u^{I}x^{j}\in U\right)  . \label{pf.Lemma18.Bj=}%
\end{equation}

Let us now show that%
\begin{equation}
\text{every }J\in\mathbb{N}\text{ satisfying }\left(  \mathfrak{B}\left(
j\right)  \text{ for every }j\in\mathbb{N}\text{ such that }j<J\right)  \text{
satisfies }\mathfrak{B}\left(  J\right)  . \label{L18-indJ-long}%
\end{equation}

[\textit{Proof of (\ref{L18-indJ-long}).} Let $J\in\mathbb{N}$ be such that%
\begin{equation}
\left(  \mathfrak{B}\left(  j\right)  \text{ for every }j\in\mathbb{N}\text{
such that }j<J\right)  . \label{L18-indJass-long}%
\end{equation}
We must prove that $\mathfrak{B}\left(  J\right)  $ holds.

The definition of the assertion $\mathfrak{B}\left(  J\right)  $ yields%
\[
\mathfrak{B}\left(  J\right)  =\left(  \text{if }J<\mu+\nu\text{, then }%
u^{I}x^{J}\in U\right)  .
\]

Assume that $J<\mu+\nu$. Then, for every $j\in\mathbb{N}$ such that $j<J$, the
assertion $\mathfrak{B}\left(  j\right)  $ holds (due to
(\ref{L18-indJass-long})). In other words, for every $j\in\mathbb{N}$ such
that $j<J$, we have $\left(  \text{if }j<\mu+\nu\text{, then }u^{I}x^{j}\in
U\right)  $ (because this is precisely what the assertion $\mathfrak{B}\left(
j\right)  $ says) and therefore $u^{I}x^{j}\in U$ (since $j<\mu+\nu$
automatically holds\footnote{because $j<J<\mu+\nu$}). Thus we have shown that
\begin{equation}
u^{I}x^{j}\in U\text{ for every }j\in\mathbb{N}\text{ such that }j<J.
\label{L18-indL1-long}%
\end{equation}
In other words,%
\begin{equation}
u^{I}x^{j}\in U\text{ for every }j\in\left\{  0,1,\ldots,J-1\right\}
\label{L18-indL1-long'}%
\end{equation}
(since the numbers $j\in\left\{  0,1,\ldots,J-1\right\}  $ are precisely the
numbers $j\in\mathbb{N}$ such that $j<J$). Hence,
\begin{align}
&  \sum\limits_{j\in\left\{  0,1,\ldots,J-1\right\}  }a_{j}u^{I}x^{j}\in
U\label{L18-indL1-long''}\\
&  \ \ \ \ \ \ \ \ \ \ \text{for every }\left(  a_{j}\right)  _{j\in\left\{
0,1,\ldots,J-1\right\}  }\in A^{\left\{  0,1,\ldots,J-1\right\}  }\nonumber
\end{align}
(since $U$ is an $A$-module, and thus is closed under $A$-linear combination).

Also, if $i\in\mathbb{N}$ satisfies $i<I$, then the assertion $\mathfrak{A}%
\left(  i\right)  $ holds (by (\ref{L18-indIass-long})). In view of the
definition of $\mathfrak{A}\left(  i\right)  $, we can restate this as
follows: If $i\in\mathbb{N}$ satisfies $i<I$, then every $j\in\mathbb{N}$
satisfies $\left(  \text{if }j<\mu+\nu\text{, then }u^{i}x^{j}\in U\right)  $.
In other words,%
\begin{equation}
u^{i}x^{j}\in U\text{ for every }i\in\mathbb{N}\text{ and }j\in\mathbb{N}%
\text{ such that }i<I\text{ and }j<\mu+\nu. \label{L18-indL2-long}%
\end{equation}
Hence,%
\begin{equation}
u^{i}x^{j}\in U\text{ for every }\left(  i,j\right)  \in\left\{
0,1,\ldots,I-1\right\}  \times\left\{  0,1,\ldots,\mu+\nu-1\right\}
\label{L18-indL2-long'}%
\end{equation}
(because for every $\left(  i,j\right)  \in\left\{  0,1,\ldots,I-1\right\}
\times\left\{  0,1,\ldots,\mu+\nu-1\right\}  $, we have $i<I$ (since
$i\in\left\{  0,1,\ldots,I-1\right\}  $) and $j<\mu+\nu$ (since $j\in\left\{
0,1,\ldots,\mu+\nu-1\right\}  $) and therefore $u^{i}x^{j}\in U$ (by
(\ref{L18-indL2-long}))). Hence,%
\begin{align}
&  \sum\limits_{\left(  i,j\right)  \in\left\{  0,1,\ldots,I-1\right\}
\times\left\{  0,1,\ldots,\mu+\nu-1\right\}  }a_{i,j}u^{i}x^{j}\in
U\label{L18-indL2-long''}\\
&  \ \ \ \ \ \ \ \ \ \ \text{for every }\left(  a_{i,j}\right)  _{\left(
i,j\right)  \in\left\{  0,1,\ldots,I-1\right\}  \times\left\{  0,1,\ldots
,\mu+\nu-1\right\}  }\in A^{\left\{  0,1,\ldots,I-1\right\}  \times\left\{
0,1,\ldots,\mu+\nu-1\right\}  }\nonumber
\end{align}
(since $U$ is an $A$-module, and thus is closed under $A$-linear combination).

Now,%
\begin{align}
&  \left\langle u^{I}\right\rangle _{A}\cdot\underbrace{\left\langle
x^{0},x^{1},\ldots,x^{J-1}\right\rangle _{A}}_{=\left\langle x^{j}\ \mid
\ j\in\left\{  0,1,\ldots,J-1\right\}  \right\rangle _{A}}\nonumber\\
&  =\left\langle u^{I}\right\rangle _{A}\cdot\left\langle x^{j}\ \mid
\ j\in\left\{  0,1,\ldots,J-1\right\}  \right\rangle _{A}=\left\langle
u^{I}x^{j}\ \mid\ j\in\left\{  0,1,\ldots,J-1\right\}  \right\rangle
_{A}\nonumber\\
&  =\left\{  \sum\limits_{j\in\left\{  0,1,\ldots,J-1\right\}  }a_{j}%
u^{I}x^{j}\ \mid\ \left(  a_{j}\right)  _{j\in\left\{  0,1,\ldots,J-1\right\}
}\in A^{\left\{  0,1,\ldots,J-1\right\}  }\right\}  \subseteq U
\label{L18-indL1final-long}%
\end{align}
(by (\ref{L18-indL1-long''})).

Furthermore,%
\begin{align}
&  \underbrace{\left\langle u^{0},u^{1},\ldots,u^{I-1}\right\rangle _{A}%
}_{=\left\langle u^{i}\ \mid\ i\in\left\{  0,1,\ldots,I-1\right\}
\right\rangle _{A}}\cdot\underbrace{\left\langle x^{0},x^{1},\ldots,x^{\mu
+\nu-1}\right\rangle _{A}}_{=\left\langle x^{j}\ \mid\ j\in\left\{
0,1,\ldots,\mu+\nu-1\right\}  \right\rangle _{A}}\nonumber\\
&  =\left\langle u^{i}\ \mid\ i\in\left\{  0,1,\ldots,I-1\right\}
\right\rangle _{A}\cdot\left\langle x^{j}\ \mid\ j\in\left\{  0,1,\ldots
,\mu+\nu-1\right\}  \right\rangle _{A}\nonumber\\
&  =\left\langle u^{i}x^{j}\ \mid\ \left(  i,j\right)  \in\left\{
0,1,\ldots,I-1\right\}  \times\left\{  0,1,\ldots,\mu+\nu-1\right\}
\right\rangle _{A}\nonumber\\
&  =\left\{  \sum\limits_{\left(  i,j\right)  \in\left\{  0,1,\ldots
,I-1\right\}  \times\left\{  0,1,\ldots,\mu+\nu-1\right\}  }a_{i,j}u^{i}%
x^{j}\right. \nonumber\\
&  \ \ \ \ \ \ \ \ \ \ \left.  \ \mid\ \left(  a_{i,j}\right)  _{\left(
i,j\right)  \in\left\{  0,1,\ldots,I-1\right\}  \times\left\{  0,1,\ldots
,\mu+\nu-1\right\}  }\in A^{\left\{  0,1,\ldots,I-1\right\}  \times\left\{
0,1,\ldots,\mu+\nu-1\right\}  }\vphantom{\sum_{i}}\right\} \nonumber\\
&  \subseteq U \label{L18-indL2final-long}%
\end{align}
(by (\ref{L18-indL2-long''})).

From $J<\mu+\nu$, we obtain $J\leq\mu+\nu-1$ (since $J$ and $\mu+\nu$ are
integers). We are now going to show that $u^{I}x^{J}\in U$.

Trivially, we have\footnote{Here, an expression like \textquotedblleft$I\geq
m\ \wedge\ J\geq\mu$\textquotedblright\ should be read as \textquotedblleft%
$\left(  I\geq m\right)  \ \wedge\ \left(  J\geq\mu\right)  $%
\textquotedblright.}%
\[
\left(  I\geq m\ \wedge\ J\geq\mu\right)  \ \vee\ \left(  I<m\ \wedge
\ J\geq\mu\right)  \vee\ \left(  I\geq n\ \wedge\ J<\mu\right)  \ \vee
\ \left(  I<n\ \wedge\ J<\mu\right)
\]
\footnote{since
\begin{align*}
&  \underbrace{\left(  I\geq m\ \wedge\ J\geq\mu\right)  \ \vee\ \left(
I<m\ \wedge\ J\geq\mu\right)  }_{\substack{=\ \left(  I\geq m\ \vee
\ I<m\right)  \ \wedge\ \left(  J\geq\mu\right)  \\=\ \left(  J\geq\mu\right)
\\\text{(since }\left(  I\geq m\ \vee\ I<m\right)  \text{ is true)}}%
}\vee\ \underbrace{\left(  I\geq n\ \wedge\ J<\mu\right)  \ \vee\ \left(
I<n\ \wedge\ J<\mu\right)  }_{\substack{=\ \left(  I\geq n\ \vee\ I<n\right)
\ \wedge\ \left(  J<\mu\right)  \\=\ \left(  J<\mu\right)  \\\text{(since
}\left(  I\geq n\ \vee\ I<n\right)  \text{ is true)}}}\\
&  =\ \left(  J\geq\mu\right)  \ \vee\ \left(  J<\mu\right)  \ =\ \text{true}%
\end{align*}
}. Hence, one of the following four cases must hold:

\textit{Case 1:} We have $I\geq m\ \wedge\ J\geq\mu$.

\textit{Case 2:} We have $I<m\ \wedge\ J\geq\mu$.

\textit{Case 3:} We have $I\geq n\ \wedge\ J<\mu$.

\textit{Case 4:} We have $I<n\ \wedge\ J<\mu$.

Let us first consider Case 1. In this case, we have $I\geq m$ and $J\geq\mu$.
Hence, $I-m\geq0$ (since $I\geq m$) and $J-\mu\geq0$ (since $J\geq\mu$). Thus,%
\begin{align*}
&  \underbrace{u^{I}}_{\substack{=u^{I-m}u^{m}\\\text{(since }I\geq m\text{)}%
}}\underbrace{x^{J}}_{\substack{=x^{\mu}x^{J-\mu}\\\text{(since }J\geq
\mu\text{)}}}\\
&  =u^{I-m}\underbrace{u^{m}x^{\mu}}_{\substack{\in\left\langle u^{0}%
,u^{1},\ldots,u^{m-1}\right\rangle _{A}\cdot\left\langle x^{0},x^{1}%
,\ldots,x^{\mu}\right\rangle _{A}+\left\langle u^{0},u^{1},\ldots
,u^{m}\right\rangle _{A}\cdot\left\langle x^{0},x^{1},\ldots,x^{\mu
-1}\right\rangle _{A}\\\left(  \text{by (\ref{L18-2})}\right)  }}x^{J-\mu}\\
&  \in u^{I-m}\left(  \left\langle u^{0},u^{1},\ldots,u^{m-1}\right\rangle
_{A}\cdot\left\langle x^{0},x^{1},\ldots,x^{\mu}\right\rangle _{A}\right. \\
&  \ \ \ \ \ \ \ \ \ \ \ \ \ \ \ \ \ \ \ \ \left.  +\left\langle u^{0}%
,u^{1},\ldots,u^{m}\right\rangle _{A}\cdot\left\langle x^{0},x^{1}%
,\ldots,x^{\mu-1}\right\rangle _{A}\right)  x^{J-\mu}\\
&  =\underbrace{u^{I-m}\left\langle u^{0},u^{1},\ldots,u^{m-1}\right\rangle
_{A}}_{\substack{=\left\langle u^{I-m}u^{0},u^{I-m}u^{1},\ldots,u^{I-m}%
u^{m-1}\right\rangle _{A}\\=\left\langle u^{\left(  I-m\right)  +0},u^{\left(
I-m\right)  +1},\ldots,u^{\left(  I-m\right)  +\left(  m-1\right)
}\right\rangle _{A}\\=\left\langle u^{I-m},u^{I-m+1},\ldots,u^{I-1}%
\right\rangle _{A}\\\subseteq\left\langle u^{0},u^{1},\ldots,u^{I-1}%
\right\rangle _{A}\\\text{(since }\left\{  I-m,I-m+1,\ldots,I-1\right\}
\subseteq\left\{  0,1,\ldots,I-1\right\}  \\\text{(since }I-m\geq0\text{))}%
}}\cdot\underbrace{\left\langle x^{0},x^{1},\ldots,x^{\mu}\right\rangle
_{A}x^{J-\mu}}_{\substack{=\left\langle x^{0}x^{J-\mu},x^{1}x^{J-\mu}%
,\ldots,x^{\mu}x^{J-\mu}\right\rangle _{A}\\=\left\langle x^{0+\left(
J-\mu\right)  },x^{1+\left(  J-\mu\right)  },\ldots,x^{\mu+\left(
J-\mu\right)  }\right\rangle _{A}\\=\left\langle x^{J-\mu},x^{J-\mu+1}%
,\ldots,x^{J}\right\rangle _{A}\\\subseteq\left\langle x^{0},x^{1}%
,\ldots,x^{\mu+\nu-1}\right\rangle _{A}\\\text{(since }\left\{  J-\mu
,J-\mu+1,\ldots,J\right\}  \subseteq\left\{  0,1,\ldots,\mu+\nu-1\right\}
\\\text{(since }J-\mu\geq0\text{ and }J\leq\mu+\nu-1\text{))}}}\\
&  \ \ \ \ \ \ \ \ \ \ +\underbrace{u^{I-m}\left\langle u^{0},u^{1}%
,\ldots,u^{m}\right\rangle _{A}}_{\substack{=\left\langle u^{I-m}u^{0}%
,u^{I-m}u^{1},\ldots,u^{I-m}u^{m}\right\rangle _{A}\\=\left\langle u^{\left(
I-m\right)  +0},u^{\left(  I-m\right)  +1},\ldots,u^{\left(  I-m\right)
+m}\right\rangle _{A}\\=\left\langle u^{I-m},u^{I-m+1},\ldots,u^{I}%
\right\rangle _{A}\\\subseteq\left\langle u^{0},u^{1},\ldots,u^{I}%
\right\rangle _{A}\\\text{(since }\left\{  I-m,I-m+1,\ldots,I\right\}
\subseteq\left\{  0,1,\ldots,I\right\}  \\\text{(since }I-m\geq0\text{))}%
}}\cdot\underbrace{\left\langle x^{0},x^{1},\ldots,x^{\mu-1}\right\rangle
_{A}x^{J-\mu}}_{\substack{=\left\langle x^{0}x^{J-\mu},x^{1}x^{J-\mu}%
,\ldots,x^{\mu-1}x^{J-\mu}\right\rangle _{A}\\=\left\langle x^{0+\left(
J-\mu\right)  },x^{1+\left(  J-\mu\right)  },\ldots,x^{\left(  \mu-1\right)
+\left(  J-\mu\right)  }\right\rangle _{A}\\=\left\langle x^{J-\mu}%
,x^{J-\mu+1},\ldots,x^{J-1}\right\rangle _{A}\\\subseteq\left\langle
x^{0},x^{1},\ldots,x^{J-1}\right\rangle _{A}\\\text{(since }\left\{
J-\mu,J-\mu+1,\ldots,J-1\right\}  \subseteq\left\{  0,1,\ldots,J-1\right\}
\\\text{(since }J-\mu\geq0\text{))}}}
\end{align*}%
\begin{align*}
&  \subseteq\underbrace{\left\langle u^{0},u^{1},\ldots,u^{I-1}\right\rangle
_{A}\cdot\left\langle x^{0},x^{1},\ldots,x^{\mu+\nu-1}\right\rangle _{A}%
}_{\substack{\subseteq U\\\text{(by (\ref{L18-indL2final-long}))}}}\\
&  \ \ \ \ \ \ \ \ \ \ +\underbrace{\left\langle u^{0},u^{1},\ldots
,u^{I}\right\rangle _{A}}_{=\left\langle u^{0},u^{1},\ldots,u^{I-1}%
\right\rangle _{A}+\left\langle u^{I}\right\rangle _{A}}\cdot\left\langle
x^{0},x^{1},\ldots,x^{J-1}\right\rangle _{A}\\
&  \subseteq U+\underbrace{\left(  \left\langle u^{0},u^{1},\ldots
,u^{I-1}\right\rangle _{A}+\left\langle u^{I}\right\rangle _{A}\right)
\cdot\left\langle x^{0},x^{1},\ldots,x^{J-1}\right\rangle _{A}}_{=\left\langle
u^{0},u^{1},\ldots,u^{I-1}\right\rangle _{A}\cdot\left\langle x^{0}%
,x^{1},\ldots,x^{J-1}\right\rangle _{A}+\left\langle u^{I}\right\rangle
_{A}\cdot\left\langle x^{0},x^{1},\ldots,x^{J-1}\right\rangle _{A}}\\
&  =U+\left\langle u^{0},u^{1},\ldots,u^{I-1}\right\rangle _{A}\cdot
\underbrace{\left\langle x^{0},x^{1},\ldots,x^{J-1}\right\rangle _{A}%
}_{\substack{\subseteq\left\langle x^{0},x^{1},\ldots,x^{\mu+\nu
-1}\right\rangle _{A}\\\text{(since }\left\{  0,1,\ldots,J-1\right\}
\subseteq\left\{  0,1,\ldots,\mu+\nu-1\right\}  \\\text{(since }J-1\leq
J\leq\mu+\nu-1\text{))}}}+\left\langle u^{I}\right\rangle _{A}\cdot
\left\langle x^{0},x^{1},\ldots,x^{J-1}\right\rangle _{A}\\
&  \subseteq U+\underbrace{\left\langle u^{0},u^{1},\ldots,u^{I-1}%
\right\rangle _{A}\cdot\left\langle x^{0},x^{1},\ldots,x^{\mu+\nu
-1}\right\rangle _{A}}_{\substack{\subseteq U\\\text{(by
(\ref{L18-indL2final-long}))}}}+\underbrace{\left\langle u^{I}\right\rangle
_{A}\cdot\left\langle x^{0},x^{1},\ldots,x^{J-1}\right\rangle _{A}%
}_{\substack{\subseteq U\\\text{(by (\ref{L18-indL1final-long}))}}}\\
&  \subseteq U+U+U\subseteq U\ \ \ \ \ \ \ \ \ \ \left(  \text{since }U\text{
is an }A\text{-module}\right)  .
\end{align*}
Thus, we have proved that $u^{I}x^{J}\in U$ holds in Case 1.

Let us next consider Case 2. In this case, we have $I<m$ and $J\geq\mu$. Thus,
$I\in\left\{  0,1,\ldots,m-1\right\}  $ (since $I<m$ and $I\in\mathbb{N}$) and
$J\in\left\{  \mu,\mu+1,\ldots,\mu+\nu-1\right\}  $ (since $J\geq\mu$ and
$J<\mu+\nu$). Thus,%
\begin{align*}
\left(  I,J\right)   &  \in\left\{  0,1,\ldots,m-1\right\}  \times\left\{
\mu,\mu+1,\ldots,\mu+\nu-1\right\} \\
&  \subseteq\left(  \left\{  0,1,\ldots,n-1\right\}  \times\left\{
0,1,\ldots,\mu-1\right\}  \right) \\
&  \ \ \ \ \ \ \ \ \ \ \cup\left(  \left\{  0,1,\ldots,m-1\right\}
\times\left\{  \mu,\mu+1,\ldots,\mu+\nu-1\right\}  \right) \\
&  =S\ \ \ \ \ \ \ \ \ \ \left(  \text{by (\ref{pf.Lemma18.S=})}\right)  .
\end{align*}
Hence, $u^{I}x^{J}\in U$ (by (\ref{L18-U-long}), applied to $I$ and $J$
instead of $i$ and $j$). Thus, we have proved that $u^{I}x^{J}\in U$ holds in
Case 2.

Let us next consider Case 3. In this case, we have $I\geq n$ and $J<\mu$.
Hence, $I-n\geq0$ (since $I\geq n$) and $J+\nu\leq\mu+\nu-1$ (since
$\underbrace{J}_{<\mu}+\nu<\mu+\nu$, and since $J+\nu$ and $\mu+\nu$ are
integers). Thus,%
\begin{align*}
\underbrace{u^{I}}_{\substack{=u^{I-n}u^{n}\\\text{(since }I\geq n\text{)}%
}}x^{J}  &  =u^{I-n}\underbrace{u^{n}}_{\substack{\in\left\langle u^{0}%
,u^{1},\ldots,u^{n-1}\right\rangle _{A}\cdot\left\langle x^{0},x^{1}%
,\ldots,x^{\nu}\right\rangle _{A}\\\left(  \text{by (\ref{L18-1})}\right)
}}x^{J}\\
&  \in\underbrace{u^{I-n}\left\langle u^{0},u^{1},\ldots,u^{n-1}\right\rangle
_{A}}_{\substack{=\left\langle u^{I-n}u^{0},u^{I-n}u^{1},\ldots,u^{I-n}%
u^{n-1}\right\rangle _{A}\\=\left\langle u^{\left(  I-n\right)  +0},u^{\left(
I-n\right)  +1},\ldots,u^{\left(  I-n\right)  +\left(  n-1\right)
}\right\rangle _{A}\\=\left\langle u^{I-n},u^{I-n+1},\ldots,u^{I-1}%
\right\rangle _{A}\\\subseteq\left\langle u^{0},u^{1},\ldots,u^{I-1}%
\right\rangle _{A}\\\text{(since }\left\{  I-n,I-n+1,\ldots,I-1\right\}
\subseteq\left\{  0,1,\ldots,I-1\right\}  \\\text{(since }I-n\geq0\text{))}%
}}\cdot\underbrace{\left\langle x^{0},x^{1},\ldots,x^{\nu}\right\rangle
_{A}x^{J}}_{\substack{=\left\langle x^{0}x^{J},x^{1}x^{J},\ldots,x^{\nu}%
x^{J}\right\rangle _{A}\\=\left\langle x^{0+J},x^{1+J},\ldots,x^{\nu
+J}\right\rangle _{A}\\=\left\langle x^{J},x^{J+1},\ldots,x^{J+\nu
}\right\rangle _{A}\\\subseteq\left\langle x^{0},x^{1},\ldots,x^{\mu+\nu
-1}\right\rangle _{A}\\\text{(since }\left\{  J,J+1,\ldots,J+\nu\right\}
\subseteq\left\{  0,1,\ldots,\mu+\nu-1\right\}  \\\text{(since }J\geq0\text{
and }J+\nu\leq\mu+\nu-1\text{))}}}\\
&  \subseteq\left\langle u^{0},u^{1},\ldots,u^{I-1}\right\rangle _{A}%
\cdot\left\langle x^{0},x^{1},\ldots,x^{\mu+\nu-1}\right\rangle _{A}\subseteq
U\ \ \ \ \ \ \ \ \ \ \left(  \text{by (\ref{L18-indL2final-long})}\right)  .
\end{align*}
Thus, we have proved that $u^{I}x^{J}\in U$ holds in Case 3.

Finally, let us consider Case 4. In this case, we have $I<n$ and $J<\mu$.
Thus, $I\in\left\{  0,1,\ldots,n-1\right\}  $ (since $I<n$ and $I\in
\mathbb{N}$) and $J\in\left\{  0,1,\ldots,\mu-1\right\}  $ (since $J<\mu$ and
$J\in\mathbb{N}$). Thus,%
\begin{align*}
\left(  I,J\right)   &  \in\left\{  0,1,\ldots,n-1\right\}  \times\left\{
0,1,\ldots,\mu-1\right\} \\
&  \subseteq\left(  \left\{  0,1,\ldots,n-1\right\}  \times\left\{
0,1,\ldots,\mu-1\right\}  \right) \\
&  \ \ \ \ \ \ \ \ \ \ \cup\left(  \left\{  0,1,\ldots,m-1\right\}
\times\left\{  \mu,\mu+1,\ldots,\mu+\nu-1\right\}  \right) \\
&  =S\ \ \ \ \ \ \ \ \ \ \left(  \text{by (\ref{pf.Lemma18.S=})}\right)  ,
\end{align*}
so that $u^{I}x^{J}\in U$ (by (\ref{L18-U-long}), applied to $I$ and $J$
instead of $i$ and $j$). Thus, we have proved that $u^{I}x^{J}\in U$ holds in
Case 4.

Therefore, we have proved that $u^{I}x^{J}\in U$ holds in each of the four
Cases 1, 2, 3 and 4. Hence, $u^{I}x^{J}\in U$ always holds.

Now, forget our assumption that $J<\mu+\nu$. Hence, we have proved that if
$J<\mu+\nu$, then $u^{I}x^{J}\in U$. In other words, we have proved the
assertion $\mathfrak{B}\left(  J\right)  $ (because $\mathfrak{B}\left(
J\right)  =\left(  \text{if }J<\mu+\nu\text{, then }u^{I}x^{J}\in U\right)  $).

Thus, we have proved (\ref{L18-indJ-long}).]

Hence, Proposition \ref{prop.strind2} yields that%
\[
\text{every }j\in\mathbb{N}\text{ satisfies }\mathfrak{B}\left(  j\right)
\text{.}%
\]
In other words,%
\[
\text{every }j\in\mathbb{N}\text{ satisfies }\left(  \text{if }j<\mu
+\nu\text{, then }u^{I}x^{j}\in U\right)
\]
(because of (\ref{pf.Lemma18.Bj=})). In other words, the assertion
$\mathfrak{A}\left(  I\right)  $ holds (because \newline$\mathfrak{A}\left(
I\right)  =\left(  \text{every }j\in\mathbb{N}\text{ satisfies }\left(
\text{if }j<\mu+\nu\text{, then }u^{I}x^{j}\in U\right)  \right)  $).

Thus, we have proved (\ref{L18-indI-long}).]

Hence, Proposition \ref{prop.strind1} yields that%
\[
\text{every }i\in\mathbb{N}\text{ satisfies }\mathfrak{A}\left(  i\right)
\text{.}%
\]
In other words,%
\[
\text{every }i\in\mathbb{N}\text{ satisfies }\left(  \text{every }%
j\in\mathbb{N}\text{ satisfies }\left(  \text{if }j<\mu+\nu\text{, then }%
u^{i}x^{j}\in U\right)  \right)
\]
(since $\mathfrak{A}\left(  i\right)  =\left(  \text{every }j\in
\mathbb{N}\text{ satisfies }\left(  \text{if }j<\mu+\nu\text{, then }%
u^{i}x^{j}\in U\right)  \right)  $). This is equivalent to
(\ref{L18-indU-long}). Thus, (\ref{L18-indU-long}) is proven.]

We have $0<\mu+\nu$ (since $\mu+\nu\in\mathbb{N}^{+}$). Thus, we can apply
(\ref{L18-indU-long}) to $i=0$ and $j=0$. As a result, we obtain $u^{0}%
x^{0}\in U$. In view of $\underbrace{u^{0}}_{=1}\underbrace{x^{0}}_{=1}=1$,
this rewrites as $1\in U$.

Furthermore, if $i\in\mathbb{N}$ and $j\in\mathbb{N}$ satisfy $j<\mu+\nu$,
then%
\[
\underbrace{u\cdot u^{i}}_{=u^{i+1}}x^{j}=u^{i+1}x^{j}\in U
\]
(by (\ref{L18-indU-long}) (applied to $i+1$ instead of $i$)). Hence,
\begin{equation}
u\cdot u^{i}x^{j}\in U\text{ for every }\left(  i,j\right)  \in S,
\label{L18-indUkonsequenz-long}%
\end{equation}
because every $\left(  i,j\right)  \in S$ satisfies $i\in\mathbb{N}$ and
$j\in\mathbb{N}$ and $j<\mu+\nu$ (by (\ref{L18-banal-long})). Hence,
$\sum\limits_{\left(  i,j\right)  \in S}a_{i,j}\underbrace{u\cdot u^{i}x^{j}%
}_{\substack{\in U\\\text{(by (\ref{L18-indUkonsequenz-long}))}}}\in U$ for
every $\left(  a_{i,j}\right)  _{\left(  i,j\right)  \in S}\in A^{S}$ (since
$U$ is an $A$-module and thus is closed under $A$-linear combination).

Now, from $U=\left\langle u^{i}x^{j}\ \mid\ \left(  i,j\right)  \in
S\right\rangle _{A}$, we obtain%
\begin{align*}
uU  &  =u\left\langle u^{i}x^{j}\ \mid\ \left(  i,j\right)  \in S\right\rangle
_{A}=\left\langle u\cdot u^{i}x^{j}\ \mid\ \left(  i,j\right)  \in
S\right\rangle _{A}\\
&  =\left\{  \sum\limits_{\left(  i,j\right)  \in S}a_{i,j}u\cdot u^{i}%
x^{j}\ \mid\ \left(  a_{i,j}\right)  _{\left(  i,j\right)  \in S}\in
A^{S}\right\}  \subseteq U
\end{align*}
(because $\sum\limits_{\left(  i,j\right)  \in S}a_{i,j}u\cdot u^{i}x^{j}\in
U$ for every $\left(  a_{i,j}\right)  _{\left(  i,j\right)  \in S}\in A^{S}$).

Altogether, $U$ is an $\left(  n\mu+m\nu\right)  $-generated $A$-submodule of
$B$ such that $1\in U$ and $uU\subseteq U$. Thus, $u\in B$ satisfies Assertion
$\mathcal{C}$ of Theorem~\ref{Theorem1} with $n$ replaced by $n\mu+m\nu$.
Hence, $u\in B$ satisfies the four equivalent assertions $\mathcal{A}$,
$\mathcal{B}$, $\mathcal{C}$ and $\mathcal{D}$ of Theorem~\ref{Theorem1} with
$n$ replaced by $n\mu+m\nu$. Consequently, $u$ is $\left(  n\mu+m\nu\right)
$-integral over $A$. This proves Lemma~\ref{Lemma18}.
\end{proof}
\end{verlong}

We record a weaker variant of Lemma~\ref{Lemma18}:

\begin{lemma}
\label{Lemma19} Let $A$ be a ring. Let $B$ be an $A$-algebra. Let $x\in B$ and
$y\in B$ be such that $xy\in A$. Let $m\in\mathbb{N}$ and $n\in\mathbb{N}$.
Let $u\in B$. Let $\mu\in\mathbb{N}$ and $\nu\in\mathbb{N}$ be such that
$\mu+\nu\in\mathbb{N}^{+}$. Assume that%
\begin{equation}
u^{n}\in\left\langle u^{0},u^{1},\ldots,u^{n-1}\right\rangle _{A}%
\cdot\left\langle x^{0},x^{1},\ldots,x^{\nu}\right\rangle _{A} \label{L19-1}%
\end{equation}
and that%
\begin{align}
u^{m}  &  \in\left\langle u^{0},u^{1},\ldots,u^{m-1}\right\rangle _{A}%
\cdot\left\langle y^{0},y^{1},\ldots,y^{\mu}\right\rangle _{A}\nonumber\\
&  \ \ \ \ \ \ \ \ \ \ +\left\langle u^{0},u^{1},\ldots,u^{m}\right\rangle
_{A}\cdot\left\langle y^{1},y^{2},\ldots,y^{\mu}\right\rangle _{A}.
\label{L19-2}%
\end{align}
Then, $u$ is $\left(  n\mu+m\nu\right)  $-integral over $A$.
\end{lemma}

\begin{vershort}
\begin{proof}
[Proof of Lemma~\ref{Lemma19}.](Again, this proof appears in greater detail in
\cite{verlong}.) We have%
\begin{equation}
\left\langle y^{0},y^{1},\ldots,y^{\mu}\right\rangle _{A}x^{\mu}%
\subseteq\left\langle x^{0},x^{1},\ldots,x^{\mu}\right\rangle _{A},
\label{L19-Pa-short}%
\end{equation}
since every $i\in\left\{  0,1,\ldots,\mu\right\}  $ satisfies%
\begin{align}
y^{i}\underbrace{x^{\mu}}_{=x^{\mu-i}x^{i}}  &  =y^{i}x^{\mu-i}x^{i}%
=\underbrace{x^{i}y^{i}}_{\substack{=\left(  xy\right)  ^{i}\in
A\\\text{(since }xy\in A\text{)}}}x^{\mu-i}\in Ax^{\mu-i}=\left\langle
x^{\mu-i}\right\rangle _{A}\label{L19-P1-short}\\
&  \subseteq\left\langle x^{0},x^{1},\ldots,x^{\mu}\right\rangle
_{A}\ \ \ \ \ \ \ \ \ \ \left(  \text{since }\mu-i\in\left\{  0,1,\ldots
,\mu\right\}  \right)  .\nonumber
\end{align}
Besides,%
\begin{equation}
\left\langle y^{1},y^{2},\ldots,y^{\mu}\right\rangle _{A}x^{\mu}%
\subseteq\left\langle x^{0},x^{1},\ldots,x^{\mu-1}\right\rangle _{A},
\label{L19-Pb-short}%
\end{equation}
since every $i\in\left\{  1,2,\ldots,\mu\right\}  $ satisfies%
\begin{align*}
y^{i}x^{\mu}  &  \in\left\langle x^{\mu-i}\right\rangle _{A}%
\ \ \ \ \ \ \ \ \ \ \left(  \text{by (\ref{L19-P1-short})}\right) \\
&  \subseteq\left\langle x^{0},x^{1},\ldots,x^{\mu-1}\right\rangle
_{A}\ \ \ \ \ \ \ \ \ \ \left(  \text{since }\mu-i\in\left\{  0,1,\ldots
,\mu-1\right\}  \right)  .
\end{align*}

Now, (\ref{L19-2}) yields%
\begin{align*}
u^{m}x^{\mu}  &  \in\left(  \left\langle u^{0},u^{1},\ldots,u^{m-1}%
\right\rangle _{A}\cdot\left\langle y^{0},y^{1},\ldots,y^{\mu}\right\rangle
_{A}\right. \\
&  \ \ \ \ \ \ \ \ \ \ \left.  +\left\langle u^{0},u^{1},\ldots,u^{m}%
\right\rangle _{A}\cdot\left\langle y^{1},y^{2},\ldots,y^{\mu}\right\rangle
_{A}\right)  x^{\mu}\\
&  =\left\langle u^{0},u^{1},\ldots,u^{m-1}\right\rangle _{A}\cdot
\underbrace{\left\langle y^{0},y^{1},\ldots,y^{\mu}\right\rangle _{A}x^{\mu}%
}_{\substack{\subseteq\left\langle x^{0},x^{1},\ldots,x^{\mu}\right\rangle
_{A}\\\left(  \text{by (\ref{L19-Pa-short})}\right)  }}\\
&  \ \ \ \ \ \ \ \ \ \ +\left\langle u^{0},u^{1},\ldots,u^{m}\right\rangle
_{A}\cdot\underbrace{\left\langle y^{1},y^{2},\ldots,y^{\mu}\right\rangle
_{A}x^{\mu}}_{\substack{\subseteq\left\langle x^{0},x^{1},\ldots,x^{\mu
-1}\right\rangle _{A}\\\left(  \text{by (\ref{L19-Pb-short})}\right)  }}\\
&  \subseteq\left\langle u^{0},u^{1},\ldots,u^{m-1}\right\rangle _{A}%
\cdot\left\langle x^{0},x^{1},\ldots,x^{\mu}\right\rangle _{A}\\
&  \ \ \ \ \ \ \ \ \ \ +\left\langle u^{0},u^{1},\ldots,u^{m}\right\rangle
_{A}\cdot\left\langle x^{0},x^{1},\ldots,x^{\mu-1}\right\rangle _{A}.
\end{align*}
In other words, (\ref{L18-2}) holds. Also, (\ref{L18-1}) holds (because
(\ref{L19-1}) holds, and because (\ref{L18-1}) is the same as (\ref{L19-1})).
Thus, Lemma~\ref{Lemma18} yields that $u$ is $\left(  n\mu+m\nu\right)
$-integral over $A$. This proves Lemma~\ref{Lemma19}.
\end{proof}
\end{vershort}

\begin{verlong}
\begin{proof}
[Proof of Lemma~\ref{Lemma19}.]Fix $p\in\mathbb{N}$.

Let $i\in\left\{  p,p+1,\ldots,\mu\right\}  $. Thus, $i\geq p$ and $i\leq\mu$.
From $i\in\left\{  p,p+1,\ldots,\mu\right\}  $, we obtain $\mu-i\in\left\{
0,1,\ldots,\mu-p\right\}  $, so that $\left\{  \mu-i\right\}  \subseteq
\left\{  0,1,\ldots,\mu-p\right\}  $. Also, $i\leq\mu$, thus $\mu-i\geq0$, so
that%
\begin{align}
y^{i}\underbrace{x^{\mu}}_{=x^{\mu-i}x^{i}}  &  =y^{i}x^{\mu-i}x^{i}%
=\underbrace{x^{i}y^{i}}_{\substack{=\left(  xy\right)  ^{i}\in
A\\\text{(since }xy\in A\text{)}}}x^{\mu-i}\in Ax^{\mu-i}=\left\langle
x^{\mu-i}\right\rangle _{A}\label{L19-P1}\\
&  \subseteq\left\langle x^{0},x^{1},\ldots,x^{\mu-p}\right\rangle _{A}
\label{L19-P2}%
\end{align}
(since $\left\{  \mu-i\right\}  \subseteq\left\{  0,1,\ldots,\mu-p\right\}  $).

Now, forget that we fixed $i$. We thus have proven (\ref{L19-P2}) for each
$i\in\left\{  p,p+1,\ldots,\mu\right\}  $. Hence, every $\left(  a_{i}\right)
_{i\in\left\{  p,p+1,\ldots,\mu\right\}  }\in A^{\left\{  p,p+1,\ldots
,\mu\right\}  }$ satisfies%
\[
\sum\limits_{i\in\left\{  p,p+1,\ldots,\mu\right\}  }a_{i}\underbrace{y^{i}%
x^{\mu}}_{\substack{\in\left\langle x^{0},x^{1},\ldots,x^{\mu-p}\right\rangle
_{A}\\\text{(by (\ref{L19-P2}))}}}\in\sum\limits_{i\in\left\{  p,p+1,\ldots
,\mu\right\}  }a_{i}\left\langle x^{0},x^{1},\ldots,x^{\mu-p}\right\rangle
_{A}\subseteq\left\langle x^{0},x^{1},\ldots,x^{\mu-p}\right\rangle _{A}%
\]
(because $\left\langle x^{0},x^{1},\ldots,x^{\mu-p}\right\rangle _{A}$ is an
$A$-module). In other words,%
\begin{equation}
\left\{  \sum\limits_{i\in\left\{  p,p+1,\ldots,\mu\right\}  }a_{i}y^{i}%
x^{\mu}\ \mid\ \left(  a_{i}\right)  _{i\in\left\{  p,p+1,\ldots,\mu\right\}
}\in A^{\left\{  p,p+1,\ldots,\mu\right\}  }\right\}  \subseteq\left\langle
x^{0},x^{1},\ldots,x^{\mu-p}\right\rangle _{A}. \label{L19-P2c}%
\end{equation}
Now,%
\begin{align}
\underbrace{\left\langle y^{p},y^{p+1},\ldots,y^{\mu}\right\rangle _{A}%
}_{=\left\langle y^{i}\ \mid\ i\in\left\{  p,p+1,\ldots,\mu\right\}
\right\rangle _{A}}x^{\mu}  &  =\left\langle y^{i}\ \mid\ i\in\left\{
p,p+1,\ldots,\mu\right\}  \right\rangle _{A}x^{\mu}\nonumber\\
&  =\left\langle y^{i}x^{\mu}\ \mid\ i\in\left\{  p,p+1,\ldots,\mu\right\}
\right\rangle _{A}\nonumber\\
&  =\left\{  \sum\limits_{i\in\left\{  p,p+1,\ldots,\mu\right\}  }a_{i}%
y^{i}x^{\mu}\ \mid\ \left(  a_{i}\right)  _{i\in\left\{  p,p+1,\ldots
,\mu\right\}  }\in A^{\left\{  p,p+1,\ldots,\mu\right\}  }\right\} \nonumber\\
&  \subseteq\left\langle x^{0},x^{1},\ldots,x^{\mu-p}\right\rangle _{A}
\label{L19-Pa}%
\end{align}
(by (\ref{L19-P2c})).

Forget that we fixed $p$. We thus have proven (\ref{L19-Pa}) for each
$p\in\mathbb{N}$. Applying (\ref{L19-Pa}) to $p=0$, we find%
\begin{equation}
\left\langle y^{0},y^{1},\ldots,y^{\mu}\right\rangle _{A}x^{\mu}%
\subseteq\left\langle x^{0},x^{1},\ldots,x^{\mu-0}\right\rangle _{A}%
=\left\langle x^{0},x^{1},\ldots,x^{\mu}\right\rangle _{A} \label{L19-Pa0}%
\end{equation}
(since $\mu-0=\mu$). Applying (\ref{L19-Pa}) to $p=1$, we find%
\begin{equation}
\left\langle y^{1},y^{2},\ldots,y^{\mu}\right\rangle _{A}x^{\mu}%
\subseteq\left\langle x^{0},x^{1},\ldots,x^{\mu-1}\right\rangle _{A}.
\label{L19-Pa1}%
\end{equation}

Now, (\ref{L19-2}) yields%
\begin{align*}
&  u^{m}x^{\mu}\\
&  \in\left(  \left\langle u^{0},u^{1},\ldots,u^{m-1}\right\rangle _{A}%
\cdot\left\langle y^{0},y^{1},\ldots,y^{\mu}\right\rangle _{A}+\left\langle
u^{0},u^{1},\ldots,u^{m}\right\rangle _{A}\cdot\left\langle y^{1},y^{2}%
,\ldots,y^{\mu}\right\rangle _{A}\right)  x^{\mu}\\
&  =\left\langle u^{0},u^{1},\ldots,u^{m-1}\right\rangle _{A}\cdot
\underbrace{\left\langle y^{0},y^{1},\ldots,y^{\mu}\right\rangle _{A}x^{\mu}%
}_{\substack{\subseteq\left\langle x^{0},x^{1},\ldots,x^{\mu}\right\rangle
_{A}\\\left(  \text{by (\ref{L19-Pa0})}\right)  }}+\left\langle u^{0}%
,u^{1},\ldots,u^{m}\right\rangle _{A}\cdot\underbrace{\left\langle y^{1}%
,y^{2},\ldots,y^{\mu}\right\rangle _{A}x^{\mu}}_{\substack{\subseteq
\left\langle x^{0},x^{1},\ldots,x^{\mu-1}\right\rangle _{A}\\\left(  \text{by
(\ref{L19-Pa1})}\right)  }}\\
&  \subseteq\left\langle u^{0},u^{1},\ldots,u^{m-1}\right\rangle _{A}%
\cdot\left\langle x^{0},x^{1},\ldots,x^{\mu}\right\rangle _{A}+\left\langle
u^{0},u^{1},\ldots,u^{m}\right\rangle _{A}\cdot\left\langle x^{0},x^{1}%
,\ldots,x^{\mu-1}\right\rangle _{A}.
\end{align*}
In other words, (\ref{L18-2}) holds. Also, (\ref{L18-1}) holds (because
(\ref{L19-1}) holds, and because (\ref{L18-1}) is the same as (\ref{L19-1})).
Thus, Lemma~\ref{Lemma18} yields that $u$ is $\left(  n\mu+m\nu\right)
$-integral over $A$. This proves Lemma~\ref{Lemma19}.
\end{proof}
\end{verlong}

We now come to something trivial:

\begin{lemma}
\label{Lemma20} Let $A$ be a ring. Let $B$ be an $A$-algebra. Let $x\in B$.
Let $n\in\mathbb{N}$. Let $u\in B$. Assume that $u$ is $n$-integral over
$A\left[  x\right]  $. Then, there exists some $\nu\in\mathbb{N}^{+}$ such
that%
\[
u^{n}\in\left\langle u^{0},u^{1},\ldots,u^{n-1}\right\rangle _{A}%
\cdot\left\langle x^{0},x^{1},\ldots,x^{\nu}\right\rangle _{A}.
\]

\end{lemma}

\begin{vershort}
\begin{proof}
[Proof of Lemma~\ref{Lemma20}.]Again, see \cite{verlong} for more details on
this argument; here we only show a quick sketch: Since $u$ is $n$-integral
over $A\left[  x\right]  $, there exists a monic polynomial $P\in\left(
A\left[  x\right]  \right)  \left[  X\right]  $ with $\deg P=n$ and $P\left(
u\right)  =0$. Denoting the coefficients of this polynomial $P$ by $\alpha
_{0},\alpha_{1},\ldots,\alpha_{n}$ (where $\alpha_{n}=1$), we can rewrite the
equality $P\left(  u\right)  =0$ as $u^{n}=-\sum\limits_{i=0}^{n-1}\alpha
_{i}u^{i}$. Note that $\alpha_{i}\in A\left[  x\right]  $ for all $i$. Now,
there exists some $\nu\in\mathbb{N}^{+}$ such that $\alpha_{i}\in\left\langle
x^{0},x^{1},\ldots,x^{\nu}\right\rangle _{A}$ for every $i\in\left\{
0,1,\ldots,n-1\right\}  $ (because for each $i\in\left\{  0,1,\ldots
,n-1\right\}  $, we have $\alpha_{i}\in A\left[  x\right]  =\bigcup
\limits_{\nu=0}^{\infty}\left\langle x^{0},x^{1},\ldots,x^{\nu}\right\rangle
_{A}$, so that $\alpha_{i}\in\left\langle x^{0},x^{1},\ldots,x^{\nu_{i}%
}\right\rangle _{A}$ for some $\nu_{i}\in\mathbb{N}$; now take $\nu
=\max\left\{  \nu_{0},\nu_{1},\ldots,\nu_{n-1},1\right\}  $). This $\nu$ then
satisfies%
\begin{align*}
u^{n}  &  =-\sum\limits_{i=0}^{n-1}\alpha_{i}u^{i}=-\sum\limits_{i=0}%
^{n-1}\underbrace{u^{i}}_{\in\left\langle u^{0},u^{1},\ldots,u^{n-1}%
\right\rangle _{A}}\underbrace{\alpha_{i}}_{\in\left\langle x^{0},x^{1}%
,\ldots,x^{\nu}\right\rangle _{A}}\\
&  \in\left\langle u^{0},u^{1},\ldots,u^{n-1}\right\rangle _{A}\cdot
\left\langle x^{0},x^{1},\ldots,x^{\nu}\right\rangle _{A},
\end{align*}
and Lemma~\ref{Lemma20} is proven.
\end{proof}
\end{vershort}

\begin{verlong}
\begin{proof}
[Proof of Lemma~\ref{Lemma20}.]There exists a monic polynomial $P\in\left(
A\left[  x\right]  \right)  \left[  X\right]  $ with $\deg P=n$ and $P\left(
u\right)  =0$ (since $u$ is $n$-integral over $A\left[  x\right]  $). Consider
this $P$. Since $P\in\left(  A\left[  x\right]  \right)  \left[  X\right]  $
is a monic polynomial with $\deg P=n$, there exist elements $\alpha_{0}%
,\alpha_{1},\ldots,\alpha_{n-1}$ of $A\left[  x\right]  $ such that $P\left(
X\right)  =X^{n}+\sum\limits_{i=0}^{n-1}\alpha_{i}X^{i}$. Consider these
$\alpha_{0},\alpha_{1},\ldots,\alpha_{n-1}$. Substituting $u$ for $X$ in the
equality $P\left(  X\right)  =X^{n}+\sum\limits_{i=0}^{n-1}\alpha_{i}X^{i}$,
we find $P\left(  u\right)  =u^{n}+\sum\limits_{i=0}^{n-1}\alpha_{i}u^{i}$.
Comparing this with $P\left(  u\right)  =0$, we obtain $u^{n}+\sum
\limits_{i=0}^{n-1}\alpha_{i}u^{i}=0$. Hence, $u^{n}=-\sum\limits_{i=0}%
^{n-1}\alpha_{i}u^{i}$.

For every $i\in\left\{  0,1,\ldots,n-1\right\}  $, we have $\alpha_{i}\in
A\left[  x\right]  $, and thus there exist some $\nu_{i}\in\mathbb{N}$ and
some $\left(  \beta_{i,0},\beta_{i,1},\ldots,\beta_{i,\nu_{i}}\right)  \in
A^{\nu_{i}+1}$ such that $\alpha_{i}=\sum\limits_{k=0}^{\nu_{i}}\beta
_{i,k}x^{k}$. Consider these $\nu_{i}$ and $\left(  \beta_{i,0},\beta
_{i,1},\ldots,\beta_{i,\nu_{i}}\right)  $. Hence, for every $i\in\left\{
0,1,\ldots,n-1\right\}  $, we have
\begin{equation}
\alpha_{i}=\sum\limits_{k=0}^{\nu_{i}}\beta_{i,k}x^{k}\in\left\langle
x^{0},x^{1},\ldots,x^{\nu_{i}}\right\rangle _{A}. \label{L20.pf.4}%
\end{equation}

Let $\nu=\max\left\{  \nu_{0},\nu_{1},\ldots,\nu_{n-1},1\right\}  $. Thus,
$\nu$ is an integer satisfying $\nu\geq1$ (since $1\in\left\{  \nu_{0},\nu
_{1},\ldots,\nu_{n-1},1\right\}  $); hence, $\nu\in\mathbb{N}^{+}$.
Furthermore, for every $i\in\left\{  0,1,\ldots,n-1\right\}  $, we have
$\nu_{i}\in\left\{  \nu_{0},\nu_{1},\ldots,\nu_{n-1}\right\}  \subseteq
\left\{  \nu_{0},\nu_{1},\ldots,\nu_{n-1},1\right\}  $ and thus $\nu_{i}%
\leq\max\left\{  \nu_{0},\nu_{1},\ldots,\nu_{n-1},1\right\}  =\nu$, hence
$\left\{  0,1,\ldots,\nu_{i}\right\}  \subseteq\left\{  0,1,\ldots
,\nu\right\}  $, and thus
\begin{align}
\alpha_{i}  &  \in\left\langle x^{0},x^{1},\ldots,x^{\nu_{i}}\right\rangle
_{A}\ \ \ \ \ \ \ \ \ \ \left(  \text{by (\ref{L20.pf.4})}\right) \nonumber\\
&  \subseteq\left\langle x^{0},x^{1},\ldots,x^{\nu}\right\rangle _{A}
\label{L20.pf.5}%
\end{align}
(since $\left\{  0,1,\ldots,\nu_{i}\right\}  \subseteq\left\{  0,1,\ldots
,\nu\right\}  $). Therefore,%
\begin{align*}
u^{n}  &  =-\sum\limits_{i=0}^{n-1}\alpha_{i}u^{i}=-\sum\limits_{i=0}%
^{n-1}\underbrace{u^{i}}_{\substack{\in\left\langle u^{0},u^{1},\ldots
,u^{n-1}\right\rangle _{A}\\\text{(since }i\in\left\{  0,1,\ldots,n-1\right\}
\text{)}}}\underbrace{\alpha_{i}}_{\substack{\in\left\langle x^{0}%
,x^{1},\ldots,x^{\nu}\right\rangle _{A}\\\text{(by \eqref{L20.pf.5})}}}\\
&  \in-\sum\limits_{i=0}^{n-1}\left\langle u^{0},u^{1},\ldots,u^{n-1}%
\right\rangle _{A}\cdot\left\langle x^{0},x^{1},\ldots,x^{\nu}\right\rangle
_{A}\\
&  \subseteq\left\langle u^{0},u^{1},\ldots,u^{n-1}\right\rangle _{A}%
\cdot\left\langle x^{0},x^{1},\ldots,x^{\nu}\right\rangle _{A}%
\end{align*}
(since $\left\langle u^{0},u^{1},\ldots,u^{n-1}\right\rangle _{A}%
\cdot\left\langle x^{0},x^{1},\ldots,x^{\nu}\right\rangle _{A}$ is an
$A$-module). This proves Lemma~\ref{Lemma20}.
\end{proof}
\end{verlong}

\subsection{Integrality over $A\left[  x\right]  $ and over $A\left[
y\right]  $ implies integrality over $A\left[  xy\right]  $}

A consequence of Lemma~\ref{Lemma19} and Lemma~\ref{Lemma20} is the following theorem:

\begin{theorem}
\label{Theorem21} Let $A$ be a ring. Let $B$ be an $A$-algebra. Let $x\in B$
and $y\in B$ be such that $xy\in A$. Let $m\in\mathbb{N}$ and $n\in\mathbb{N}%
$. Let $u\in B$. Assume that $u$ is $n$-integral over $A\left[  x\right]  $,
and that $u$ is $m$-integral over $A\left[  y\right]  $. Then, there exists
some $\lambda\in\mathbb{N}$ such that $u$ is $\lambda$-integral over $A$.
\end{theorem}

\begin{proof}
[Proof of Theorem~\ref{Theorem21}.]Since $u$ is $n$-integral over $A\left[
x\right]  $, Lemma~\ref{Lemma20} yields that there exists some $\nu
\in\mathbb{N}^{+}$ such that%
\[
u^{n}\in\left\langle u^{0},u^{1},\ldots,u^{n-1}\right\rangle _{A}%
\cdot\left\langle x^{0},x^{1},\ldots,x^{\nu}\right\rangle _{A}.
\]
In other words, there exists some $\nu\in\mathbb{N}^{+}$ such that
(\ref{L19-1}) holds. Consider this $\nu$.

\begin{vershort}
Since $u$ is $m$-integral over $A\left[  y\right]  $, Lemma~\ref{Lemma20}
(with $x$, $n$ and $\nu$ replaced by $y$, $m$ and $\mu$) yields that there
exists some $\mu\in\mathbb{N}^{+}$ such that%
\begin{equation}
u^{m}\in\left\langle u^{0},u^{1},\ldots,u^{m-1}\right\rangle _{A}%
\cdot\left\langle y^{0},y^{1},\ldots,y^{\mu}\right\rangle _{A}. \label{T21-P1}%
\end{equation}
Consider this $\mu$. Hence, (\ref{L19-2}) holds as well (because
(\ref{T21-P1}) is even stronger than (\ref{L19-2})).
\end{vershort}

\begin{verlong}
Since $u$ is $m$-integral over $A\left[  y\right]  $, Lemma~\ref{Lemma20}
(with $x$, $n$ and $\nu$ replaced by $y$, $m$ and $\mu$) yields that there
exists some $\mu\in\mathbb{N}^{+}$ such that%
\[
u^{m}\in\left\langle u^{0},u^{1},\ldots,u^{m-1}\right\rangle _{A}%
\cdot\left\langle y^{0},y^{1},\ldots,y^{\mu}\right\rangle _{A}.
\]
Consider this $\mu$. Hence,%
\begin{align*}
u^{m}  &  \in\left\langle u^{0},u^{1},\ldots,u^{m-1}\right\rangle _{A}%
\cdot\left\langle y^{0},y^{1},\ldots,y^{\mu}\right\rangle _{A}\\
&  \subseteq\left\langle u^{0},u^{1},\ldots,u^{m-1}\right\rangle _{A}%
\cdot\left\langle y^{0},y^{1},\ldots,y^{\mu}\right\rangle _{A}+\left\langle
u^{0},u^{1},\ldots,u^{m}\right\rangle _{A}\cdot\left\langle y^{1},y^{2}%
,\ldots,y^{\mu}\right\rangle _{A}.
\end{align*}
In other words, (\ref{L19-2}) holds. From $\mu\in\mathbb{N}^{+}$ and $\nu
\in\mathbb{N}^{+}$, we obtain $\mu+\nu\in\mathbb{N}^{+}$.
\end{verlong}

Since both (\ref{L19-1}) and (\ref{L19-2}) hold, Lemma~\ref{Lemma19} yields
that $u$ is $\left(  n\mu+m\nu\right)  $-integral over $A$. Thus, there exists
some $\lambda\in\mathbb{N}$ such that $u$ is $\lambda$-integral over $A$
(namely, $\lambda=n\mu+m\nu$). This proves Theorem~\ref{Theorem21}.
\end{proof}

We record a generalization of Theorem~\ref{Theorem21} (which will turn out to
be easily seen equivalent to Theorem~\ref{Theorem21}):

\begin{theorem}
\label{Theorem22} Let $A$ be a ring. Let $B$ be an $A$-algebra. Let $x\in B$
and $y\in B$. Let $m\in\mathbb{N}$ and $n\in\mathbb{N}$. Let $u\in B$. Assume
that $u$ is $n$-integral over $A\left[  x\right]  $, and that $u$ is
$m$-integral over $A\left[  y\right]  $. Then, there exists some $\lambda
\in\mathbb{N}$ such that $u$ is $\lambda$-integral over $A\left[  xy\right]  $.
\end{theorem}

\begin{proof}
[Proof of Theorem~\ref{Theorem22}.]Let $C$ denote the $A$-subalgebra $A\left[
xy\right]  $ of $A$. Thus, $C=A\left[  xy\right]  $ is an $A$-subalgebra of
$B$, hence a subring of $B$. Thus, $C\left[  x\right]  $ is a $C$-subalgebra
of $B$, hence a subring of $B$. Note that $C=A\left[  xy\right]  =A\left[
yx\right]  $ (since $xy=yx$).

Furthermore, $A\left[  x\right]  $ is a subring of $C\left[  x\right]
$\ \ \ \ \footnote{\textit{Proof.} Both $A\left[  x\right]  $ and $C\left[
x\right]  $ are subrings of $B$.
\par
Now, let $\gamma\in A\left[  x\right]  $. Thus, there exist some
$p\in\mathbb{N}$ and some elements $a_{0},a_{1},\ldots,a_{p}$ of $A$ such that
$\gamma=\sum_{i=0}^{p}a_{i}x^{i}$. Consider this $p$ and these $a_{0}%
,a_{1},\ldots,a_{p}$. For each $i\in\left\{  0,1,\ldots,p\right\}  $, we have
$\underbrace{a_{i}}_{\in A}\cdot1_{B}\in A\cdot1_{B}\subseteq A\left[
xy\right]  =C$ (since $C=A\left[  xy\right]  $). Hence, $\sum_{i=0}^{p}\left(
a_{i}\cdot1_{B}\right)  x^{i}\in C\left[  x\right]  $. In view of
\[
\sum_{i=0}^{p}\left(  a_{i}\cdot1_{B}\right)  x^{i}=\sum_{i=0}^{p}a_{i}%
\cdot\underbrace{1_{B}x^{i}}_{=x^{i}}=\sum_{i=0}^{p}a_{i}x^{i}=\gamma
\ \ \ \ \ \ \ \ \ \ \left(  \text{since }\gamma=\sum_{i=0}^{p}a_{i}%
x^{i}\right)  ,
\]
this rewrites as $\gamma\in C\left[  x\right]  $.
\par
Forget that we fixed $\gamma$. We thus have shown that $\gamma\in C\left[
x\right]  $ for each $\gamma\in A\left[  x\right]  $. In other words,
$A\left[  x\right]  \subseteq C\left[  x\right]  $. Hence, $A\left[  x\right]
$ is a subring of $C\left[  x\right]  $ (since both $A\left[  x\right]  $ and
$C\left[  x\right]  $ are subrings of $B$).}. Thus, $C\left[  x\right]  $ is
an $A\left[  x\right]  $-algebra. Also, $B$ is a $C\left[  x\right]  $-algebra
(since $C\left[  x\right]  $ is a subring of $B$). Since $u$ is $n$-integral
over $A\left[  x\right]  $, Lemma \ref{lem.I} (applied to $B$, $C\left[
x\right]  $, $A\left[  x\right]  $ and $u$ instead of $B^{\prime}$,
$A^{\prime}$, $A$ and $v$) yields that $u$ is $n$-integral over $C\left[
x\right]  $. The same argument (but applied to $y$, $x$, $n$ and $m$ instead
of $x$, $y$, $m$ and $n$) shows that $u$ is $m$-integral over $C\left[
y\right]  $ (since $C=A\left[  yx\right]  $).

Now, $B$ is a $C$-algebra (since $C$ is a subring of $B$) and we have $xy\in
A\left[  xy\right]  =C$. Hence, Theorem~\ref{Theorem21} (applied to $C$
instead of $A$) yields that there exists some $\lambda\in\mathbb{N}$ such that
$u$ is $\lambda$-integral over $C$ (because $u$ is $n$-integral over $C\left[
x\right]  $, and because $u$ is $m$-integral over $C\left[  y\right]  $). In
other words, there exists some $\lambda\in\mathbb{N}$ such that $u$ is
$\lambda$-integral over $A\left[  xy\right]  $ (since $C=A\left[  xy\right]
$). This proves Theorem~\ref{Theorem22}.
\end{proof}

\subsection{Generalization to ideal semifiltrations}

Theorem~\ref{Theorem22} has a ``relative version'':

\begin{theorem}
\label{Theorem23} Let $A$ be a ring. Let $B$ be an $A$-algebra. Let $\left(
I_{\rho}\right)  _{\rho\in\mathbb{N}}$ be an ideal semifiltration of $A$. Let
$x\in B$ and $y\in B$.

\textbf{(a)} Then, $\left(  I_{\rho}A\left[  x\right]  \right)  _{\rho
\in\mathbb{N}}$ is an ideal semifiltration of $A\left[  x\right]  $. Besides,
$\left(  I_{\rho}A\left[  y\right]  \right)  _{\rho\in\mathbb{N}}$ is an ideal
semifiltration of $A\left[  y\right]  $. Besides, $\left(  I_{\rho}A\left[
xy\right]  \right)  _{\rho\in\mathbb{N}}$ is an ideal semifiltration of
$A\left[  xy\right]  $.

\textbf{(b)} Let $m\in\mathbb{N}$ and $n\in\mathbb{N}$. Let $u\in B$. Assume
that $u$ is $n$-integral over $\left(  A\left[  x\right]  ,\left(  I_{\rho
}A\left[  x\right]  \right)  _{\rho\in\mathbb{N}}\right)  $, and that $u$ is
$m$-integral over $\left(  A\left[  y\right]  ,\left(  I_{\rho}A\left[
y\right]  \right)  _{\rho\in\mathbb{N}}\right)  $. Then, there exists some
$\lambda\in\mathbb{N}$ such that $u$ is $\lambda$-integral over $\left(
A\left[  xy\right]  ,\left(  I_{\rho}A\left[  xy\right]  \right)  _{\rho
\in\mathbb{N}}\right)  $.
\end{theorem}

Our proof of this theorem will rely on a lemma:

\begin{lemma}
\label{lem.N}Let $A$ be a ring. Let $B$ be an $A$-algebra. Let $v\in B$. Let
$\left(  I_{\rho}\right)  _{\rho\in\mathbb{N}}$ be an ideal semifiltration of
$A$. Lemma \ref{lem.J} (applied to $A^{\prime}=A\left[  v\right]  $) yields
that $\left(  I_{\rho}A\left[  v\right]  \right)  _{\rho\in\mathbb{N}}$ is an
ideal semifiltration of $A\left[  v\right]  $. Consider the polynomial ring
$A\left[  Y\right]  $ and its $A$-subalgebra $A\left[  \left(  I_{\rho
}\right)  _{\rho\in\mathbb{N}}\ast Y\right]  $. We know that $A\left[  \left(
I_{\rho}\right)  _{\rho\in\mathbb{N}}\ast Y\right]  $ is a subring of
$A\left[  Y\right]  $, and (as explained in Definition~\ref{Definition7}) the
polynomial ring $\left(  A\left[  v\right]  \right)  \left[  Y\right]  $ is an
$A\left[  Y\right]  $-algebra (since $A\left[  v\right]  $ is an $A$-algebra).
Hence, $\left(  A\left[  v\right]  \right)  \left[  Y\right]  $ is an
$A\left[  \left(  I_{\rho}\right)  _{\rho\in\mathbb{N}}\ast Y\right]
$-algebra (since $A\left[  \left(  I_{\rho}\right)  _{\rho\in\mathbb{N}}\ast
Y\right]  $ is a subring of $A\left[  Y\right]  $). On the other hand,
$\left(  A\left[  v\right]  \right)  \left[  \left(  I_{\rho}A\left[
v\right]  \right)  _{\rho\in\mathbb{N}}\ast Y\right]  \subseteq\left(
A\left[  v\right]  \right)  \left[  Y\right]  $.

\textbf{(a)} We have%
\begin{equation}
\left(  A\left[  v\right]  \right)  \left[  \left(  I_{\rho}A\left[  v\right]
\right)  _{\rho\in\mathbb{N}}\ast Y\right]  =\left(  A\left[  \left(  I_{\rho
}\right)  _{\rho\in\mathbb{N}}\ast Y\right]  \right)  \left[  v\right]  .
\label{eq.lem.N.a}%
\end{equation}

\textbf{(b)} Let $u\in B$. Let $n\in\mathbb{N}$. Then, the element $u$ of $B$
is $n$-integral over $\left(  A\left[  v\right]  ,\left(  I_{\rho}A\left[
v\right]  \right)  _{\rho\in\mathbb{N}}\right)  $ if and only if the element
$uY$ of the polynomial ring $B\left[  Y\right]  $ is $n$-integral over the
ring $\left(  A\left[  \left(  I_{\rho}\right)  _{\rho\in\mathbb{N}}\ast
Y\right]  \right)  \left[  v\right]  $.
\end{lemma}

\begin{proof}
[Proof of Lemma \ref{lem.N}.]\textbf{(a)} We have proven Lemma \ref{lem.N}
\textbf{(a)} during the proof of Theorem~\ref{Theorem9} \textbf{(b)}.

\textbf{(b)} The ring $B$ is an $A\left[  v\right]  $-algebra (since $A\left[
v\right]  $ is a subring of $B$). Hence, Theorem~\ref{Theorem7} (applied to
$A\left[  v\right]  $ and $\left(  I_{\rho}A\left[  v\right]  \right)
_{\rho\in\mathbb{N}}$ instead of $A$ and $\left(  I_{\rho}\right)  _{\rho
\in\mathbb{N}}$) yields that the element $u$ of $B$ is $n$-integral over
$\left(  A\left[  v\right]  ,\left(  I_{\rho}A\left[  v\right]  \right)
_{\rho\in\mathbb{N}}\right)  $ if and only if the element $uY$ of the
polynomial ring $B\left[  Y\right]  $ is $n$-integral over the ring $\left(
A\left[  v\right]  \right)  \left[  \left(  I_{\rho}A\left[  v\right]
\right)  _{\rho\in\mathbb{N}}\ast Y\right]  $. In view of (\ref{eq.lem.N.a}),
this rewrites as follows: The element $u$ of $B$ is $n$-integral over $\left(
A\left[  v\right]  ,\left(  I_{\rho}A\left[  v\right]  \right)  _{\rho
\in\mathbb{N}}\right)  $ if and only if the element $uY$ of the polynomial
ring $B\left[  Y\right]  $ is $n$-integral over the ring $\left(  A\left[
\left(  I_{\rho}\right)  _{\rho\in\mathbb{N}}\ast Y\right]  \right)  \left[
v\right]  $. This proves Lemma \ref{lem.N} \textbf{(b)}.
\end{proof}

\begin{proof}
[Proof of Theorem~\ref{Theorem23}.]\textbf{(a)} Since $\left(  I_{\rho
}\right)  _{\rho\in\mathbb{N}}$ is an ideal semifiltration of $A$,
Lemma~\ref{lem.J} (applied to $A^{\prime}=A\left[  x\right]  $) yields that
$\left(  I_{\rho}A\left[  x\right]  \right)  _{\rho\in\mathbb{N}}$ is an ideal
semifiltration of $A\left[  x\right]  $.

\begin{vershort}
Similarly, the other two statements of Theorem~\ref{Theorem23} \textbf{(a)}
are proven.
\end{vershort}

\begin{verlong}
Since $\left(  I_{\rho}\right)  _{\rho\in\mathbb{N}}$ is an ideal
semifiltration of $A$, Lemma \ref{lem.J} (applied to $A^{\prime}=A\left[
y\right]  $) yields that $\left(  I_{\rho}A\left[  y\right]  \right)
_{\rho\in\mathbb{N}}$ is an ideal semifiltration of $A\left[  y\right]  $.

Since $\left(  I_{\rho}\right)  _{\rho\in\mathbb{N}}$ is an ideal
semifiltration of $A$, Lemma \ref{lem.J} (applied to $A^{\prime}=A\left[
xy\right]  $) yields that $\left(  I_{\rho}A\left[  xy\right]  \right)
_{\rho\in\mathbb{N}}$ is an ideal semifiltration of $A\left[  xy\right]  $.
\end{verlong}

Thus, Theorem~\ref{Theorem23} \textbf{(a)} is proven.

\textbf{(b)} For every $v\in B$, the family $\left(  I_{\rho}A\left[
v\right]  \right)  _{\rho\in\mathbb{N}}$ is an ideal semifiltration of
$A\left[  v\right]  $ (by Lemma~\ref{lem.J}, applied to $A^{\prime}=A\left[
v\right]  $), and thus we can consider the polynomial ring $\left(  A\left[
v\right]  \right)  \left[  Y\right]  $ and its $A\left[  v\right]
$-subalgebra $\left(  A\left[  v\right]  \right)  \left[  \left(  I_{\rho
}A\left[  v\right]  \right)  _{\rho\in\mathbb{N}}\ast Y\right]  $. For every
$v\in B$, the polynomial ring $B\left[  Y\right]  $ is an $\left(  A\left[
v\right]  \right)  \left[  Y\right]  $-algebra (as explained in
Definition~\ref{Definition7}), since $B$ is an $A\left[  v\right]
$-algebra\footnote{because $A\left[  v\right]  $ is a subring of $B$}. Hence,
this ring $B\left[  Y\right]  $ is an $\left(  A\left[  v\right]  \right)
\left[  \left(  I_{\rho}A\left[  v\right]  \right)  _{\rho\in\mathbb{N}}\ast
Y\right]  $-algebra as well (because $\left(  A\left[  v\right]  \right)
\left[  \left(  I_{\rho}A\left[  v\right]  \right)  _{\rho\in\mathbb{N}}\ast
Y\right]  $ is a subring of $\left(  A\left[  v\right]  \right)  \left[
Y\right]  $). Similarly, the ring $B\left[  Y\right]  $ is an $A\left[
\left(  I_{\rho}\right)  _{\rho\in\mathbb{N}}\ast Y\right]  $-algebra.

Lemma \ref{lem.N} \textbf{(b)} (applied to $v=x$) yields that the element $u$
of $B$ is $n$-integral over $\left(  A\left[  x\right]  ,\left(  I_{\rho
}A\left[  x\right]  \right)  _{\rho\in\mathbb{N}}\right)  $ if and only if the
element $uY$ of the polynomial ring $B\left[  Y\right]  $ is $n$-integral over
the ring $\left(  A\left[  \left(  I_{\rho}\right)  _{\rho\in\mathbb{N}}\ast
Y\right]  \right)  \left[  x\right]  $. But since the element $u$ of $B$ is
$n$-integral over $\left(  A\left[  x\right]  ,\left(  I_{\rho}A\left[
x\right]  \right)  _{\rho\in\mathbb{N}}\right)  $, this yields that the
element $uY$ of the polynomial ring $B\left[  Y\right]  $ is $n$-integral over
the ring $\left(  A\left[  \left(  I_{\rho}\right)  _{\rho\in\mathbb{N}}\ast
Y\right]  \right)  \left[  x\right]  $.

Lemma \ref{lem.N} \textbf{(b)} (applied to $y$ and $m$ instead of $v$ and $n$)
yields that the element $u$ of $B$ is $m$-integral over $\left(  A\left[
y\right]  ,\left(  I_{\rho}A\left[  y\right]  \right)  _{\rho\in\mathbb{N}%
}\right)  $ if and only if the element $uY$ of the polynomial ring $B\left[
Y\right]  $ is $m$-integral over the ring $\left(  A\left[  \left(  I_{\rho
}\right)  _{\rho\in\mathbb{N}}\ast Y\right]  \right)  \left[  y\right]  $. But
since the element $u$ of $B$ is $m$-integral over $\left(  A\left[  y\right]
,\left(  I_{\rho}A\left[  y\right]  \right)  _{\rho\in\mathbb{N}}\right)  $,
this yields that the element $uY$ of the polynomial ring $B\left[  Y\right]  $
is $m$-integral over the ring $\left(  A\left[  \left(  I_{\rho}\right)
_{\rho\in\mathbb{N}}\ast Y\right]  \right)  \left[  y\right]  $.

Thus we know that $uY$ is $n$-integral over the ring $\left(  A\left[  \left(
I_{\rho}\right)  _{\rho\in\mathbb{N}}\ast Y\right]  \right)  \left[  x\right]
$, and that $uY$ is $m$-integral over the ring $\left(  A\left[  \left(
I_{\rho}\right)  _{\rho\in\mathbb{N}}\ast Y\right]  \right)  \left[  y\right]
$. Hence, Theorem~\ref{Theorem22} (applied to $A\left[  \left(  I_{\rho
}\right)  _{\rho\in\mathbb{N}}\ast Y\right]  $, $B\left[  Y\right]  $ and $uY$
instead of $A$, $B$ and $u$) yields that there exists some $\lambda
\in\mathbb{N}$ such that $uY$ is $\lambda$-integral over $\left(  A\left[
\left(  I_{\rho}\right)  _{\rho\in\mathbb{N}}\ast Y\right]  \right)  \left[
xy\right]  $. Consider this $\lambda$.

Lemma \ref{lem.N} \textbf{(b)} (applied to $xy$ and $\lambda$ instead of $v$
and $n$) yields that the element $u$ of $B$ is $\lambda$-integral over
$\left(  A\left[  xy\right]  ,\left(  I_{\rho}A\left[  xy\right]  \right)
_{\rho\in\mathbb{N}}\right)  $ if and only if the element $uY$ of the
polynomial ring $B\left[  Y\right]  $ is $\lambda$-integral over the ring
$\left(  A\left[  \left(  I_{\rho}\right)  _{\rho\in\mathbb{N}}\ast Y\right]
\right)  \left[  xy\right]  $. But since the element $uY$ of the polynomial
ring $B\left[  Y\right]  $ is $\lambda$-integral over the ring $\left(
A\left[  \left(  I_{\rho}\right)  _{\rho\in\mathbb{N}}\ast Y\right]  \right)
\left[  xy\right]  $, this yields that the element $u$ of $B$ is $\lambda
$-integral over $\left(  A\left[  xy\right]  ,\left(  I_{\rho}A\left[
xy\right]  \right)  _{\rho\in\mathbb{N}}\right)  $. Thus,
Theorem~\ref{Theorem23} \textbf{(b)} is proven.
\end{proof}

\subsection{Second proof of Corollary~\ref{Corollary3}}

We notice that Corollary~\ref{Corollary3} can be derived from
Lemma~\ref{Lemma18}:

\begin{vershort}
\begin{proof}
[Second proof of Corollary~\ref{Corollary3}.]Let $n=1$. Let $m=1$. From $n=1$,
we obtain $\left\langle u^{0},u^{1},\ldots,u^{n-1}\right\rangle _{A}%
=\left\langle u^{0}\right\rangle _{A}=\left\langle 1_{B}\right\rangle _{A}$
(since $u^{0}=1_{B}$). Hence,%
\begin{align}
\underbrace{\left\langle u^{0},u^{1},\ldots,u^{n-1}\right\rangle _{A}%
}_{=\left\langle 1_{B}\right\rangle _{A}}\cdot\left\langle v^{0},v^{1}%
,\ldots,v^{\alpha}\right\rangle _{A} &  =\left\langle 1_{B}\right\rangle
\cdot\left\langle v^{0},v^{1},\ldots,v^{\alpha}\right\rangle _{A}\nonumber\\
&  =\left\langle 1_{B}v^{0},1_{B}v^{1},\ldots,1_{B}v^{\alpha}\right\rangle
_{A}\nonumber\\
&  =\left\langle v^{0},v^{1},\ldots,v^{\alpha}\right\rangle _{A}%
.\label{pf.Corollary3.short.1}%
\end{align}
The same argument (applied to $m$ and $\beta$ instead of $n$ and $\alpha$)
yields%
\begin{equation}
\left\langle u^{0},u^{1},\ldots,u^{m-1}\right\rangle _{A}\cdot\left\langle
v^{0},v^{1},\ldots,v^{\beta}\right\rangle _{A}=\left\langle v^{0},v^{1}%
,\ldots,v^{\beta}\right\rangle _{A}.\label{pf.Corollary3.short.2}%
\end{equation}

Now, we have%
\[
u^{n}\in\left\langle u^{0},u^{1},\ldots,u^{n-1}\right\rangle _{A}%
\cdot\left\langle v^{0},v^{1},\ldots,v^{\alpha}\right\rangle _{A}%
\]
\footnote{\textit{Proof.} From $n=1$, we obtain%
\begin{align*}
u^{n}  & =u^{1}=u=\sum\limits_{i=0}^{\alpha}\underbrace{s_{i}}_{\in A}v^{i}%
\in\left\langle v^{0},v^{1},\ldots,v^{\alpha}\right\rangle _{A}=\left\langle
u^{0},u^{1},\ldots,u^{n-1}\right\rangle _{A}\cdot\left\langle v^{0}%
,v^{1},\ldots,v^{\alpha}\right\rangle _{A}\\
& =\left\langle u^{0},u^{1},\ldots,u^{n-1}\right\rangle _{A}\cdot\left\langle
v^{0},v^{1},\ldots,v^{\alpha}\right\rangle _{A}\ \ \ \ \ \ \ \ \ \ \left(
\text{by (\ref{pf.Corollary3.short.1})}\right)  .
\end{align*}
} and%
\begin{align*}
u^{m}v^{\beta} &  \in\left\langle u^{0},u^{1},\ldots,u^{m-1}\right\rangle
_{A}\cdot\left\langle v^{0},v^{1},\ldots,v^{\beta}\right\rangle _{A}\\
&  \ \ \ \ \ \ \ \ \ \ +\left\langle u^{0},u^{1},\ldots,u^{m}\right\rangle
_{A}\cdot\left\langle v^{0},v^{1},\ldots,v^{\beta-1}\right\rangle _{A}%
\end{align*}
\footnote{\textit{Proof.} From $m=1$, we obtain $u^{m}=u^{1}=u$ and thus%
\begin{align*}
\underbrace{u^{m}}_{=u}v^{\beta} &  =uv^{\beta}=\sum\limits_{i=0}^{\beta}%
t_{i}v^{\beta-i}=\sum\limits_{i=0}^{\beta}t_{\beta-i}\underbrace{v^{\beta
-\left(  \beta-i\right)  }}_{=v^{i}}\\
&  \ \ \ \ \ \ \ \ \ \ \left(  \text{here we substituted }\beta-i\text{ for
}i\text{ in the sum}\right)  \\
&  =\sum\limits_{i=0}^{\beta}\underbrace{t_{\beta-i}}_{\in A}v^{i}%
\in\left\langle v^{0},v^{1},\ldots,v^{\beta}\right\rangle _{A}=\left\langle
u^{0},u^{1},\ldots,u^{m-1}\right\rangle _{A}\cdot\left\langle v^{0}%
,v^{1},\ldots,v^{\beta}\right\rangle _{A}\ \ \ \ \ \ \ \ \ \ \left(  \text{by
(\ref{pf.Corollary3.short.2})}\right)  .
\end{align*}
}. Thus, Lemma~\ref{Lemma18} (applied to $v$, $\beta$ and $\alpha$ instead of
$x$, $\mu$ and $\nu$) yields that $u$ is $\left(  n\beta+m\alpha\right)
$-integral over $A$ (since $\beta+\alpha=\alpha+\beta\in\mathbb{N}^{+}$). This
means that $u$ is $\left(  \alpha+\beta\right)  $-integral over $A$ (because
$\underbrace{n}_{=1}\beta+\underbrace{m}_{=1}\alpha=\beta+\alpha=\alpha+\beta
$). This proves Corollary~\ref{Corollary3} once again.
\end{proof}
\end{vershort}

\begin{verlong}
\begin{proof}
[Second proof of Corollary~\ref{Corollary3}.]Let $n=1$. Let $m=1$. From $n=1$,
we obtain $\left\langle u^{0},u^{1},\ldots,u^{n-1}\right\rangle _{A}%
=\left\langle u^{0},u^{1},\ldots,u^{0}\right\rangle _{A}=\left\langle
u^{0}\right\rangle _{A}=\left\langle 1_{B}\right\rangle _{A}$ (since
$u^{0}=1_{B}$). Similarly, from $m=1$, we obtain $\left\langle u^{0}%
,u^{1},\ldots,u^{m-1}\right\rangle _{A}=\left\langle 1_{B}\right\rangle _{A}$.

Now, we have%
\[
u^{n}\in\left\langle u^{0},u^{1},\ldots,u^{n-1}\right\rangle _{A}%
\cdot\left\langle v^{0},v^{1},\ldots,v^{\alpha}\right\rangle _{A}%
\]
\footnote{\textit{Proof.} From $n=1$, we obtain%
\[
u^{n}=u^{1}=u=\sum\limits_{i=0}^{\alpha}\underbrace{s_{i}}_{\in A}v^{i}%
\in\left\langle v^{0},v^{1},\ldots,v^{\alpha}\right\rangle _{A}=\left\langle
u^{0},u^{1},\ldots,u^{n-1}\right\rangle _{A}\cdot\left\langle v^{0}%
,v^{1},\ldots,v^{\alpha}\right\rangle _{A},
\]
since
\begin{align*}
\underbrace{\left\langle u^{0},u^{1},\ldots,u^{n-1}\right\rangle _{A}%
}_{=\left\langle 1_{B}\right\rangle _{A}}\cdot\left\langle v^{0},v^{1}%
,\ldots,v^{\alpha}\right\rangle _{A}  &  =\left\langle 1_{B}\right\rangle
\cdot\left\langle v^{0},v^{1},\ldots,v^{\alpha}\right\rangle _{A}=\left\langle
1_{B}v^{0},1_{B}v^{1},\ldots,1_{B}v^{\alpha}\right\rangle _{A}\\
&  =\left\langle v^{0},v^{1},\ldots,v^{\alpha}\right\rangle _{A}.
\end{align*}
} and%
\begin{align*}
u^{m}v^{\beta}  &  \in\left\langle u^{0},u^{1},\ldots,u^{m-1}\right\rangle
_{A}\cdot\left\langle v^{0},v^{1},\ldots,v^{\beta}\right\rangle _{A}\\
&  \ \ \ \ \ \ \ \ \ \ +\left\langle u^{0},u^{1},\ldots,u^{m}\right\rangle
_{A}\cdot\left\langle v^{0},v^{1},\ldots,v^{\beta-1}\right\rangle _{A}%
\end{align*}
\footnote{\textit{Proof.} We have%
\begin{align}
\underbrace{\left\langle u^{0},u^{1},\ldots,u^{m-1}\right\rangle _{A}%
}_{=\left\langle 1_{B}\right\rangle _{A}}\cdot\left\langle v^{0},v^{1}%
,\ldots,v^{\beta}\right\rangle _{A}  &  =\left\langle 1_{B}\right\rangle
_{A}\cdot\left\langle v^{0},v^{1},\ldots,v^{\beta}\right\rangle _{A}%
=\left\langle 1_{B}v^{0},1_{B}v^{1},\ldots,1_{B}v^{\beta}\right\rangle
_{A}\nonumber\\
&  =\left\langle v^{0},v^{1},\ldots,v^{\beta}\right\rangle _{A}.
\label{pf.Corollary3.2nd.fn2.1}%
\end{align}
From $m=1$, we obtain $u^{m}=u^{1}=u$ and thus%
\begin{align*}
\underbrace{u^{m}}_{=u}v^{\beta}  &  =uv^{\beta}=\sum\limits_{i=0}^{\beta
}t_{i}v^{\beta-i}=\sum\limits_{i=0}^{\beta}t_{\beta-i}\underbrace{v^{\beta
-\left(  \beta-i\right)  }}_{=v^{i}}\\
&  \ \ \ \ \ \ \ \ \ \ \left(  \text{here we substituted }\beta-i\text{ for
}i\text{ in the sum}\right) \\
&  =\sum\limits_{i=0}^{\beta}\underbrace{t_{\beta-i}}_{\in A}v^{i}%
\in\left\langle v^{0},v^{1},\ldots,v^{\beta}\right\rangle _{A}\\
&  =\left\langle u^{0},u^{1},\ldots,u^{m-1}\right\rangle _{A}\cdot\left\langle
v^{0},v^{1},\ldots,v^{\beta}\right\rangle _{A}\ \ \ \ \ \ \ \ \ \ \left(
\text{by (\ref{pf.Corollary3.2nd.fn2.1})}\right) \\
&  \subseteq\left\langle u^{0},u^{1},\ldots,u^{m-1}\right\rangle _{A}%
\cdot\left\langle v^{0},v^{1},\ldots,v^{\beta}\right\rangle _{A}+\left\langle
u^{0},u^{1},\ldots,u^{m}\right\rangle _{A}\cdot\left\langle v^{0},v^{1}%
,\ldots,v^{\beta-1}\right\rangle _{A}.
\end{align*}
}. Thus, Lemma~\ref{Lemma18} (applied to $v$, $\beta$ and $\alpha$ instead of
$x$, $\mu$ and $\nu$) yields that $u$ is $\left(  n\beta+m\alpha\right)
$-integral over $A$ (since $\beta+\alpha=\alpha+\beta\in\mathbb{N}^{+}$). This
means that $u$ is $\left(  \alpha+\beta\right)  $-integral over $A$ (because
$\underbrace{n}_{=1}\beta+\underbrace{m}_{=1}\alpha=1\beta+1\alpha
=\beta+\alpha=\alpha+\beta$). This proves Corollary~\ref{Corollary3} once again.
\end{proof}
\end{verlong}

\end{document}